\documentclass[10pt]{article}

\newif\ifdraft\draftfalse

\usepackage{amsthm}
\usepackage{amsmath}
\usepackage{amsxtra}
\usepackage{amssymb}
\usepackage{enumitem}
\usepackage{mathrsfs}
\usepackage{turnstile}
\usepackage{thmtools}
\usepackage{url}
\usepackage{authblk}
\usepackage{IEEEtrantools}
\usepackage[colorlinks=true,linkcolor=blue]{hyperref}
\usepackage{latexsym}
\usepackage{verbatim}
\usepackage{changepage}
\usepackage{wasysym}

\declaretheoremstyle[bodyfont=\sl]{slanted}

\declaretheorem[name=Definition,style=definition,qed=$\dashv$,numberwithin=section]{definition}
\declaretheorem[name=Definition,style=definition,numbered=no,qed=$\dashv$]{definition*}
\declaretheorem[name=Definition,style=definition,qed=$\dashv$,sibling=definition]{dfn}
\declaretheorem[name=Definition,style=definition,numbered=no,qed=$\dashv$]{dfn*}

\declaretheorem[name=Theorem,style=slanted,sibling=dfn]{tm}
\declaretheorem[name=Theorem,style=slanted,sibling=dfn]{tm*}
\declaretheorem[name=Theorem,style=slanted,sibling=dfn]{theorem}
\declaretheorem[name=Theorem,style=slanted,numbered=no]{theorem*}
\declaretheorem[name=Lemma,style=slanted,sibling=dfn]{lemma}
\declaretheorem[name=Lemma,style=slanted,sibling=dfn]{lem}

\declaretheorem[name=Corollary,style=slanted,sibling=dfn]{corollary}
\declaretheorem[name=Corollary,style=slanted,numbered=no]{corollary*}
\declaretheorem[name=Corollary,style=slanted,sibling=dfn]{cor}
\declaretheorem[name=Corollary,style=slanted,numbered=no]{cor*}
\declaretheorem[name=Remark,style=definition,sibling=dfn]{rem}

\declaretheorem[name=Fact,style=slanted,sibling=dfn]{fact}
\declaretheoremstyle[headfont=\scshape]{claimstyle}
\declaretheorem[name=Claim,style=claimstyle]{clm}

\declaretheorem[name=Claim,style=claimstyle]{clmtwo}
\declaretheorem[name=Claim,style=claimstyle]{clmthree}
\declaretheorem[name=Claim,style=claimstyle]{clmfour}
\declaretheorem[name=Claim,style=claimstyle]{clmfive}

\declaretheorem[name=Claim,style=claimstyle,numbered=no]{clm*}

\declaretheorem[name=Subclaim,style=claimstyle,numbered=no]{sclm*}

\declaretheorem[name=Subsubclaim,style=claimstyle,numbered=no]{ssclm*}

\newcommand{\inflatearrow}{\rightsquigarrow}

\newcommand{\ph}{\mathfrak{P}}

\newcommand{\sw}{\mathrm{sw}}
  \newcommand{\swsw}{\mathrm{swsw}}
  \newcommand{\stk}{\mathrm{stk}}
  \newcommand{\Ttvec}{{\vec{\Tt}}}
  \newcommand{\Uuvec}{{\vec{\Uu}}}
    \newcommand{\Vvvec}{{\vec{\Vv}}}
  \newcommand{\msP}{\mathscr{P}}
  \newcommand{\msF}{\mathscr{F}}
  \newcommand{\PP}{\mathbb{P}}

\newcommand{\lgcd}{\mathrm{lgcd}}
\newcommand{\wW}{\mathscr{W}}

\newcommand{\rest}{{\upharpoonright}}

\newcommand{\Msw}{{M_{\rm sw}}}

\newcommand{\ext}{\mathrm{ext}}

\newcommand{\forces}{\Vdash}

\renewcommand{\models}{\vDash}

\newcommand{\dom}{{\rm dom}}

\newcommand{\lh}{{\rm lh}}

\newcommand{\crit}{{\rm crit }}

\newcommand{\M}{\mathcal{M}}
\newcommand{\N}{\mathcal{N}}
\newcommand{\U}{\mathcal{U}}
\newcommand{\T}{\mathcal{T}}
\newcommand{\J}{\mathcal{J}}

\def\cof{\mathop{\rm cof}\nolimits}

\def\and{\mathrel{\kern1pt\&\kern1pt}}

\def\<#1>{\langle\,#1\,\rangle}

\newcommand{\ueq}{\ \widehat{=}\ }
\newcommand{\vV}{\mathscr{V}}
\newcommand{\HOD}{\mathrm{HOD}}
\newcommand{\sub}{\subseteq}
\newcommand{\om}{\omega}
\newcommand{\Coll}{\mathrm{Coll}}
\newcommand{\Mswsw}{M_{\mathrm{swsw}}}
\newcommand{\OR}{\mathrm{OR}}
\newcommand{\es}{\mathbb{E}}
\newcommand{\sss}{\mathrm{sh}}

\newcommand{\dsr}{\mathrm{dsr}}

\newcommand{\Ll}{\mathcal{L}}

\newcommand{\Rr}{\mathcal{R}}
\newcommand{\Ss}{\mathcal{S}}
\newcommand{\Tt}{\mathcal{T}}
\newcommand{\Uu}{\mathcal{U}}
\newcommand{\Vv}{\mathcal{V}}
\newcommand{\Ww}{\mathcal{W}}
\newcommand{\Xx}{\mathcal{X}}
\newcommand{\Yy}{\mathcal{Y}}
\newcommand{\Zz}{\mathcal{Z}}
\newcommand{\conc}{\ \widehat{\ }\ }
\newcommand{\Ult}{\mathrm{Ult}}
\newcommand{\core}{\mathfrak{C}}
\newcommand{\pins}{\triangleleft}
\newcommand{\sats}{\models}
\newcommand{\id}{\mathrm{id}}
\newcommand{\ins}{\trianglelefteq}
\newcommand{\dirlim}{\mathrm{dirlim}}
\newcommand{\Hull}{\mathrm{Hull}}
\newcommand{\inter}{\cap}
\newcommand{\com}{\circ}

\newcommand{\BB}{\mathbb{B}}
\newcommand{\cHull}{\mathrm{cHull}}

\newcommand{\rg}{\mathrm{rg}}
\newcommand{\her}{\mathcal{H}}
\newcommand{\lpole}{\left\lfloor}
\newcommand{\rpole}{\right\rfloor}
\newcommand{\univ}[1]{\lpole #1\rpole}
\newcommand{\vareps}{\varepsilon}
\newcommand{\dropset}{\mathscr{D}}
\newcommand{\pred}{\mathrm{pred}}

\newcommand{\tu}{\textup}
\newcommand{\stack}{\mathrm{stack}}

\newcommand{\eqdef}{=_{\mathrm{def}}}

\newcommand{\passive}{{\mathrm{pv}}}
\newcommand{\pow}{\mathcal{P}}

\newcommand{\sn}{\mathrm{sn}}

\newcommand{\jbar}{\bar{j}}

\newcommand{\cut}{\backslash}
\newcommand{\rSigma}{\mathrm{r}\Sigma}

\newcommand{\ZFC}{\mathrm{ZFC}}
\newcommand{\elem}{\preceq}
\renewcommand{\P}{\mathcal{P}}
\newcommand{\steelurl}{https://math.berkeley.edu/~steel}
\title{Varsovian models II}

\author[$\dagger$]{Grigor Sargsyan}
\author[$\ddag$]{Ralf Schindler}
\author[$\star$]{Farmer Schlutzenberg}

\affil[$\dagger$]{Institute of Mathematics,	
	Polish Academy of Sciences\\ gsargsyan@impan.pl}
\affil[$\ddag$]{WWU M\"unster, Institut f\"ur Mathematische Logik und Grundlagenforschung, Einsteinstra{\ss}e 62,
	48149 M\"unster, Germany\\ rds@wwu.de}
	\affil[$\star$]{Institut f\"ur Diskrete Mathematik und Geometrie\\
	TU Wien\footnote{Schlutzenberg email: afirstname dot alastname at tuwien dot ac dot at}}
         
\date{\today}

\newcounter{tempcounter}

\begin{document}

\maketitle

\begin{abstract}Assume sufficient large cardinals.
	Let $M_{\sw n}$ be the minimal  iterable proper class $L[E]$ model satisfying ``there are  $\delta_0<\kappa_0<\ldots<\delta_{n-1}<\kappa_{n-1}$ such that the $\delta_i$ are Woodin cardinals and the $\kappa_i$ are strong cardinals''. Let $M=M_{\sw2}$.
	We identify an inner model $\vV_2^M$ of $M$, which is a proper class model satisfying ``there are 2 Woodin cardinals'', and is iterable both in $V$ and in $M$,
	and closed under its own iteration strategy. The construction also yields significant information about the extent to which $M$ knows its own iteration strategy. We characterize the universe of $\vV_2^M$ as the mantle and the least ground of $M$, and as $\HOD^{M[G]}$ for $G\sub\Coll(\om,\lambda)$ being $M$-generic with $\lambda$ sufficiently large.
	These results correspond to facts already known for
	$M_{\sw1}$, and the proofs are an elaboration of those, but there are substantial new issues  and new methods with which to handle them.
	\end{abstract}

\tableofcontents

\section{Introduction}

The first generation of canonical inner models for large cardinals are those of the form $M=L[E]$ (or $L_\alpha[E]$) where $E$ is a sequence
of (partial) measures or extenders with various nice properties. The second generation are those of the form $M=L[E,\Sigma]$ (or $L_\alpha[E,\Sigma]$), with $E$ as before, but $\Sigma$ is a (partial)
iteration strategy for $M$. We refer to the former as \emph{mice} or \emph{extender models}, and the latter as \emph{strategy mice} or \emph{strategic extender models}. Strategy mice arise naturally
as  HODs of determinacy models, and this phenomenon has been extensively studied. (The universe of) a strategy mouse $\vV^{\Msw}$
was also found in \cite{vm1} to be the mantle of
and a certain HOD associated to the mouse $M_{\sw}=M_{\sw1}$
(the ``minimal'' proper class mouse with a strong cardinal above a Woodin cardinal).
While mice with Woodin cardinals (and which model ZFC, for example) can
only compute restricted fragments of their own iteration strategies, strategy mice can be fully self-iterable.

One can contemplate the relationship
between the two hierarchies; a key issue is the consistency strength
of large cardinals when exhibited in the respective models:
how do large cardinal hypotheses in (fully iterable) mice compare in consistency strength to those in (fully iterable) strategy mice, particularly for strategy mice which  are closed under their own strategy?
Continuing the line of investigation of \cite{vm1},
the present paper derives\footnote{Disclaimer: The ``proofs'' (and some definitions) presented here are not quite complete, because their full exposition  depends on an integration, omitted here, of the method of $*$-translation (see \cite{closson}) with the techniques we develop. The integration itself is a straightforward matter of combining the two things. But because $*$-translation itself is already quite detailed, its inclusion would have added significantly to the length
	of the paper. It will be covered instead in \cite{*-trans_add}.} the existence of a fully iterable proper class strategy mouse $\vV=L[E,\Sigma]$, closed under its strategy,
and containing two Woodin cardinals, from the existence and full iterability of the mouse  $M_{\swsw}^\#$. This is the least active mouse
$N$ such that letting $\kappa=\crit(F)$ where $F$ is the active extender of $N$, then $N|\kappa\models$``There are ordinals $\delta_0<\kappa_0<\delta_1<\kappa_1$ such that
each $\delta_i$ is a Woodin cardinal and each $\kappa_i$ is a strong cardinal''. Letting $M_{\swsw}$ be the proper class model left behind after iterating $F$ out of the universe,
the strategy mouse
$\vV$ will be an inner model of $M_{\swsw}$.
(We also obtain Silver indiscernibles for $\vV$.)
The analysis also shows that $M_{\swsw}$ computes
substantial fragments of its own iteration strategy,
thereby contributing to the investigation of self-iterability
in mice as in \cite{sile}, but here beyond the tame level.

Now recall that if $W$ is a model of {\sf ZFC}, then
$P \subseteq W$ is a {\em ground of} $W$ iff $P$ is also a model of {\sf ZFC} and there
is some poset ${\mathbb P} \in P$ and some $g$ which is $(P,{\mathbb P})$-generic
with $W=P[g]$. (Note this implies that $P$ is transitive in the sense of $W$ and contains all of the ordinals of $W$; by the Woodin/Laver ground definability result  \cite{stgeol}, \cite{laver_vlc}, $P$ is also definable from parameters over $W$.) The intersection of all grounds of $W$ is called the
{\em mantle} $\mathbb{M}^W$ of $W$. Recall $W$ is called a {\em bedrock} iff $W$  has no non-trivial grounds, or equivalently, $W=\mathbb{M}^W$.  See \cite{stgeol} and
\cite{usuba_ddg} for more general background on these topics not specific to inner model theory.

The reason that mice modelling ZFC +  Woodin cardinals do not compute their own iteration strategies is connected with the fact that
they have proper (set-)grounds. The standard examples
of such grounds arise from Woodin's genericity iterations.
This phenomenon has led to \emph{inner model theoretic geology}, which has proven to be an exciting and
fruitful area of set theory. Its program is to analyze the collection of grounds and the mantle of  given canonical inner models. See \cite{vm0}
and \cite{vm1}, which address exactly this kind of problem, and are precursors to the current work. See also  \cite{local_mantles_of_Lx_v2}, parts  of which were
motivated by the current work. The theme uncovered in these works is roughly that the mantle of a (sufficiently canonical) mouse tends to itself be a mouse or a strategy mouse, and hence can be analyzed in high detail.

The paper \cite{vm0} proves that if $M=L[E]$ is a tame proper class mouse
with a Woodin cardinal but
no strong cardinal, and some further  technical assumptions
hold, then the mantle of $M$ is itself a mouse, but  is not a ground of $L[E]$;
see \cite[\S3.4]{vm0} and specifically \cite[Theorem 3.33]{vm0}.
As an example, the mantle of $M_1$ (the minimal proper class mouse with one Woodin cardinal) is the model left behind after iterating
the unique measure on the least measurable of $M_1$ out of the universe, and note this model has no measurable cardinals.
The situation is entirely different if $L[E]$ has a strong cardinal.

Let $\Msw^\#$ denote the minimal active mouse $N$
such that letting $\kappa=\crit(F)$ where $F$ is the
active extender of $N$, then $N|\kappa\models$``there is a strong cardinal above a Woodin cardinal'', and suppose
this mouse is fully iterable (for all set-sized trees).
Let $\Mswsw$ be the proper class mouse left by iterating
$F$ out of the universe.
 It is shown in \cite{vm1} that
there are only set many grounds of $\Msw$ and that
the mantle of $\Msw$ is
itself a
ground of $\Msw$ and hence a bedrock. There is therefore some analogy here between
$\Msw$ and $V$ in the presence of an extendible cardinal; see \cite[Theorem 1.3]{usuba_extendible}. The mantle of $\Msw$, however, also
has an interesting structural analysis,
as it is the universe of the strategy mouse $\vV^{\Msw}$
mentioned earlier. It is, moreover, a canonical ``least'' inner model which has a Woodin cardinal and knows how to fully iterate itself; see \cite[Lemma 2.20]{vm1}.

In personal communication with the second author  \cite{email_woodin_to_schindler},
W.~Hugh Woodin expressed suspicion that the mantle of any proper class mouse $L[E]$
with a strong cardinal above a Woodin cardinal might perhaps
contain non-trivial strategy information
at its least Woodin cardinal and not at any larger Woodin.

A reasonable candidate for testing this suspicion and for extending
the analysis of \cite{vm1} is
the  big brother of $\Msw$,  namely $M=\Mswsw$, introduced above,
and studied in this paper. We will show  that
the strategy mouse $\vV^{M}$, also introduced above,
has universe the mantle of $M$, and so in fact,
this mantle contains two Woodins together with non-trivial (and is closed under) strategy information for both of them.
There is therefore a stronger analogy between
hod mice  (see \cite{hod_mice}) and
mantles of extender models $L[E]$
than was previously
expected. This universe is also a ground
of $M$, and hence is a bedrock.
We will also show that $\vV^{M}$
has universe the eventual generic HOD of $M$; that is, its universe is $\HOD^{M[G]}$ whenever $\lambda$ is a sufficiently large ordinal and
$G\sub\Coll(\om,\lambda)$ is $M$-generic.

In some more detail,
we will first isolate the  {\em first Varsovian model} $\vV_1={\vV}_1^{M}$ of $M$
and show that
${\vV}_1$ is a ground of $M$,  contains exactly two Woodin cardinals
and a strong above them, and knows how to iterate itself fully for trees based on its least Woodin. This model is at first constructed in the form of ``$L[\M_\infty,*]$'', very much like in the construction of \cite{vm1}, which also mirrors Woodin's analysis of $\HOD^{L[x,G]}$.
We then show that this model admits a stratification as a fine structural strategy premouse.
The  indexing used for the stratification is new,
and this indexing is important in the overall analysis we give. It is moreover
determined in a very strong sense by the hierarchy of $M$ -- the extender sequence of $\vV_1$ is in fact given by simply restricting the extenders on the sequence of $M$ above a certain point, some of which correspond to strategy.
 We  then go on to isolate the \emph{second Varsovian model} $\vV_2=\vV_2^M$ of $M$,
which will be constructed inside $\vV_1$ (so $\vV_2\subseteq\vV_1\subseteq M$), using an elaboration of
the construction of $\vV_1$ in $M$.
We then analyze the model and compute an iteration strategy for it,
and establish the remaining facts mentioned above:
 the universe of $\vV_2$ is the mantle and eventual generic HOD of $M$,  $\vV_2$
 contains exactly two Woodin cardinals and knows fully how to iterate itself.
  We also show that the universe of $\vV_1$ is the $\kappa_0$-mantle of $M$,
where $\kappa_0$ is the least strong of $M$.
The overall picture and process is expected to generalize to $n<\omega$ iterations (working in  the appropriate
starting mouse)  and beyond.

The reader who is familiar with  \cite{hod_mice}, for example, will encounter a lot of
parallels between our analysis and the theory of hod mice; a key difference,
though, is
that our treatment is purely combinatorial and ``inner model theoretic'',
using no descriptive set theory.
Familiarity with \cite{vm1} certainly helps,
since the current paper is in large part an extension of that one,
and some arguments covered in \cite{vm1} are omitted here.
But the reader who is reasonably familiar with inner model theory
in general should be able to refer to \cite{vm1} as needed.

The paper is organized as follows. There are some preliminaries and notation
listed at the end of this section.
In \S\ref{sec:ground_generation}, we present the general method of assigning the (first)
Varsovian model ${\mathscr V}^{L[E]}$ to an extender model $L[E]$, and
prove key facts about it, under certain hypotheses. In \S\ref{sec:Mswsw}, we describe
some key properties of $\Mswsw$ and its iteration strategy, which
will be essential throughout. In \S\ref{sec:vV_1}, the first
Varsovian model $\vV_1$ of $M$, and its iteration strategy $\Sigma_{\vV_1}$, are defined and analyzed. This analysis
is centered
around the stratification of $\vV_1$ as a strategy premouse. We also
give  natural characterizations of the universe of $\vV_1$. In \S\ref{sec:vV_2},
we identify $\vV_2$,
also stratifying it as a strategy premouse.
We  show that $\vV_2$ has two Woodin cardinals,
is fully iterable, and is closed under its iteration strategy.
We expect that the hypothesis used to construct such a model
is in some sense optimal.
We finally show in \S\ref{subsec:vV_2_is_mantle} that  the universe of $\vV_2$ is the mantle of $M$ and is $\HOD^{M[G]}$ for sufficiently large collapse generics $G$.

The work presented here was started by the first two authors, extending their \cite{vm1}.
In the early stages, significant progress was made,
but without a full development of the level-by-level fine structural correspondence
presented in this paper
between the models $M=\Mswsw$, $\vV_1$ and $\vV_2$;
such a correspondence was considered to some extent, but then put aside in favour of other methods.
During this time, the first author developed an approach
to computing the mantle of $M$
which does not use the level-by-level correspondence,
but this has not been published.
Later, the second author returned
to the level-by-level correspondence,
and developed some of the main ideas in its connection.
Following this, in September 2017, the second and third authors then
began discussing this approach.
Over the next few months, building on what had already been established,
they (mostly)
completed the analysis via this approach,
leading to the current presentation (some details being added over time somewhat later). Some of the evolution
of ideas was documented by the second author's talks
at the 4th M\"unster conference on
inner model theory, July 17--Aug 01, 2017, and at the
1st Girona conference on
inner model theory,
July 16--27, 2018, and in the handwritten notes \cite{ralf_notes_vm2}.

The early development, worked out by the first two authors,
directly yielded parts of the present paper, as well as
precursors to some other parts.
Some version of probably the most central concept in the paper,
the strategy mouse hierarchy used in
Definition \ref{dfn:vV_1} (which is also a precursor of Definition \ref{dfn:vV_2}), is due to the first two
authors, as is \S\ref{subsec:Q}; the setup for the first direct limit system  in
\S\S\ref{subsec:models_of_first_system},\ref{first_dir_limit_system},\ref{sec:vV_1_as_M_infty[*]}
is much as in
\cite{vm1} and is basically due to them,
although the approach used in
\S\ref{section-short-tree-strategy-for-M} for computing short tree
strategy, and some other uses of normalization, are due to the
3rd author.
The 2nd author is responsible for the majority of
\S\ref{sec:ground_generation}, including Definition \ref{defn_bukowsky-poset}, for
the computation of $\HOD_{\mathscr{E}}^{M[G]}$ in Theorem
\ref{tm:M_infty[*]=HOD_E} via extending Lemma \ref{restr_of_extenders_are_there} (and the idea to
consider $\HOD_{\mathscr{E}}$),
and the modified P-construction (Definition \ref{dfn:modified_P-con}).
The 2nd and 3rd authors jointly established Lemma
\ref{lem:modified_P-con_works},
the construction of the second direct limit system in
\S\ref{subsec:second_dls}, and the strategy mouse hierarchy used in Definition \ref{dfn:vV_2} (adapting \ref{dfn:vV_1}).
The (self-)iterability of $\vV_1$ and $\vV_2$ is also mostly due to the 2nd and 3rd authors,
integrating some of the earlier work of the first two.
The 3rd author is responsible for Lemma \ref{restr_of_extenders_are_there},
that $\vV_1\sub\M_\infty[*]$,
Lemmas \ref{lem:pi_infty^+}, \ref{local-definability-of-that-structure},
\ref{lem:Psi_sn_good},
\ref{lem:Sigma_vV_1_vshc},
\ref{lem:vV_2_Psi^sn_good},
\ref{lem:vV_2_Sigma_vV_1_vshc},
Definitions \ref{dfn:Psi_sn_1},
\ref{dfn:Psi_vV_2^sn},
\S\S \ref{subsec:Vsp},
\ref{subsec:M-iteration_on_M_infty},
\ref{subsubsec:DSR_trees},
\ref{subsubsec:def_Sigma_vV,sss^dsr}, \ref{subsec:vV_2_is_mantle},
 and the original version of \S\ref{subsec:kappa_0-mantle}.

\subsection{Notation and Background}\label{subsec:notation}

\emph{General}: Given structures $P,Q$, $\univ{P}$ denotes the universe of $P$,
and $P\ueq Q$ means $\univ{P}=\univ{Q}$.

\emph{Premice}: All premice in the paper are Jensen-indexed ($\lambda$-indexed).
Given premouse $N=(U,\es,F)$
with universe $U$, internal extender sequence $\es$ and active
extender $F$, we write $\univ{N}=U$,
$\es^N=\es$, $F^N=F$, and $\es_+^N=\es\conc F$.
Write $N^\passive=(U,\es,\emptyset)$
for its passivization,
$N||\alpha$ for the initial segment $P$ of $N$
with $\OR^P=\alpha$ (inclusive of active extender)
and $N|\alpha=(N||\alpha)^\passive$.
Write $\lh(F)=\OR^N$.
Given premice $M,N$, we write $M\ins N$ iff $M=N||\alpha$ for some $\alpha\leq\OR^N$,
and $M\pins N$ iff $M\ins N$ but $N\not\ins M$.
Given also $m,n\leq\om$ such that $M$ is $m$-sound and $N$ is $n$-sound,
we write $(M,m)\ins(N,n)$ iff $M\ins N$ and if $M=N$ then $m\leq n$,
and write $(M,m)\pins(N,n)$ iff $(M,m)\ins(N,n)$ but $(N,n)\not\ins(M,m)$.
For $\eta\leq\OR^N$, we say $\eta$ is a \emph{cutpoint}
of $N$ iff for all $E\in\es_+^N$, if $\crit(E)<\eta$ then $\lh(E)\leq\eta$,
and a \emph{strong cutpoint} iff for all $E\in\es_+^N$,
if $\crit(E)\leq\eta$ then $\lh(E)\leq\eta$.
For $\xi<\delta\in\OR^N$, $\BB^N_{\delta,\xi}$ denotes the
$\delta$-generator
extender algebra at $\delta$,
with axioms induced by extenders $E\in\es$ with $\nu(E)$
inaccessible in $N$ and $\xi\leq\crit(E)$.
And $\BB_{\delta}^N=\BB^N_{\delta,0}$.

\emph{Hulls}: In general, $\Hull^M_t(X)$ denotes the structure
whose universe is the collection of
elements of $M$, definable over $M$ from parameters in $X$,
with definitions of ``kind $t$'', and whose predicates are just the restrictions of those of $M$. Here ``kind $t$'' depends on context,
but the main example is that if $M$ is a premouse, then the universe of $\Hull_{n+1}^M(X)$ is the collection of all $y\in M$ such that for some $\rSigma_{n+1}$ formula $\varphi$ and $\vec{x}\in X^{<\omega}$,
$y$ is the unique $z\in M$ such that $M\sats\varphi(z,\vec{x})$.
When it makes sense,
$\cHull_t^M(X)$ denotes the transitive collapse of $\Hull_{n+1}^M(X)$
(including the collapses of predicates).

\emph{Ultrapowers}: Let $E$ be an extender over $N$.
Write $i_E^N:N\to\Ult(N,E)$ for the ultrapower map,
and $i_E^{N,n}\to\Ult_n(N,E)$ for the degree-$n$ ultrapower and associated map. Write $\kappa_E=\crit(E)$ for the critical point of $E$,
$\lambda_E=\lambda(E)=i_E(\kappa_E)$, and
$\delta(E)$ for the measure space of $E$; in particular, if $E$ is short
then $\delta(E)=\crit(E)+1$.

\emph{Iteration trees}:
A fine structural iteration tree $\Tt$ consists of tree order $<^\Tt$,
tree-predecessor function $\alpha+1\mapsto\pred^\Tt(\alpha+1)$,
model-dropping node-set $\mathscr{D}^\Tt\sub\lh(\Tt)$, model-or-degree-dropping
node-set $\mathscr{D}^\Tt_{\deg}\sub\lh(\Tt)$, models $M^\Tt_\alpha$ and degrees
$\deg^\Tt_\alpha$ (for $\alpha<\lh(\Tt)$), extenders
$E^\Tt_\alpha\in\es_+(M^\Tt_\alpha)$ and model pre-images $M^{*\Tt}_{\alpha+1}\ins M^\Tt_\beta$
where $\beta=\pred^\Tt(\alpha+1)$
(for $\alpha+1<\lh(\Tt)$),
and here $M^\T_{\alpha+1}=\Ult_d(M^{*\Tt}_{\alpha+1},E^\Tt_\alpha)$
where $d=\deg^\Tt_\alpha$,
and if $\alpha\leq^\Tt\beta$ and $(\alpha,\beta]_\Tt\inter\dropset^\Tt=\emptyset$,
iteration maps $i^\Tt_{\alpha\beta}:M^\Tt_\alpha\to M^\Tt_\beta$, and if $\alpha$ is also a successor ordinal,  $i^{*\Tt}_{\alpha\beta}:M^{*\Tt}_\alpha\to M^\Tt_\beta$.
(Note we are only indicating notation above; the definition of \emph{iteration tree}
has more demands.)

Let $N$ be an $n$-sound premouse, where $n\leq\om$, and $\Tt$ a fine structural
iteration tree.
Recall that $\Tt$ is \emph{$n$-maximal on $N$} iff (i) $(M^\Tt_0,\deg^\Tt_0)=(N,n)$,
(ii) $\lh(E^\Tt_\alpha)<\lh(E^\Tt_\beta)$ for $\alpha+1<\beta+1<\lh(\Tt)$,
(iii) $\pred^\Tt(\alpha+1)$ is the least $\beta$ such that $\crit(E^\Tt_\alpha)<\lambda(E^\Tt_\beta)$, and
(iv) $(M^{*\Tt}_{\alpha+1},\deg^\Tt_{\alpha+1})$ is the lex-largest $(P,p)$ such that
$(M^\Tt_\beta||\lh(E^\Tt_\beta),0)\ins(P,p)\ins(M^\Tt_\beta,\deg^\Tt_\beta)$
and $\crit(E^\Tt_\alpha)<\rho_p^P$.
We say $\Tt$ is \emph{normal} if it is $n$-maximal
for some $n$. For $\delta$ an $N$-cardinal, we say $\Tt$ is \emph{based on $N|\delta$}
iff for all $\alpha+1<\lh(\Tt)$, if $[0,\alpha]_\Tt$ does not drop
in model then $\lh(E^\Tt_\alpha)\leq i^\Tt_{0\alpha}(\delta)$.
We say $\Tt$ is \emph{above $\kappa$} iff $\crit(E^\Tt_\alpha)\geq\kappa$
for all $\alpha+1<\lh(\Tt)$, and \emph{strictly above $\kappa$}
iff $\crit(E^\Tt_\alpha)>\kappa$ for all $\alpha+1<\lh(\Tt)$.

\emph{P-construction}: Given a premouse $M$ and $N\in M$,
$\mathscr{P}^M(N)$, or just $\mathscr{P}(N)$ if $M$ is understood,
denotes the P-construction as computed in $M$ over base set $N$.
In \cite[\S1]{sile}, this model would be denoted ${\cal P}(M,N,-)$.
If $\Tt$ is a limit length iteration tree, $\mathscr{P}^N(\Tt)$ abbreviates $\mathscr{P}^N(M(\Tt))$, and if $\Tt$ is the trivial tree (that is, uses no extenders)
then $\mathscr{P}^N(\Tt)$ denotes $N$.
(The latter notation is just convenient when we set up indices for the direct limit systems,
as then the trivial tree $\Tt$ on $M$ indexes the base of the system computed in $M$.)

\begin{rem}For our overall purposes  Jensen indexing for premice is natural.
	However, genericity iterations are essential, which are somewhat cumbersome with Jensen indexing and Jensen iteration rules (as for \emph{$n$-maximality} above).
	The process for this is described in \cite[Theorem 5.8]{iter_for_stacks}.
	We f also use \emph{genericity inflation},
sketched in \S\ref{section-short-tree-strategy-for-M},
	and \emph{minimal genericity inflation}, see \cite[\S5.2***]{fullnorm_v3}.)
\end{rem}

\section{Ground 
generation}\label{sec:ground_generation}\label{ground_generation_stuff}
In this section we shall present an abstract version of the construction
of a Varsovian model $\vV$ derived from a given inner model $M$ (satisfying
the requirements below),
and prove that $\vV$ is a ground of $M$. It will take some time
to lay out the required hypotheses (\ref{item:M_pc_ZFC}--\ref{item:unif_grds});
we will also collect some facts along the way.

Fix $M,(d,\preceq),(\mathcal{P}_p\colon p\in d)$ such that
\begin{enumerate}[label=\tu{(}ug\arabic*\tu{)}]
\setcounter{enumi}{0}
\item\label{item:M_pc_ZFC} $M$ is a proper class transitive model of ZFC,
\item\label{item:d,preceq_directed_po} $(d,\preceq)\in M$ is a  directed partial
order,
\item\label{item:P_i_system} $({\cal P}_p \colon p \in d)$ is an indexed system
of transitive proper class inner
models
of $M$ which is an $M$-class; that is,
each ${\cal P}_p$ is a transitive proper class
inner
model of $M$, and
$\{ (p,x) \colon p \in d \wedge x \in {\cal P}_p \}$ is an $M$-class.
\footnote{In practice, $M$ and all ${\cal P}_p$, $p\in d$, will be
(pure or strategic)
premice, hence inner models constructed from a distinguished (class sized)
predicate, in which
case our definability hypothesis is supposed to mean that the collection of
predicates constructing the ${\cal P}_p$, $p\in d$, is definable over $M$. }
\setcounter{tempcounter}{\value{enumi}}
\end{enumerate}
Suppose that in $V$ there is a system $(\pi_{pq} \colon p,q \in d \wedge p
\preceq q)$ such that:

\begin{enumerate}[label=\tu{(}ug\arabic*\tu{)}]
\setcounter{enumi}{\value{tempcounter}}
\item\label{item:pi_ij_elem} $\pi_{pq} \colon {\cal P}_p \rightarrow
{\cal P}_q$ is elementary whenever $p\preceq q$,
\item\label{item:pi_ij_commute} the maps are commuting; that is,
$\pi_{qr}\com\pi_{pq}=\pi_{pr}$
for $p\preceq q\preceq r$.
\setcounter{tempcounter}{\value{enumi}}
\end{enumerate}
 Let
\[\mathscr{D}^{\ext}=\big(\left<\P_p\colon p\in d\right>,\left<\pi_{pq}\colon p,q\in d\wedge p\preceq q\right>\big) \]
be the directed system (the \emph{ext} stands for \emph{external}). Define the direct limit (model and maps)
\begin{eqnarray}\label{dir_limit}
({\cal M}_\infty^{\ext}, \pi_{p\infty} \colon p \in d) = \mbox{ dir lim }\mathscr{D}.
\end{eqnarray}
Suppose
\begin{enumerate}[label=\tu{(}ug\arabic*\tu{)}]
\setcounter{enumi}{\value{tempcounter}}
 \item\label{item:M_infty_wfd}  $\M_\infty^{\ext}$ is wellfounded; we take it
transitive.
\setcounter{tempcounter}{\value{enumi}}
\end{enumerate}
Note that the system $\mathscr{D}^{\ext}$ is not assumed to be an $M$-class,
hence neither $\M_\infty^{\ext}$.
But suppose that $\mathscr{D}^{\ext}$ is
``covered'' by an $M$-class,
in the sense that there is an $M$-class
$(d^+,\preceq)$ (we use the same symbol $\preceq$, since there will be no possibility of confusion) such that:
\begin{enumerate}[label=\tu{(}ug\arabic*\tu{)}]
\setcounter{enumi}{\value{tempcounter}}
\item\label{item:d^+} $d^+ \subseteq d \times ([{\rm
OR}]^{<\omega}\cut\emptyset)$
and $\preceq$ is a directed partial order on $d^+$,
\item\label{item:d^+_order} if $(p,s),(q,t)\in d^+$ then $(p,s) \preceq (q,t)$
iff $p\preceq q$ and
$s\subseteq t$
\item\label{item:reduce_t} if $(p,t)\in d^+$ and $\emptyset\neq s\subseteq t$
then $(p,s)\in
d^+$,
\item\label{item:increase_i} if $(p,s) \in d^+$, $q \in
d$ and $p \preceq q$, then
$(q,s) \in d^+$.
\setcounter{tempcounter}{\value{enumi}}
\end{enumerate}
Further, there is
a system
\begin{eqnarray}\label{2nd_dir_lim_system}
\mathscr{D}=\big(\left<H^p_s\colon (p,s)\in d^+\right>,
\left< \pi_{ps,qt} \colon (p,s), (q,t) \in d^+\wedge (p,s)
\preceq (q,t)\right>\big)
\end{eqnarray}
such that:
\begin{enumerate}[label=\tu{(}ug\arabic*\tu{)}]
\setcounter{enumi}{\value{tempcounter}}
\item\label{item:D^+_M-def} $\mathscr{D}$ is an $M$-class,
\item\label{item:H_i^s} for all $(p,s) \in d^+$, $H^p_s$ is an elementary
substructure of
${\cal Q}^p_s={\cal P}_p|\max(s)$,\footnote{Here if
${\cal P}_p$ is a (possibly strategy)
premouse, then  this is precisely defined, and is passive
(strategy) premouse; in general ${\cal P}_p$ should
be stratified in an $\OR$-indexed increasing chain of $\Sigma_0$-elementary
substructures
and ${\cal Q}^p_s$ should be the proper level of that hierarchy indexed at
$\max(s)$.}
\item\label{item:pi_is_js} for all $(p,s),(q,s)\in d^+$ with $p\preceq q$, the
map
$\pi_{ps,qs}:H^p_s\to H^q_s$
is elementary,
\item\label{item:pi_is_it} for all $(p,t)\in d^+$ and $s\sub t$, we have
$H^p_s\elem_0 H^p_t$,
and  the map
\[ \pi_{ps,pt}:H^p_s\to H^p_t \]
is the inclusion map (hence $\Sigma_0$-elementary),
\item\label{item:pisjt_commute} the maps $\pi_{ps,qt}$ commute, in that
$\pi_{qt,ru}\com\pi_{ps,qt}=\pi_{ps,ru}$,
\setcounter{tempcounter}{\value{enumi}}
\end{enumerate}

Note
if $(p,s),(q,t)\in d^+$ and $(p,s)\preceq(q,t)$,
then $(p,s)\preceq(q,s)\preceq(q,t)$, so by
\ref{item:pi_is_js}, \ref{item:pi_is_it}, \ref{item:pisjt_commute},
\[ \pi_{ps,qt}=\pi_{qs,qt}\com\pi_{ps,qs}:H^p_s\to H^q_t \]
is $\Sigma_0$-elementary and has the same graph as has $\pi_{ps,qs}$,
and in particular, the graph is independent of $t$.
And note that $\pi_{ps,qs} \sub \pi_{pt,qt}$ whenever
$s\sub t$ and
$(p,t) \in d^+$ and $p\preceq q\in d$, because here $(p,s)\in d^+$ and
\[
\pi_{pt,qt}\com\pi_{ps,pt}=\pi_{ps,qt}=\pi_{qs,qt}
\com\pi_{ps,qs}, \]
but  $\pi_{ps,pt}$ and $\pi_{qs,qt}$ are just inclusion maps.

\begin{dfn}\label{dfn:ug_stable}
Given $\alpha\in\OR$ and $p\in d$, say $\alpha$ is \emph{$p$-stable}
iff $\pi_{pq}(\alpha)=\alpha$ for all $q\in d$ with $p\preceq q$.
Say $s\in[\OR]^{<\om}$ is \emph{$p$-stable} iff $\alpha$ is $p$-stable for each $\alpha\in s$. Call $(p,s) \in d^+$ {\em true} iff $s$ is $p$-stable and for all $q,r \in d$ with
$p \preceq q \preceq r$, we have $\pi_{qs,rs} =
\pi_{qr} \upharpoonright H^q_s$.
\end{dfn}

\begin{lem}
For each $s\in[\OR]^{<\om}\cut\{\emptyset\}$, there is $p\in d$ such that $s$ is $p$-stable.
\end{lem}
The proof is standard, using the wellfoundedness of $\M_\infty^{\ext}$. Assume further:
\begin{enumerate}[label=\tu{(}ug\arabic*\tu{)}]
\setcounter{enumi}{\value{tempcounter}}
\item\label{item:every_s_gets_stable} for all $s\in[\OR]^{<\om}\cut\{\emptyset\}$,
there is $p \in d$ with
$(p,s)\in d^+$ and $(p,s)$ true.
\item\label{item:every_x_gets_captured} for all $p \in d$ and $x \in
{\cal P}_p$, there exists $s$ such that $(p,s) \in
d^+$, $(p,s)$ is true and $x \in H^p_s$.
\setcounter{tempcounter}{\value{enumi}}
\end{enumerate}

Define the direct limit
$({\cal M}_\infty, \pi_{ps,\infty} \colon (p,s) \in d^+)
 = \mbox{ dir lim }\mathscr{D}$.

\begin{lemma}\label{c0}
$\M_\infty=\M_\infty^{\ext}$ is $M$-definable.
\end{lemma}

\begin{proof} Our assumptions immediately give that $\M_\infty$ is
$M$-definable. Consider the equality. We proceed as in the proof of
\cite[Lemma 2.4]{vm1} or the first few claims in \cite{Theta_Woodin_in_HOD};
the last few properties listed above been abstracted from  those proofs.
We will define a map $\chi \colon {\cal M}_\infty\rightarrow {\cal
M}_\infty^{\ext}$ and
show that $\chi$ is the identity.

Let $(p,s) \in d^+$ and $x \in H^p_s$. By \ref{item:increase_i} and
\ref{item:every_s_gets_stable}, we
may fix $q$ such that $(p,s)
\preceq (q,s)$ and $(q,s)$ is true. Define
$$\chi(\pi_{ps,\infty}(x))= \pi_{q,\infty} \circ \pi_{ps,qs}(x).$$
By commutativity and truth (trueness), this
does not depend on the choice of $q$, so  $\chi$ is well-defined.
Note that $\chi$ is $\Sigma_0$-elementary and cofinal, hence fully elementary,
by \cite[Theorem II.1, p.~54; Remark II.2, p.~55]{gaifman}.
If $p \in d$ and $x \in {\cal P}_p$, then by \ref{item:every_x_gets_captured}
there is  $s$
with $(p,s)$  true and $x \in H^p_s$. Hence $\pi_{p,\infty}(x) =
\chi(\pi_{ps,\infty}(x))
\in {\rm ran}(\chi)$. So $\chi$ is surjective,
so $\chi=\id$. \end{proof}

Given $p \in d$, write ${\cal C}_p = \{ q \in d \colon p \preceq q \}$,
and given $(p,s)\in d^+$, write
${\cal C}_{(p,s)}=\{ (q,t) \in d^+ \colon (p,s) \preceq (q,t) \}$.
Call a set $\mathcal{C}\sub d$ a \emph{cone} iff $\mathcal{C}_p\sub\mathcal{C}$
for some $p\in d$, and $\mathcal{C}'\sub d^+$ a \emph{cone}
iff $\mathcal{C}_{(p,s)}\sub\mathcal{C}'$ for some $(p,s)\in d^+$.

\begin{lem}\label{lem:cones}
	We have:
	\begin{enumerate}
	\item There is a cone $\mathcal{C}\sub d$
		such that  $\pi_{pq}(\rho)=\rho$ for all  $p,q\in\mathcal{C}$ with $p\preceq q$.
\item\label{item:cone_C'} There is a cone $\mathcal{C}'\sub d^+$ such that $\rho\in H^p_s
	\text{ and } \pi_{ps,qt}(\rho)=\rho$
	for all  $(p,s),(q,t)\in\mathcal{C}'$ with $(p,s)\preceq(q,t)$, and
	\end{enumerate}
	\end{lem}
\begin{proof}
Let $s=\{\rho,\rho+1\}$ and (using \ref{item:every_s_gets_stable}) let  $p\in d$ be such that
$(p,s)\in d^+$ and $(p,s)$ is true (hence $s$ is $p$-stable).
Then $\mathcal{C}_p$ and $\mathcal{C}_{(p,s)}$ work.
	\end{proof}
Note that part \ref{item:cone_C'} of the previous lemma is understood by $M$.
Using this, we can define the associated \emph{$*$-map} $\OR\to\OR$.
For $\alpha\in\OR$ write
\begin{eqnarray}\label{defn_star_function}
\rho^* = \pi_{ps,\infty}(\rho),\end{eqnarray}
where $(p,s)$ is any element of  any cone $\mathcal{C}'$ witnessing
part \ref{item:cone_C'} of the lemma. Note this is well-defined,
and  $*$ is a class of $M$.

\begin{lemma}\label{c1}
Let $\rho\in\OR$, and $\mathcal{C},\mathcal{C}'$ be cones
witnessing Lemma \ref{lem:cones}.
Then:
\begin{enumerate}
\item\label{item:cone_ext_system} $\rho^* = \pi_{p\infty}(\rho)$ for all $p
\in {\cal C}$,
\item\label{item:cone_cov_system} $\rho \in H^p_s$ and $\rho^* =
\pi_{ps,\infty}(\rho)$ for all
$(p,s) \in {\cal C}'$,
 \item\label{item:min_it_image} $\rho^* = \mbox{min} \{
\pi_{p\infty}(\rho)
\colon p \in d \}$.\end{enumerate}
\end{lemma}
\begin{proof}
	Parts \ref{item:cone_ext_system} and \ref{item:cone_cov_system}
	follow directly from Lemma \ref{lem:cones}
	and the fact that the function $\chi$ defined
	in the proof of Lemma \ref{c0} is the identity.

Part \ref{item:min_it_image}: Let ${\cal C} \subseteq d$ witness Lemma \ref{lem:cones}.
Let $p\in d$. Then there is
$q \in {\cal C}$ with $p\preceq q$. Therefore $\pi_{pq}(\rho)\geq\rho$,
so $\pi_{p\infty}(\rho)\geq\pi_{q\infty}(\rho)$, which suffices.
\end{proof}

There is another important characterization of $\rho \mapsto \rho^*$,
given some further properties. Assume:
\begin{enumerate}[label=\tu{(}ug\arabic*\tu{)}]
\setcounter{enumi}{\value{tempcounter}}
\item\label{item:i_0_exists} There is a unique $\preceq$-minimal $p_0\in d$.
Moreover, $M={\cal P}_{p_0}$ and
$\pi_{p_0q}(p_0)=q$
for each $q\in d$.
\setcounter{tempcounter}{\value{enumi}}
\end{enumerate}
So $\pi_{p_0q}:M\to{\cal P}_q$ for all $q\in d$, and
$\pi_{p_0\infty}:M\to\M_\infty$.
The next hypotheses guarantee a homogeneity property of
the system, in that each ${\cal P}_p$
may equally serve as a base.
Let $\mathscr{D}^{\P_p}=\pi_{p_0p}(\mathscr{D})$, $\M_\infty^{\P_p}=\pi_{p_0p}(\M_\infty)$, etc, for $p\in d$.\footnote{Note that we write $\pi_{p_0p}(\mathscr{D})$, not $\pi_{p_0p}(\mathscr{D}^{\ext})$;
of the course the latter does not make sense. We know $\M_\infty^{\ext}=\M_\infty$,
but $\M_\infty^{\M_\infty}$ is of course computed in $\M_\infty$
as the direct limit of $\mathscr{D}^{\M_\infty}$. At this stage
it is not relevant whether there is some external system of elementary
embeddings associated with $\mathscr{D}^{\M_\infty}$ analogous to $\mathscr{D}^{\ext}$.}
Let $c_p=d_p\inter d$ and $c^+_p=d^+_p\inter d^+$.
Suppose:

\begin{enumerate}[label=\tu{(}ug\arabic*\tu{)}]
\setcounter{enumi}{\value{tempcounter}}

\item\label{item:c_i_dense} For all $p\in d$, $c_p$ is dense
in $(d_p,\preceq_p)$ and dense in $(d,\preceq)$,
and ${\preceq_p}\rest c_p={\preceq}\rest c_p$,
\item\label{item:struc_emb_correct} For all $p\in d$ and all $(q,s),(r,t)\in
c_p^+$
with $(q,s)\preceq(r,t)$, we have $(\P_q)^{\P_p}=\P_q$
and $(H^q_s)^{\P_p}=H^q_s$
and $(\pi_{qs,rt})^{\P_p}=\pi_{qs,rt}$.
\item\label{item:all_s_eventually_internally_in} For all $s\in[\OR]^{<\om}\cut\{\emptyset\}$
there is $p\in d$ with $(p,s)$ true
and  $(p,s)\in c^+_p$.

\setcounter{tempcounter}{\value{enumi}}
\end{enumerate}

Using these properties, it is now straightforward to deduce:

\begin{lem}\label{lem:M_inf,*_hom} For each $p\in d$, we have:
\begin{enumerate}
	\item 	 $c_p^+$ is dense in $(d_p^+,\preceq_p)$
	and dense in $(d^+,\preceq)$, and ${\preceq_p}\rest c_p^+={\preceq_p}\rest c_p^+$.

\item  The direct limit $\M_\infty^{\P_p}$ of $\mathscr{D}^{\P_p}$ is just $\M_\infty$ , and the associated $*$-map $*^{\P_p}$ is just $*$, so $\pi_{p_0p}(\M_\infty)=\M_\infty$
 and $\pi_{p_0p}(*)=*$.
  \end{enumerate}
\end{lem}

\begin{dfn}
For $p\in d$ and $s,t\in[\OR]^{<\om}\cut\emptyset$,
say $(p,s,t)$ is \emph{embedding-good}
iff $(p,t)\in d^+$,
$\pi_{ps,\infty}\in H^p_t$
and
$\pi_{pt,qt}(\pi_{ps,\infty})=
\pi_{qs,\infty}$  for all $q\in\mathcal{C}_p$.
\end{dfn}
Note that \emph{embedding-good} is an $M$-class.
\begin{lem}\label{lem:get_embedding-good}\
	\begin{enumerate}
	\item\label{item:all_s_ev_in_in_all_q} If $(p,s)$ is as in \ref{item:all_s_eventually_internally_in}
	then $(q,s)$ is true and $(q,s)\in d^+\inter d^+_q$ for all $q\in \mathcal{C}_p$.

\item\label{item:all_s_exists_embedding-good} For each $s\in[\OR]^{<\om}\cut\emptyset$
 there is  $p\in d$ such that for each $q\in\mathcal{C}_p$,
 we have $(q,s)$ true and $(q,s)\in d^+\inter d^+_q$,
 and for each $x\in \P_q$
 there is $t\in[\OR]^{<\om}$ such that $(q,t)$ is true and $x\in H^q_t$ and
 $(q,s,t)$ is embedding-good.
\end{enumerate}\end{lem}
\begin{proof}
	Part \ref{item:all_s_ev_in_in_all_q}:
We have $(q,s)\in d^+$ and $(q,s)$ true
	because $(p,s)\in d^+$ and $(p,s)$ is true,
	and $(q,s)\in d_q^+$ because $(p,s)\in d_p^+$ and
	$\pi_{pq}((p,s))=(q,s)$
	by \ref{item:i_0_exists} and because $s$ is $p$-stable.

	Part \ref{item:all_s_exists_embedding-good}:
Fix $p\in d$ with $(p,s)$ true and $(p,s)\in d^+\inter d^+_p$
(using \ref{item:all_s_eventually_internally_in}).
Let $q\in\mathcal{C}_p$.
By part \ref{item:all_s_ev_in_in_all_q} and \ref{item:i_0_exists}--\ref{item:struc_emb_correct},
$(\pi_{rs,\infty})^{\P_r}=\pi_{rs,\infty}$
for all $r\in\mathcal{C}_q$.
Since $\pi_{qr}(s)=s$ for such $r$,
we get
$\pi_{qr}(\pi_{qs,\infty})=\pi_{rs,\infty}$.
Let $x\in \mathcal{P}_q$.
Using \ref{item:every_x_gets_captured}, let
$t\in[\OR]^{<\om}$ such that $(q,t)\in d^+$ is true
and $x,\pi_{qs,\infty}\in H^q_t$. Then $q,t$ works.
\end{proof}

Let $\mathscr{G}$ be the class of all embedding-good tuples.
Define
$\mathscr{D}^{\M_\infty}=\pi_{0\infty}(\mathscr{D})$
and  ${\cal M}_\infty^{\M_\infty}=\pi_{0\infty}(\M_\infty)$.
Working in $M$,  define
$\pi_\infty:\M_{\infty}\to\M_{\infty}^{\M_\infty}$
by
\[\pi_\infty=
\bigcup_{(p,s,t)\in\mathscr{G}}
\pi_{pt,\infty}(\pi_{ps,\infty}).\]

\begin{lem}\label{lem:pi^infty}
 $\pi_{\infty}:\M_{\infty}\to\M_{\infty}^{\M_\infty}$ is elementary
 and $\pi_{\infty}(\rho)=\rho^*$ for  $\rho\in\OR$.
 Moreover, $\pi_{p_0q}(\pi_\infty)=\pi_\infty$
 for all $q\in d$.
\end{lem}
\begin{proof}
The well-definedness and elementarity of $\pi_{\infty}$ is left to the reader.
Fix $\rho\in\OR$. Let $(p,s)\in d^+$ with $\rho\in\rg(\pi_{ps,\infty})$,
taking $p$ as in Lemma \ref{lem:get_embedding-good}  part \ref{item:all_s_exists_embedding-good} with respect to $s$.
Note we may assume that $\pi_{pq}(\rho)=\rho$ for all $q\in\mathcal{C}_p$.
Let $\pi_{ps,\infty}(\bar{\rho})=\rho$.
Let $t\in[\OR]^{<\om}\cut\{\emptyset\}$ be such that $(p,t)$ is true,
$\bar{\rho},\rho,\pi_{ps,\infty}\in H^p_t$
and $(p,s,t)$ is embedding-good.  Then
\begin{eqnarray*}
\rho^* & = & \pi_{p\infty}(\rho) \\
{} & = & \pi_{pt,\infty}(\pi_{ps,\infty}({\bar \rho})) \\
{} & = & \pi_{pt,\infty}(\pi_{ps,\infty})(\pi_{pt,\infty}({\bar
\rho})) \\
{} & = & \pi_\infty(\rho).
\end{eqnarray*}
The ``moreover'' clause is as in Lemma \ref{lem:M_inf,*_hom}.
\end{proof}

We now define the associated {\em Varsovian model} $\vV$ as
\begin{eqnarray}\label{defn_varsovian_model}
\vV = L[{\cal M}_\infty, \pi_\infty].
\end{eqnarray}
So by Lemmas \ref{lem:M_inf,*_hom}
and \ref{lem:pi^infty}, $\vV$ is a class of  ${\cal
P}_p$, for all $p \in d$, and $\pi_{p_0p}(\vV)=\vV$.
Let $\vV^{\M_\infty}=\pi_{p_0\infty}(\vV)$, so $\vV^{\M_\infty}$ is defined over $\M_\infty$
just as $\vV$ over $M$. So letting $\pi_\infty^{\M_\infty}=\pi_{p_0\infty}(\pi_\infty)$,
\[ \vV^{\M_\infty}=L[\M_\infty^{\M_\infty},\pi_\infty^{\M_\infty}].\]

\begin{lem}\label{lem:pi_infty^+} $\pi_\infty$ extends
uniquely to an elementary
\[ \pi_\infty^+:\vV\to\vV^{\M_\infty}\]
such that $\pi_\infty^+(\pi_\infty)=\pi_\infty^{\M_\infty}$ \tu{(}and $\pi_\infty^+(\M_\infty)=\pi_\infty(\M_\infty)=\M_\infty^{\M_\infty}$\tu{)}.
Moreover, $\pi^+_\infty$ is $\vV$-definable from $\M_\infty,\pi_\infty$.
\end{lem}

\begin{proof}

Since every element of $\vV$ is definable over $\vV$ from $\M_\infty$, $\pi_\infty$ and
 some ordinal,  it suffices to see that for all formulas $\varphi$ and ordinals $\alpha$, we have
\[ \vV\sats\varphi(\M_\infty,\pi_\infty,\alpha)\ \Longleftrightarrow\ \vV^{\M_\infty}\sats\varphi(\M_\infty^{\M_\infty},\pi_\infty^{\M_\infty},\pi_\infty(\alpha)).\]
But $\vV\sats\varphi(\M_\infty,\pi_\infty,\alpha)$
iff
\[ \mathcal{P}_p\sats\text{``}\vV\sats\varphi(\M_\infty,\pi_\infty,\alpha)\text{''} \]
for each $p\in d$. Taking $p$ such that $\alpha$ is $p$-stable,
and then applying $\pi_{p\infty}$,
note that the latter holds iff $\M_\infty$ satisfies the corresponding formula
regarding $\pi_\infty(\alpha)$; that is, iff \[ \vV^{\M_\infty}\sats\varphi(\M_\infty^{\M_\infty},\pi_\infty^{\M_\infty},\pi_\infty(\alpha)).\qedhere\]
	\end{proof}

It therefore makes sense to define,  for any  $x \in {\cal M}_\infty$,
\begin{eqnarray}\label{star-infty}
	x^* = \pi_\infty(x),
\end{eqnarray}
and for $x\in\vV$,
\begin{eqnarray}\label{star-infty^+}
	x^{*+} = \pi_\infty^+(x);
\end{eqnarray}
that is, $*$ and $\pi_\infty$ denote the same function, as do $*+$ and $\pi_\infty^+$.

We next formulate a few more assumptions
which ensure that certain sets are generic over $\vV$.
Let $\delta\in\OR$ and $\mathbb{B}\in M$. Let
$\delta_p=\pi_{p_0p}(\delta)$ and $\delta_\infty=\pi_{p_0\infty}(\delta)$, etc.
Assume:

\begin{enumerate}[label=\tu{(}ug\arabic*\tu{)}]
\setcounter{enumi}{\value{tempcounter}}
\item\label{item:delta_reg,B_delta-cc_cba} $M\sats$``$\delta$ is regular and $\BB$ is a $\delta$-cc~complete Boolean algebra'', and

\item\label{item:delta_infty-cc} $\vV\sats$``$\delta_\infty$ is
regular and $\BB_\infty$ is
$\delta_\infty$-cc''.
\setcounter{tempcounter}{\value{enumi}}
\end{enumerate}

Now work in $\vV$.
Let $\Ll$ be the  infinitary
propositional language, with
propositional symbols $P_{\xi}$ for each ordinal $\xi$,
generated by closing under
under negation and under conjunctions and disjunctions of length ${<
\delta_\infty}$ (so if
$\left<\varphi_\alpha\right>_{\alpha<\theta}\sub\Ll_\mu$
where $\theta<\delta_\infty$, then
$\bigwedge_{\alpha<\theta} \varphi_\alpha$
and $\bigvee_{\alpha<\theta}\varphi_\alpha$ are also  in $\Ll_\mu$).
(Note $\Ll$ is a proper class of $\vV$.)

Working in any outer universe of $\vV$, given a set $B$ of ordinals,
 the satisfaction relation $B
\sats \varphi$  for $\varphi\in\Ll_\mu$
is defined recursively as usual; that is,
 $B \models P_\xi$ iff $\xi \in B$;
 $B \models \lnot \varphi$ iff $B \not\models
\varphi$;
  $B \models \bigwedge_{\alpha<\theta} \varphi_\alpha$ iff $B \models
\varphi_\alpha$
for all $\alpha<\theta$;
and  $B \models \bigvee_{\alpha<\theta} \varphi_\alpha$ iff $B
\models
\varphi_\alpha$
for some $\alpha<\theta$.

Fix some $\BB_p$-name $\tau_p\in \mathcal{P}_p$ for a set of ordinals, for some $p\in d$, and let $\tau_j=\pi_{pq}(\tau_p)$ for $q\in\mathcal{C}_p$,
and $\tau_\infty=\pi_{p\infty}(\tau_p)$.

\begin{definition}\label{defn_bukowsky-poset}
	Work in $\vV$. Let $\mathbb{L}$ be the poset whose conditions are formulas $\varphi\in\Ll$
	such that there is $p\in\BB_\infty$ such that \[ p\forces^{\M_\infty}_{\BB_\infty}\text{``}\tau_\infty\sats\varphi^{*+}\text{''},\]
	and with ordering $\varphi\leq\psi$ iff for every $p\in\BB_\infty$, we have \[ p\forces^{\M_\infty}_{\BB_\infty}\text{``}\tau_\infty\sats(\varphi^{*+}\Rightarrow\psi^{*+})\text{''.}\qedhere\]
	\end{definition}

Although $\Ll$ is proper class, it is easy to see that this is equivalent to a set forcing.
Note here that since $\varphi\in\vV$, $\varphi^{*+}$ is well-defined
and
\[ \varphi^{*+}\in\vV^{\M_\infty}\sub\M_\infty,\]
and the forcing assertions above make sense, as $\tau_\infty\in\M_\infty$ is a $\BB_\infty$-name
and $\varphi^{*+},\psi^{*+}\in\M_\infty$. Note that by modding out by the equivalence relation
\[ \varphi\approx\psi\Longleftrightarrow\varphi\leq\psi\leq\varphi,\]
 we get a forcing-equivalent poset defined as follows:

\begin{definition}\label{dfn:L_forcing_2}
Work in $\vV$. Let $\mathbb{L}\sub\BB_\infty$ be the forcing whose conditions are those Boolean values
in $\BB_\infty$
of the form
\[ ||\text{``}\tau_\infty\sats\varphi^{*+}\text{''}||_{\BB_\infty}^{\M_\infty},\]
where $\varphi\in\Ll$, excluding the $0$-condition of $\BB_\infty$, and with ordering induced by $\BB_\infty$.
\end{definition}

Note that the forcing $\mathbb{L}$ depends on the name $\tau_\infty$;
if we want to make this explicit, we will write $\mathbb{L}(\tau_\infty)$.

\begin{lemma}\label{c2}
$\vV \models$ ``$\mathbb{L}$ has the $\delta_\infty$-c.c.''
\end{lemma}

\begin{proof}
	Since $\vV\sats$``$\BB_\infty$ has the $\delta_\infty$-cc'' (by \ref{item:delta_infty-cc}),
	this is an immediate consequence of Definition \ref{dfn:L_forcing_2}.\end{proof}

\begin{lemma}\label{c3}  Let
	 $A\in M$ be a set of ordinals and $p_1\in d$ and suppose that for all $q\in\mathcal{C}_{p_1}$ there is $g\in M$ such that $g$ is
		$({\cal P}_q,{\mathbb B}_q)$-generic
		and $(\tau_q)_g = A$.\footnote{In our applications, where $M$
			will be a
			(pure or strategic) premouse, $A$ will typically be a canonical code
			for
			$M|\mu$, and the name $\tau_q$ will provide
			a canonical translation of the pair
			$({\cal P}_q|\mu,{\dot g})$ into $M|\mu$,
			where $\dot g$ is the generic filter.}
Then the filter
\[ G_{A} = \{ \varphi \in\mathbb{L} \bigm| A \models \varphi \}; \]
or equivalently, if
using Definition \ref{dfn:L_forcing_2} to define $\mathbb{L}$, the filter
\[ G_A=\Big\{||\text{``}\tau_\infty\sats\varphi^{*+}\text{''}||^{\M_\infty}_{\BB_\infty}\Bigm| \varphi\in\Ll\wedge A\sats\varphi\Big\}, \]
 is
$(\vV,\mathbb{L})$-generic, and $A\in\vV[G_A]$ (so note
$\vV[A]=\vV[G_A]$).
\end{lemma}

\begin{proof} Easily, $G_{A}$ is a filter. We verify
genericity, and then clearly $A\in\vV[G_A]$.
Let $\left<\varphi_\alpha\right>_{\alpha<\theta} \in \vV$ be a maximal antichain
of $\mathbb{L}$. We must see $G_{A}$ meets
$\left<\varphi_\alpha\right>_{\alpha<\theta}$, or equivalently,
that $A\sats\varphi_\alpha$ for some $\alpha$. Supposing otherwise, $A\sats\psi$ where
\[ \psi=\bigwedge_{\alpha<\theta}
\lnot
\varphi_\alpha, \]
and by Lemma \ref{c2}, $\theta<\delta_\infty$, so  $\psi\in\Ll$.

For each $q\in\mathcal{C}_{p_1}$, we have $\vV \subseteq {\cal P}_q$, so $\mathbb{L}
\in {\cal
P}_q$,
and by
hypothesis there is a $(\P_q,\BB_q)$-generic
$g\in M$
with $(\tau_q)_g =
A$, so
\begin{eqnarray}\label{a_condition}
\exists p \in \BB_q \, [p \forces^{{\cal
P}_q}_{\BB_q}
\text{``}\tau_q\models \psi
\text{''}].
\end{eqnarray}
Considering the definition of $\pi_\infty^+$,
note that we may take $q\in\mathcal{C}_p$ such that $\pi_{q\infty}(\psi)=\pi_\infty^+(\psi)=\psi^{*+}$,
so that (\ref{a_condition}) implies
$$\exists p \in {\mathbb B}_\infty \,[ p \Vdash^{{\cal
M}_\infty}_{{\mathbb
B}_\infty}
\text{``}\tau_\infty \models \psi^{*+}
\text{''}]{\rm , }$$
so $\psi\in\mathbb{L}$.
But the elementarity of $\pi_\infty^+$ easily gives that $\psi\perp
\varphi_\alpha$ for every
$\alpha<\theta$,
so $\left<\varphi_\alpha\right>_{\alpha<\theta}$ is not a maximal antichain. Contradiction!
\end{proof}

Finally, suppose:

\begin{enumerate}[label=\tu{(}ug\arabic*\tu{)}]
	\setcounter{enumi}{\value{tempcounter}}
	\item \label{item:unif_grds} For every ordinal $\mu$ with $\mathrm{card}^M(V_\mu^M)=\mu$, there is a set $A'\in\pow(\mu)^M$
	coding $V_\mu^M$, and there is $p_1\in d$
such that $\mu$ is $p_1$-stable,
	and there is a $\BB_{p_1}$-name $\tau'_{p_1}\in \mathcal{P}_{p_1}$
	 such that the hypotheses of Lemma \ref{c3} hold for $A',p_1,\tau'_{p_1}$.
	\setcounter{tempcounter}{\value{enumi}}
\end{enumerate}

So under these assumptions  for a given $\mu,A',\tau'_{p_1}$,
the conclusion of Lemma \ref{c3} holds with respect to  $\mathbb{L}(\tau'_\infty)$.
Assumption \ref{item:unif_grds} basically says that the $\mathcal{P}_p$ form a system of grounds
for $M$ in a ``uniform'' manner.

\begin{dfn}\label{def_unif_grounds}
For $M$, etc, as above,
we say that  $\mathscr{D},\mathscr{D}^+$ provide {\em uniform grounds} for $M$ iff
conditions
\ref{item:M_pc_ZFC}--\ref{item:unif_grds} hold.
\end{dfn}

\begin{tm}\label{tm:Varsovian_is_ground}
Under the uniform grounds assumptions,
$\vV$ is a ground of $M$,  via a forcing
$\mathbb{P}$ such that
$M\sats$``$\PP$ has cardinality $\leq 2^{\delta_\infty}$''
and $\vV\sats$``$\mathbb{P}$ is $\delta_\infty$-cc''.
Therefore $\delta_\infty$ is a regular cardinal in $M$.

However, $\M_\infty$ is \emph{not} a ground of $M$.
\end{tm}

\begin{proof}
We have that every set (of ordinals) in $M$ is  generic over $\vV$ for some
$\mathbb{L}(\tau'_\infty)\sub\BB_\infty$. Since there
are only set-many such forcings, $\vV$ is in fact a ground of $M$
for some such $\mathbb{L}(\tau'_\infty)$.
Moreover, this forcing is $\delta_\infty$-cc in $\vV$,  by
 \cite[Theorem 2.2]{schindler_buk},
we can find a forcing $\mathbb{P}\in\vV$ as desired.

The ``therefore'' clause now follows (recall $\delta_\infty$ is a regular cardinal of $\vV$).

Now $\vV^{\M_\infty}$ is a ground of $\M_\infty$ (since $\vV$ is a ground of $M$).
Suppose $\M_\infty$ is a ground of $M$.
Then $\vV^{\M_\infty}$ is also a ground of $M$.
But $M$ defines $\pi_\infty^+:\vV\to\vV^{\M_\infty}$,
which is then a non-trivial elementary embedding between two  grounds of $M$,
contradicting \cite[Theorem 8]{gen_kunen_incon}.
\end{proof}

When we produce instances of uniform grounds later, we will actually know more:
we will have $V_{\delta_\infty}^{\M_\infty}=V_{\delta_\infty}^{\vV}$ and
$\delta_\infty$ Woodin in $\vV$ (hence also in $\M_\infty$, so $\delta$ Woodin in $M$, which will be an assumption),
and $\BB\sub\delta$, so $\mathbb{L}\sub\delta_\infty$, so some $\mathbb{L}(\tau'_\infty)$
will be a $\PP$ as above, but in fact of cardinality $\delta_\infty$ in $\vV$ and hence also in $M$.

\section{The model $\Mswsw$}\label{sec:Mswsw}
In this section introduce the mouse $\Mswsw$ we
will be analyzing,
and establish some of its basic properties,
as well as some of those of its iteration strategy.

\begin{dfn}Let $\psi_{\swsw}$ be the statement, in the passive premouse
language, asserting  ``There are ordinals
$\delta_0<\kappa_0<\delta_1<\kappa_1$
with $\delta_i$  Woodin and $\kappa_i$ strong for $i\leq 1$,
as witnessed by $\es$''.  Let $M^\#$ be the least active mouse
such that $M^\#|\mu\sats\psi_{\swsw}$ where $\mu=\crit(F^{M^\#})$.\footnote{By
\cite{mim},  Woodinness and strength is automatically witnessed by
$\es^{\M^\#}$, as a consequence of iterability, but we will also
consider premice $N\sats\psi_{\swsw}$ which need not be iterable.}
Then $\Mswsw$ denotes the proper class model left behind
by iterating $F^{M^\#}$ out of
the universe.
Note $\rho_\om({M^\#})=\rho_1(M^\#)=\om$, $p_1^{M^\#}=\emptyset$ and
$M^\#$ is $\om$-sound. We assume
throughout that $M^\#$ exists and is $(\om,\OR,\OR)$-iterable.\footnote{We
could
probably just work with $(\om,\om_1+1)$-iterability.
By \cite[Theorems 9.1, 9.3]{iter_for_stacks}, because $M^\#$ is $\om$-sound and projects
to $\om$,
$(\om,\om_1+1)$-iterability for $M^\#$ implies $(\om,\om_1,\om_1+1)^*$-iterability,
and similarly, $(\om,\OR)$-iterability for $M^\#$ implies
$(\om,\OR,\OR)$-iterability.}
We usually write $M=\Mswsw$.

Let $\Sigma$ denote the  $(\om,\OR)$-iteration strategy
(that is, for $\om$-maximal, hence normal, trees, of  set length)
for $M$ which is
induced by the unique $(\om,\OR)$-strategy $\Sigma_{M^\#}$ for $M^\#$.
Let $\Gamma=\Sigma^{\stk}$ denote
the optimal-$(\om,\OR,\OR)$-strategy for $M$ which is induced
by $\Sigma$ via the normalization process of \cite{fullnorm_v3} (see Fact \ref{fact:Sigma_properties} below, especially item
\ref{item:Slist_} there).
\end{dfn}

Certain aspects of  normalization, used to define $\Gamma=\Sigma^{\stk}$ from $\Sigma$,
will be used in the paper. The main features we need are the properties of $\Sigma$ mentioned in Fact \ref{fact:Sigma_properties} below,
which can be black-boxed. Some of the details of the normalization process will also come up to some extent later on, but the reader unfamiliar with those
details should still be able to follow most of the arguments in the paper.

\begin{rem}\label{rem:E^M_def}
$M$ knows enough of  $\Sigma$
 that $M|\om_1^M$ is definable over the universe of $M$ (without parameters).
 Therefore by \cite[Theorem 1.1]{V=HODX_pub}, $\es^M$ is definable over the
universe
 of $M$ without parameters. Thus, when we talk about definability
 over $M$, it does not matter whether we are given $\es^M$
 as a predicate or not. However, if $g$ is $M$-generic,
 then $\HOD^{M[g]}$ can differ from $\HOD^{M[g]}_{\es^M}$, for example.
\end{rem}

\begin{dfn}\label{defn_correct_normal_iterate}
	If $\Ttvec$ is a stack on $M$ via $\Gamma$,
	then $\Gamma_{\Ttvec,N}$ denotes the tail stacks strategy for $N$ induced by
	$\Gamma$, i.e. $\Gamma_{\Ttvec,N}(\Uuvec)=\Gamma(\Ttvec\conc\Uuvec)$.
	Also $\Sigma_{\Ttvec,N}$ denotes the normal part of $\Gamma_{\Ttvec,N}$.
	Actually by what follows below, we can and usually do write $\Gamma_N$ and $\Sigma_N$.
	\end{dfn}

Recall
that $\Gamma=\Sigma^{\stk}$ is the strategy
for stacks induced by $\Sigma$.

\begin{fact}\label{fact:Sigma_properties}  We have:
\begin{enumerate}[label=\tu{(}$\Sigma$\arabic*\tu{)}]
\item\label{item:Slist_} $\Sigma$ is the unique $(\om,\OR)$-strategy
for $M$, so satisfies both strong hull condensation and minimal
hull condensation,
and therefore by \cite{fullnorm_v3}:
\begin{itemize}[label=--]
 \item every iterate of $M$ via $\Gamma$ is also an iterate via
$\Sigma$,
 \item if $G$ is $V$-generic then $\Sigma,\Gamma$ extend
 canonically to $V[G]$, with the same properties there;
 with an abuse of notation, we continue to write $\Sigma,\Gamma=\Sigma^{\stk}$ for
these extensions, or may write $\Sigma^{V[G]}$ or $\Gamma^{V[G]}$ to emphasize the distinction.
\end{itemize}
\item\label{item:Slist_positional} $\Gamma$ is fully positional, in that
whenever $\vec{\Tt},\vec{\Uu}$ are two stacks via $\Gamma$ with the same last model $N$,
then $\Gamma_{\vec{\Tt},N}=\Gamma_{\vec{\Uu},N}$, irrespective of drops. However, positionality will only be relevant in the non-dropping case.
\item\label{item:Slist_commuting} $\Gamma$ is
commuting,
i.e., if $\Ttvec\conc\Uuvec$ and $\Ttvec\conc\Vvvec$ are
non-dropping stacks via $\Gamma$ with a common last model, then
$i^\Uuvec=i^\Vvvec$; see \cite[***Theorem 10.4]{fullnorm_v3}.
\item For all $\Ttvec$ via $\Gamma$, with last model $N$,
$\Sigma_{\Ttvec,N}$ has minimal hull condensation
and $\Gamma_{\Ttvec,N}=(\Sigma_{\Ttvec,N})^{\stk}$; see \cite[***Theorem 10.2]{fullnorm_v3}.\footnote{In order
	to define $(\Sigma_{\Ttvec,N})^{\stk}$, one also needs that
	$N$ is $n$-standard, where $n=\deg^{\Ttvec}_\infty$, but this follows
	from the fact that $M$ is $0$-standard, by \cite[***Remark 2.2]{fullnorm_v3}.}
Thus, every iterate of $N$ via $\Gamma_{\Ttvec,N}$ is also an iterate via $\Sigma_{\Ttvec,N}$,
in a unique manner,
\item\label{item:Slist_strategy_agreement} If $\Tt,\Uu$ are via $\Sigma$, of successor length,
with non-dropping final branches, $P=M^\Tt_\infty$ and $Q=M^\Uu_\infty$,
$\eta\in\OR^M$,  $\eta'=i^\Tt(\eta)=i^\Uu(\eta)$
and $P|\eta'=Q|\eta'$
then $\Sigma_{\Tt,M^\Tt_\infty}$ and $\Sigma_{\Uu,M^\Uu_\infty}$
agree with one another in their action on trees $\Vv$ based on $P|\eta'$.
See \cite[***Theorem 10.5]{fullnorm_v3}.\footnote{The precise version of this fact
	might be simplified by the fact that $M$ is below superstrong.}
\end{enumerate}
\end{fact}

\begin{rem}Very strong hull condensation (\cite{steel_local_HOD_comp})
	implies minimal hull condensation (\cite{fullnorm_v3}), which implies
	minimal inflation condensation (\cite{fullnorm_v3}).
	For the normalization process of \cite{fullnorm_v3},
	minimal inflation condensation is sufficient,
	but for the generic absoluteness results,
    minimal hull condensation is used.
	\end{rem}

 \begin{dfn}
 		We say that a stack $\Ttvec$ on $M$ is \emph{correct} if it is via $\Gamma$.
 	We say that $N$ is a  \emph{$\Sigma$-iterate} of $M$
 	iff there is a correct stack $\Ttvec$ on $M$ with last model $N=M^{\Ttvec}_\infty$.
 By the properties above, we may in fact take $\Tt$ via $\Sigma$ (hence normal),
and note that this $\Tt$
 is uniquely determined
 by $N$ (and $\Sigma$); we write $\Tt_N=\Tt$. A $\Sigma$-iterate
 is a \emph{dropping iterate}
 iff $b^\Tt$ drops, and otherwise is \emph{non-dropping}.

Let $N$ be a non-dropping $\Sigma$-iterate.
Then a \emph{$\Sigma_N$-iterate} is similarly
an iterate of $N$ via $\Sigma_N$ (equivalently, via $\Gamma_N$).\footnote{We don't
	need to iterate dropping iterates of $M$ further.}
If $P$ is a non-dropping $\Sigma_N$-iterate,
let $i_{NP}:N\to P$ be the iteration map (via $\Sigma_N$).
Given $\delta\leq\OR^N$, we say
that $N$ is \emph{$\delta$-sound}
 iff, letting $\Tt=\Tt_N$, we have
$N=\Hull^N(\delta\cup\rg(i^\Tt))$;
equivalently,  $\nu(E^\Tt_\alpha)\leq\delta$
 for all $\alpha+1<\lh(\Tt)$.
\end{dfn}

\begin{dfn}\label{dfn:inds} Let $\mathscr{I}^M$ denote the class of
critical points of the linear iteration of $F^{M^\#}$ which
produces $M$.
For $N$ as above, let $\mathscr{I}^N=i_{MN}``\mathscr{I}^M$.
\end{dfn}

\begin{dfn}[$\Mswsw$-like]\label{dfn:Mswsw-like}
A premouse $N$ is \emph{$\Mswsw$-like}
 iff $N$ is proper class and satisfies a certain finite sub-theory $T$ of
the theory of $M$, including $\psi_{\swsw}+$``I have no active proper segment
$R$ such that $R|\crit(F^R)\sats\psi_{\swsw}$''. We will not spell $T$ out
exactly, but the reader should
 add statements to it as needed to make certain arguments work.
For an $\Mswsw$-like  model $N$, write
\begin{eqnarray}\label{notation-delta-kappa-gamma}
\begin{cases}
\delta_0^N = \mbox{ the least Woodin cardinal of } N \\
\kappa_0^N = \mbox{ the least strong cardinal of } N \\
\kappa_0^{+N} = (\kappa_0^N)^{+N} \\
\delta_1^N = \mbox{ the least Woodin cardinal of } N \mbox{ above } \kappa_0^N \\
\kappa_1^N = \mbox{ the least strong cardinal of } N \mbox{ above } \delta_0^N \\
\kappa_1^{+N}  = (\kappa_1^N)^{+N}\\
\end{cases}
\end{eqnarray}
If $N=M$, then we may suppress the superscript $N$, so $\delta_0 =
\delta_0^{M}$, etc.
\end{dfn}

\begin{fact}\label{fact:inds_preserved}
	Let $N$ be a  non-dropping $\Sigma$-iterate of $M$
	and $\Tt=\Tt_N$.
	Let $\delta=\delta(\Tt)=\sup_{\alpha+1<\lh(\Tt)}\nu(E^\Tt_\alpha)$.
	Then $\delta\leq\kappa_1^N<\min(\mathscr{I}^N)$
	and $\mathscr{I}^N$ is the unique club class of indiscernibles $I$
	such that $N=\Hull_1^N(I\cup\delta)$,
	or alternatively such that $N=\Hull^N_1(I\cup\kappa_1^P)$.
\end{fact}

The following two lemmas are instances of branch condensation (see \cite{hod_mice}) and are simple variants of
\cite[Lemma 2.1]{vm1}; we fill in a couple of key points which were omitted
from that proof, however.

\begin{lem}[Branch condensation A]\label{lem:branch_con}
Let $\Uu_0$ be a successor length tree on $M$, via $\Sigma$,
based on $M|\delta_0^M$,
with $b^{\Uu_0}$ non-dropping.
Let $\Tt,\Uu$ be on $N=M^{\Uu_0}_\infty$, via $\Sigma_N$,
based on $N|\delta_0^N$, with $\Tt$ of limit length and $\Uu$ successor
with $b^\Uu$ non-dropping.
Let $G$ be $V$-generic.
Let $b,k\in V[G]$ where $b$ is a non-dropping $\Tt$-cofinal branch
and
\[ k:M^\Tt_b|\delta_0^{M^\Tt_b}\to M^\Uu_\infty|\delta_0^{M^\Uu_\infty} \]
is elementary with $k\com i^\Tt_b\sub i^\Uu_{0\infty}$.
Then $b=\Sigma_N(\Tt)$.
\end{lem}
\begin{proof}
Because $\Sigma$ extends to $V[G]$,
with corresponding properties there (cf.~Fact \ref{fact:Sigma_properties}\ref{item:Slist_}),
we may assume  $G=\emptyset$.
Let $c=\Sigma_N(\Tt)$. Let $P_b=M^\Tt_b$ and $P_c=M^\Tt_c$.

Suppose first  $\delta(\Tt)<\delta_0^{P_b}$.
Then there is a Q-structure $Q'\pins P_b$ for $\delta(\Tt)$,
and because $\delta(\Tt)$ is a cardinal of $P_b$,
$M(\Tt)$ has no Woodin cardinals, so
$\delta(\Tt)$ is a strong cutpoint of $Q'$.
Because we have $k$, $Q'$ is iterable.

If $c$ is non-dropping and $\delta(\Tt)=\delta_0^{P_c}$
then we can compare $Q'$ versus $P_c$ for a contradiction.
So in any case, $Q=Q(\Tt,c)$ exists. Since $M(\Tt)$ has no Woodins,
$\delta(\Tt)$ is also a strong cutpoint of $Q$,
so we can compare and get $Q=Q'$, so $b=c$.

Now suppose $\delta(\Tt)=\delta_0^{P_b}$.
Then we can argue as in the proof of \cite[Lemma 2.1]{vm1};
however, we fill in a seemingly key point: We extend $k$ to
\[ k^+:P_b\to M^\Uu_\infty\]
with $k^+\com i^\Tt_b=i^\Uu_{0\infty}$ as in \cite{vm1}.
Now $P_b=P_c$ (this was not mentioned in \cite{vm1});
for $P_b$ is iterable and is $\delta_0^{P_b}$-sound,
and likewise for $P_c$, but both are models of ``I am $\Mswsw$'',
so comparison gives $P_b=P_c$. And because $P_b=P_c$
and $i^\Tt_b,i^\Tt_c$ fix all sufficiently large indiscernibles,
we can indeed conclude that $i^\Tt_b\rest X=i^\Tt_c\rest X$,
where
\[ X=\Hull^N(I^N)\inter\delta_0^N, \]
where $I^N$ is the class of $N$-indiscernibles. So by the Zipper Lemma,
we get $b=c$.
\end{proof}

There is also a version at $\delta_1$. We won't directly use this,
but will use a variant, which will use a similar proof:

\begin{lem}[Branch condensation B]\label{lem:branch_con_2}
Let $\Uu_0$ be  successor length  on $M$, via $\Sigma$,
based on $M|\delta_1^M$,
with $b^{\Uu_0}$ non-dropping.
Let $\Tt,\Uu$ be  on $N=M^{\Uu_0}_\infty$, via $\Sigma_N$,
based on $N|\delta_1^N$, with $\Tt$ of limit length and above
$\kappa_0^{+N}$, $\Uu$ successor length
with $b^\Uu$ non-dropping. Let $G$ be $V$-generic.
Let $b,k\in V[G]$ where $b$ is a non-dropping $\Tt$-cofinal branch
and
\[ k:M^\Tt_b|\delta_1^{M^\Tt_b}\to M^\Uu_\infty|\delta_1^{M^\Uu_\infty} \]
is elementary with $k\com i^\Tt_b=i^\Uu_{0\infty}$.
Then $b=\Sigma_N(\Tt)$.
\end{lem}
\begin{proof}
Again we may assume $b,k\in V$. Let $\alpha\in b^{\Uu_0}$
be least with either
$\alpha+1=\lh(\Uu_0)$ or $\kappa_0(M^{\Uu_0}_\alpha)<\crit(i^{\Uu_0}_{\alpha
\infty})$. Let $\bar{N}=M^{\Uu_0}_\alpha$. So
$\kappa_0^N=\kappa_0^{\bar{N}}$ and  $\bar{N}$
 is $\kappa_0^N$-sound and $\pi:\bar{N}\to N$ where
$\pi=i^{\Uu_0}_{\alpha\infty}$ and $\crit(\pi)>\kappa_0^N$.
 Let $P_b=M^\Tt_b$, $c=\Sigma_N(\Tt)$  and $P_c=M^\Tt_c$.
If $\delta(\Tt)<\delta_1^{P_b}$ then let $Q_b\pins P_b$ be the Q-structure
for $\delta(\Tt)$; this exists because $\delta_1^N$ is the least
Woodin of $N$ above $\kappa_0^N$,
and $\Tt$ is above $\kappa_0^{+N}$.
Otherwise let $Q_b=P_b$. This time, $Q_b$ can have extenders $E$ overlapping
$\delta(\Tt)$,
but only with $\crit(E)=\kappa_0^N$.
Let $Q_c\ins P_c$ be likewise.

Define phalanxes
$\bar{\mathfrak{P}}=((\bar{N},\kappa_0^N),Q_b,\delta(\Tt))\text{ and
}\bar{\mathfrak{Q} } =((\bar{N},\kappa_0^N),Q_c,\delta(\Tt))$.
We claim $\bar{\mathfrak{P}},\bar{\mathfrak{Q}}$ are
iterable.\footnote{The notation
indicates that we start iterating the phalanx with extenders of index $>\delta(\Tt)$,
and extenders with critical point $\kappa_0$ apply to $\bar{N}$.}
Given this,
we can compare $\bar{\mathfrak{P}},\bar{\mathfrak{Q}}$,
and because $\bar{N}$ is $\kappa_0^N$-sound
and $Q_b,Q_c$ are $\delta(\Tt)$-sound, we get $Q_b=Q_c$,
so if $b\neq c$ then $P_b=Q_b=Q_c=P_c$ and $\delta_1^{P_b}=\delta(\Tt)$,
and we reach a contradiction like before.

Define phalanxes
$\mathfrak{P}=((N,\kappa_0^N),P_b,\delta(\Tt))$ and
$\mathfrak{Q}=((N,\kappa_0^N),P_c,\delta(\Tt))$.
It suffices to see $\mathfrak{P},\mathfrak{Q}$ are iterable,
because then we  can reduce
trees on $\bar{\mathfrak{P}}$ to trees on $\mathfrak{P}$,
using $\pi:\bar{N}\to N$, and likewise
 $\bar{\mathfrak{Q}}$ to $\mathfrak{Q}$.

But $\mathfrak{Q}$ is iterable because $c=\Sigma(\Tt)$.
For $\mathfrak{P}$, we have $k^+:P_b\to
M^\Uu_\infty$ defined as before
(note that the same definition still works in case $\delta(\Tt)<\delta_1^{P_b}$).
So $P_b$ is iterable.
But $i^\Tt_b:N\to P_b$ with $\crit(i^\Tt_b)>\kappa_0^N$.
So we can lift trees on $\mathfrak{P}$ to trees on $P_b$
using the maps $(i^\Tt_b,\id)$. So $\mathfrak{P}$ is iterable.
\end{proof}

\section{The first Varsovian model $\vV_1$} \label{section_first_varsovian}\label{sec:vV_1}
We begin by  identifying a natural direct limit system, giving  uniform
grounds for $M=\Mswsw$, in the sense
Definition \ref{def_unif_grounds} in
\S\ref{ground_generation_stuff},
hence yielding a Varsovian model, which we denote $\vV_1$
(we will later define a second Varsovian model $\vV_2$).
The direct limit system will be defined
analogously
to that of \cite[\S2]{vm1}.
The main difference is in the increased large cardinal level.
A smaller difference, one of approach,
is that we use normalization, which means that we can focus
on normal trees, instead of stacks.

\subsection{The models for the   system}\label{subsec:models_of_first_system}

\begin{dfn}Let $\Tt$ be a limit length normal tree on $M$, based on
	$M|\delta_0^M$, via $\Sigma$.\footnote{Recall $\Sigma$ is the
		$(0,\OR)$-strategy (that is, for $\om$-maximal,
		hence normal, trees) for $M$.}
	Let $b=\Sigma(\Tt)$.
	We say that $\Tt$ is \emph{short} iff either $b$ drops or
	$\delta(\Tt)<i^\Tt_{0b}(\delta_0^M) = \delta_0^{{\cal M}_b^\Tt}$; otherwise
	$\Tt$ is \emph{maximal}.
	Let $\Sigma_\sss$ be the restriction of $\Sigma$ to short trees.

	If $P$ is a non-dropping $\Sigma$-iterate of $M$
	and $\Tt$ is limit length normal on $P$ and based on $P|\delta_0^P$,
	we define \emph{short/maximal} for $\Tt$ analogously,
	and $\Sigma_{P,\sss}$ is the restriction of $\Sigma_P$
	to short trees.
\end{dfn}

\begin{definition}\label{defn_points_from_the_system}
	Let ${\mathbb U}$
	consist of all iteration trees ${\cal U}\in M|\kappa_0$ on $M$, such that either $\Uu$ is trivial, or
	\begin{enumerate}[label=\tu{(}\alph*\tu{)}]
		\item\label{item:U_basics} ${\cal U}$ is based on
		$M|\delta_0$, via $\Sigma_\sss$ (hence is $\om$-maximal),
		\item ${\cal U}$ is maximal,
	\end{enumerate}
	and for some strong cutpoint $\eta<\kappa_0$ of $M$, writing
	$\delta=\delta(\Uu)$
	and $R=M(\Uu)$,
	\begin{enumerate}[resume*]
		\item ${\rm lh}({\cal U})=\eta^{+M} = \delta$,
		\item ${\cal U}$ is definable from parameters over $M|\delta$,
		\item\label{item:M|delta_generic_over_R} $M|\delta$ is $R$-generic for
		$\BB_\delta^R$,
		\item $P\eqdef\mathscr{P}^M(R)$ is proper class;\footnote{Recall the notation $\mathscr{P}^M$ from \S\ref{subsec:notation}.}
		\footnote{Recall $\delta$ is Woodin in
			$P$, as witnessed by $\es^R$, and $V_{\delta}^P$ is the universe of $R$.}
		hence $P$ is a ground of $M$ via $\BB^R_\delta$,
		and in fact $P[M|\delta]\ueq M$. We write here also $\mathscr{P}^M(\Uu)=P$.\qedhere
	\end{enumerate}
\end{definition}

The proof of \cite[Lemma 2.2]{vm1} or the first few claims in
\cite{Theta_Woodin_in_HOD} give:

\begin{lemma}\label{lemma_1}\label{lem:inds_fixed}
	Let ${\cal U} \in {\mathbb U}$, $b = \Sigma({\cal U})$ and
	$P=\msP^M(\U)$. Then  ${\cal M}_{b}^{{\cal U}}=P$,
	$\mathscr{I}^M=\mathscr{I}^P$
	and $i_{0b}^{\cal T}\rest\mathscr{I}^M=\id$.
\end{lemma}

We have ${\mathbb U}\sub M$ by definition,
but because of the requirement
that $\Uu\in{\mathbb U}$ be via $\Sigma_\sss$ (hence via $\Sigma$),
it is not immediate that ${\mathbb U}\in M$.
But
we show in the next section that it is,
and that ${\mathbb U}$ is rich, with the following properties:
The restriction of $\Sigma_\sss$ to $M$
is known to $M$,
and
whenever $P=\msP^M({\cal U})$
for some ${\cal U} \in {\mathbb U}$, the restriction of $\Sigma_{P,\sss}$ to $M$
is known to $M$ (and moreover, these are preserved by $i_{MP}$).
Pseudo-genericity iterations can be formed using these strategies
to produce trees in ${\mathbb U}$.
Any two such models $P_1,P_2$ can be pseudo-compared
with these strategies.
Moreover,
every maximal tree ${\cal T} \in M|\kappa_0$
via $\Sigma_\sss$ is ``absorbed'' by
some ${\cal X} \in {\mathbb U}$.

\subsection{The short tree strategy $\Sigma_\sss$ for $M$}\label{section-short-tree-strategy-for-M}

We now show that
$M$ is closed under $\Sigma_\sss$
and $\Sigma_\sss\rest M$ is a class of $M$,
and that the same also holds
with $M[g]$ replacing $M$,
for any $M$-generic $g$
(where $g$ need not be in $V$);
in fact when $g=\emptyset$ the class can be taken lightface.

Let $\Tt$ be via $\Sigma$ of limit length, and  $b=\Sigma(\Tt)$.
Suppose we want to
compute $b$.
Since $\Sigma$ has strong hull condensation,
it suffices to find a tree $\Xx$ via $\Sigma$
and $\Tt$-cofinal branch $b'$ and a tree embedding
$\Pi:\Tt\conc b'\to\Xx$, for then $b'=b$.

Suppose also $\Tt\in M[g]$ and is based on $M|\delta_0$.
Working in $M[g]$ we want to (i) determine whether $\Tt$ is short,
and (ii) if short, compute $\Sigma_\sss(\Tt)$. If it happens that $\Tt$
incorporates, in an appropriate manner,
a genericity iteration for making $\es^M$ generic, then we will be able to use
P-constructions (combined with $*$-translation, discussed below)
to achieve both of these goals. In the general case, we use the method
of \emph{\tu{(}genericity\tu{)} inflation}
to reduce $\Tt$ to a tree $\Xx$ which does incorporate such a genericity iteration
(see  \cite[\S5.2]{iter_for_stacks}, which adapts methods for tame mice from \cite[\S1]{sile}).
We give here a sketch of the relevant methods from \cite{iter_for_stacks}, restricted to our context; but the reader should
consult \cite{iter_for_stacks} for details.

Suppose also that $g$ is $M$-generic for some $\PP\in V_\theta^M$,
and $\Tt\in V_\theta^{M[g]}$.
If $\theta<\kappa_0$ let $U=M$; otherwise let $E\in\es^M$ be $M$-total with
$\crit(E)=\kappa_0$
and $V_\theta^M\sub U=\Ult(M,E)$ (so $\theta<\lambda(E)=\kappa_0^U$).
Let $\eta$ be a strong cutpoint of $U$ with $\theta\leq\eta<\kappa_0^U$.
Following \cite{iter_for_stacks}, let $\Xx$ be the genericity inflation (explained
further below) of $\Tt$
for making $U|\delta(\Xx)$ generic for the $\delta(\Xx)$-generator extender
algebra,
incorporating an initial linear iteration which moves the
least measurable of $M(\Xx)$ beyond $\eta$, and incorporating linear iterations
past $*$-translations of Q-structures.\footnote{The technique of inserting
	linear iterations past Q-structures
	comes from \cite{odle_v2}, where there are details of such a construction
	given.} The $*$-translation is due to Steel, Neeman, Closson;
see \cite{closson}, together with an amendment in
\cite{*-trans_add}.
The $*$-translations of Q-structures
are segments of $U$ which compute the Q-structures which guide
branch choices for $\Xx$.

Here is a sketch of the relevant material from \cite{iter_for_stacks}.
We define $\Xx\rest(\alpha+1)$ by induction on $\alpha$.
Suppose we have $\Xx\rest(\alpha+1)$ defined, but have not yet succeeded in
finding
$\Sigma(\Tt)$.
We will have an ordinal $\eta_\alpha\leq\OR(M^\Xx_\alpha)$ defined, and
possibly have an ordinal $\beta_\alpha<\lh(\Tt)$
and a lifting map
\[ \sigma_\alpha:M^\Tt_{\beta_\alpha}||\lh(E^\Tt_{\beta_\alpha})\to
M^\Xx_\alpha||\eta_\alpha \]
defined. (At $\alpha=0$ we have $\beta_0=0$ and $\eta_0=\lh(E^\Tt_0)$
and $\sigma_0=\id$.) We set $E^\Xx_{\alpha}=F^{M^\Xx_\alpha||\eta_\alpha}$,
unless there is an extender $G\in\es^{M^\Xx_\alpha}$ with $\lh(G)<\eta_\alpha$
such that either (i) $G$ induces an extender algebra axiom
which is not satisfied by $\es^U$,
and $G$ satisfies some further conditions
as explained in \cite{iter_for_stacks}\footnote{It suffices
	that $\nu_G$ is inaccessible in $M^\Xx_\alpha|\eta_\alpha$,
	but one must also consider other extenders, including partial ones,
	because of the nature of genericity iteration with Jensen indexing.}
or (ii) $G$ is a measure to be used for one of the linear iterations mentioned
above.\footnote{The linear iterations need to be set up appropriately,
	to ensure that the process does not last too long; similar details are dealt with
	in the  comparison arguments in \cite{odle_v2}.}
We say $E^\Xx_\alpha$ is either \emph{copied} from $\Tt$
(when $\lh(E^\Xx_\alpha)=\eta_\alpha$ and $(\beta_\alpha,\sigma_\alpha)$ are
defined) or is \emph{inflationary} (otherwise).
The stages $\alpha$ for which $(\beta_\alpha,\sigma_\alpha)$ is not defined
correspond to a drop in model in $\Xx$, below the image of the relevant extender
from $\Tt$,
and arise because of the nature of genericity iteration with Jensen indexing.
Let $\gamma=\pred^\Xx(\alpha+1)$.
If $E^\Xx_\alpha$ is copied then $\beta_{\alpha+1}=\beta_\alpha+1$
is defined, and $\sigma_{\alpha+1}$ is the restriction
of a map given by the Shift Lemma applied
to $\sigma_\alpha$ and another map $\pi$ (whose domain
is $M^{*\Tt}_{\beta_\alpha+1}$; we have not specified $\pi$
here).
If $E^\Xx_\alpha$ is
inflationary
then $\beta_{\alpha+1}$ is defined just in
case $\beta_\gamma$ is defined and $E^\Xx_\alpha$
is total over $M^{\Xx}_{\gamma}||\eta_\gamma$,
and in this case $\beta_{\alpha+1}=\beta_\gamma$
and $\sigma_{\alpha+1}=i^{*\Xx}_{\alpha+1}\com\sigma_\gamma$.

Now consider a limit stage $\lambda$. The first
thing to do is either compute $c=\Sigma(\Xx\rest\lambda)$
(if $\lambda<(\eta^+)^U$), or declare
$\Xx,\Tt$ maximal (if $\lambda=(\eta^+)^U$). Let
$\delta=\delta(\Xx\rest\lambda)$.
If $M(\Xx\rest\lambda)$ is a Q-structure for itself
then $b$ is trivial, and arguing as in \cite{odle_v2}
shows that in this case, $\lambda<(\eta^+)^U$
(the argument is mostly standard, but some variant
details arise, which are discussed there).
So suppose otherwise. Then $\lambda=\delta$
and  $\Xx\rest\delta$ is definable from parameters over $(U|\delta)[g]$. This is
because $\eta<\delta$,
$\Tt\in (U|\eta)[g]$, the process for determining
$\eta_\alpha,E^\Xx_\alpha$ is locally definable,
and the $*$-translations of Q-structures used to compute the branches of
$\Xx\rest\lambda$
are all proper segments of $U|\delta$, because of the linear iterations past
these $*$-translations.
Moreover, $U|\delta$ is generic for the $\delta$-generator extender algebra of
$M(\Xx\rest\lambda)$.

From now on, let us assume that $g=\emptyset$ for simplicity;
since $\eta$ is a strong cutpoint of $U$, the general case
only involves shifting to $U[g]$.
Let $c=\Sigma(\Xx\rest\delta)$. Let $Q=Q(\Xx\rest\delta,c)$,
considering $\Xx$ as a tree on $M^\#$.
(Maybe $\Xx\rest\delta$ is not short.)
Then $Q$ could have extenders overlapping $\delta$.
But the $*$-translation $Q^*$ of $(Q,\Xx\rest\delta)$ is a premouse extending
$U|\delta$
and having no overlaps of $\delta$, and in fact,
either (i) $\delta<(\eta^+)^U$ and $\Xx\rest\delta$ is short and $Q^*\pins U$,
or (ii) $\lambda=(\eta^+)^U$ and $\Xx\rest\lambda$ is maximal and $Q^*=U^\#$.
So $U$ can see which of case (i) and (ii) we are in, and in case (i),
compute $Q^*,Q,c$ (as $Q^*$ is the unique segment of $U$ whose inverse
$*$-translation
is well-defined and terminates with a Q-structure for $M(\Xx\rest\lambda)$,
which is then $Q$).
Moreover, the branch $c$ is determined by the $*$-translation $Q^*$ of a
Q-structure, as
promised earlier.

Suppose $\lambda<(\eta^+)^U$.
So we have computed $c=\Sigma(\Xx\rest\lambda)$
in $U[g]$.
By \cite{iter_for_stacks},  this determines
either (i) some $\beta_\lambda<\lh(\Tt)$
(and possibly a $\sigma_\lambda$ as before),
in which case we continue the process;
or (ii) a $\Tt$-cofinal branch $b$
and a tree embedding
\[ \Pi:\Tt\conc b\hookrightarrow\Xx\conc c\]
with $b$ mapped cofinally into $c$,
and $b$ is encoded into $c$, in such a manner  that
$U[g]$ can compute $b,\Pi$ from $(\Tt,\Xx,c)$.

Now suppose that the process reaches $\Xx$ of length $(\eta^+)^U$.
So $\Xx$ is maximal and $Q^*=U^\#$.
Let $c=\Sigma(\Xx)$ and $b=\Sigma(\Tt)$.
So $i^\Xx_c(\delta_0^M)=\delta(\Xx\rest\lambda)=(\eta^+)^U$.
Again by \cite{iter_for_stacks}, there is a tree embedding
\[ \Pi:\Tt\conc b\hookrightarrow \Xx\conc c\]
which maps $b$ cofinally in $c$,
and since $\Tt,\Xx$ are based on $M|\delta_0^M$, then $\Tt$ is maximal.
Also in this case, considering $\Tt,\Xx$ as trees on $M$, instead of on $M^\#$,
we get that $M^\Xx_c=\mathscr{P}^U(\Xx)$ (the P-construction
of $U$ above $\M(\Xx)$, which is the analogue
of the inverse $*$-translation of $U$  in this case), so if $U=M$ (and still $g=\emptyset$)
then $\Xx\in\mathbb{U}$.

This completes the sketch. For further details the reader should refer to
\cite{iter_for_stacks}, augmented by \cite{closson}, \cite{*-trans_add}
and \cite{V=HODX_pub}.

\begin{dfn}
	For a non-dropping $\Sigma$-iterate $P$ of $M$, $\Sigma_{P,\sss}$
	(the \emph{short tree strategy} for $P$) denotes the restriction
	of $\Sigma_P$ to short trees. Also, $\Sigma_{P|\delta_0^P}$
	denotes the ($0$-maximal) strategy for $P|\delta_0^P$ induced by $\Sigma$
	(including maximal trees)
	and $\Sigma_{P|\delta_0^P,\sss}$ denotes its restriction to short trees.
	\end{dfn}

Note that by Fact \ref{fact:Sigma_properties}\ref{item:Slist_strategy_agreement},
the notations $\Sigma_{P|\delta_0^P}$ and $\Sigma_{P|\delta_0^P,\sss}$
are unambiguous; that is, if $P\neq Q$ are both non-dropping $\Sigma$-iterates of $M$
with $P|\delta_0^P=Q|\delta_0^Q$, then $\Sigma_{P}$ agrees with $\Sigma_Q$
in terms of their action on trees based on $P|\delta_0^P$.
Of course $\Sigma_{P,\sss}$ is equivalent to $\Sigma_{P|\delta_0^P,\sss}$,
except that the two strategies have different base models. This is useful
notationally below, where we can refer directly to $P|\delta_0^P$ but maybe not to $P$.

We summarize the main results of this section in the following two lemmas:

\begin{lem}\label{lem:Sigma_sss_in_M[g]_definability} Let $g$ be $M$-generic.
	Then:
	\begin{enumerate}
		\item  $M[g]$ is closed under $\Sigma_\sss$.
		\item $\es^M,\ \Sigma_\sss\rest M[g]\text{ and }\dom(\Sigma_\sss\rest
		M[g])$
		are classes of $M[g]$, definable over $M[g]$ (as a coarse structure) from the
		parameter
		$M|(\lambda^{+\om})^M$ where $g\sub M|\lambda$,
		uniformly in $\lambda$.
		\item If $g=\emptyset$ then these are in fact lightface classes of the universe
		$\univ{M}$ of $M$.
		\item\label{item:U_M-def} Therefore $\mathbb U$ is lightface $M$-definable,
		as is
		$\left<\mathscr{P}^M(\Uu)\right>_{\Uu\in\mathbb{U}}$
		(recall $\mathscr{P}^N(\Uu)=N$ if $\Uu$ is  the trivial tree on $N$),
		\item\label{item:M_cl_under_Sigma_P,sss} For each non-dropping $\Sigma$-iterate $P$ of $M$ with $\bar{P}=P|\delta_0^P\in M$, $M$ is closed under  $\Sigma_{\bar{P},\sss}$, and $\Sigma_{\bar{P},\sss}\rest M$ is definable
		over $M$ from $\bar{P}$, uniformly in $\bar{P}$.
Therefore  the function
		\[ S:\bar{P}\mapsto\Sigma_{\bar{P},\sss}\rest M,\]
		with domain the class of all such $\bar{P}\in M$,
		 is lightface $M$-definable.
		\item\label{item:Sigma_N,sss} The corresponding facts
		hold after replacing $M$ by $N$ and $\Sigma_{\sss}$
		by $\Sigma_{N,\sss}$ and $\mathbb U$ by $\mathbb U^N$, for any non-dropping
		$\Sigma$-iterate $N$ of $M$. Moreover,
		\[ i_{MN}(\Sigma_{M,\sss}\rest
		M)=\Sigma_{N,\sss}\rest N, \]
		and with
		$S$ from part \ref{item:M_cl_under_Sigma_P,sss},
		$i_{MN}(S)$ has the corresponding domain in $N$, and
		$i_{MN}(S)(\bar{P})=\Sigma_{\bar{P},\sss}\rest N$ for each $\bar{P}\in\dom(i_{MN}(S))$.
	\end{enumerate}
\end{lem}
\begin{proof}
	By the previous discussion, $M[g]$ is closed under $\Sigma_\sss$.
	Moreover,
	\[ \Sigma_\sss\rest M[g]\text{ and }\dom(\Sigma_\sss\rest M[g])\]
	are definable over $M[g]$ from the predicate $\es^M$.
	But the universe $\univ{M}$ of $M$ is definable over $M[g]$
	from $M|(\lambda^{+\om})^M$ by Woodin-Laver \cite{laver_vlc},
	\cite{woodin_CH_multiverse_Omega}.
	By \cite{V=HODX_pub}, $\es^M$ is definable over $\univ{M}$
	from $M\rest\om_1^M$, but the latter
	is $(\om,\om_1+1)$-iterable
	in $\univ{M}$ (via $\Sigma_\sss$),
	and is therefore definable without parameters there (which is relevant to the
	case that $g=\emptyset$).
	Part \ref{item:M_cl_under_Sigma_P,sss}
	is a straightforward adaptation;
	in fact, note that trees via $\Sigma_{\bar{P},\sss}$ normalize
	to trees via $\Sigma_{\sss}$.
	Part \ref{item:Sigma_N,sss} is also straightforward,
	using the uniformity of the process.\footnote{However, working inside $N$,
		if $\Tt$ on $N|\delta_0^N$ is maximal and we minimally inflate $\Tt$
		to produce $\Xx$, and build the proper class model $\mathscr{P}^N(M(\Xx))$
		by P-construction,
		and $c=\Sigma_N(\Xx)$, then it need not be that $M^\Xx_c=\mathscr{P}^N(M(\Xx))$.
		But in this case, $M^\Xx_c$ and $\mathscr{P}^N(M(\Xx))$ will still compare to a common model above $\delta(\Xx)$. Related issues will be discussed further in \S\ref{subsec:M-iteration_on_M_infty}.}
\end{proof}

\begin{lemma}\label{absorbing-a-tree-by-a-genericity-tree}
	Let $g$ be $M$-generic for $\PP$ and ${\cal T} \in M[g]$ be
	a limit length normal tree on $M$ which is based on
	$M|\delta_0^M$ and via $\Sigma$.
	If $\Tt\in M|\kappa_0$ let $U=M$, and otherwise let $E\in\es^M$ be $M$-total
	with $\crit(E)=\kappa_0$
	and
	\[ \PP\in V_\theta^M\sub U=\Ult(M,E) \]
	and $\Tt\in V_\theta^{M[g]}$.
	Let $b=\Sigma(\Tt)$. Let $\eta$ be a strong cutpoint of $U$
	with $\theta\leq\eta<\kappa_0^U$. Then there is
	$\Xx=\Xx_{\Tt,\eta}\in U[g]\sub M[g]$ such that:
	\begin{enumerate}[label=\tu{(}\arabic*\tu{)}]
		\item\label{item:gen_inf_X_props_was_a} ${\cal X}$ is a limit length tree on $U$
		(but is equivalent to one on $M$), based on $M|\delta_0^M$, via $\Sigma_U$ (hence $\om$-maximal);
		let $c=\Sigma(\Xx)$ and $\delta=\delta(\Xx)$,
		\item $\delta\leq(\eta^+)^U$,
		\item\label{item:gen_inf_X_props_was_b} $U|\delta$ is $M(\Xx)$-generic
		for $\BB_\delta^{M(\Xx)}$,
		\item\label{item:gen_inf_X_props_was_c} If ${\cal T}$ is maximal then
		$\Xx$ is maximal,
		$\delta_0^{M^\Xx_c}=\delta=(\eta^+)^U$
		and $M^\Xx_c=\mathscr{P}^U(\Xx)$.
		\item Suppose $\Tt$ is short. Then $\Xx$ is short and
		$\eta<\delta(\Xx)<(\eta^+)^U$, and there is $R\pins U|(\eta^+)^U$
		which computes the Q-structure $Q(\Xx,c)$ via
		inverse $*$-translation above $M(\Xx)$.

		\item\label{item:gen_inf_X_props_was_e} There
		is a tree embedding $\Pi:\Tt\conc b\hookrightarrow\Xx\conc c$,
		and $b,\Pi$ can be computed locally from $(\Tt,\Xx,c)$
		(hence if $\Tt,\Xx$ are short then $b\in U[g]$).
		\item If $\Tt\in M|\kappa_0$
		(so $U=M$) and $\Tt$ is maximal then $\Xx\in{\mathbb U}$.
	\end{enumerate}
\end{lemma}

\begin{dfn}We may also express the situation
	of the preceding lemma
	by  saying
	that ${\cal T}\conc b$  is
	{\em absorbed by} ${\cal X}\conc c$,
	or  ${\cal
		T}$ is \emph{absorbed by} ${\cal X}$.
\end{dfn}

\subsection{The first direct limit system}\label{first_dir_limit_system}

\subsubsection{The external direct limit system $\mathscr{D}^{\ext}$}

We now  define a system of uniform grounds for $M$.
In the notation of \S\ref{sec:ground_generation},
we use index set
\[ d=\{M|\delta_0\}\cup\{\M(\Uu)\bigm|\Uu\in\mathbb U\text{ is non-trivial}\}.\]
For $p\in d$,
the associated model is $\P_p=\mathscr{P}^M(p)$.
Of course $d$ and $\mathbb U$ are essentially equivalent.
By Lemma \ref{lem:Sigma_sss_in_M[g]_definability}, $(d,\left<\P_p\right>_{p\in d})$
is lightface $M$-definable.
Write
${\mathscr F}=\{\mathcal{P}_p\mid p\in d\}$.

We now define the partial order $\preceq$ on $d$, and
maps $\pi_{pq}$.
Let $\T,{\cal U} \in {\mathbb U}$ and
$P=\msP^M({\cal T})$ and $Q=\msP^M({\cal U})$.
Set\[\M(\Tt)\preceq\M(\Uu)\Longleftrightarrow Q\text{ is a
}\Sigma_P\text{-iterate of }P.\]
We also define the
order $\preceq$ on $\mathscr{F}$ by $P\preceq Q$ iff $\M(\Tt)\preceq\M(\Uu)$.
The associated embedding $\pi_{\M(\Tt)\M(\Uu)}$ is just the iteration map $i_{PQ}$.
We remark that if $P\neq Q$ then
the tree witnessing that $P\preceq Q$ is of the form $\Vv\conc
\Sigma_P(\Vv)$, with $\Vv$ via
$\Sigma_{P,\sss}$ and  $\Vv\in M$, but $\Sigma_P(\Vv)\notin M$.\footnote{$M$
	cannot have an elementary embedding $P\to Q$, because $P,Q$ are both grounds of
	$M$ and by \cite{gen_kunen_incon}. Therefore $\Sigma_P(\Vv)\notin M$.
	$\Vv$ is determined by $P,Q,\Sigma_{P,\sss}$ (in fact,
	just by $P,Q$, because all the relevant Q-structures
	are segments of $Q$), so $\Vv\in M$.}
We write $\T_{PQ}$ for $\Vv\conc\Sigma_P(\Vv)$, and
$\T_{PQ,\sss}=\Vv$.
While $i_{PQ}$ is not amenable to $M$ (if non-identity),
we do have:

\begin{lemma}\label{lem:vV_1_preceq_props} $\preceq$ is a directed partial order, is lightface $M$-definable,
	and the associated embeddings commute:
	if $P\preceq Q\preceq R$ then $i_{QR}\com i_{PQ}=i_{PR}$.
\end{lemma}
\begin{proof}[Proof Sketch]
	For the definability, note that $P\preceq Q$
	iff the pseudo-comparison of $P,Q$, using $\Sigma_{P,\sss}$ to iterate
	$P$,\footnote{Or
		just reading Q-structures from segments of $Q$ to compute branches.}
	yields a limit length tree $\Tt$ on $P$ with $\mathcal{M}(\Tt)=Q|\delta_0^Q$
	(so $Q$ does not move in the pseudo-comparison).

	The fact that $\preceq$
	is a partial order, and the commutativity, follows from
	the properties of $\Sigma^{\stk}$
	in Fact \ref{fact:Sigma_properties}.

	For directedness, given $P,Q\in\mathscr{F}$,
	witnessed by trees $\Tt,\Uu\in\mathbb{U}$, form
	a simultaneous pseudo-comparison and $\es^M$-genericity iteration
	of $P,Q$
	in $M$, using $\Sigma_{P,\sss},\Sigma_{Q,\sss}$, producing
	trees $\Vv,\Ww$ respectively, and $R\in\mathscr{F}$
	with $P,Q\preceq R$;
	note that if we normalize the stack $(\Tt,\Vv)$, or the stack
	$(\Uu,\Ww)$, we get the same normal tree $\Xx\in{\mathbb U}$; here
	$R=\mathscr{P}^M(\Xx)$.
\end{proof}

Now define  $\mathscr{D}^{\ext}=(P,Q,i_{PQ} \colon P\preceq Q
\in {\mathscr F} )$.
By the lemma, $\mathscr{D}^{\ext}$ is
a direct limit system
with properties \ref{item:M_pc_ZFC},
\ref{item:d,preceq_directed_po}, \ref{item:P_i_system},
\ref{item:pi_ij_elem},
\ref{item:pi_ij_commute}.
Note that \ref{item:i_0_exists} holds,
with $p_0=M|\delta_0$.
Define the direct limit model and maps
\begin{eqnarray}\label{defn_M+}
	(\M_\infty^{\ext}, (i_{P\infty} \colon P \in {\mathscr F}))= \dirlim\ \mathscr{D}^{\ext}.
\end{eqnarray}
Notice that even though ${\mathscr F}$ is a definable collection of classes in $M$,
this system is not ``in'' $M$, as the maps $i_{NP}$ (when non-identity)
are not amenable to $M$. As usual, $\M_\infty^{\ext}$ is wellfounded, giving \ref{item:M_infty_wfd}.

\begin{dfn}\label{dfn:tau^P}
	For $P\in\mathscr{F}$, let $\tau^P$ be the canonical class $\BB^P_{\delta_0^P}$-name
	for $M$; that is, $\tau^P$ is the name for the class ``premouse'' $N$
	such that $N|\delta_0^P$ is  given by the extender algebra generic,
	and $\es^{N}\rest[\delta_0^P,\infty)$ is given by extending $\es^P$ via the usual
	extension to small generic extensions (equivalently, $\es^{N}\rest[\delta_0^P,\infty)$ agrees
	with $\es^P\rest[\delta_0^P,\infty)$ on the ordinals). (Of course,
	for some generics, this might not actually yield a premouse, but
	with $g$ the generic for adding $M|\delta_0^P$, we have $(\tau^P)_g=M$.)
	Note that $i_{MP}(\tau^M)=\tau^P$.
\end{dfn}

\begin{lemma}\label{lem:c_i_dense}
	\ref{item:c_i_dense} holds:
	for each $P\in\mathscr{F}$, $c^P=d\cap d^P$ is dense
	in $(\mathscr{F}^P,\preceq^P)$ and dense in $(\mathscr{F},\preceq)$,
	and ${\preceq^P}\rest c^P={\preceq}\rest c^P$.\end{lemma}
\begin{proof}
	Let $P\in\mathscr{F}$.
	The fact that ${\preceq^P}\rest c^P={\preceq}\rest c^P$ is
	by Lemma \ref{lem:Sigma_sss_in_M[g]_definability} part \ref{item:Sigma_N,sss}.
	So let $Q\in\mathscr{F}^P$
	and $R\in\mathscr{F}$. We must find some $S\in\mathscr{F}^P\inter\mathscr{F}$
	with $Q,R\preceq S$. Let $\sigma$
	be some $\BB_{\delta_0^P}^P$-name
	such that $\sigma_g=R$. Let $\eta$ be a strong cutpoint of $P$
	such that $\delta_0^P<\eta$ and $Q|\delta_0^Q,\sigma\in P|\eta$.
	Let $p_1\in\BB_{\delta_0^P}^P$ be the Boolean value
	of the statement ``$\tau^P$ is an $\Mswsw$-like premouse
	and $\sigma\in\mathscr{F}^{\tau^P}$''.
	Then working in $P$, we can form a tree $\Tt$ on $P$
	by Boolean-valued
	comparison
	of $P,Q$ and all interpretations of
	$\sigma$ below $p_1$, with $\es^P$-genericity
	and Boolean-valued $\tau^P$-genericity iteration folded in,
	and using the short tree strategies to iterate.
	(The Boolean-valued genericity iteration means that we use extenders $E$
	under the usual circumstances as for genericity iteration,
	and given that there is some
	$\BB_{\delta_0^P}^P$ condition
	forcing that $E$ induces an axiom false for $\tau^P$.)
	For each limit $\lambda\leq(\eta^+)^P$,
	if $M(\Tt\rest\lambda)$ is not a Q-structure
	for itself then $\delta(\Tt\rest\lambda)=\lambda$ and
	$\Tt\rest\lambda$ is definable from parameters
	over $P|\lambda$, because (i) the segments
	of $\tau^P$ are determined level-by-level by $\es^P$ above $\delta_0^P$,
	and (ii) for all limits $\xi<\lambda$,
	the Q-structures guiding branch choice at stage $\xi$
	do not overlap $\delta(\Tt\rest\xi)$,
	and (iii) the Q-structures $Q_\xi$ of all trees at stage $\xi$
	are identical, and hence $Q_\xi\pins M(\Tt\rest\lambda)$;
	this means that we can use P-construction to compute Q-structures,
	and obtain an iterate in $\mathscr{F}\cap\mathscr{F}^P$.
	(For (ii), suppose $\xi$ is least such that a Q-structure
	overlaps $\delta(\Tt\rest\xi)$. Then there are fatal drops passed before stage $\xi$.
	This has to originally arise from a disagreement between the extender sequences of some projecting structures, as opposed to extenders used
	for genericity iteration purposes (as the latter are only used  when they are sufficiently total;
	cf.~the process in \cite{iter_for_stacks}). But then we can find some mutual generics witnessing this disagreement, and because the short tree strategies extend to generic extensions (because $\tau^P$ is forced $\Mswsw$-like), and given the fatal drop, these strategies suffice to complete the comparison between the conflicting projecting structures in the generic extension, which leads to the usual contradiction. Because $Q_\xi$ does not overlap
	$\delta(\Tt\rest\xi)$, note that no extenders in $\es_+^{Q_\xi}$
	of length $>\delta(\Tt\rest\xi)$ will be used for genericity iteration
	purposes.)
\end{proof}

\subsubsection{The internal direct limit system $\mathscr{D}$}

\begin{dfn}\label{dfn:M_internal_dls}
	Work in $M$. Given $P \in {\mathscr F}$ and $s\in[\OR]^{<\om}\cut\{\emptyset\}$,
	say $P$ is \emph{weakly $s$-iterable} iff for all $Q \in {\mathscr F}$
	with $P\preceq Q$, letting $\Tt=\Tt_{PQ,\sss}$,
	there is $\lambda\in\OR$ such that  $\Coll(\om,\lambda)$
	forces the existence of  a $\Tt$-cofinal branch $b$ such that
	\begin{eqnarray}\label{weak-s-it-requirement}
		i^\Tt_{b}(s)=s \mbox{ and }
		i^\Tt_{b}(P|\max(s))=Q|\max(s)
	\end{eqnarray}
	(in particular, $s$ is in the wellfounded part of $M^\Tt_b$).
	We say that $P$ is \emph{$s$-iterable} iff every $Q\in {\mathscr F}$ with $P \preceq Q$ is weakly $s$-iterable.

	Given an $s$-iterable $P$, define
	\[ \gamma^P_s=\sup(\Hull^{P|\max(s)}(s^-)\inter\delta_0^P), \]
	\[ H^P_s=\Hull^{P|\max(s)}(\gamma^P_s\cup s^-). \]
	(Note that the hulls here are uncollapsed. Recall that $P|\max(s)=(P||\max(s))^\passive$ is passive by definition.) Given
	also a $t$-iterable $Q$ with $s\sub t$ and
	$P\preceq Q$,
	define
	\[ \pi_{Ps,Qt}:H^P_s\to H^Q_t \]
	as the common restriction of iteration maps $i^\Tt_{b}$
	for $b$ witnessing the weak $s$-iterability requirement
	(\ref{weak-s-it-requirement}).
	\footnote{Notice that $\pi_{Ps,Qt}$ does not
		depend on $t$, because
		$\pi_{Ps,Qt}$ and $\pi_{Ps,Qs}$ have the same graph.
		In
		\cite{vm1}, the notation for the analogous map,
		$\pi_{PQ}^s$,
		does not mention $t$.}
	(Those restrictions agree pairwise by the Zipper Lemma.)
	Say that $P$  is \emph{strongly $s$-iterable}
	iff $P$ is $s$-iterable and whenever
	$Q,R\in{\mathscr F}$
	and $P\preceq Q\preceq R$ (hence $P\preceq R$), then
	\[ \pi_{Ps,Rs}=\pi_{Qs,Rs} \com\pi_{Ps,Qs}.\]

	Let $\mathscr{F}^+=\{(P,s)\bigm|P\in\mathscr{F}\text{ and }P\text{ is strongly
	}s\text{-iterable}\}$, and similarly let
	$d^+=\{(P|\delta_0^P,s)\bigm|(P,s)\in\mathscr{F}^+\}$.
	The order $\preceq$ on $d^+$
	is now determined by \ref{item:d^+_order}:
	for $(\bar{P},s),(\bar{Q},t)\in d^+$,
	we have  $(\bar{P},s)\preceq(\bar{Q},t)$ iff $\bar{P}\preceq\bar{Q}$ and
	$s\sub t$. Define the order $\preceq$ on $\mathscr{F}^+$
	in the same manner.
	Clearly
	\[ (P,s)\preceq(Q,t)\preceq (R,u)\implies
	\pi_{Ps,Ru}=\pi_{Qt,Ru}\com\pi_{Ps,Qt}.\]

	Define the system $\mathscr{D}=
	(H^P_s,H^Q_t,\pi_{Ps,Qt} \colon (P,s) \preceq (Q,t) \in \mathscr{F}^+)$.

	Given $P\in\mathscr{F}$ and $s\in[\OR]^{<\om}$, recall that $s$ is
	\emph{$P$-stable}
	iff $\pi_{PQ}(s)=s$ for every $Q\in\mathscr{F}$ with $P\preceq Q$.
\end{dfn}

\begin{rem}\label{rem:M_s-it_implies_strong_s-it}
	Even though it is superfluous, we note that
	$s$-iterability actually implies strong $s$-iterability. This follows
	from calculations in \cite{fullnorm_v3}. For let
	$\Tt=\Tt_{PQ,\sss}$,
	$\Uu=\Uu_{QR,\sss}$ and $\Xx=\Xx_{PR,\sss}$. Say that a $\Tt$-cofinal branch
	is $\Tt$-\emph{good} iff $M^\Tt_b$ is $\delta(\Tt)+1$-wellfounded
	and $i^\Tt_b(\delta_0^P)=\delta(\Tt)$; likewise for the other trees.
	Then the $\Xx$-good branches
	$d$ correspond exactly to pairs $(b,c)$ such that
	$b$ is $\Tt$-good and $c$ is  $\Uu$-good;
	and moreover, the corresponding iteration maps $\rest\delta_0^P$
	then commute (see \cite[***Theorem 10.8]{fullnorm_v3}).
	Let $b$ be $\Tt$-good and witness
	weak $s$-iterability. Let $\Tt'$ be the $0$-maximal
	tree on $P|\max(s)$ given by $\Tt$. So $M^{\Tt'}_b=\Ult_0(P|\max(s),E^\Tt_b)$,
	where $E^\Tt_b$ is the branch extender. By a condensation argument due to J.\ Steel,
	$M^{\Tt'}_b\ins M^\Tt_b$, and since $i^\Tt_b(\max(s))=\max(s)$,
	clearly $\OR(M^{\Tt'}_b)=\max(s)$, so $M^{\Tt'}_b=Q|\max(s)$,
	and similarly $i^{\Tt'}_b(s^-)=s^-$.
	Likewise for $c$ being $\Uu$-good and witnessing weak $s$-iterability,
	and $0$-maximal tree $\Uu'$ on $Q|\max(s)$.
	Let $d=d_{b,c}$ be the corresponding branch, and $\Xx'$ the $0$-maximal
	tree on $P|\max(s)$. Then we get $M^{\Xx'}_d=M^{\Uu'}_c=R|\max(s)$,
	and $i^{\Xx'}_d(s^-)=i^{\Uu'}_c(i^{\Tt'}_b(s^-))=s^-$,
	and commutativity in general. But note that these embeddings
	agree with the covering direct limit maps (consider the natural factor map
	 $M^{\Tt'}_b\to M^\Tt_b|\max(s)$), and therefore
	we get strong $s$-iterability.
\end{rem}

The following is proved by the usual arguments
(recall that since $\M_\infty^{\ext}$ is wellfounded, for all $s\in[\OR]^{<\om}$,
there is $P\in\mathscr{F}$ such that $s$ is $P$-stable):

\begin{lemma}\label{lem:first_properties_int_sys}
	We have:
	\begin{enumerate}[label=\tu{(}\alph*\tu{)}]
		\item\label{item:stable_b} if $P \in {\mathscr
			F}$ and $s\in[\OR]^{<\om}\cut\{\emptyset\}$ and $s$ is $P$-stable, then
		$(P,s)\in\mathscr{F}^+$ and $(P,s)$ is true
		(see Definition \ref{dfn:ug_stable}).
		\item\label{item:stable_c}  $(\mathscr{F}^+,\preceq)$ is
		directed -- for $(P,s)$, $(Q,t) \in \mathscr{F}^+$
		there is $(R,u) \in \mathscr{F}^+$ with $(P,s) \preceq (R,u)$ and
		$(Q,t) \preceq (R,u)$ \tu{(}note $u=s\cup t$ suffices\tu{)}.
		\item Therefore properties \ref{item:d^+},
		\ref{item:d^+_order}, \ref{item:reduce_t}, \ref{item:increase_i},
		\ref{item:every_s_gets_stable} hold.
	\end{enumerate}
\end{lemma}

The following properties follow directly from the definitions;
note that strong $s$-iterability
is used for \ref{item:pisjt_commute}:
\begin{lemma}
	$\mathscr{D}$ is lightface $M$-definable, and
	properties \ref{item:D^+_M-def}, \ref{item:H_i^s}, \ref{item:pi_is_js},
	\ref{item:pi_is_it}, \ref{item:pisjt_commute}
	hold.
\end{lemma}

For the following, recall  $\mathscr{I}^P$ from Definition \ref{dfn:inds}:
\begin{lemma}\label{lem:M_indisc_stable}
	For each $P\in\mathscr{F}$,  $\{\alpha\}$ is $P$-stable
	for every $\alpha\in\mathscr{I}^M=\mathscr{I}^P$.
	Therefore property \ref{item:every_x_gets_captured} holds,
	as witnessed by some $s\in[\mathscr{I}^M]^{<\om}$.
\end{lemma}
\begin{proof}
	The proof is standard, but we give a reminder.
	By Fact \ref{fact:inds_preserved},
	$P=\Hull^P(\mathscr{I}^P\cup\delta_0^P)$.
	By Lemma \ref{lem:inds_fixed},
	$\pi_{MP}\rest\mathscr{I}^M=\id$,
	so $\mathscr{I}^P=\mathscr{I}^M$,
	and we have $\{\alpha\}$-stability
	for each $\alpha\in\mathscr{I}^M$.
	Now let $x\in P$. By these facts, we can fix $s\in[\mathscr{I}^M]^{<\om}$
	such that $x\in\Hull^{P|\max(s)}(\delta\cup\{s^-\})$.
	But we can also arrange that $\gamma^P_s$ is as large below
	$\delta_0^P$ as desired. It follows we can get $x\in H^P_s$,
	and by Lemma \ref{lem:first_properties_int_sys}, this works.
\end{proof}

We now have enough  properties from
\S\ref{ground_generation_stuff}
to define (working in $M$) the direct limit
\begin{equation}\label{internal_dir_limit}
	({\cal M}_\infty, \pi_{Ps,\infty} \colon (P,s) \in \mathscr{F}^+)
	= \dirlim\ \mathscr{D},\end{equation}
and the $*$-map,
and Lemmas \ref{c0} and \ref{c1} hold,
so in particular, $\chi:{\cal M}_\infty \to {\cal
	M}_\infty^{\ext}$ is the identity and $\M_\infty=\M_\infty^{\ext}$. Property \ref{item:struc_emb_correct}
is straightforward (the main point is that if $P\in\mathscr{F}$
and $\bar{Q}\in d^P\inter d$ then
$\P_{\bar{Q}}^P=\mathscr{P}^P(\bar{Q})=\mathscr{P}^M(\bar{Q})$,
because $\es^P,\es^M$ agree above $\delta_0^P$).
For property \ref{item:all_s_eventually_internally_in}, given $s$, note that
any $P\in\mathscr{F}$ such that $s$ is $P$-stable works.

So far we have verified  \ref{item:M_pc_ZFC}--\ref{item:all_s_eventually_internally_in}.
For the remaining
properties
we use $\delta=\delta_0^M$ and $\BB=\BB_{\delta_0}^M$
(the $\delta_0$-generator extender algebra of $M$ at $\delta_0$).
This immediately secures \ref{item:delta_reg,B_delta-cc_cba}.
Recall we defined $\tau^P$ in Definition \ref{dfn:tau^P}, and
$\delta_\infty=i_{M\infty}(\delta_0)=\delta_0^{\M_\infty}$.

\begin{lemma}\label{key_facts_about_V1} We have:
	\begin{enumerate}
		\item \label{item:tau_g_gives_M|alpha}	For each $M$-stable $\alpha\in\OR$
		and each $P\in\mathscr{F}$,
		letting $\tau^P\rest\alpha=i_{MP}(\tau^M\rest\alpha)$
		and $g$
		be the $P$-generic filter for $\BB_{\delta_0^P}^P$
		given by $M|\delta_0^P$, then
		$(\tau^P\rest\alpha)_g=M|\alpha$. Moreover,
		$M\ueq P[g]\ueq P[M|\delta_0^P]$.

		\item\label{item:unif_grds_prop_holds}	Property \ref{item:unif_grds} holds.

		\item\label{item:kappa_0^M=least_meas}  $\kappa_0^M$ is the least measurable
		cardinal of ${{\cal M}_\infty}$.
		\item\label{item:kappa_0^+M=delta_infty} $\kappa_0^{+M} = \delta_\infty$.
	\end{enumerate}
\end{lemma}
\begin{proof}
	Part \ref{item:tau_g_gives_M|alpha} is already clear, and part \ref{item:unif_grds_prop_holds}
	is an easy consequence.
	Part \ref{item:kappa_0^M=least_meas}
	is also clear.
	Part \ref{item:kappa_0^+M=delta_infty}
	is by the proof of \cite[Lemma 2.7(b)]{vm1}.\footnote{Actually, an easy cardinality calculation
		shows that $\delta_\infty\leq\kappa_0^{+M}$, and we will show
		directly later that $\delta_\infty$ is Woodin in $\vV_1$ and $\vV_1$ is a ground of $M$
		via a forcing which has the $\delta_\infty$-cc in $\vV$, and hence $\delta_\infty$ is regular in $M$,
		so $\delta_\infty=\kappa_0^{+M}$, without using the proof of \cite[Lemma 2.7(b)]{vm1}.}
\end{proof}

\begin{lemma}\label{lem:i_E(M_infty)_is_iterate}  Let $E\in\es^M$ be $M$-total with $\crit(E)=\kappa_0$ and let $\N=i_E(\M_\infty)$. We have:
	\begin{enumerate}
		\item\label{item:M,M_infty_same_indisc}	$\mathscr{I}^{\M_\infty}=\mathscr{I}^M$, and so $i_{M\M_\infty}\rest\mathscr{I}^{M}=\id=*\rest\mathscr{I}^M$.
		\item  $\N$ is a $\delta_0^\N$-sound $\Sigma_{\M_\infty}$-iterate
		of $\M_\infty$.
		\item\label{item:P,Ult(M,E)_same_indisc} $\mathscr{I}^{\N}=\mathscr{I}^{\Ult(M,E)}$.
		\item\label{item:i_E_is_correct_it_map_on_M_infty} $i^M_E\rest\M_\infty:\M_\infty\to \N$
		is the $\Sigma_{\M_\infty}$-iteration map.
	\end{enumerate}
\end{lemma}
\begin{proof}
	We delay the proof until Lemma \ref{lem:M_infty_of_kappa_0-sound_iterate},
	which is more general.\footnote{It is fine to read
		\ref{lem:P-bar_is_iterate}--\ref{lem:M_infty_of_kappa_0-sound_iterate} at this point,
		which covers what is needed.}
\end{proof}

\begin{lemma}\label{M_knows_how_to_iterate_Minfty_up_to_its_woodin}
	The following are true.
	\begin{enumerate}[label=\tu{(}\alph*\tu{)}]
		\item\label{item:M_infty_strat_thru_delta_0} The restriction of
		$\Sigma_{{{\cal M}_\infty}}$ to trees in $M$ and based on ${{\cal
				M}_\infty}|\delta_\infty$, is lightface definable over $M$.
		\item\label{item:M_infty_strat_thru_delta_0_in_M[g]} Let $\lambda\in\OR$
		and $g$ be in some generic extension of $V$, and be
		${\mathbb P}$-generic over $M$ for some ${\mathbb
			P} \in M|\lambda$.
		Let $\Sigma'_{\M_\infty}$ be the restriction of $\Sigma_{{{\cal M}_\infty}}$
		to
		trees in $M[g]$ and based on ${{\cal M}_\infty}|\delta_\infty$.
		Then $\Sigma'_{\M_\infty}$ is definable over the universe of $M[g]$
		from the parameter
		$x=M|(\lambda^{+\om})^M$, uniformly in $\lambda$.\footnote{\label{ftn:trees_not_in_V}If $g\notin V$,
			we are extending $\Sigma$ and $\Sigma^{\stk}$ canonically to $V[g]$ as in
			Fact \ref{fact:Sigma_properties}\ref{item:Slist_}.}
	\end{enumerate}
\end{lemma}

\begin{proof} The short tree strategy for $\M_\infty$ is computed just like for $M$,
	and the definability is  like in \ref{lem:Sigma_sss_in_M[g]_definability}.
	The computation of branches at maximal stages is just like \cite[Lemma 2.9(a),(b)]{vm1},
	supplemented
	by Lemmas \ref{lem:branch_con} and \ref{lem:i_E(M_infty)_is_iterate},
	and for the definability, use the definability  of $\es^M$ from $x$ in $M[g]$ (see \ref{lem:Sigma_sss_in_M[g]_definability}). Here is a sketch for $g=\emptyset$. Let $E\in\es^M$ be $M$-total, with $\crit(E)=\kappa_0$
	and $\Tt\in M|\lambda(E)$ a maximal tree on $\M_\infty|\delta_\infty$. Let $U=\Ult(M,E)$.
	By Lemma \ref{lem:i_E(M_infty)_is_iterate}, $\N=i_E(\M_\infty)$
	is a $\delta_0^\N$-sound iterate of $\M_\infty$ and
	$i^M_E\rest\M_\infty:\M_\infty\to \N$
	is the correct iteration map.  Now let $P=M(\Tt)$.
	Then $U\sats$``$P$ is a maximal $\Sigma_{\sss}$-iterate of $M|\delta_0^M$'',
	and therefore $U\sats$``$\N|\delta_0^\N$ is a maximal $\Sigma_{P,\sss}$-iterate of $P$'',
	considering statements satisfied by $M$ regarding such iterates. But $U$ is correct about this.
	Let $\Ss$ be the tree on $P$ leading to $\N|\delta_0^\N$.
	Working in a generic extension of $M$, find a $\Tt$-cofinal branch
	$b$ and $\Ss$-cofinal branch $c$ such that $i^M_E\rest(\M_\infty|\delta_\infty)=i^\Ss_c\com i^\Tt_b$, and
	then verify that $b=\Sigma_{\M_\infty}(\Tt)$, using Lemma \ref{lem:branch_con}.\footnote{\label{ftn:find_branch_via_normalization}There is an alternate
		proof which uses properties of normalization and is more direct.
		Let $\Ww\conc d$ be the tree leading from $\M_\infty$ to $\N$
		(with final branch $d$). We have $\Ww,d\in M$ by Lemma \ref{lem:i_E(M_infty)_is_iterate}. But $\Ww$ (of limit length)
		is the normalization of the stack $(\Tt,\Ss)$ (the trees in the first given proof).
		Letting $b,c$ be the correct branches through $\Tt,\Ss\in M$,
		$d$ determines (together with $\Tt,\Ss,\Ww$) the pair $(b,c)$ uniquely
		via normalization.}
\end{proof}

\begin{rem}\label{rem:M_computes_stacks_strategy_for_M_infty}
	The strategy $\Sigma_{\M_\infty}$
	also has minimal hull condensation,
	so we get the canonical stacks strategy
	$(\Sigma_{\M_\infty})^{\stk}$ induced by $\Sigma_{\M_\infty}$,
	which agrees with the tail  strategy $\Gamma_{\M_\infty}=(\Sigma^{\stk})_{\M_\infty}$,
	by Fact \ref{fact:Sigma_properties}.
	It is easily definable from $\Sigma_{\M_\infty}$,
	and for stacks based on $\M_\infty|\delta_\infty$,
	we only need the normal strategy for trees based there.
	So $M[g]$ can also compute the restriction of
	$(\Sigma_{\M_\infty})^{\stk}=\Gamma_{\M_\infty}$
	to stacks based on $\M_\infty|\delta_\infty$, which are in $M[g]$.\end{rem}

Note that we have now verified all of the requirements
for uniform grounds,
excluding  \ref{item:delta_infty-cc}. This will be established
later in Lemma \ref{still-a-woodin-in-v}.

\subsection{The first Varsovian model as $\M_\infty[*]$}\label{sec:vV_1_as_M_infty[*]}
In \S\ref{sec:ground_generation} we defined the elementary
maps $\pi_\infty:\M_\infty\to\M_\infty^{\M_\infty}$
and $\pi_\infty^+:\vV\to\vV^{\M_\infty}$. We now want to show that $\M_\infty^{\M_\infty}$ is an iterate of $\M_\infty$ and $\pi_\infty$
is a correct iteration map. We also want to generalize this to other iterates
of $M$, but in general one must be a little careful.
\begin{dfn}\label{dfn:M_infty^ext^N}
	Given an $\Mswsw$-like premouse $N$, let $\mathscr{D}^N$ and $\M_\infty^N$
	be defined over $N$ just as how $\mathscr{D}$, $\M_\infty$ are defined
	over $M$.
	If $N$ is a correct iterate of $\Mswsw$, also define
	$(\M_\infty^{\ext})_N$ (the external direct limit) relative to $N$, as for $M$: given a maximal tree $\Tt\in\mathbb{U}^N$ (considered as a tree on $N$),
	let $b=\Sigma_N(\Tt)$ and $M_{\Tt}=M^\Tt_b$,
	and let $(\M_\infty^{\ext})_N$ be the direct limit of these models $M_\Tt$
	under the iteration maps (by Lemma \ref{lem:Sigma_sss_in_M[g]_definability},
	and in particular its part \ref{item:Sigma_N,sss}, these trees $\Tt$ are indeed according to $\Sigma_N$).\footnote{These models
		can in general be distinct from the models $\P^N_{M(\Tt)}$
		computed by $N$ via P-construction,
		which is why the care mentioned above is needed.}
	If in fact $M_\Tt=\mathscr{P}^N_{M(\Tt)}$ (the model indexed by $M(\Tt)$ in the covering system $\mathscr{D}^{N}$) for each such $\Tt$, then define $\chi_N:\M_\infty^N\to(\M_\infty^{\ext})_N$ as in \S\ref{sec:ground_generation}.
\end{dfn}

\begin{lem}\label{lem:delta_0-sound_iterate_M_infty}
	Let $N$ be a $\delta_0^N$-sound, non-dropping correct iterate of $M$.
	Then
	$M_\Tt=\mathscr{P}^N(M(\Tt))$ for each $\Tt\in\mathbb{U}^N$,
	$\M_\infty^N=(\M_\infty^{\ext})_N$ and $\chi^N=\id$,
	and $\M_\infty^N$ is a $\delta_0^{\M_\infty^N}$-sound,
	non-dropping correct  iterate of both $M$ and $N$.
\end{lem}

\begin{proof}The proof is just like for $M$, using the $\delta_0^{N}$-soundness of $N$
	(and resulting $\delta_0^{\M_\infty^N}$-soundness of $\M_\infty^N$)
	as a substitute for the fact that
	$M=\Hull_1^M(\mathscr{I}^M)$,
	to see that the models of $\mathscr{D}^N$ really are iterates of $N$.
\end{proof}

We will see later, however, that if $N$ is a $\Sigma$-iterate of $M$
which is $\kappa_0^N$-sound but non-$\delta_0^N$-sound, then $(\M^{\ext}_\infty)_N\neq\M_\infty^N$,
so we need a little more care in this case.

By the lemma, $\M_\infty^{\M_\infty}=(\M_\infty^{\ext})_{\M_\infty}$ is an iterate of $\M_\infty$.
Now recall that $\pi_\infty:\M_\infty\to\M_\infty^{\M_\infty}$
is the union of all
$\pi_{pt,\infty}(\pi_{ps,\infty})$
for embedding-good tuples $(p,s,t)$,
and that $\pi_\infty(x)=x^*$, and
if $\rho\in\OR$ then
\begin{eqnarray}\label{defn_first_star_function}
	\pi_\infty(\rho)=\rho^* = \min( \{ \pi_{N\infty}(\rho) \colon N \in
	\mathscr{F}
	\})=\pi_{Ps,\infty}(\rho),
\end{eqnarray}
where $(P,s)\in\mathscr{F}^+$ is any pair with $\rho\in s$
and $\rho<\max(s)$.
As usual we have:
\begin{lem}\label{lem:pi_infty_is_iteration_map}
	$\pi_\infty:\M_\infty\to\M_\infty^{\M_\infty}$
	 is the iteration map
	according to $\Sigma_{\M_\infty}$.

	Similarly, let $N$ be as in Lemma \ref{lem:delta_0-sound_iterate_M_infty}
	(so $\M_\infty^N$ a $\delta_0^{\M_\infty^N}$-sound correct iterate,
	and likewise $\M_\infty^{\M_\infty^N}$).
	Let $\pi_\infty^N=i_{MN}(\pi_\infty)$. Then
	$\pi_\infty^N$ is the iteration map $\M_\infty^N\to\M_\infty^{\M_\infty^N}$
	according to $\Sigma_{\M_\infty^N}$.
\end{lem}

Write $N=\M_\infty$.
As in \cite{vm1}, since the tree from $N$
to $\M_\infty^N$ is based on $N|\delta_0^N$,
$\pi_\infty:N\to\M_\infty^N$ is determined by
$\pi=\pi_\infty\rest(N|\delta_0^N)$
(as $\pi_\infty$ is the ultrapower map by the extender
derived from $\pi$), which is in turn determined by
the pair
$(\pi_\infty\rest\delta_0^N,\M_\infty^N|\delta_0^{\M_\infty^N})$.
Since $\M_\infty^N|\delta_0^{\M_\infty^N}\in N$,
it follows that
\[
L[\M_\infty,*]=L[\M_\infty,\pi_\infty]=
L[\M_\infty,\pi_\infty\rest\delta_\infty].\]
Moreover, $\pi_\infty$ is definable over this universe
from the predicate $N=\M_\infty$
(given $N$, we can recover $\M_\infty^N$
and the maximal tree $\Tt$ leading from $N$ to $\M_\infty^N$,
but in $L[N,\pi_\infty]$, there is a unique $\Tt$-cofinal
branch $b$ with $M^\Tt_b=\M_\infty^N$;
but $i^\Tt_b=\pi_\infty$).

\begin{dfn}\label{dfn:M_infty[*]}
	The {\em first Varsovian model} of $M$
	(cf.~(\ref{defn_varsovian_model}))
	is the structure
	\begin{eqnarray}\label{defn_first_varsovian_model}
		\M_\infty[*]=(L[\M_\infty,*],\M_\infty,*);
	\end{eqnarray}
	that is, $\M_\infty[*]$ has universe $L[\M_\infty,*]$
	and predicates $\M_\infty$ and $*$.
\end{dfn}

(By the preceding comments, it would suffice
to just have the predicate $\M_\infty$,
but we include $*$ for convenience.)
Later we will introduce a second presentation $\vV_1$
of $\M_\infty[*]$, constructed from a different predicate
(but giving the same universe).
However, the two predicates will be inter-definable
over that universe.

Before we introduce that presentation,
we first develop some properties of
$\M_\infty[*]$ using the presentation above.
We may at times
blur the distinction between the universe $L[\M_\infty,*]$
and the structure $\M_\infty[*]$, but for definability
issues over $\M_\infty[*]$, we can by default use $(\M_\infty,*)$
as a predicate.

\subsection{$\HOD^{M[G]}_{\mathscr{E}}$}\label{subsec:HOD_E}

Up until this point, the paper has  covered
material which is mostly a direct adaptation from that of \cite{vm1}.
But in this section we begin to see some new features.
In \cite{vm1} it is shown that the Varsovian model
has universe that of $\HOD^{\Msw[g]}$, where $g\sub\Coll(\om,{<\kappa})$
is $\Msw$-generic, for $\kappa$ the strong cardinal of $\Msw$.
In this section we will establish an analogue of this fact.

Let $G$ be $(M,\Coll(\om,{<\kappa_0}))$-generic.
Note that \emph{if} $\HOD^{M[G]}$ is the universe of $\M_\infty[*]$,
then it would follow as in \cite{vm1} that $\M_\infty[*]$ is closed
under maximal branches according to $\Sigma_{\M_\infty}$
(those branches are in $M$ by \ref{M_knows_how_to_iterate_Minfty_up_to_its_woodin},
and have length of uncountable cofinality in $M[G]$,
and hence are unique there). Thus, such a fact is at least useful in
verifying that the first
Varsovian model can iterate its own least Woodin cardinal,
which one would like to prove.

Also, in order to proceed to the next step of the mantle analysis,
one might want to consider iteration trees on $\M_\infty[*]$,
based on $\M_\infty[*]|\delta_1^M$ (we will show that $\delta_1^M$
is Woodin in $\M_\infty[*]$).
But because $\M_\infty[*]$ is built from the somewhat cumbersome combination of
$\M_\infty$ and $*$,
the nature of its large cardinals above  $\delta_0^{\M_\infty}$
(and so far also below there, though that is resolved by standard methods) is somewhat unclear, as is its fine structure.

Now if we are to expect $\M_\infty[*]$ to be closed
under $\Sigma_{\M_\infty}$,
a possible goal is to find a presentation of it
as a strategy mouse, built from extenders
and strategy for
$\Sigma_{\M_\infty}$,
with a fine structural hierarchy;
with this target in mind, we write $\vV_1$ for the desired model of this form
(whatever its eventual presentation might be).
The first two authors considered various candidates
for such a presentation, with one possibility
being a construction by level-by-level correspondence between $\vV_1$
and $M$, via a modified kind of P-construction.
This P-construction would result from restricting the extenders
of $M$ to segments of $\vV_1$, starting above some point $\theta$ not too far
above $\kappa_0$. (An early candidate was $\theta=(\kappa_0^{+3})^M$,
but the second and third authors later reduced this to an optimal starting point,
which we use.)
Of course standard P-constructions  build a ground
of the outer model, and this feature was expected,
via a Bukowsky-style forcing
as in \cite{vm1}. Here we use the forcing $\mathbb{L}$
from \S\ref{sec:ground_generation}, which was eventually isolated by the second author.
But  note that a new feature in this  P-construction
would be that some extenders
(those extenders $E$ with $\crit(E)=\kappa_0$)
overlap the forcing $\mathbb{L}$.
This would cause a problem with a standard
P-construction (where the base forcing
is produced by genericity iteration).
But such  extenders
$E$  yield strategy
information, via the process in the proof of \ref{M_knows_how_to_iterate_Minfty_up_to_its_woodin}
(whereas those with $\crit(E)>\kappa_0$ would be as usual).
So it appears that one might construct $\vV_1$ with such a
P-construction, with extenders $E$ having $\crit(E)=\kappa_0$
corresponding to strategy information,
to be added fine structurally to the relevant
segment of $\vV_1$.
This could then lead to a model closed under $\Sigma_{\M_\infty}$ for trees
based on $\delta_\infty$, as is desired.
The model $\vV_1$ should also inherit iterability for itself, from the
iterability of $M$, in a manner similar to standard P-constructions.

A first basic question is whether the model $\vV_1$ constructed as above will end up $\sub\M_\infty[*]$. We now make the key observation which shows that it does.
A useful first step is to restrict one's attention to the action of the extenders on
the ordinals; this will be enough for the P-construction.
For the purpose of the next lemma, let us write
\begin{equation}\label{seq-restr-to-V1}
	{\mathbb F}^M_{>\kappa_0} =
	\{ (\nu,\alpha,\beta)\in\OR^3 \bigm| \nu > \kappa_0^M ,\ \es_\nu^M \not=
	\emptyset ,
	\text{ and } \es_\nu^M(\alpha)=\beta \} {\rm , }
\end{equation}
where, since we are using Jensen indexing, the $M$-extender $\es_\nu^M$ in
(\ref{seq-restr-to-V1}) is
an elementary embedding
$j:M|\mu^{+M|\nu}\to
M|\nu$
where $\mu={\rm crit}(\es_\nu^M)$, so
$(\mathbb{F}^M_{>\kappa_0})_\nu=j\rest\mu^{+M|\nu}$.

In the following lemma, recall that for definability
over $\M_\infty[*]$, we  by default have the predicate
$(\M_\infty,*)$ available for free:

\begin{lemma}\label{restr_of_extenders_are_there}
	${\mathbb F}^{M}_{>\kappa_0}$ is lightface definable over $\M_\infty[*]$.
\end{lemma}
\begin{proof}
	Let $(\nu,\alpha,\beta)\in\OR^3$ with $\nu>\kappa_0$.
	We claim that $(\nu,\alpha,\beta)\in{\mathbb F}^{M}_{>\kappa_0}$
	iff
	\[ F\neq\emptyset\text{
		and }F(\alpha^*)=\beta^*\text{ where }F=F^{\M_\infty||\nu^*},\]
	which suffices. This equivalence holds as for every $P\in\mathscr{F}$,
	we have $F^{P||\nu}\rest\OR=F^{M||\nu}\rest\OR$,
	since $\nu>\crit(F^{M||\nu})\geq\kappa_0$ and
	$P$ results from a P-construction of $M$ above some point below $\kappa_0$.
\end{proof}

One can now proceed directly with the P-construction,
using $(\M_\infty|\kappa_0^{+\M_\infty},*\rest\delta_0^{\M_\infty})$ and $\mathbb{F}^M_{>\kappa_0}$ to define it.
But we
postpone this, and
first establish some other characterizations of the
universe of $\M_\infty[*]$.

The proof of the preceding lemma
can be extended to show that $\M_\infty[*]$ has universe
$\HOD^{M[G]}_{\mathscr{E}}$,
for a certain collection $\mathscr{E}$ of premice in $M[G]$,
and $G$ as above.
Thus, we use $\HOD_{\mathscr{E}}$ here in place of the use of $\HOD$
in \cite{vm1}. We describe how this works next.
What follows is slightly related to some methods from \cite{odle_v2}; also cf.~\ref{grounds-identify-V1}.

\begin{dfn}
	Let $L[\es]$ be a premouse, and let $\mu$ be a strong
	cutpoint
	of $L[\es]$. If $g$ is ${\rm Col}(\omega,\mu)$-generic over $L[\es]$, then
	every extender $\es_\nu$ from $\es$ with $\nu>\mu$ (and hence $\crit(\es_\nu)
	>\mu$) lifts canonically to an extender $\es^g_\nu$
	over $(L[\es]|\nu)[g]$. Let us write $\es^g=\left<\es_\nu^g\right>_{\nu>\mu}$.
	Then $L[\es][g]$ gets reorganized as a premouse over
	$(L[\es]||\mu,g)$ with
	extender sequence $\es^g$;
	so $L[\es][g]\ueq L[\es^g](L[\es]||\mu,g)$.

	Let $L[\es_0]$, $L[\es_1]$ be proper class premice and $\mu\in\OR$.
	Write
	\[ \es_0 \sim^\mu\es_1 \]
	iff $\mu$ is a strong cutpoint
	in both $L[\es_0]$ and $L[\es_1]$ and there are $g_0,g_1$ with $g_i$ being
	${\rm Col}(\omega,\mu)$-generic over $L[\es_i]$ and
	$L_{\mu+\omega}[\es_0][g_0]\ueq L_{\mu+\omega}[\es_1][g_1]$ and $(\es_0)^{g_0}
	=(\es_1)^{g_1}$. So ``$\es_0 \sim^\mu \es_1$'' expresses the fact that above
	$\mu$,
	$\es_0$ and $\es_1$ are intertranslatable: for every $\nu>\mu$, $(\es_0)_\nu
	= (\es_1)_\nu^{g_1} \cap L[\es_0]|\nu$, and vice versa. Write
	$$\es_0 \sim^{<\mu} \es_1$$ iff there is some ${\bar \mu}<\mu$ with
	$\es_0 \sim^{\bar \mu} \es_1$.
	Both $\sim^\mu$ and $\sim^{<\mu}$ are equivalence relations.

	Let $L[\es]$ be a proper class premouse, let $\mu$ be inaccessible in
	$L[\es]$, and let $H$ be ${\rm Col}(\omega,{<\mu})$-generic over $L[\es]$.
	We denote by $\mathscr{E}_\es^{L[\es][H]}$ the
	collection of all $\es'$ such that
	$\es'|{\mu} \in L[\es][H]$ and $L[\es][H]\sats$``$\es' \sim^{<\mu} \es$''.
\end{dfn}

\begin{rem}\label{rem:uniform_def_of_mathscrE}Note that if
	$\es'\in\mathscr{E}_\es^{L[\es][H]}$ then
	\begin{enumerate}[label=(\roman*)]
		\item There are $\bar{\mu}<\mu$
		and generics
		$g,g'\in
		L[\es][H]$
		witnessing that $\es'\sim^{\bar{\mu}}\es$.
		\item $\es'$ is $\Sigma_1$-definable inside $L[\es][H]$ from the set parameter
		$\es'|{\bar \mu}$ and the class parameter $\es$, uniformly
		in $\es',\bar{\mu}$.
		\item There is  $H'$ which is ${\rm Col}(\omega,{<\mu})$-generic over
		$L[\es']$,
		with
		$L[\es'][H']\ueq L[\es][H]$.
		For any such $H'$,
		we have $\mathscr{E}^{L[\es][H]}_{\es}=\mathscr{E}^{L[\es'][H']}_{\es'}$.
		\item  $\mathscr{E}_\es^{L[\es][H]}$ is definable over $L[\es][H]$ from
		$\es,\mu$, uniformly in $\es,\mu,H$.
		\item $\mathscr{E}_\es^{L[\es][H]}$
		is definable over $L[\es][H]$ from
		$\es',\mu$, for all $\es'\in\mathscr{E}^{L[\es][H]}_{\es}$,
		uniformly in $\es,\mu,H,\es'$.
	\end{enumerate}
\end{rem}

Fix now $G\sub{\rm Col}(\omega,<\kappa_0)$ which
is $M$-generic.
We write $\mathscr{E}= \mathscr{E}^{L[\es^M][G]}$.

The following is now immediate.

\begin{lemma}\label{lemma_on_E}
	If $N=L[\es'] \in \msF$, then there is
	$G'$ which is ${\rm Col}(\omega,{<\kappa_0})$-generic over $N$ such that
	$N[G']\ueq M[G]$, and for any such $G'$,
	we have $\mathscr{E}_{\es'}^{L[\es'][G']}=\mathscr{E}$.
\end{lemma}

\begin{theorem}\label{tm:M_infty[*]=HOD_E}
	We have:
	\begin{enumerate}[label=\tu{(}\roman*\tu{)}]
		\item\label{item:M_infty[*]_def_from_mathscrE}
		$\M_\infty[*]$ \tu{(}including its
		predicates\tu{)} is definable over
		$M[G]$ from the parameter $\mathscr{E}$, and
		\item\label{item:M_infty[*]_has_universe_HOD_mathscrE} $\M_\infty[*]$ has
		universe $\HOD_{\mathscr{E}}^{M[G]}$.
	\end{enumerate}
\end{theorem}

\begin{rem}
	This leaves the question: is $\HOD^{M[G]}$ the universe of $\M_\infty[*]$?
\end{rem}

\begin{proof}
	We first verify \ref{item:M_infty[*]_def_from_mathscrE}.
	Write $\es=\es^M$.
	Say that $\es'$ is \emph{$\Mswsw$-like}
	iff $L[\es']$ is $\Mswsw$-like.
	Fix an $\Mswsw$-like $\es'\in\mathscr{E}$.
	We claim that ${\cal M}_\infty$ and
	$*$ are defined over $L[\es']$ in the same manner as over $M$,
	which suffices.
	For the systems $\mathscr{F}^{L[\es']}$ and $\mathscr{F}^M$ have
	cofinally many points in common, which easily suffices.
	To see this fact, use a Boolean-valued comparison argument
	as in the last part of the proof of \cite[Claim 2.11]{vm1}
	(comparison with simultaneous genericity iteration,
	against $L[\es'']$ for various
	$\Mswsw$-like $\es''\in\mathscr{E}_{\es'}^{L[\es'][G']}$).
	Because ``$\Mswsw$-like'' includes short-tree iterability, etc,
	we can indeed form this iteration
	successfully.

	Part \ref{item:M_infty[*]_has_universe_HOD_mathscrE}: By
	\ref{item:M_infty[*]_def_from_mathscrE},
	$\M_\infty[*]\sub\HOD^{M[G]}_\mathscr{E}$.
	We now prove the converse. Let $A$ be a set of ordinals,
	$\varphi$ a formula,  $\alpha\in\OR$ such that
	for every $\xi\in\OR$, $$\xi \in A \Longleftrightarrow M[G] \models
	\varphi(\xi,\alpha,\mathscr{E}).$$
	So for every $\es'\in\mathscr{E}$
	and $G'$ with $L[\es'][G']\ueq M[G]$,
	and every $\xi\in\OR$, we have $$\xi \in A
	\Longleftrightarrow L[\es'][G']\models
	\varphi(\xi,\alpha,\mathscr{E}_{\es'}^{L[\es'][G']})$$
	(since  $\mathscr{E}^{L[\es'][G']}_{\es'}=\mathscr{E}$,
	by Remark \ref{rem:uniform_def_of_mathscrE}).

	Given an $\Mswsw$-like $\es'$,
	write $\dot{\mathscr{E}}^{L[\es']}$ for the natural
	$\Coll(\om,{<\kappa_0^{L[\es']}})$-name
	for $\mathscr{E}^{L[\es'][G']}_{\es'}$ (for $G'$
	the $\Coll(\om,{<\kappa_0^{L[\es']}})$-generic filter;
	cf.~the remarks
	on the uniform definability of $\mathscr{E}^{L[\es'][G']}_{\es'}$ above).

	Let $\xi\in\OR$. Pick
	$N=L[\es'] \in {\mathscr F}$ with $\xi^* = \pi_{N{\cal
			M}_\infty}(\xi)$ and $\alpha^*=\pi_{N\cal {{\cal M}_\infty}}(\alpha)$, and let
	$G'$ be as in Lemma \ref{lemma_on_E} for $N$. Then
	\begin{eqnarray*}
		\xi \in A & \Longleftrightarrow & M[G] \models
		\varphi(\xi,\alpha,\mathscr{E})\\
		&\Longleftrightarrow& N[G'] \models
		\varphi(\xi,\alpha,\mathscr{E}^{N[G']}_{\es'}) \\
		{} & \Longleftrightarrow & \Vdash^N_{{\rm Col}(\omega,<\kappa_0)}
		\varphi({\check \xi},{\check \alpha}, \dot{\mathscr{E}}^N)\\
		{} & \Longleftrightarrow & \Vdash^{{\cal M}_\infty}_{{\rm
				Col}(\omega,<\kappa_0^{{\cal M}_\infty})} \varphi({\check
			\xi^*},\check{\alpha^*},\dot{\mathscr{E}}^{{\cal M}_\infty}).
	\end{eqnarray*}
	Therefore $A \in {\cal M}_\infty[*]$.
\end{proof}

\subsection{Uniform grounds}

Recall that $\delta_\infty=\delta^{\M_\infty}$ is the least Woodin of $\M_\infty$. The following lemma completes the proof that the first
direct limit system for $M$ provides uniform grounds
(\S\ref{ground_generation_stuff}):
\begin{lemma}\label{still-a-woodin-in-v}
	We have:
	\begin{enumerate}[label=\tu{(}\alph*\tu{)}]
		\item\label{item:sys_1_V_delta_infty_preserved}
		$V_{\delta_\infty}^{\M_\infty[*]} = V_{\delta_\infty}^{{{\cal M}_\infty}}$.
		\item\label{item:sys_1_delta_infty_Woodin_in_M_infty[*]} $\delta_\infty$ is
		\tu{(}the least\tu{)} Woodin in $\M_\infty[*]$.
		\item\label{item:V_delta_and_Woodin_from_gg}
		Property \ref{item:delta_infty-cc} of uniform grounds holds for $\vV_1,\delta_\infty$;
		that is, $\vV_1\sats$``$\delta_\infty$ is
		regular and $\BB_\infty$ is
		$\delta_\infty$-cc''. Moreover, $\vV_1\sats$``$\BB_\infty$ is a complete Boolean algebra''.
	\end{enumerate}
\end{lemma}

\begin{proof} \ref{item:sys_1_V_delta_infty_preserved}: The usual considerations
	show that $* \upharpoonright
	\eta \in {{\cal M}_\infty}$ for every $\eta<\delta_\infty$.
	Combined with the proof  of Theorem \ref{tm:M_infty[*]=HOD_E},
	this suffices. (Cf.\ \cite[Claim 2.12 (b)]{vm1}.)

	\ref{item:sys_1_delta_infty_Woodin_in_M_infty[*]}: By the proof of
	\cite[Theorem 2.19]{vm1} or of \cite[Claim 13]{Theta_Woodin_in_HOD}.

	\ref{item:V_delta_and_Woodin_from_gg}:
	Property \ref{item:delta_infty-cc} holds because $\delta_\infty$ is Woodin in $\M_\infty[*]$
	and $\BB_\infty$ is the extender algebra.
	The ``moreover'' clause follows from this and \ref{item:sys_1_V_delta_infty_preserved}.
\end{proof}

So by Theorem \ref{tm:Varsovian_is_ground}, $\M_\infty[*]$ is a ground
for $M$, and in fact as in its proof, there is some $M$-stable $\mu\in\OR$
such that $M|\mu$ is $(\M_\infty[*],\mathbb{L})$-generic,
where $\mathbb{L}=\mathbb{L}(\mu^*)$, and $\M_\infty[*][M|\mu]\ueq M$.
We will actually refine this result in
Lemma \ref{V_is_a_ground}.

We can immediately deduce the following corollary, which however
will be extended in Lemma \ref{V_1_knows_how_to_iterate_Minfty_up_to_its_woodin}:

\begin{cor}\label{cor:M_infty[*]_closed_under_maximal_branches}
	For all maximal trees $\Tt\in\M_\infty[*]$ via $\Sigma_{\M_\infty}$, based on $\M_\infty|\delta_\infty$, we have $b=\Sigma_{\M_\infty}(\Tt)\in\M_\infty[*]$,
	and $b$ is the unique $\Tt$-cofinal branch in $\M_\infty[*]$.
\end{cor}
\begin{proof} By Lemma \ref{M_knows_how_to_iterate_Minfty_up_to_its_woodin}, $b\in M$.
	And by Lemma \ref{key_facts_about_V1}, $\delta_\infty$ is regular in $M$ and in $M[G]$,
	whenever $G\sub\Coll(\om,{<\kappa_0})$ is $M$-generic.
	Therefore $M[G]$ contains no other $\Tt$-cofinal branch.
	Since $\M_\infty[*]$ is the universe
	of $\HOD^{M[G]}_{\mathscr{E}}$ (with notation as in Theorem \ref{tm:M_infty[*]=HOD_E}),
	and $\Tt\in\M_\infty[*]$, it follows that $b\in\M_\infty[*]$ also.
\end{proof}

\subsection{The structure $\mathcal{Q}$}\label{subsec:Q}

Let $U=E_\nu^{{{\cal M}_\infty}}$ be the least total measure on the ${{\cal
		M}_\infty}$-sequence with critical
point $\kappa_0^{{{\cal M}_\infty}}$. Fix a natural tree ${\cal T} \in
{\mathbb U}^{{\rm ult}({{\cal M}_\infty};U)}$ with $\delta({\cal
	T})=\kappa_0^{+{{\cal M}_\infty}}$
and which makes ${{\cal M}_\infty}|\kappa_0^{+{{\cal M}_\infty}}$ generic,
after  iterating the least measurable out to $\kappa_0^{\M_\infty}$. As ${\cal
	T}$
lives on $\M_\infty |\delta_0^{{{\cal M}_\infty}}$, we may and shall construe ${\cal T}$
as a tree on ${{\cal M}_\infty}$ rather
than on ${\rm ult}({{\cal M}_\infty};U)$. If $b=\Sigma_{{{\cal M}_\infty}}({\cal
	T})$,
then $$\mathscr{P}^{{\rm ult}({{\cal M}_\infty};U)}({\cal M}({\cal T}))={\cal
	M}_b^{\cal T}$$ and $\delta({\cal T})=
\pi_{0,b}^{\cal T}(\delta_0^{{{\cal M}_\infty}})$. Also,
\begin{eqnarray}\label{ext=ult}
	\mathscr{P}^{{\rm ult}({{\cal M}_\infty};U)}({\cal M}({\cal T}))[{{\cal
			M}_\infty}|\theta_0^{{{\cal M}_\infty}}]={\rm ult}({{\cal M}_\infty};U).
\end{eqnarray}
Let us write \begin{eqnarray}\label{defn_i}
	\iota = \pi_{0,b}^{\cal T} \upharpoonright \delta_0^{{\cal M}_\infty}.
\end{eqnarray}
By Corollary \ref{cor:M_infty[*]_closed_under_maximal_branches},
$\iota \in \M_\infty[*]$.
Hence
\begin{eqnarray}\label{i_is_in_V}
	{\cal Q}=({{\cal M}_\infty}|\kappa_0^{+{\cal M}_\infty};\iota)
\end{eqnarray}
is an amenable structure and is an element of $\M_\infty[*]$.

\begin{lemma}[\textbf{Soundness of ${\cal Q}$}]\label{soundness}  ${\cal Q}={\rm
		Hull}^{\cal Q}(\iota)$.
\end{lemma}

\begin{proof} Let $$\sigma \colon {\bar {\cal Q}} = ({\bar M};{\bar \iota})
	\cong X =
	{\rm Hull}^{\cal Q}(\iota) \prec {\cal Q}.$$
	As $\iota = \{ (\xi,\eta) \colon \xi<\delta_0^{{{\cal M}_\infty}} \wedge
	\eta=\iota(\xi) \}$, $\delta_0^{{{\cal M}_\infty}} \subset X$, so that $\sigma
	\upharpoonright \delta_0^{{{\cal M}_\infty}}
	= {\rm id}$ and ${{\cal M}_\infty}|\delta_0^{{{\cal M}_\infty}} \triangleleft
	{\bar M}$.
	Let ${\bar {\cal T}}$ be defined over ${\bar M}$ as ${\cal T}$ was defined
	over ${{\cal M}_\infty}|\theta_0^{{{\cal M}_\infty}}$. The tree
	${\bar T}$ is on ${\bar M}$ which lives on ${{\cal M}_\infty}|\delta_0^{{{\cal
				M}_\infty}}$, but we may and shall construe ${\bar T}$
	as a tree on ${{\cal M}_\infty}$, and as such ${\bar {\cal T}}$ is according to
	$\Sigma_{{{\cal M}_\infty}}$.

	Let ${\bar b} = \Sigma_{{{\cal M}_\infty}}({\bar {\cal T}})$.
	By branch condensation Lemma \ref{lem:branch_con}, ${\bar b}$ is the pullback of $b$ via $\sigma$.
	We will have that ${\bar \iota} = \pi_{0,{\bar b}}^{\bar {\cal T}}
	\upharpoonright
	\delta_0^{{{\cal M}_\infty}}$ and
	there is a canonical elementary embedding $${\hat \sigma} \colon {\cal M}_{\bar
		b}^{{\bar {\cal T}}} \rightarrow {\cal M}_b^{\cal T}$$
	defined by $${\hat \sigma}(\pi_{0,{\bar b}}^{\bar {\cal T}}(f)(a)) =
	\pi_{0,{b}}^{ {\cal T}}(f)(\sigma(a)){\rm , }$$
	where $f \in {{\cal M}_\infty}$ and $a \in [{\rm OR} \cap {\bar M}]^{<\omega}$.
	We will have that ${\cal M}({\bar {\cal T}}) = {\rm dom}(\sigma) \cap {\rm
		dom}({\hat
		\sigma})$ and $\sigma$ and ${\hat \sigma}$ agree on this common part of their
	domains.

	If $E_\nu^{{\cal M}({\bar {\cal T}})}$ is a total extender from the
	$\M(\bar{\Tt})$-sequence, then by the elementarity of $\sigma$,
	${\bar M}$ satisfies all the axioms of the extender algebra of ${\cal M}({\bar
		{\cal T}})$
	given by $E_\nu^{{\cal M}({\bar {\cal T}})}$,
	as ${{\cal M}_\infty}|\theta_0^{{{\cal M}_\infty}}$ satisfies all the axioms of
	the extender algebra of ${\cal M}({\cal T})$ given by $E_{\sigma(\nu)}^{{{\cal
				M}_\infty}}$. We may conclude that
	${\bar M}$ is generic over ${\cal M}_{\bar b}^{{\bar {\cal T}}}$.
	If $\bar{g}$ is the associated generic over ${\cal M}_{\bar b}^{{\bar {\cal
				T}}}$
	and $g$ that over ${\cal M}_b^{\cal T}$,
	then ${\hat \sigma} \mbox{``}\bar{g}
	= \sigma \mbox{``} \bar{g} \subset g$,
	and hence we may lift ${\hat \sigma}$ to an elementary embedding $${\hat
		\sigma}^* \colon
	{\cal M}_{\bar b}^{\bar {\cal T}}[{\bar M}] \rightarrow
	{\cal M}_b^{\cal T}[{{\cal M}_\infty}|\kappa_0^{+{{\cal M}_\infty}}] = {\rm
		ult}({{\cal M}_\infty};U){\rm , }$$
	defined by ${\hat \sigma}^*(\tau^{\bar{g}})= ({\hat \sigma}(\tau))^{g}$.

	But notice that $\delta_0^{{{\cal M}_\infty}} \cup \{ {\rm crit}(U) \} \subset
	{\rm ran}({\hat
		\sigma}^*)$, so that by the $\delta_0^{{{\cal M}_\infty}}$-soundness of ${{\cal
			M}_\infty}$ and
	by the choice of $U$ as a measure, ${\hat
		\sigma}^*$ must be the identity,
	and therefore so is $\sigma$.
\end{proof}

\begin{corollary}\label{cor:size_of_gamma}\
	\begin{enumerate}
		\item[(a)] $\kappa_0^{+{{\cal M}_\infty}}$
		has size $\delta_0^{\vV_1}$
		inside $\M_\infty[*]$.
		\item[(b)] $\kappa_0^{++{{\cal M}_\infty}}=\kappa_0^{++M}$.
	\end{enumerate}
\end{corollary}

\begin{proof}
	(a) follows from Lemma \ref{soundness} together with (\ref{i_is_in_V}).
	(b) is then immediate by $\kappa_0^{++M} \in {\rm Card}^{{{\cal
				M}_\infty}}$.
\end{proof}

This corollary should be compared with Lemma \ref{local-definability-of-that-structure}, to come.

\subsection{The ${\kappa_0}$-mantle of $M$}\label{subsec:kappa_0-mantle}

We now give another  characterization of the universe of $\M_\infty[*]$,
though no results outside of this subsection actually depend on it.

The following definitions are essentially
taken from \cite{usuba_extendible}, though the notation and terminology is different.
If $W$ is an inner model and $\lambda$ is a cardinal of $W$,
then ${\bar W} \subset W$ is a $\lambda$ {\em ground of} $W$ iff
${\bar W}$ is an inner model of $W$ and
there is some poset ${\mathbb P} \in {\bar W}$ which has size $\lambda$ in
${\bar W}$ and
some $g$ which is ${\mathbb P}$-generic over ${\bar W}$ such that $W={\bar
	W}[g]$.
${\bar W}$ is called a $<\lambda$ {\em ground of} $W$ iff there is some
${\bar \lambda}<\lambda$ such that ${\bar W}$ is a ${\bar \lambda}$ ground
of $W$. The $<\lambda$ {\em mantle of} $W$ is the intersection of all $<\lambda$
grounds of $W$. We write $\mathbb{M}^W_{<\lambda}$ for the ${<\lambda}$-mantle
of $W$.

The main result of this subsection is
that $\M_\infty[*]$ has universe $\mathbb{M}^M_{<\kappa_0}$;
see \ref{prop:vV_1_is_<kappa_0-mantle}.
The following fact and its proof are similar to parts of
\cite[***Lemmas 3.11 and 4.1]{V=HODX_pub}
and \cite[***Lemmas 5.4 and 5.12]{odle_v2};
in particular, we make use of the condensation stacks from \cite{V=HODX_pub}.

Let $\mu<\kappa_0$ be a cardinal strong
cutpoint of $M$, and $\eta=\mu^{+M}$.
Let $g$ be $\Coll(\omega,\mu)$-generic over $M$.
Let $\mathscr{P}^{M[g]}_M$ denote
the set of all $N\in M[g]$  such that:
\begin{enumerate}[label=--]
	\item
	$N$ is a premouse of ordinal height $\eta$,
	\item  $N$
	has a largest cardinal $\lambda<\eta$,
	and $\lambda$ is a strong cutpoint of $N$,
	\item
	$M[g]\sats$``$N$ is fully iterable above $\lambda$''
	\item there is a proper class premouse
	$N'$ with $N\pins N'$
	and $M[g]\sats$``$\es^{N'} \sim^{<\eta}\es^M$''.
\end{enumerate}
Note that because of the restriction on $\es^{N'}$
above $\eta$, we don't actually need to quantify over proper classes here,
and clearly $\mathscr{P}_M^{M[g]}$ is definable over $M[g]$
from the class $M$. We now refine this fact:

\begin{lemma}\label{grounds-identify-V1}
	Let $g$ be $\Coll(\omega,\mu)$-generic over $M$. Then (i)
	$\mathscr{P}^{M[g]}_M$ is definable over $M[g]$ from no parameters, uniformly
	in $\mu,g$.
	Further, (ii) for all $N,N'$ as in the definition of $\mathscr{P}^{M[g]}_M$,
	$N'$ is definable over $M[g]$ from the parameter
	$N$, uniformly in $\mu,g,N,N'$.
\end{lemma}

\begin{proof}
	Note that $\eta=\om_1^{M[g]}$ and
	$\mathrm{HC}^{M[g]}$ is the universe of $(M|\eta)[g]$.
	Now working in  $M[g]$, let $\mathscr{P}'$ be the set
	of all premice $N$ of height $\eta$ such
	that for some $h$,
	\begin{enumerate}[label=--]
		\item $N$ has a largest cardinal, $\lambda$, which is a strong cutpoint
		of $N$,
		\item $N$  is fully iterable above $\lambda$,
		\item $h$ is $\Coll(\omega,\lambda)$-generic over $N$
		and $N[h]$ has universe
		${\rm HC}$,
		\item the condensation stack $(N[h])^+$ above $N[h]$
		(relativized as a premouse over $(N|\lambda,h)$)
		is well-defined, hence is proper class with universe $V$.\footnote{Recall
			we are working in $M[g]$,
			so ``$V$'' refers to $M[g]$ here. Because
			we relativize over $(N|\lambda,h)$, $N[h]$
			plays the role that $P|\om_1^P$ plays in the unrelativized condensation stack
			for a premouse $P$.}
	\end{enumerate}

	For $N\in\mathscr{P}'$, we write $\lambda^N$ for the largest
	cardinal of $N$, and with $h$ as above, we consider $N[h]$ as an
	$(N|\lambda^N,h)$-premouse. Now $\mathscr{P}'\neq\emptyset$,
	and in fact $M|\eta\in\mathscr{P}'$, as witnessed by
	$g$. For by \cite{V=HODX_pub}, the condensation stack above $(M|\eta)[g]$,
	as computed in $M[g]$, is just $M[g]$ (arranged as a premouse
	over $(M|\mu,g)$).

	We will show that $\mathscr{P}^{M[g]}_M=\mathscr{P}'$, which gives part (i).
	The first direction is basically as in \cite[Lemma 5.4 part 1]{odle_v2}:

	\begin{clm}$\mathscr{P}^{M[g]}_M\sub\mathscr{P}'$.\end{clm}
	\begin{proof}
		Let $N\in\mathscr{P}^{M[g]}_M$, as witnessed by $N'$.
		Everything is clear enough except for the fact that, in $M[g]$,
		the condensation stack $N[h]^+$ is well-defined.
		But $N'[h]$ has universe that of $M[g]$,
		and there is $\alpha<\eta$ such that $N'[h]$ is iterable above $\alpha$ (in
		$V$), and we can take $\alpha$ here such that $N|\alpha$ projects to
		$\lambda^N$. By this iterability,
		$N'[h]$, \emph{when considered
			as an $(N|\alpha,h)$-premouse},
		therefore is just the condensation stack above $N[h]$.
		But
		$N'[h]|\eta=N[h]$ is also iterable in $M[g]$ (above $\lambda^N$;
		that is, as an
		$(N|\lambda^N,h)$-premouse).
		So $N[h]$ satisfies all standard condensation facts (as an
		$(N|\lambda^N,h)$-premouse).
		So we can argue as in the proof of \cite[Lemma 5.4 part 1]{odle_v2}
		to see that $N'[h]$, \emph{when considered as an $(N|\lambda^N,h)$-premouse},
		is also the condensation stack
		above $N[h]$, as computed in  $M[g]$, as desired.
	\end{proof}

	\begin{clm}\label{clm:N_in_P'_has_convergence_pt} Let $N \in \mathscr{P}'$,
		witnessed by $h$. Then there is
		$\alpha$ such that $\mu,\lambda^N<\alpha<\eta$
		and $M[g]||\alpha$ and $N[h]||\alpha$
		have the same universe, with
		largest cardinal $\om_1^{M[g]||\alpha}<\alpha$,
		and the two structures $M[g]||\alpha$ and $N[h]||\alpha$ are inter-definable
		from parameters
		and project $\leq\om_1^{M[g]||\alpha}$.
	\end{clm}
	\begin{proof}
		We basically follow the proof of \cite[Lemma 3.11]{V=HODX_pub}.
		Let $M'$ be $M[g]$ arranged as a premouse
		over $(M|\mu,g)$,
		and $N'$ be the condensation stack over $N[h]$ as computed in $M[g]$.

		So $M',N'$ both have universe $M[g]$,
		and in particular, $\eta=\om_1^{M'}=\om_1^{N'}$ and
		$(\her_{\om_2})^{M'}=(\her_{\om_2})^{N'}$.
		We first find $\alpha'<\om_2^{M[g]}$
		such that $M'||\alpha'$ and $N'||\alpha'$ have the same universe
		and are inter-definable from parameters (and some more).
		Define $\left<M_n,N_n\right>_{n<\om}$ as follows.
		Let $M_0=M'||\eta$ and $N_0=N'||\eta$.
		Now given $M_n,N_n$,
		let $N_{n+1}$ be the least $S\pins N'$ such that
		$N_n\pins S\pins N'$ and $M_n\in S$ and $\rho_\om^S=\eta$.
		Define $M_{n+1}$ symmetrically.

		Let $\widetilde{M}=\stack_{n<\om}M_n$ and $\widetilde{N}=\stack_{n<\om}N_n$.
		Note that $\widetilde{M},\widetilde{N}$ have the same universe
		$U$, which has largest cardinal $\eta$, and therefore
		$\widetilde{M}\pins M'$ and $\widetilde{N}\pins N'$
		(that is, $M'||\OR^U$ and $N'||\OR^U$ are both passive,
		as $\eta$ is a strong cutpoint of $M',N'$).
		Note that $\widetilde{M}$ is definable from the parameter
		$M_0$ over $U$; in fact, $\widetilde{M}$ is the Jensen stack
		above $M_0$ there; cf.~\cite{V=HODX_pub}.
		It follows that $\widetilde{M}$ is $\Sigma_1^U(\{M_0\})$.
		Likewise for $\widetilde{N}$
		and $N_0$. Using this, note that
		also $\left<M_n,N_n\right>_{n<\om}$
		is $\Sigma_1^U(\{(M_0,N_0)\})$,
		and  so $\rho_1^{\widetilde{M}}=\rho_1^{\widetilde{N}}=\eta$.

		Now as in \cite{V=HODX_pub},
		we can find $\bar{\eta}<\eta$ and $x\in U$ such that
		$\mu,\lambda^N<\bar{\eta}$ and
		the hulls $\Hull_1^{\widetilde{M}}(\bar{\eta}\cup\{x\})$
		and
		$\Hull_1^{\widetilde{N}}(\bar{\eta}\cup\{x\})$
		have the same elements, and letting  $\bar{M},\bar{N}$
		be the transitive collapses
		and $\pi:\bar{M}\to\widetilde{M}$
		and $\sigma:\bar{N}\to\widetilde{N}$
		the uncollapse maps
		(so $\bar{M},\bar{N}$ have the same universe $\bar{U}$
		and $\pi,\sigma$ the same graphs),
		and such that  $M_0,N_0\in\rg(\pi)$ and $\crit(\pi)=\bar{\eta}$ and
		$\pi(\bar{\eta})=\eta$ and
		$\bar{M},\bar{N}$ are $1$-sound with
		$\rho_1^{\bar{M}}=\bar{\eta}=\rho_1^{\bar{N}}$,
		so $\bar{M}\pins M'$ and $\bar{N}\pins N'$.
		Note that the
		$\Sigma_1^U(\{N_0\})$ definition
		of $\widetilde{N}$
		reflects down to
		a $\Sigma_1^{\bar{U}}(\{\bar{N}_0\})$ definition
		of $\bar{N}$,
		where $\pi(\bar{N}_0)=N_0$.
		Likewise vice versa.
		Therefore $\bar{M},\bar{N}$ have the same universe
		and are inter-definable from
		parameters.
		So letting $\alpha=\OR^{\bar{M}}=\OR^{\bar{N}}$, we are done.
	\end{proof}

	\begin{clm}$\mathscr{P}'\sub\mathscr{P}^{M[g]}_M$.\end{clm}
	\begin{proof}Let $N,h,\alpha$ be as in Claim
		\ref{clm:N_in_P'_has_convergence_pt}.
		Then note that $\alpha$ is a strong cutpoint of $M$
		and of $N$ and of the condensation stack $N[h]^+$
		above $N[h]$ (as an $(N|\lambda^N,h)$-premouse,
		as computed
		in $M[g]$).
		So
		the iterability of $M|\eta$ and $N$ above $\alpha$
		(in $M[g]$)
		implies that
		$(\es^M_\nu)^g = (\es_\nu^N)^h$
		for each $\nu>\alpha$.
		Thus, the extender sequences of $M|\eta$ and $N$ are intertranslatable
		(modulo a generic) above $\alpha$.
		So by a proof almost identical to \cite[***Lemma 5.4 part 1]{odle_v2},
		$N[h]^+$
		and $M[g]$ (as an $(M|\mu,g)$-premouse)
		have the same extender sequence
		above $\alpha$. Now let $N^+$
		be the result of the
		P-construction
		of $N[h]^+$ above $N$. Because $\eta=\om_1^{N[h]^+}$,
		this works fine structurally, giving a proper class premouse $N^+$
		extending $N$.
		But since $N[h]^+$ and $M[g]$ agree above $\alpha$,
		it follows that
		$\es^{N^+}
		\sim^{<\eta} \es^M$. So $N\in\mathscr{P}$.
	\end{proof}

	This proves part (i).
	For (ii),
	working in $M[g]$, given $N\in\mathscr{P}^{M[g]}_M=\mathscr{P}'$,
	we first define the condensation stack $N[h]^+$
	above $N[h]$, and then the P-construction $N$ of $N[h]^+$
	above $N$, which gives the desired $N'$.
\end{proof}

\begin{lemma}\label{Fscr_is_dense_in_grounds}$\mathscr{F}$ is dense in the
	${<\kappa_0}$-grounds of $M$,
	so $\bigcap \mathscr{F}=\mathbb{M}_{<\kappa_0}^M$.
\end{lemma}
\begin{proof}
	Let $\mu<\kappa_0$
	be a regular cardinal strong cutpoint of $M$, and let $W$ be a $<\mu$-ground of
	$M$,
	via a generic filter $k$ (so $W[k]\cong M$).
	Let $g$ be $(M,\Coll(\om,\mu))$-generic and $h$ be
	$(W,\Coll(\om,\mu))$-generic,
	with $W[h]\cong M[g]$. Let $\widetilde{W}$ be the model produced by
	the P-construction of $M$
	over  $(\her_{\mu^+})^W$.
	By Lemma \ref{grounds-identify-V1}, $\widetilde{W}\sub W$
	and $\widetilde{W}$ is definable from parameters over $W$.
	And $\widetilde{W}$ is a ground of $M$; in fact $\widetilde{W}[k]\cong M$,
	as $M|(\mu^+)^M$ is definable over
	$(\her_{\mu^+})^M=(\her_{\mu^+})^W[k]$
	from the parameter $M|\mu$, via the Jensen stack.\footnote{It follows by the
		standard forcing argument that $W$ is the actually the universe of
		$\widetilde{W}$.}

	Now working in $\widetilde{W}$, we can compute some $N\in\mathscr{F}$
	by forming a Boolean-valued comparison/genericity iteration above $(\mu^+)^W$,
	to compute $N|\delta_0^N$, and then using P-construction to compute the rest
	of
	$N$.
\end{proof}

\begin{theorem}\label{prop:vV_1_is_<kappa_0-mantle}
	$\M_\infty[*]$ has universe
	$\bigcap {\mathscr F}=\mathbb{M}_{<\kappa_0}^M$.
\end{theorem}

\begin{proof}
	We already know
	$\M_\infty[*]\sub\bigcap\mathscr{F}=\mathbb{M}_{<\kappa_0}^M$
	by Lemma
	\ref{Fscr_is_dense_in_grounds}.

	Let us verify $\mathbb{M}^M_{<\kappa_0}\sub\M_\infty[*]$.\footnote{The original argument
		used for the proof that $\mathbb{M}^M_{<\kappa_0}\sub\M_\infty[*]$,
		found by the 3rd author,
		was slightly different; that argument is sketched for the analogous
		result Corollary \ref{cor:HOD^M[G],kappa_1-mantle}. The 2nd author then adapted that one to yield the one presented here. In either form, the argument is related to Usuba's ZFC proof of the fact that if $\kappa$ is extendible then $\mathbb{M}_{<\kappa}=\mathbb{M}$.
		Related arguments have since been used by the second author in \cite{ralf_kappa_mantle} and the third author in \cite{choice_principles_in_local_mantles}, \cite{local_mantles_of_Lx_v2}.}
	Let  $U\in\es^M$ be the order $0$ total measure on $\kappa_0$. Let
	\[ j \colon M \rightarrow M'=\Ult(M,U)\] be the ultrapower map. For
	$P \in {\mathscr F}$, $i_{MP}(U)=U \cap P\in\es^P$ is the order $0$ total
	measure on $\kappa_0$ in $P$. Let $j^P:P\to\Ult(P,i_{MP}(U))$ be the ultrapower map.
	Note  $j\rest\OR=j^P\rest\OR$.
	Let $U_\infty=i_{M\infty}(U)$ and $j_\infty:\M_\infty\to\M'_\infty=\Ult(\M_\infty,U_\infty)$
	be the ultrapower map.

	Arguing much as in the proof of Lemma \ref{restr_of_extenders_are_there},
	$j \upharpoonright
	{\rm OR}$ is
	definable over $\M_\infty[*]$:
	for all $\eta,\xi\in\OR$ and  $P\in\mathscr{F}$, we have
	\[
	\eta = j(\xi) \Longleftrightarrow \eta=j^P(\xi)\Longleftrightarrow \eta^* = j_\infty(\xi^*).
	\]
	(For the second equivalence, just take $P$ such that $\eta,\xi$ are $P$-stable.)

	Now let $X\in\mathbb{M}_{<\kappa_0}^M$ be a set of ordinals.
	By the preceding paragraph,
	\begin{equation}\label{eqn:j_rest_sup(X)_in} j \upharpoonright \sup(X) \in \M_\infty[*].\end{equation}
	By elementarity, $j(X)\in\mathbb{M}^{M'}_{<j(\kappa_0)}$.
	It is straightforward to verify that $\M_\infty'$ is the
	direct limit of
	a system $\mathscr{F}'$ of uniform grounds of $M'$ in much the same way as ${{\cal
			M}_\infty}$ is the
	direct limit
	of the system ${\mathscr F}$ of uniform grounds of $M$;
	here the models $P'\in\mathscr{F}'$
	are exactly those of the form $P'=\Ult(P,i_{MP}(U))$
	for some $P\in\mathscr{F}$.
	So by Theorem \ref{tm:Varsovian_is_ground},
	$$Q = \M_\infty'[*]$$
	is a $<j(\kappa_0)$-ground of $M'$.
	So $j(X) \in Q$, but note $Q \sub \M_\infty[*]$, so $j(X)\in\M_\infty[*]$,
	but then using line (\ref{eqn:j_rest_sup(X)_in}) we get $X\in\M_\infty[*]$, since
	\[ \alpha\in X\Longleftrightarrow j(\alpha)\in j(X).\qedhere\]
\end{proof}

\subsection{The first Varsovian model as the strategy mouse $\vV_1$}\label{subsec:vV_1}

We will now give another presentation
of $\M_\infty[*]$, as a  strategy mouse
$\vV_1$ in a fine
structural hierarchy
$\left<\vV_1||\nu\right>_{\nu\in\OR}$,
as sketched at the beginning
of \S\ref{subsec:HOD_E}. To motivate this, first notice the following.

A routine first observation is:

\begin{lemma}\label{lem:F^P_from_F^P_rest_OR}
	Let $P$ be an active (Jensen indexed) premouse.
	Then $F^P$ is $\Sigma_1$-definable over $(P^\passive,F^P\rest\OR)$,
	uniformly in $P$.
\end{lemma}

We remark that it is important here that we have the universe and (internal)
extender sequence $\es^P$ of $P$ available.

Let
${\mathbb L}=\mathbb{L}^{\M_\infty[*]}(\kappa_0)$ for the poset
of Definition \ref{defn_bukowsky-poset}, for adding generic subset of $\kappa_0$.

\begin{lemma}\label{V_is_a_ground}
	$M|\kappa_0$ is ${\mathbb L}$-generic over
	$\M_\infty[*]$ and
	$\M_\infty[*][M|\kappa_0] \ueq M$.
\end{lemma}

\begin{proof} Since we have verified the properties
	for uniform grounds, that $M|\kappa_0$ is ${\mathbb L}$-generic over
	$\M_\infty[*]$ follows from
	\S\ref{ground_generation_stuff}.
	We aim to show that $\M_\infty[*][M|\kappa_0] \ueq M$ by performing a
	``$P$-construction.''

	We identify the sequence
	$\left<M|\nu, M||\nu \right>_{\nu \in\OR}$ inside
	$\M_\infty[*][M|\kappa_0]$,
	from the parameter $M|\kappa_0$ and classes $\M_\infty,*$, as
	follows. We start with $M|\kappa_0$ given. Fix $\nu$ with $\kappa_0 < \nu \in
	{\rm OR}$.
	The sequence $\left<M||\beta\right>_{\beta<\nu}$ determines $M|\nu$.
	By Lemma \ref{restr_of_extenders_are_there}, $\M_\infty[*]$ knows
	whether
	$\es_\nu^M\neq\emptyset$, and if $\es_\nu^M\neq\emptyset$, knows
	$\es_\nu^M
	\upharpoonright\OR$,
	uniformly in $\nu$ (definably from $\M_\infty,*$). But from the pair
	$( M|\nu , \es_\nu^M \rest\OR)$ we can compute $\es_\nu^M$,
	also uniformly in $\nu$, by Lemma \ref{lem:F^P_from_F^P_rest_OR}.
\end{proof}

\begin{dfn}\label{dfn:vV_1}
	We now define a class structure $\vV_1$, structured analogously to a premouse,
	built from a sequence $\es^{\vV_1}=\left<\es^{\vV_1}_\nu\right>_{\nu\in\OR}$
	of
	extenders.
	However, some of the extenders will be (properly) long, and will not cohere the sequence. We write
	$\vV_1||\nu$ and $\vV_1|\nu$ with the usual meaning.
	For those segments $\vV_1||\nu$ active with long extenders,
	some of the premouse axioms will fail (in particular, coherence).\footnote{The
		segments $\vV_1||\nu$ where
		$\nu=\gamma^{\vV_1}$ or [$\nu>\gamma^{\vV_1}$ and $\crit(\es^M_\nu)=\kappa_0$]
		do not satisfy the usual premouse axioms with respect to their active extender.}
	Write $\gamma^{\vV_1}=\kappa_0^{+\M_\infty}$.

	The map $\pi_\infty:\M_\infty\to\M_\infty^{\M_\infty}$
	(see (\ref{star-infty}) and the
	preceding discussion) is an
	iteration map. We define (recursively on $\nu$):
	\begin{eqnarray}
		\es_\nu^{\vV_1} =
		\begin{cases}
			\es_\nu^{{\cal M}_\infty} & \mbox{ if } \nu< \gamma^{\vV_1} \\
			\pi_\infty \upharpoonright ({\cal M}_\infty|\delta_0^{{\cal M}_\infty}) & \mbox{ if }
			\nu = \gamma^{\vV_1} \\
			\es^{M}_\nu \upharpoonright (\vV_1|\nu) & \mbox{ if } \nu >
			\gamma^{\vV_1}.
		\end{cases}
	\end{eqnarray}
	(We will verify in Lemma \ref{local_correspondence} that this definition makes sense;
	in particular, that $\es^M_{\nu}``(\vV_1|\nu)\sub\vV_1|\nu$ when $\nu>\gamma^{\vV_1}$.)
	Let us also write
	\begin{eqnarray}\label{predicate-of-first-varsovian}
		\es^{\vV_1}& =& \{ (\nu , x , y ) \colon \es_\nu^{\vV_1} \not= \emptyset
		\text{ and } y = \es_\nu^{\vV_1}(x) \},\\
		\es^{\vV_1} \rest \nu& =& \{ ({\bar \nu},x,y) \in {\mathbb E}^{\vV_1}
		\colon {\bar \nu}<\nu \}.
	\end{eqnarray}

	We   define the  structure
	\[ \vV_1=(L[{\mathbb E}^{\vV_1}];\in,\mathbb{E}^{\vV_1}).\]
	Like with $\M_\infty[*]$, when we discuss definability
	or write an equation with $\vV_1$, or some similar structure,
	then we refer to the structure  itself, not just its universe
	$\univ{\vV_1}=L[\es^{\vV_1}]$.
	In particular, definability over $\vV_1$ has the
	class predicate $\es^{\vV_1}$ available by default,
	just like for premice.
	(But when no confusion can arise,
	such as with expressions like ``$x\in\vV_1$'' or ``$x\sub\vV_1$'',
	we really mean ``$x\in\univ{\vV_1}$'' or ``$x\sub\univ{\vV_1}$'',
	etc.)

	Write $e^{\vV_1}=\es^{\vV_1}_{\gamma^{\vV_1}}$,
	so $e^{\vV_1}$ is the least (properly) long extender of $\es^{\vV}$,
	which is just the extender induced by $*$.
\end{dfn}

\begin{dfn}[Fine structure of $\vV_1$]\label{dfn:vV_1_fine_structure}
	The fine structural concepts ($\rSigma_{n+1}$, $\rho_{n+1}$, $p_{n+1}$,
	$(n+1)$-solidity, etc, for $n<\om$) are defined for segments $P\ins\vV_1$,
	and also for $P^\passive$, just as for standard premice with Jensen indexing
	below superstrong:
	$\rSigma_1^P=\Sigma_1^P$ (without any constant symbols in the fine structural
	language), which determines everything else via the usual
	recursion.
\end{dfn}

We will show that all segments of $\vV_1$ are well-defined, sound,
and establish a fine structural correspondence between
segments of $\vV_1$ and segments of $M$, above a certain starting point.
The first non-trivial instance of these facts is  given by
\ref{V_is_a_ground} together with the next lemma;
it uses techniques reminiscent of those in \cite{V=HODX_pub}.
It results in a tighter bound on
$\kappa_0^{+\M_\infty}=\gamma^{\vV_1}$
than that given by Corollary \ref{cor:size_of_gamma}.

\begin{lemma}\label{local-definability-of-that-structure}
	Let $\bar{\vV}=\vV_1||\gamma^{\vV_1}$ \tu{(}so $\bar{\vV}$ has $e^{\vV_1}$ active\tu{)}. Then:
	\begin{enumerate}[label=\tu{(}\alph*\tu{)}]
		\item\label{item:forcing_def} ${\mathbb L}$ is
		$\Sigma_1$-definable over $\bar{\vV}$.
		\item\label{item:gamma_structure_def}
		$\bar{\vV}$ is isomorphic to a structure which is
		definable without parameters over
		$M|\kappa_0^{+M}$.
		\item\label{item:vV_1||gamma_is_sound} $\bar{\vV}$ is sound, with
		$\rho_\om^{\bar{\vV}}=\rho_1^{\bar{\vV}}=\delta_\infty$ and $p_1^{\bar{\vV}}=\emptyset$.
		\item\label{item:better_size_of_gamma}
		$\OR^{\bar{\vV}}<\xi_0$,
		where $\xi_0$ is the least $\xi>\kappa_0^{+M}$ such that $M|\xi$ is
		admissible.
		Therefore  $\vV_1||\nu$ is passive for every $\nu\in(\gamma^{\vV_1},\xi_0]$.
	\end{enumerate}
\end{lemma}
\begin{proof}
	Part \ref{item:forcing_def} follows from an
	inspection of \S\ref{ground_generation_stuff}.

	Part \ref{item:gamma_structure_def}:
	Let
	$\overline{\mathscr{F}}^+$ be like
	$\mathscr{F}^+$,
	but consisting of pairs $(P,s)$ such that there is $P'$
	with $(P',s)\in\mathscr{F}^+$ and $P=P'|\kappa_0^{+M}$
	and $s\sub
	\kappa_0^{+M}$.
	(Let the associated ordering, models and embeddings
	be the corresponding ones of $\mathscr{F}^+$.)
	Let $\overline{{\cal M}}_\infty$ be the direct limit of
	$\overline{\mathscr{F}}^+$,
	and $\bar{*}$ the associated $*$-map. Let
	\[ \overline{\sigma}:\overline{\cal M}_\infty\to {{\cal
			M}_\infty}|(\kappa_0^{{\cal M}_\infty})^{+\M_\infty}\]
	be the natural map
	determined by how
	$\overline{\mathscr{F}}^+$ sits within $\mathscr{F}^+$.
	Noting that
	$\overline{\mathscr{F}}^+$ is definable over
	$M|\kappa_0^{+M}$,
	it suffices to see that $\bar{\sigma}$ is the identity.

	Fix $s\in[\mathscr{I}^M]^{<\om}$ with $s\neq\emptyset$. For $P \in
	\mathscr{F}$, let $K^P$ be the transitive collapse of
	$$H^P=\Hull^{P|\max(s)}(s^-\cup\kappa_0^M),$$
	and let $\tau^P:K^P\to H^P$ be the uncollapse map, and
	$\bar{s}^P=t\cup\{\OR(K^P)\}$ where $\tau^P(t)=s^-$.

	We claim that $\pi_{MP}(\bar{s}^M)=\bar{s}^M$, and therefore
	$(P|\kappa_0^{+M},\bar{s}^P)\in\overline{\mathscr{F}}^+$.
	For this, note $\pi_{MP}(s,\kappa_0)=(s,\kappa_0)$, so
	$\pi_{MP}(\bar{s}^M)=\bar{s}^P$.
	But $M$ is a small (of size $<\kappa_0$) forcing extension of $P$,
	which implies
	$H^M \cap \OR =
	H^P\inter\OR$, so $\bar{s}^P=\bar{s}^M$, as required.

	So write $\bar{s}$ for the common value of $\bar{s}^P$.
	One can now use the argument of the proof of Lemma \ref{c0}
	(which showed that the natural
	map $\chi:\M_\infty^+\to\M_\infty$ is the identity),
	but replacing the use of
	$s$ there with
	$\bar{s}$. It follows that $\bar{\sigma}=\id$.

	Part \ref{item:vV_1||gamma_is_sound}:
	Since $\bar{\vV}\in\M_\infty[*]$,
	it suffices that $\bar{\vV}=\Hull_1^{\bar{\vV}}(\delta_\infty)$.\footnote{Let
		$\mathcal{Q}$ be the structure defined in \S\ref{subsec:Q}.
		We already know $\mathcal{Q}=\Hull_1^{\mathcal{Q}}(\delta)$,
		but it seems that the branch through the genericity tree involved there
		might not be computable from $*$. So the soundness
		of $\bar{\vV}$ is not an obvious corollary.}
	Let $\alpha<\gamma^{\vV_1}=\OR^{\bar{\vV}}$; we want to see that
	$\gamma\in\Hull_1^{\bar{\vV}}(\delta_\infty)$.
	Fix a non-empty $s\in[\mathscr{I}^M]^{<\om}$ and
	$N\in\mathscr{F}$ such that
	$\alpha\in\rg(\pi_{Ns,\infty})$.
	Let $\bar{s}$ be as above.
	As before, $\bar{s}$ is $N$-stable,
	and note that $\alpha\in\rg(\pi_{N\bar{s},\infty})$
	(because $\alpha<\kappa_0^{+M}$,
	if $\pi_{Ns,\infty}(\bar{\alpha})=\alpha$
	then $\pi_{N\bar{s},\infty}(\bar{\alpha})=\alpha$).
	But then, as desired, we have
	\[
	\alpha\in\Hull^{\M_\infty|\max(\bar{s})^*}((\bar{s}
	^-)^*\cup\delta_\infty)\sub\Hull_1^{\bar{\vV}}(\delta_\infty).\]

	Part \ref{item:better_size_of_gamma}
	follows from
	\ref{item:gamma_structure_def}
	and the definition of $\es^{\vV_1}$ above $\gamma^{\vV_1}$.
\end{proof}

The levels of $\vV_1$ correspond tightly to the levels of $M$, as follows.

\begin{lemma}\label{local_correspondence}
	Let $g=g_{M|\kappa_0}$
	be the $(\M_\infty[*],{\mathbb L})$-generic determined by $M|\kappa_0$.
	(so  $\M_\infty[*][g]\ueq M$). For every
	$\nu\in\OR$ we
	have:
	\begin{enumerate}
		\item\label{item:vV_1||nu_in_Minfty*} $\vV_1|\nu$ and $\vV_1||\nu$ are in
		$\M_\infty[*]$,
		\item\label{item:vV_1||nu_sound} $\vV_1|\nu$ and $\vV_1||\nu$ are sound,
		\item\label{item:fs_correspondence} Suppose $\nu\geq\xi_0$.\footnote{
			Also, $M|\theta_0^M$ and $\vV_1||\gamma^{\vV_1}$ are ``generically equivalent in
			the codes'',
			and letting
			\[ f:(\theta_0^M,\xi_0)\to(\gamma^{\vV_1},\xi_0) \]
			be the unique surjective order-preserving map,
			then $M|\alpha=M||\alpha$ are likewise equivalent
			with $\vV_1|f(\alpha)=\vV_1||f(\alpha)$ for all $\alpha\in\dom(f)$,
			but we will not need this.}
		Then
		\begin{enumerate}[label=\tu{(}\alph*\tu{)}]
			\item ${\mathbb L} \in \vV_1|\nu$ and $g$ is $(\vV_1|\nu,\mathbb{L})$-generic,
			\item $(\vV_1|\nu)[g] =^* M|\nu$,\footnote{The notation is explained in
				\ref{rem:=^*}.}
			\item $(\vV_1||\nu)[g] =^* M||\nu$,
			\item\label{item:premouse_axioms} if $\es^M_\nu\neq\emptyset$ and
			$\crit(\es^M_\nu)>\kappa_0$ then $\vV_1||\nu$ satisfies the usual premouse
			axioms with respect to its active predicate (with Jensen indexing;
			in particular, $\es^{\vV_1}_\nu$ is an extender over $\vV_1|\nu$), and
			\item\label{item:non_coherence} if $E=\es^M_\nu\neq\emptyset$ and
			$\crit(E)=\kappa_0$ then $\es^{\vV_1}_\nu$ is a
			long $(\delta_\infty,\nu)$-extender
			over $\M_\infty$ and
			\[
			\Ult(\M_\infty|\delta_\infty,\es^{\vV_1}_\nu)=i^M_{E}
			(\M_\infty|\delta_\infty)=\M_\infty^{\Ult(M,E)}|i^M_E(\delta_\infty)\]
			is a lightface proper class of $\vV_1|\nu$, uniformly in $\nu$.
		\end{enumerate}
	\end{enumerate}
\end{lemma}
\begin{rem}\label{rem:=^*}
	Here the notation $=^*$ is the usual one in this context,
	meaning that (i) the two structures have the same universe,
	(ii) for each $\alpha\in[\xi_0,\nu)$ (or $[\xi_0,\nu]$),
	$\es^{\vV_1}_\alpha=\es^M_\alpha\rest(\vV_1|\alpha)$ (which is already true by
	definition),
	and conversely, if $\crit(\es^M_\alpha)>\kappa_0$
	then $\es^M_\alpha$ is the canonical small forcing
	extension of $\es^{\vV_1}_\alpha$ to $M|\alpha$
	and if $\crit(\es^M_\alpha)=\kappa_0$
	then $\es^M_\alpha$ is determined by
	$\es^M_\alpha\rest\OR=\es^{\vV_1}_\alpha\rest\OR$ and $M|\alpha$ as usual for a
	premouse,
	and (iii) the structures $\vV_1|\nu$ and $M|\nu$ (or $\vV_1||\nu$ and $M||\nu$)
	have corresponding fine structure in the manner usual for P-constructions as in \cite{sile},
	with matching projecta and parameters, etc.
\end{rem}

\begin{proof}
	The lemma holds for $\nu<\gamma^{\vV_1}$
	directly by definition,
	for $\nu=\gamma^{\vV_1}$ by
	\ref{local-definability-of-that-structure}.
	For $\gamma^{\vV_1}<\nu\leq\xi_0$ it is a straightforward consequence:
	we have $\vV_1||\xi_0\in\M_\infty[*]$
	since $A\in\M_\infty[*]$,
	and therefore $\vV_1||\nu$ cannot
	project $<\delta_\infty$. Note that
	$\gamma^{\vV_1}\in\Hull_1^{\vV_1|\nu}(\emptyset)$,
	just because $\vV_1||\gamma^{\vV_1}$
	is the least segment with an active extender
	of its kind. Using \ref{local-definability-of-that-structure},
	therefore $\gamma^{\vV_1}+1\sub\Hull_1^{\vV_1|\nu}(\delta_\infty)$.
	If $\nu<\om\cdot\gamma^{\vV_1}$ then we therefore
	get $\vV_1|\nu\sub\Hull_1^{\vV_1|\nu}(\delta_\infty)$,
	and hence $\vV_1|\nu$ is sound.
	If instead $\om\cdot\gamma^{\vV_1}\leq\nu$
	then note that $(\vV_1|\nu)[g]$ has universe that of $M|\nu$,
	and the $\Sigma_0^{M|\nu}$ forcing relation
	is $\Delta_1^{\vV_1|\nu}$ (lightface, as
	$\gamma^{\vV_1}\in\Hull_1^{\vV_1|\nu}(\emptyset)$).
	But $M|\nu=\Hull_1^{M|\nu}(\kappa_0^{+M}+1)$,
	and since all of the $\Sigma_1$ facts witnessing
	this get forced, it follows that
	$\vV_1|\nu=\Hull_1^{\vV_1|\nu}(\delta_\infty)$,
	so $\vV_1|\nu$ is sound.

	For $\nu>\xi_0$ we discuss parts \ref{item:vV_1||nu_in_Minfty*}
	and
	\ref{item:fs_correspondence}\ref{item:premouse_axioms},\ref{item:non_coherence};
	for the rest one mostly uses standard calculations as for P-constructions.
	Suppose we have established that $(\vV_1|\nu)[g]=^*M|\nu\in\M_\infty[*]$,
	and $\es^M_\nu\neq\emptyset$, so
	$\es^{\vV_1}_\nu=\es^M_\nu\rest(\vV_1|\nu)\neq\emptyset$.
	We already know that $\es^{\vV_1}_\nu\rest\OR\in\M_\infty[*]$ (by
	Lemma \ref{restr_of_extenders_are_there}).
	We must verify that
	$\es^{\vV_1}_\nu\in\M_\infty[*]$ (and uniformly so),
	and hence $\vV_1||\nu\in\M_\infty[*]$,
	and that $\vV_1||\nu$ has the right properties.

	Suppose $\kappa_0<\crit(\es^M_\nu)$.
	Then $\es^{\vV_1}_\nu$ is an extender over $\vV_1|\nu$
	satisfying the usual requirements for premice,
	by the usual proof (using induction and that $M|\nu$ is a small forcing
	extension of $\vV_1|\nu$). It follows that $\es^{\vV_1}_\nu$
	can be computed from the pair $(\vV_1|\nu,\es^{\vV_1}_\nu\rest\OR)$,
	as usual. But $\vV_1|\nu$  is in $\M_\infty[*]$ by induction.

	Now suppose that $\kappa_0=\crit(\es^M_\nu)$.
	Then $E=\es^{\vV_1}_\nu$ is a (long) $(\delta_\infty,\nu)$-extender
	over $\M_\infty|\delta_\infty$, but this time $E$ does not cohere $\vV_1|\nu$.
	In order to compute the full $E$ from $E\rest\OR$,
	one also needs the target model
	\[ U=\Ult(\M_\infty|\delta_\infty,E)=\M_\infty^{\Ult(M,\es^M_\nu)}|\nu.\]
	By
	\ref{local-definability-of-that-structure}, this is computed by the local
	covering system of $M|\nu$ (as defined in that proof,
	but over $M|\nu$, not $M|(\kappa_0^{+M})$).
	But since $\vV_1|\nu=^*M|\nu$ and the forcing
	$\mathbb{L}\in \vV_1|\lambda$ where $\lambda$
	is the largest cardinal of $\vV_1|\nu$,
	and $\M_\infty|\delta_\infty\in\vV_1|\lambda$,
	there is a canonical definition of $U$ over $\vV_1|\nu$
	(uniform in all such $\nu$).
	That is, although we don't have the full $M|\nu$ directly available,
	the agreement between $\es^M$ and $\es^{\vV_1}$
	ensures that the short tree strategy for $\M_\infty$ is computed almost like
	when we do have $M|\nu$: Given a strong cutpoint $\gamma$ of $\vV_1|\nu$
	with $\delta_\infty^{+(\vV_1|\nu)}<\gamma$,
	let $G$ be $(\vV_1|\nu,\Coll(\om,\gamma))$-generic.
	Then $(\vV_1|\nu)[G]$ can be arranged as a premouse over $(\vV_1||\gamma,G)$,
	and note that we can also take $G$ such that there is an $(M,\Coll(\om,\gamma))$-generic
	$G'$ such that $(M|\nu)[G']=^*(\vV_1|\nu)[G]$ (with $(M||\gamma,G')$ equivalent intercomputable with $(\vV_1||\gamma,G)$).
	Therefore we can use $(\vV_1|\nu)[G]$ to compute short tree strategy for $\M_\infty|\delta_\infty$ in the same
	manner we use $(M|\nu)[G']$ (working above $\gamma$), and by homogeneity,  this computation restricted
	to trees in $\vV_1|\nu$ is independent of the choice of $G$. The computation of maximal trees
	(above $\gamma$) is similarly absolute, and note that the P-constructions determined by these trees
	also agree between $M|\nu$ and $\vV_1|\nu$. The system computed in $\vV_1|\nu$ is also
	easily dense in that of $M|\nu$. Therefore $\vV_1|\nu$ computes $\M_\infty^{\Ult(M,E)}|\nu$,
	as desired.\footnote{Note that the foregoing proof did not  use
		Lemmas \ref{lem:i_E(M_infty)_is_iterate} or \ref{M_knows_how_to_iterate_Minfty_up_to_its_woodin}, which we are yet to actually prove; it does not matter here
		whether $\M_\infty^{\Ult(M,\es^M_\nu)}|\nu$ is a correct iterate of $\M_\infty|\delta_0^{\M_\infty}$.
		But in any case,  we could have proved those lemmas at the point they appeared in the text.}
\end{proof}

\begin{lemma}\label{lem:vV_1_M_infty[*]_same_univ}
	$\M_\infty[*]$ and $\vV_1$
	\tu{(}as defined in \ref{dfn:M_infty[*]} and \ref{dfn:vV_1}\tu{)}
	have the same universe.
\end{lemma}

\begin{proof} We have
	$\vV_1\sub\M_\infty[*]$ by Lemma \ref{local_correspondence} part
	\ref{item:vV_1||nu_in_Minfty*}.

	Let us show
	$\M_\infty[*]\sub\vV_1$. Write
	$\mathbb{L}=\mathbb{L}_{\theta_0^M}^{\vV_1}$.
	By \ref{local_correspondence} part \ref{item:fs_correspondence},
	$\mathbb{L}\in\vV_1$,
	and since $\vV_1\sub\M_\infty[*]$,
	therefore $g=g_{M|\theta_0^M}$ is $\mathbb{L}$-generic
	over both $\M_\infty[*]$ and $\vV_1$. But then as in the proof of
	Lemma \ref{V_is_a_ground}, $\vV_1[g]$ has universe $\univ{M}$
	(as does $\M_\infty[*][g]$).

	Now let $x \in \M_\infty[*]$ be a set of ordinals.
	We show that $x\in\vV_1$. Let $\tau$ be an ${{\mathbb L}}$-name
	in $\vV_1$ such that $x=\tau^g$. Let $p \in g$ be such that $p
	\Vdash_{\M_\infty[*]}^{{\mathbb L}} \tau = {\check x}$. It is easy to see that
	then
	\[ x = \{ \xi \colon p \Vdash_{\vV_1}^{{\mathbb L}} {\check \xi}
	\in \tau \} \in \vV_1.\]

	So the two models have the same universe.
\end{proof}

The preceding fact will be refined later in Lemma \ref{lem:vV_1_M_infty[*]_inter-def}.

\begin{rem}\label{rem:vV_1_as_strategy_premouse}
	We may also reorganize $\vV_1$ as a strategy premouse,
	by representing the information contained in the long extenders in
	$\es^{\vV_1}$
	differently.
	These extenders are easily seen to be intertranslatable with
	fragments of $\Sigma_{\M_\infty}$ for trees based on
	$\M_\infty|\delta_\infty$. Namely, let us define a sequence
	$(\mathbb{F}_\nu^{\vV_1} \colon
	\nu \in {\rm OR})$ as follows. Except for those $\nu$
	where $\es^{\vV_1}_\nu$ is long, we set
	$\mathbb{F}^{\vV_1}_\nu=\es^{\vV_1}_\nu$.
	If $\es^{\vV_1}_\nu$ is long, then
	$\mathbb{F}^{\vV_1}_\nu=\Sigma_{\M_\infty}(\Tt)$,
	where $\Tt$ is the normal tree on $\M_\infty$ leading from
	$\M_\infty|\delta_\infty$ to $i^M_E(\M_\infty|\delta_\infty)$.
	Then easily, Lemma \ref{local_correspondence} holds also
	after replacing $\es^{\vV_1}$ with $\mathbb{F}^{\vV_1}$,
	and $\es^{\vV_1},\mathbb{F}^{\vV_1}$ are level-by-level intertranslatable.
	So $L[{\mathbb F}^{\vV_1}]\ueq\vV_1$.\footnote{However,
		it is not clear whether $\vV_1$ can be arranged
		as a strategy mouse in one of the more traditional hierarchies,
		like those used for HOD mice, or the least (tree) branch hierarchy.}
\end{rem}

\begin{lemma}\label{V_1_knows_how_to_iterate_Minfty_up_to_its_woodin}
	We have:
	\begin{enumerate}[label=\tu{(}\alph*\tu{)}]
		\item\label{item:vV_1_M_infty_strat_thru_delta_0} The restriction of
		$\Sigma_{{{\cal M}_\infty}}$ to trees in $\vV_1$ and based on ${{\cal
				M}_\infty}|\delta_\infty$, is lightface definable over $\vV_1$ (so by
		Lemma \ref{lem:vV_1_M_infty[*]_inter-def} below, it is also lightface definable over $\M_\infty[*]$).
		\item\label{item:vV_1_M_infty_strat_thru_delta_0_in_M[g]} Let $g$ be in some set generic extension
		of $V$ and be set-generic over
		$\vV_1$.
		Let $\Sigma'_{\M_\infty}$ be the restriction of $\Sigma_{{{\cal M}_\infty}}$
		to
		trees in $\vV_1[g]$ and based on ${{\cal M}_\infty}|\delta_\infty$.
		Then $\Sigma'_{\M_\infty}$ is definable over the universe of $\vV_1[g]$
		from the predicate $\vV_1$.\footnote{Regarding trees $\notin V$,
			cf.~Footnote \ref{ftn:trees_not_in_V}.}
	\end{enumerate}
\end{lemma}

\begin{proof} This is much like the proof of Lemma \ref{M_knows_how_to_iterate_Minfty_up_to_its_woodin} (whose complete
	proof will rely on Lemma \ref{lem:M_infty_of_kappa_0-sound_iterate}, still to come), although now we don't
	have $M$ itself available. The computation of short tree strategy
	is as in the proof of Lemma \ref{local_correspondence} part
	\ref{item:fs_correspondence}\ref{item:non_coherence}.
	The computation of branches at maximal stages is like in the proof of Lemma \ref{M_knows_how_to_iterate_Minfty_up_to_its_woodin}.
\end{proof}

To summarize, we have isolated several representations of the universe of $\vV_1$,
indicating that  $\vV_1$ is a natural object:
\begin{eqnarray*}
	\vV_1 & \ueq & L[{\mathbb E}^{\vV_1}] \mbox{ (Definition \ref{dfn:vV_1}}) \\
	{}&\ueq& L[{\mathbb F}^{\vV_1}] \mbox{
		(Remark
		\ref{rem:vV_1_as_strategy_premouse})}\\
	{}&\ueq& {{\cal M}_\infty}[*]
	\mbox{
		(see (\ref{defn_first_varsovian_model}), Lemma \ref{lem:vV_1_M_infty[*]_same_univ})} \\
	{} &\ueq & \HOD_{\mathscr{E}}^{M[G]}  \mbox{ (Lemma
		\ref{tm:M_infty[*]=HOD_E})} \\
	{} &\ueq & \bigcap {\mathscr F} = \mbox{the } <\kappa_0 \mbox{ mantle of } M
	\mbox{ (Proposition \ref{prop:vV_1_is_<kappa_0-mantle})}
\end{eqnarray*}

\subsection{Varsovian strategy premice}\label{subsec:Vsp}
In this section we will introduce an axiomatization for premice
in the hierarchy of $\vV_1$.
But first, we refine Lemma \ref{lem:vV_1_M_infty[*]_same_univ} as follows:

\begin{lemma}\label{lem:vV_1_M_infty[*]_inter-def}\
	\begin{enumerate}
		\item \label{item:vV_1^M_infty_=_Ult(vV_1,e)}
		$\Ult(\vV_1,e^{\vV_1})=\vV_1^{\M_\infty}$.
		\item\label{item:vV_1_lightface_class_of_M_infty[*]} $\vV_1$ is a lightface
		class of
		$\M_\infty[*]$,

		\item \label{item:M_infty[*]_class_of_vV_1} $\M_\infty[*]$ is a lightface
		class of $\vV_1$.
\end{enumerate}\end{lemma}

\begin{proof}
	Part \ref{item:vV_1_lightface_class_of_M_infty[*]}:  Argue as in the proof of
	Lemma
	\ref{local_correspondence}, using again Lemma
	\ref{restr_of_extenders_are_there}.

	Part \ref{item:vV_1^M_infty_=_Ult(vV_1,e)}:	Let $e=e^{\vV_1}$.  We have $\M_\infty^{\M_\infty}=\Ult(\M_\infty,e)$ and
	$\pi_\infty:\M_\infty\to\M_\infty^{\M_\infty}$ is  the ultrapower map, so $e\sub\pi_\infty$.
	By Lemma \ref{lem:pi_infty^+}, $\pi_\infty$ extends elementarily to \[ \pi_\infty^+:\M_\infty[*]\to\M_\infty^{\M_\infty}[*^{\M_\infty}],\]
	so considering how $\vV_1$ is defined over $\M_\infty[*]$, and $\vV_1^{\M_\infty}$ over $\M_\infty^{\M_\infty}[*^{\M_\infty}]$,
	\[ \pi_\infty^+:\vV_1\to\vV_1^{\M_\infty} \]
	is also elementary.
	Since $V_{\delta_\infty}^{\M_\infty}=V_{\delta_\infty}^{\vV_1}$, it follows
	that $e$ is also derived from $\pi_\infty^+$.
	Let $\vV'=\Ult(\vV_1,e)$ and $j^+:\vV_1\to\vV'$ be the ultrapower map
	and $k^+:\vV'\to\vV_1^{\M_\infty}$ the factor map, so $k^+\com j^+=\pi_\infty^+$.
	Since $\pi_\infty\sub\pi_\infty^+$ and $\pi_\infty$ is the ultrapower map,
	we have $k^+\rest\OR=\id$ (and $j^+\rest\OR\sub\pi_\infty^+$).
	So in fact $k^+=\id$ and $\vV'=\vV_1^{\M_\infty}$.\footnote{Here is a slightly
		alternate argument. We have $\M_\infty^{\M_\infty}=\Ult(\M_\infty,e)$
		and $\pi_\infty$ is the ultrapower map. Let
		$N[e']=\Ult(\M_\infty[e],e)$
		and $j:\M_\infty[e]\to N[e']$
		be the ultrapower map. Then as before, in fact $N=\M_\infty^{\M_\infty}$
		and $\pi_\infty\sub j=\pi_\infty^+$. Moreover, by considering some fixed indiscernibles,
		\[ e'=\pi_\infty^+(e)=\bigcup_{\alpha<\delta_\infty}\pi_\infty(e\rest\alpha) \]
		is the correct extender of the iteration map $\M_\infty^{\M_\infty}\to(\M_\infty)^{\M_\infty^{\M_\infty}}$.
		Now $\vV_1$ and $\M_\infty[e]$ have the same universe,
		and $\vV_1$ is defined over $\M_\infty[e]$ via the procedure mentioned for part \ref{item:vV_1_lightface_class_of_M_infty[*]} of the lemma. So $\Ult(\vV_1,e)$
		is defined over $\M_\infty^{\M_\infty}[e']$ in the same manner.
		But $e'$ agrees with $*^{\M_\infty}$, so this definition
		actually yields $\vV_1^{\M_\infty}$.}

	Part \ref{item:M_infty[*]_class_of_vV_1}: It suffices to define $\M_\infty$
	and $*\rest\delta_0^{\M_\infty}$.
	But $*\rest\delta_0^{\M_\infty}$ is just the active extender of $\vV_1||\gamma^{\vV_1}$.
	To define $\M_\infty$, we first have
	$\M_\infty||\kappa_0^{+\M_\infty}=\vV_1|\gamma^{\vV_1}$.
	But $\vV_1^{\M_\infty}$ is a lightface class of $\vV_1$ by  part
	\ref{item:vV_1^M_infty_=_Ult(vV_1,e)}.
	And $\es^{\M_\infty}\rest[\kappa_0^{+\M_\infty},\infty)$
	is determined by $M||\kappa_0^{+\M_\infty}$ and
	$\es^{\vV_1^{\M_\infty}}\rest(\gamma^{\vV_1^{\M_\infty}},\infty)$,
	since the two sequences agree over the ordinals.
\end{proof}

\begin{dfn}
	For an $\Mswsw$-like $N$, $\M_\infty^N$, $*^N$, $\M_\infty[*]^N$ and $\vV_1^N$
	denote the lightface $N$-classes defined over $N$ just as the corresponding
	classes are defined over $M$.

	Also given an $\Mswsw$-like $N$ and $\bar{N}\ins N$
	with $\kappa_0^{+N}\leq\OR^{\bar{N}}$,
	we define $\vV_1^{\bar{N}}$ by recursion on $\OR^{\bar{N}}$
	by setting $\vV_1^{N||(\kappa_0^{+N}+\alpha)}=\vV_1^N||(\gamma+\alpha)$,
	where $\gamma=\gamma^{\vV_1^N}$.
	Noting that this definition is level-by-level,
	we similarly define $\vV_1^{\bar{N}}(\kappa)$
	whenever $\bar{N}$ is $\Mswsw$-small
	and $\kappa$ is an inaccessible limit of cutpoints and Woodins
	of $\bar{N}$ and $\kappa<\OR^{\bar{N}}$,
	level-by-level (starting by defining
	$\vV_1^{\bar{N}|\kappa^{+\bar{N}}}$
	as $\vV_1||\gamma^{\vV_1}$ is defined (in the codes)
	over $M|\kappa_0^{+M}$). We will often suppress the $\kappa$
	from the notation, writing just $\vV_1^{\bar{N}}$.
\end{dfn}

We now want to axiomatize structures in the hierarchy of $\vV_1$ to some extent:

\begin{dfn}\label{dfn:local_vV_1}
	A \emph{base Vsp} is an amenable transitive structure $\vV=(P_\infty,F)$ such that
	in some forcing extension
	there is $P$ such that:
	\begin{enumerate}
		\item $P,P_\infty$ are premice which model $\ZFC^-$ and are $\Mswsw$-small; that is,
		they have no active
		segments which satisfy ``There are $\kappa_0<\delta_1<\kappa_1$
		with $\delta_1$ Woodin and $\kappa_0,\kappa_1$ strong''.
		\item $P$ has a least Woodin cardinal $\delta_0^P$
		and a largest cardinal $\kappa_0^P>\delta_0^P$, and $\kappa_0^P$ is inaccessible in $P$
		and a limit of cutpoints of $P$; likewise for $P_\infty,\delta_0^{P_\infty},\kappa_0^{P_\infty}$,
		\item 	$\OR^P=\delta_0^{P_\infty}$,
		$\kappa_0^P$ is the least measurable of $P_\infty$, and $\vV_1^{P}=\cHull_1^{\vV}(\delta_0^{P_\infty})$,
		\item $\M_\infty^{P_\infty}$ (defined over $P_\infty$
		like $\M_\infty|\gamma^{\vV_1}$ is defined over $M|\kappa_0^{+M}$) is well-defined,
		and has least measurable $\kappa_0^{P_\infty}$ and least Woodin $\delta_0^{\M_\infty^{P_\infty}}=\OR^{P_\infty}$,
		\item $\M_\infty^{P_\infty}|\delta_0^{\M_\infty^{P_\infty}}$ is obtained
		by  iterating $P_\infty|\delta_0^{P_\infty}$, via a normal tree $\Tt$ of length $\delta_0^{\M_\infty^{P_\infty}}$,
		\item $F$ is a cofinal $\Sigma_1$-elementary
		(hence fully elementary, by $\ZFC^-$) embedding \[ P_\infty|\delta_0^{P_\infty}\to\M_\infty^{P_\infty}|\delta_0^{\M_\infty^{P_\infty}},\]
		and there is a $\Tt$-cofinal branch $b$ such that
		$F\sub i^\Tt_b$, and $i^\Tt_b(\delta_0^{P_\infty})=\delta_0^{\M_\infty^{P_\infty}}$ (so
		$b$ is intercomputable with $F$, and note that by amenability of $\vV$, $F$ is amenable to $P_\infty$, and hence so is $b$),
		\item\label{item:delta_0_stays_Woodin} $\rho_1^{\vV}=\delta_0^{P_\infty}$ and $p_1^{\vV}=\emptyset$
		and  $\delta_0^{P_\infty}$ is Woodin in $\J(\core_1(\vV))$, as witnessed by $\es^{P_\infty}$,

		\item $P$ is $(\J(\core_1(\vV)),\mathbb{L}^{\vV})$-generic,
		where $\mathbb{L}^{\vV}$ is defined over $\vV$ as
		$\mathbb{L}$ above was defined over $\vV_1||\gamma^{\vV_1}$.
		\qedhere
	\end{enumerate}
\end{dfn}

\begin{rem}\label{rem:base_Vsp}
	Let $C=\core_1(\vV)$ and $\pi:C\to\vV$ be the core map. Then
	\begin{equation}\label{eqn:rg(pi)_is_Hull} \rg(\pi)=\Hull_1^{P_\infty}(\delta_0^{P_\infty}\cup F``\delta_0^{P_\infty}). \end{equation}
	For we have $\rho_1^{\vV}=\delta_0^{P_\infty}$ and $p_1^{\vV}=\emptyset$ by hypothesis.
	So $\supseteq$ is clear, and $\sub$ is because for each $\alpha<\delta_0^{P_\infty}$,
	$F\rest\alpha$ is in the hull on the right, by calculations like with the Zipper Lemma.

	Now because $\delta_0^{P_\infty}$ is Woodin in $\J(\core_1(\vV))$, and in particular regular,
	we have that $\core_1(\vV)$ is sound with $\rho_\om^{\core_1(\vV)}=\delta_0^{P_\infty}$,
	and in particular $V_{\delta_0^{P_\infty}}^{\J(\core_1(\vV))}=V_{\delta_0^{P_\infty}}^{P_\infty}$.

	Note also that  $\mathbb{L}$ is (lightface) $\Sigma_1^{\vV_1||\gamma^{\vV_1}}$, and
	we use the natural $\Sigma_1$ definition above to define $\mathbb{L}^{\vV}$, so $\mathbb{L}^{\vV}=\mathbb{L}^{\core_1(\vV)}$. Moreover, like for $\mathbb{L}$,
	$\mathbb{L}^{\vV}$
	is a sub-algebra of the extender algebra of $P_\infty$ at $\delta_0^{P_\infty}$
	and $\J(\core_1(\vV))\sats$``$\mathbb{L}^{\vV}$ is a $\delta_0^{P_\infty}$-cc complete Boolean algebra''.

	The definition is actually specified by an (infinite) first-order theory satisfied by $\vV$, modulo the wellfoundedness of $\vV$.  (The theory needs to be infinite because	of the assertion of Woodinness in $\J(\core_1(\vV))$ in condition \ref{item:delta_0_stays_Woodin}.)
	To see this, observe that the (generic) existence of $P$ is first-order:
	Working in $\vV$, we can say that there is some condition
	of $\mathbb{L}^{\vV}$ forcing over $\core_1(\vV)$ that the generic object is a premouse
	$P$ such that the conditions above hold (and by the preceding discussion, all relevant
	antichains are in $P_\infty$); a small subtlety here
	is that we need to refer to $F$ and the hull in (\ref{eqn:rg(pi)_is_Hull})
	in order to talk about $\core_1(\vV)$ and assert that it is isomorphic to $\vV^P$; note
	that we can just talk about the relevant theories to assert this.
	(We don't demand that $\M_\infty^{P_\infty}$ be wellfounded,
	but only what was asserted above, which gives that it is wellfounded through $\delta_0^{\M_\infty^{P_\infty}}+1$.)

	Unsound base Vsps naturally arise  from iterating sound ones.
\end{rem}
\begin{dfn}\label{dfn:Vsp}
	A \emph{Varsovian strategy premouse (Vsp)} is a structure
	\[ \vV=(\J_\alpha^\es,\es,F) \]
	for some sequence $\es$ of extenders, where either $\vV$ is a premouse, or:
	\begin{enumerate}
		\item $\alpha\leq\OR$ and $\vV$ is an amenable acceptable J-structure,
		\item $\vV$ has a least Woodin cardinal $\delta_0^{\vV}$,
		and an initial segment $\vV||\gamma$ which is a  base Vsp,
		\item $\delta_0^{\vV}<\gamma$, so $\delta_0^{\vV}$ is the least Woodin of $\vV||\gamma$,
		\item if $F\neq\emptyset$ and $\gamma<\OR^\vV$ then either:
		\begin{enumerate}
			\item  $\vV$ satisfies the premouse
			axioms (for Jensen indexing) with respect to $F$, and $\gamma<\crit(F)$, or
			\item\label{item:long} We have:
			\begin{enumerate}
				\item $\vV^\passive=(\J_\alpha^\es,\es,\emptyset)\sats\ZFC^-$,
				\item $\vV$ has largest cardinal $\mu$, which is inaccessible in $\vV$
				and a limit of cutpoints of $\vV$ (where \emph{cutpoint}
				is with regard to all kinds of extenders),
				\item   $\N=\M_\infty^{\vV^\passive}$ is well-defined,
				and satisfies the axioms of a premouse (but is possibly illfounded)
				with a Woodin cardinal $\delta_0^{\N}$,
				and is $(\OR^{\vV}+1)$-wellfounded with $\delta_0^{\N}=\OR^{\vV}$,
				\item   $\N|\delta_0^{\N}$  is a proper class
				of $\vV^\passive$ and has least measurable
				$\mu$,
				\item  $F$ is a cofinal $\Sigma_1$-elementary
				embedding $F:\vV|\delta_0^{\vV}\to \N|\delta_0^{\N}$,
				\item $\N|\delta_0^{\N}$ is pseudo-normal iterate of $\vV|\delta_0^{\vV}$,
				via tree $\Tt$, and there is a $\Tt$-cofinal branch $b$ such that
				$F\sub i^\Tt_b$  (hence $b$ is amenable to $\vV$
				and inter-definable with $F$ over $\vV^\passive$),
			\end{enumerate}
		\end{enumerate}
		\item each proper segment of $\vV$ is a sound Vsp
		(defining \emph{Vsp} recursively),
		where the fine structural language for base Vsps and segments as in \ref{item:long}
		is just that with symbols for $\in,\es,F$,
		\item some $p\in\mathbb{L}^\vV=\mathbb{L}^{\vV||\gamma}$ forces
		that the generic object is a premouse $N$ of height $\delta_0^{\vV}$
		with $\vV^N=\vV||\gamma$, and there is an extension of $N$ to a  premouse $N^+$
		such that $\vV^{N^+}=\vV$ (and note then that $N^+$ is level-by-level definable over $\vV[N]$,
		via inverse P-construction).
	\end{enumerate}

	We write $\gamma^{\vV}=\gamma$ above (if $\vV$ is not a premouse).
\end{dfn}

\begin{dfn}\label{dfn:vV_1-like}
	A  Vsp $\vV$ is \emph{$\vV_1$-like} iff it is proper class
	and in some set-generic extension,  $\vV=\vV^N$ for some $\Mswsw$-like premouse $N$. (Note this is first-order over $\vV$.)

	When talking about the extenders $E\in\es_+^{\vV}$, for a Vsp $\vV$,
	we say that $E$ is \emph{short} if $\vV||\lh(E)$ satisfies the usual premouse
	axioms with respect to $E$, and \emph{long} otherwise; likewise for the corresponding
	segments. So $\vV||\gamma^\vV$ is the least long segment.

	We write $\M_\infty=\vV_1\downarrow 0$.
	Let $\vV$ be  $\vV_1$-like.
	We define $\vV\downarrow 0$ analogously (first-order over $\vV$ as in the proof of
	Lemma \ref{lem:vV_1_M_infty[*]_inter-def} part \ref{item:M_infty[*]_class_of_vV_1}).
	In fact, let us define $\vV\downarrow 0$ more generally, including the case that $\vV$ is illfounded,
	but satisfies the first order properties of a $\vV_1$-like structure.
	Also if $N$ is a premouse,
	let $N\downarrow 0=N$.
	We write  $\vV^-$ for (the premouse) $\vV|\delta_0^{\vV}$.
	(So $\vV_1^-=\M_\infty|\delta_0^{\M_\infty}$.)
	We write $\Lambda^{\vV}$ for the strategy for $\vV\downarrow 0$
	(for trees based on $\vV^-$) defined over $\vV$
	just as the corresponding restriction of $\Sigma_{\M_\infty}$
	is defined over $\vV_1$, via the proof of Lemma \ref{V_1_knows_how_to_iterate_Minfty_up_to_its_woodin}.

	We write
	$\vV_1=\vV(\M_\infty,*)=\vV(\M_\infty,*\rest\delta_\infty)=\vV(\M_\infty,e^{
		\vV_1})$.
	Given a pair
	$(N,*')$ or $(N,*'\rest\delta)$ or $(N,e)$ where
	$N$ is $\Mswsw$-like and the  pair has
	similar first-order properties
	as does $(\M_\infty,*)$ or $(\M_\infty,*\rest\delta_\infty)$ or
	$(\M_\infty,e^{\vV_1})$ respectively,
	we define $\vV(N,*')$ or $\vV(N,*'\rest\delta)$ or $\vV(N,e)$ analogously
	(via the proof of Lemma \ref{lem:vV_1_M_infty[*]_inter-def} part
	\ref{item:vV_1_lightface_class_of_M_infty[*]}).
\end{dfn}

\subsection{The action of $M$-iteration on $\M_\infty$}\label{subsec:M-iteration_on_M_infty}

We now aim to extend Lemma \ref{V_1_knows_how_to_iterate_Minfty_up_to_its_woodin}, analyzing  the
nature of $\M_\infty^N$ for iterates $N$ of $M$, and the partial
iterability of $\M_\infty^N$ in $N$.

\begin{lem}\label{lem:P-bar_is_iterate}
	Let $N$ be any non-dropping $\Sigma$-iterate of $M$.
	Let $P\in\mathscr{F}^N$. Let
	\[ \bar{P}=\cHull^P(\delta_0^P\cup\mathscr{I}^N). \]
	Then $\bar{P}$ is a $\delta_0^{\overline{P}}$-sound non-dropping $\Sigma$-iterate of $M$, $\bar{P}|\delta_0^{\overline{P}}=P|\delta_0^P$, and letting $\pi:\bar{P}\to P$ be the uncollapse map,
	then $\pi``\mathscr{I}^{\bar{P}}=\mathscr{I}^N$.\footnote{But if $N$ is not $\delta_0^N$-sound then $\bar{P}$ is not a $\Sigma_N$-iterate of $N$.}
\end{lem}
\begin{proof}
	We have $P|\delta_0^P=M^\Tt_\alpha|\delta_0^{M^\Tt_\alpha}$ for some tree $\Tt$ via $\Sigma$,
	where $[0,\alpha]_\Tt$ does not drop;
	this is because the Q-structures used to guide the short tree strategy computing $P$
	are correct. But  $\bar{P}|\delta_0^{\bar{P}}=P|\delta_0^P$ and $\bar{P}$ is an iterable, $\delta_0^{\bar{P}}$-sound,
	$\Mswsw$-like premouse, and with $\alpha$ above minimal, comparison gives  $\bar{P}=M^\Tt_\alpha$.
	Finally note that $\mathscr{I}=\pi^{-1}``\mathscr{I}^N$ is a club class of generating
	indiscernibles for $\bar{P}$, so $\mathscr{I}=\mathscr{I}^{\bar{P}}$.
\end{proof}
\begin{dfn}\label{dfn:N_alpha_P-stable}Let $N$ be a non-dropping $\Sigma$-iterate of $M$.
	Define $(\M_\infty^{\overline{\ext}})_N$ as the direct limit of
	the iterates $\bar{P}$, for $P\in\mathscr{F}^N$ (the notation $\bar{P}$ as in Lemma \ref{lem:P-bar_is_iterate}). (Cf.~Definition \ref{dfn:M_infty^ext^N}.)
	For $P\in\mathscr{F}^N$ let
	$\pi_{\bar{P}P}:\bar{P}\to P$ be the uncollapse map and
	\[ H^P=\Hull_1^P(\delta_0^P\cup\mathscr{I}^N)=\rg(\pi_{\bar{P}P}).\]

	Suppose further that $N$ is $\kappa_0^N$-sound.
	Let $\alpha\in\OR$ and $P\in\mathscr{F}^N$.
	We say that $\alpha$ is \emph{$(P,\mathscr{F}^N)$-stable}
	iff whenever $P\preceq Q\in\mathscr{F}^N$, we have $\alpha\in H^Q$ and
	\[
	\pi_{\bar{Q}Q}\com i_{\bar{P}\bar{Q}}\com\pi_{\bar{P}P}^{-1}(\alpha)=\alpha.\qedhere\]
\end{dfn}
The definition of stability above is more complicated than the version for $M$,
because it can be that $P\leq Q\in\mathscr{F}^N$ but $Q$ is not actually
an iterate of $P$ (although $Q|\delta_0^Q$ is an iterate of $P|\delta_0^P$).

\begin{lem}\label{lem:M_infty_of_iterate_N}
	Let $N$ be a non-dropping $\Sigma$-iterate of $M$.
	Then:
	\begin{enumerate}
		\item\label{item:M_infty^ext-bar^N_is_true_iterate} $P=(\M_\infty^{\overline{\ext}})_N$ is a $\delta_0^P$-sound non-dropping
		$\Sigma$-iterate of $M$.
		\item\label{item:M_infty^N_agmt_with_P}  $\M_\infty^N|\delta_\infty^N=P|\delta_0^P$.
		\item\label{item:if_N_to_kappa_0^M=M_to_kappa_0^M} If $M|\kappa_0^M\pins N$ then $\M_\infty|\delta_\infty\in N$
		and $P$ is a $\Sigma_{\M_\infty}$-iterate of $\M_\infty$.
	\end{enumerate}
\end{lem}
\begin{proof}
	Part \ref{item:M_infty^ext-bar^N_is_true_iterate} is immediate from the definitions. Part \ref{item:M_infty^N_agmt_with_P} follows from Lemma \ref{lem:P-bar_is_iterate}, since infinitely
	many	indiscernibles are fixed by the iteration maps.	For part \ref{item:if_N_to_kappa_0^M=M_to_kappa_0^M}, note that because $M|\kappa_0^M\pins N$
	and $N$ is a non-dropping iterate, in fact $M|\kappa_0^{+M}\pins N$,
	so $\M_\infty|\delta_\infty\in N$, and then it is an easy consequence of part \ref{item:M_infty^ext-bar^N_is_true_iterate}.
\end{proof}

\begin{lem}\label{lem:every_alpha_ev_stable}
	Let $N$ be a $\kappa_0^N$-sound non-dropping
	$\Sigma$-iterate of $M$.
	Then:
	\begin{enumerate}
		\item\label{item:H^P_subset_H^Q} For each $P\preceq Q\in\mathscr{F}^N$,
		we have $H^P\inter\OR\sub H^Q\inter\OR$.
		\item\label{item:each_alpha_ev_stable} For each $\alpha\in\OR$
		there is $P\in\mathscr{F}^N$
		such that
		$\alpha$ is $(P,\mathscr{F}^N)$-stable.
	\end{enumerate}
\end{lem}
\begin{proof}
	Part \ref{item:H^P_subset_H^Q}: We in fact that $H^P\inter\OR=\Hull_1^N(\delta_0^P\cup\mathscr{I}^N)$
	(which immediately gives $H^P\sub H^Q$). This is just by extender algebra genericity and definability
	of $P|\delta_0^P$ over $N|\delta_0^P$.

	Part \ref{item:each_alpha_ev_stable}:
	Since $N$ is $\kappa_0^N$-sound, we can fix $s\in\vec{\mathscr{I}^N}$
	and $\beta<\kappa_0^N$ and a term $t$ such that $\alpha=t^N(s,\beta)$.
	Then taking $P\in\mathscr{F}^N$ with $\beta<\delta_0^P$,
	we get $\alpha\in H^P$.
	Now since $(\M_\infty^{\overline{\ext}})^N$ is wellfounded,
	it suffices to see that
	\[ \pi_{\bar{Q}Q}\com i_{\bar{P}\bar{Q}}\com\pi_{\bar{P}P}^{-1}(\alpha)\geq\alpha \]
	whenever $P\preceq Q\in\mathscr{F}^N$ and $\alpha\in H^P$. So let $s\in\vec{\mathscr{I}^N}$ with $\alpha\in H^P_s$.
	Then $\pi_{Ps,Qs}(\alpha)\geq\alpha$
	whenever $P\preceq Q\in\mathscr{F}^N$,
	because $M$ satisfies the same about $i_{MN}^{-1}(s)$
	(because whenever $R\preceq S\in\mathscr{F}^M$,
	$S$ is actually an iterate of $R$,
	and these iterates are $\delta_0$-sound, etc). But note that \[ \pi_{Ps,Qs}(\alpha)=
	\pi_{\bar{Q}Q}\com i_{\bar{P}\bar{Q}}\com\pi_{\bar{P}P}^{-1}(\alpha)
	\]
	(because any generic branch witnessing the definition of $\pi_{Ps,Qs}^N$
	must move the relevant theory of indiscernibles and elements ${<\gamma^P_s}$ correctly, since these agree appropriately between $P,\bar{P}$ and $Q,\bar{Q}$). This gives the desired conclusion.
\end{proof}

The following lemma is proved like a similar fact in \cite{local_mantles_of_Lx_v2},
integrated with part of the argument for \cite[Lemma 2.9]{vm1}:
\begin{lem}\label{lem:M_infty_of_kappa_0-sound_iterate}Let $N$ be a $\kappa_0^N$-sound non-dropping $\Sigma$-iterate of $M$ and $N'$ a $\kappa_0^{N'}$-sound
	non-dropping $\Sigma_N$-iterate of $N$. Let $\bar{N}$ be the $\delta_0^N$-core of $N$.
	Then:
	\begin{enumerate}
		\item\label{item:I^M_infty=I^M} $\mathscr{I}^{\M_\infty}=\mathscr{I}^M$
		and $i_{M\M_\infty}\rest\mathscr{I}^M=\id=*\rest\mathscr{I}^M$,
		\item\label{item:I^M_infty^N=I^N} $\mathscr{I}^{\M_\infty^N}=\mathscr{I}^N$
		and $*^N\rest\mathscr{I}^N=\id$,

		\item\label{item:M_infty^N_is_ext-bar_M_infty} $\M_\infty^N=i_{MN}(\M_\infty)=(\M_\infty^{\overline{\ext}})_N$
		is a $\delta_0^{\M_\infty^N}$-sound $\Sigma_{\bar{N}}$-iterate of $\bar{N}$.
		\item\label{item:iterate_of_M_infty} If $N|\kappa_0^N\pins N'$ then
		\begin{enumerate}
			\item\label{item:M_infty^N_is_iterate}  $\M_\infty^{N'}$ is a $\Sigma_{\M_\infty^N}$-iterate
			of $\M_\infty^N$, and
			\item\label{item:i_MN_restricts_to_iteration_map} $i_{NN'}\rest\M_\infty^N$ is just the $\Sigma_{\M_\infty^N}$-iteration map
			$\M_\infty^N\to\M_\infty^{N'}$.
		\end{enumerate}
		\item\label{item:not_an_iterate} If $\bar{N}\neq N$ then $\M_\infty^N$ is not a $\Sigma_N$-iterate
		of $N$.
	\end{enumerate}
\end{lem}
\begin{proof}
	Part \ref{item:I^M_infty=I^M}:	We have $\mathscr{I}^{\M_\infty}=i_{M\M_\infty}``\mathscr{I}^M$,
	and $\M_\infty$ is $\delta_0^{\M_\infty}$-sound.
	Suppose $\kappa\in\mathscr{I}^M$ is least such that $i_{M\M_\infty}(\kappa)>\kappa$,
	fix a tuple $\vec{\kappa}\in\mathscr{I}^{\M_\infty}$ and a term $t$
	and $\alpha<\delta_0^{\M_\infty}$
	such that $\kappa=t^{\M_\infty}(\vec{\kappa},\alpha)$, and note
	we may assume that $\vec{\kappa}\cut\kappa\in\mathscr{I}^M\cut(\kappa+1)$
	by shifting this part up, but since $\vec{\kappa}\inter\kappa\sub\mathscr{I}^M$
	and $\M_\infty$ is a lightface $M$-class, this gives a contradiction.

	Part \ref{item:M_infty^N_is_ext-bar_M_infty}:
	Note that  Lemma \ref{lem:M_infty_of_iterate_N} applies.
	Like  in \S\ref{sec:ground_generation},
	we will define an elementary
	$\chi:\M_\infty^N\to(\M_\infty^{\overline{\ext}})_N$ and show
	that $\chi=\id$.
	We can cover $\mathscr{D}^N$ by
	with indices of the form $(P,u)$ with $u\in[\mathscr{I}^N]^{<\om}$.
	For given any $(Q,s)\in\mathscr{D}^N$,
	by  Lemma \ref{lem:every_alpha_ev_stable},
	we can fix
	$P\in\mathscr{F}^N$ such that $s$ is $(P,\mathscr{F}^N)$-stable
	(see Definition \ref{dfn:N_alpha_P-stable}) and $Q\preceq P$,
	and with $\delta_0^P$ large enough that there is $u\in[\mathscr{I}^N]^{<\om}$
	such that $s\in H^P_u$, which suffices.
	Because of this covering, we can define $\chi:\M_\infty^N\to(\M_\infty^{\overline{\ext}})_N$
	in the natural way; i.e.~for each such $(P,u)$ and $x\in H^P_u$,
	set $\chi(\pi_{Pu,\infty}^N(x))=i_{\bar{P}\M_\infty^N}(\pi_{\bar{P}P}^{-1}(x))$;
	note we have $u,x\in\rg(\pi_{\bar{P}P})$. It is now easy to see that $\M_\infty^N=(\M_\infty^{\overline{\ext}})_N$ and $\chi=\id$.

	Part \ref{item:I^M_infty^N=I^N}:
	By part \ref{item:I^M_infty=I^M} and since $i_{MN}``\mathscr{I}^M=\mathscr{I}^N$,
	we get $*^N\rest\mathscr{I}^N=\id$. And
	by part \ref{item:M_infty^N_is_ext-bar_M_infty} and its proof,
	$\M_\infty^N=\Hull_1^{\M_\infty^N}(\delta_0^{\M_\infty^N}\cup (*^N``\mathscr{I}^N))$.
	Since $\M_\infty^N$ is also a lightface $N$-class,
	$\mathscr{I}^N$ are model theoretic indiscernibles for $\M_\infty^N$.
	Therefore $\mathscr{I}^{\M_\infty^N}=\mathscr{I}^N$.

	Part \ref{item:iterate_of_M_infty}:
	\ref{item:M_infty^N_is_iterate} is an easy consequence
	of the fact that $\M_\infty^N$ and $\M_\infty^{N'}$ are $\delta_\infty^N$- and $\delta_\infty^{N'}$-sound
	respectively.
	For \ref{item:i_MN_restricts_to_iteration_map},
	we argue partly like in \cite[Lemma 2.9(a)]{vm1}, but somewhat differently.\footnote{
		Moreover,
		the proof
		of \cite[Lemma 2.9(a)]{vm1} has a bug: with notation as there, it talks about iteration maps
		$\pi_{j(N),N^*}$ and $\pi_{N^*,j(\M_\infty)}$, with the implication
		that $j(\M_\infty)$ is in fact an iterate of $j(N)$, but this is not true,
		as $j(N)$ is not $\delta_0^{j(N)}$-sound,
		whereas $j(\M_\infty)$ is $\delta_0^{j(\M_\infty)}$-sound.} So, note that by the preceding parts, $\M_\infty^{N'}$ is indeed an iterate of $\M_\infty^N$,
	and $i_{NN'}\rest\mathscr{I}^{\M_\infty^N}=i_{\M_\infty^N\M_\infty^{N'}}\rest\mathscr{I}^{\M_\infty^N}$,
	so  it just remains to see that \[i_{NN'}\rest\delta_0^{\M_\infty^N}= i_{\M_\infty^N\M_\infty^{N'}}\rest\delta_0^{\M_\infty^N}.\]
	So let $\alpha<\delta_0^{\M_\infty^N}$. Let $s\in[\mathscr{I}^{\M_\infty^N}]^{<\om}=[\mathscr{I}^N]^{<\om}$
	be such that $\alpha<\gamma^{\M_\infty^N}_s$. Fix $P\in\mathscr{F}^N$
	with some $\bar{\alpha}<\gamma^P_s$
	such that $\pi_{Ps,\infty}^N(\bar{\alpha})=\alpha$. Note that $\M_\infty^N$
	is a $\Sigma_{\bar{P}}$-iterate of $\bar{P}$ and
	\begin{equation}\label{eqn:i_Pbar(alpha-bar)=alpha} i_{\bar{P}\M_\infty^N}(\bar{\alpha})=\alpha.\end{equation}
	Now
	\[ i_{NN'}(\alpha)=\pi^{N'}_{P's',\infty}(\bar{\alpha}) \]
	where $P'=i_{NN'}(P)\in\mathscr{F}^{N'}$
	and $s'=i_{NN'}(s)\in\vec{\mathscr{I}^{N'}}$.
	But then with the map $\chi$ defined as earlier, but for $N'$ instead of $N$,
	\[
	i_{NN'}(\alpha)=
	\chi(\pi^{N'}_{P's',\infty}(\bar{\alpha}))=
	i_{\overline{P'}\M_\infty^{N'}}(\pi^{-1}_{\overline{P'}P'}(\bar{\alpha}))=
	i_{\bar{P}\M_\infty^{N'}}(\bar{\alpha})=
	i_{\M_\infty^N\M_\infty^{N'}}(\alpha),\]
	using that $\chi=\id$,
	$\overline{P'}=\bar{P}$, $\bar{\alpha}<\crit(\pi_{\overline{P'}P'})$, and line (\ref{eqn:i_Pbar(alpha-bar)=alpha}).

	Part \ref{item:not_an_iterate}: If $\bar{N}\neq N$ then $N$ is not $\delta_0^N$-sound,
	but then any non-dropping iterate $O$ of $N$ is non $\delta_0^O$-sound,
	so $O\neq(\M_\infty^{\overline{\ext}})_N=\M_\infty^N$.
\end{proof}

Note that with the preceding lemma, we have completed the proofs of
Lemmas \ref{lem:i_E(M_infty)_is_iterate},
\ref{M_knows_how_to_iterate_Minfty_up_to_its_woodin} and \ref{V_1_knows_how_to_iterate_Minfty_up_to_its_woodin}.
\subsection{Iterability of $\vV_1$}\label{subsec:iterability_of_vV_1}
In this subsection we will define a normal iteration strategy $\Sigma_{\vV_1}$
for $\vV_1$ in $V$.
We will first define and analyze the action of $\Sigma_{\vV_1}$ for trees based on $\M_\infty|\delta_\infty$.

\subsubsection{Tree translation from $M$ to $\vV_1$}

The iteration strategy for $\vV_1$ will be tightly connected to that for $M$,
as we describe now.
But first the basic notion under consideration:

\begin{definition}
	Let $N$ be a $\vV_1$-like Vsp. A \emph{$0$-maximal iteration tree $\Tt$ on $N$ of length $\lambda\geq 1$} is a system
	\[ \left(<_\Tt,\left<M_\alpha,m_\alpha\right>_{\alpha<\lambda},
	\left<E_\alpha\right>_{\alpha+1<\lambda}\right) \]
	with the usual properties, except that when $E_\alpha$ is a long extender (which is allowed),
	then $\pred^\Tt(\alpha+1)$ is the least $\beta\leq\alpha$
	such that $[0,\beta]_\Tt$ does not drop and $\delta_0^{M^\Tt_\beta}<\lh(E^\Tt_\alpha)$.

	We say that $\Tt$ is \emph{short-normal} iff $\Tt$ uses no long extenders.

	\emph{Iteration strategies} and \emph{iterability} for $N$ are now defined in the obvious manner.
\end{definition}

\begin{dfn}\label{dfn:short-normal}
	A \emph{short-normal tree} on a $\vV_1$-like Vsp $\vV$
	is a $0$-maximal tree that uses no long extenders.
	Note that a short-normal tree is of the form $\Tt\conc\Ss$,
	where $\Tt$ is based on $\vV|\delta_0^{\vV}$,
	and either
	\begin{enumerate}[label=\tu{(}\roman*\tu{)}]
		\item\label{item:short-extender_i} [$\Tt$ has limit length or $b^{\Tt}$ drops]
		and $\Ss=\emptyset$, or
		\item\label{item:short-extender_ii}  $\Tt$ has successor length,
		$b^{\Tt}$ does not drop and $\Ss$ is above $\delta_0^{M^{\Tt_0}_\infty}$.
	\end{enumerate}
	Say that $\Tt,\Ss$ are the \emph{lower, upper components} respectively.
\end{dfn}

\begin{definition}\label{dfn:translatable} Let $N$ be $\Mswsw$-like.
	An iteration tree $\Tt$ on $\vV_1^N$ is \emph{$\vV_1^N$-translatable}
	iff:
	\begin{enumerate}
		\item\label{item:om-max} $\Tt$ is $0$-maximal, and
		\item\label{item:kappa^+<lh(E)} $\kappa_0^{+M^\Tt_\alpha}<\lh(E^\Tt_\alpha)$
		for all
		$\alpha+1<\lh(\Tt)$ such that $[0,\alpha]_\Tt\inter\dropset^\Tt=\emptyset$.\qedhere
	\end{enumerate}
\end{definition}

\begin{rem}\label{rem:translatable}
	Under $0$-maximality,
	condition \ref{item:kappa^+<lh(E)}
	holds iff
	$\kappa_0^{+M^\Tt_\eta}<\lh(E^\Tt_\eta)$ for $\eta=0$ and for all limits $\eta$
	such that $\eta+1<\lh(\Tt)$ and $[0,\eta]_\Tt\inter\dropset^\Tt=\emptyset$
	and $i^\Tt_{0\eta}(\kappa_0^N)=\delta(\Tt\rest\eta)$.
	This follows easily from the fact that $\lh(E^\Tt_\alpha)<\lh(E^\Tt_\beta)$ for $\alpha<\beta$ (using Jensen indexing).
\end{rem}

\begin{definition}\label{dfn:vV_1-translation} Let $N$ be $\Mswsw$-like. Let $\Tt$ on $N$ be $\vV_1^N$-translatable. The \emph{$\vV_1^N$-translation} of $\Tt$
	is the $0$-maximal tree $\Uu$ on $\vV_1^N$ such that:
	\begin{enumerate}
		\item $\lh(\Uu)=\lh(\Tt)$ and $\Uu,\Tt$ have the same tree, drop and degree structure,
		\item $\lh(E^\Uu_\alpha)=\lh(E^\Tt_\alpha)$ for all $\alpha+1<\lh(\Tt)$.\qedhere
	\end{enumerate}
\end{definition}
\begin{rem}\label{rem:vV_1_of_dropped_iterate}
	Let $N$ be $\Mswsw$-like. Let $\Tt$ be a tree on $N$ and let $\alpha<_\Tt\vareps+1\leq_\Tt\beta$
	be such that $[0,\alpha]_\Tt$ does not drop, $\vareps+1\in\dropset^\Tt$,
	$\pred^\Tt(\vareps+1)=\alpha$
	and $\kappa_0^{+M^\Tt_\alpha}<\crit(E^\Tt_\vareps)$.
	Note that
	\[
	\kappa_0^{+M^\Tt_\alpha}<\gamma=\gamma^{\vV_1(M^\Tt_\alpha)}
	<\xi<\crit(E^\Tt_\vareps) \]
	where $\xi$ is the least $M^\Tt_\alpha|\kappa_0^{+M^\Tt_\alpha}$-admissible.
	Note here that $\vV_1^{M^\Tt_\beta}=\vV_1(M^\Tt_\beta)$ is the set-sized model
	$\vV$ such that $\vV||\gamma=\vV_1^{M^\Tt_\alpha}||\gamma$
	and above $\gamma$, $\es_+^\vV$ is the level-by-level translation of
	$\es_+(M^\Tt_\beta)$.
	Note that because $\gamma<\OR^{\vV_1(M^\Tt_\beta)}$,
	$M^\Tt_\beta$ is a $\kappa_0^{+M^\Tt_\alpha}$-cc forcing extension of
	$\vV_1(M^\Tt_\beta)$.\end{rem}

\begin{lemma}\label{lem:vV_1-translatable}
	Let $\Tt$ on $N$ be $\vV_1^N$-translatable, where $N$ is $\Mswsw$-like. Then:
	\begin{enumerate}
		\item The $\vV_1^N$-translation $\Uu$ of $\Tt$
		exists and is unique.
		\item $M^\Uu_\alpha=\vV_1^{M^\Tt_\alpha}$ and $\deg^\Uu_\alpha=\deg^\Tt_\alpha$ and $\gamma^{M^\Uu_\alpha}<\OR(M^\Uu_\alpha)$ for all $\alpha<\lh(\Tt)$.
		\item $i^\Uu_{\alpha\beta}=i^\Tt_{\alpha\beta}\rest M^\Uu_\alpha$ for all $\alpha<_\Tt\beta$ such that $(\alpha,\beta]_\Tt$ does not drop.
		\item $M^{*\Uu}_{\alpha+1}=\vV_1^{M^{*\Tt}_{\alpha+1}}$ for all $\alpha+1<\lh(\Tt)$.
		\item $i^{*\Uu}_{\alpha+1}=i^{*\Tt}_{\alpha+1}\rest M^{*\Uu}_{\alpha+1}$ for all $\alpha+1<\lh(\Tt)$.
	\end{enumerate}
\end{lemma}

\begin{proof}
	This is partly via the usual translation of iteration trees between models and P-constructions thereof.
	However, there is a new feature here, when $\alpha+1<\lh(\Tt)$
	and is such that $[0,\alpha+1]_\Tt$ does not drop
	and letting $\beta=\pred^\Tt(\alpha+1)$, then $\crit(E^\Tt_\alpha)=\kappa=\kappa_0(M^\Tt_\beta)$,
	so consider this situation.

	Then $E^\Uu_\alpha$ is long with space
	$\delta=\delta_0(M^\Uu_\beta)=\kappa_0^{+M^\Tt_\beta}$, and
	$E^\Uu_\alpha=E^\Tt_\alpha\rest (M^\Uu_\beta|\delta)$
	and $[0,\beta]_\Uu$ does not drop,
	and $M^\Uu_\beta=\vV_1(M^\Tt_\beta)$.
	Let $(a,f)$ be such that $f\in M^\Tt_\beta$
	and $a\in[\nu(E^\Tt_\alpha)]^{<\om}$ and
	\[ f:[\kappa]^{|a|}\to M^\Uu_\beta=\vV_1(M^\Tt_\beta).\]
	We need some $(b,g)$ with $g\in M^\Uu_\beta$ and
	$b\in[\lh(E^\Uu_\alpha)]^{<\om}$
	such that $a\sub b$ and $f^{ab}(u)=g(u)$ for $(E^\Tt_\alpha)_b$-measure one many $u$.
	We may assume $\rg(f)\sub\OR$.

	If $\rg(f)\sub\delta$, the existence of $(b,g)$ is just because
	$E^\Uu_\alpha$ is the restriction of $E^\Tt_\alpha$, and this restriction
	is cofinal in $\lh(E^\Tt_\alpha)$.
	In general we will reduce to this case.

	Now $M^\Tt_\beta$ is a $\delta$-cc forcing extension of $M^\Uu_\beta$,
	so $\rg(f)\sub X$ for some $X\in M^\Uu_\beta$, where $X$ has cardinality ${<\delta}$ in $M^\Uu_\beta$.
	Let $\eta$ be the ordertype of $X$, so $\eta<\delta$, let $\pi:X\to\eta$ be the collapse,
	and let $\widetilde{f}=\pi\com f$. So $\widetilde{f}\in M^\Tt_\beta$
	and $\rg(\widetilde{f})\sub\delta$, so we get a corresponding pair $(b,\widetilde{g})$, with $\widetilde{g}\in M^\Uu_\beta$.
	Letting $g=\pi^{-1}\com\widetilde{g}$, then $g\in M^\Uu_\beta$ and $(b,g)$ works.
\end{proof}

\subsubsection{Trees based on $\M_\infty|\delta_\infty$}

We now  transfer trees on $\M_\infty$, based on $\M_\infty|\delta_\infty$, to trees on $\vV_1$.
\begin{dfn}\label{dfn:Psi_vV_1,vV_1^-}
	Write $\Sigma_{\M_\infty,\vV_1^-}$
	for the normal strategy for $\M_\infty$ for trees based on $\vV_1^-$, induced
	by $\Sigma_{\M_\infty}$. We use analogous notation $\Sigma_{N,N|\alpha}$ more generally.
	Let $\Psi_{\vV_1,\vV_1^-}$ denote the putative normal strategy for
	trees on $\vV_1$ based on $\vV_1^-$,
	induced by $\Sigma_{\M_\infty,\vV_1^-}$.
	This makes sense
	by Lemma \ref{still-a-woodin-in-v}.
\end{dfn}

\begin{rem}\label{rem:basic_trees_via_Psi_vV_1,vV_1^-}
	Let $\Uu$ be a putative tree on $\vV_1$, based on $\vV_1^-$, via
	$\Psi_{\vV_1,\vV_1^-}$.
	Let $\alpha<\lh(\Uu)$. Suppose $[0,\alpha]_\Uu$ does not drop.
	Then
	${M^\Uu_\alpha\downarrow
		0}=i^\Uu_{0\alpha}(\M_\infty)$, and if $M^\Uu_\alpha$ is wellfounded then it is $\vV_1$-like.
	If instead $[0,\alpha]_\Uu$ drops, note that it drops below the image of $\delta_0^{\vV_1}$ and $M^\Uu_\alpha$ is a premouse (note that it is wellfounded in this case), so ${M^\Uu_\alpha\downarrow
		0}=M^\Uu_\alpha$.
\end{rem}

\begin{dfn}
	Let $\vV$ be a Vsp. Then $\Lambda^{\vV}$ denotes the partial putative strategy
	for $\vV|\delta_0^{\vV}$ determined by the long extenders of $\vV$.
	That is, $\Lambda^{\vV}(\Tt)=b$ iff  $\Tt\in\vV$, $\Tt$ is on $\vV|\delta_0^{\vV}$,
	 $\vV\sats$``$\Tt$ is via $\Sigma_{\sss}$'', and either
	 \begin{enumerate}[label=--]
	 	\item $\vV\sats$``$\Tt$ is short
	 and $b=\Sigma_{\sss}(\Tt)$'', or
	 \item $\vV\sats$``$\Tt$ is maximal''
	 and there is a long $E\in\es_+(\vV)$ such that
	 $\Tt\in\vV|\lambda(E)$ and $b$ is computed
	 via factoring through $c$ as in Footnote \ref{ftn:find_branch_via_normalization},
	 where $c$ is the cofinal branch through the tree from
	 $\vV|\delta_0^{\vV}$ to $j(\vV|\delta_0^{\vV})$
	 determined by $j=i^{\vV|\lh(E)}_E$.\qedhere
	 \end{enumerate}
	\end{dfn}

The following lemma, which is the main point of this subsubsection,
is the analogue
of \cite[Lemma 2.17]{vm1} and \cite[Claim 12]{Theta_Woodin_in_HOD}.

\begin{lem}\label{lem:Gamma_A_is_good}
	$\Psi_{\vV_1,\vV_1^-}$ yields wellfounded models.
	Moreover, let $\Tt$ be on $\M_\infty$,
	via $\Sigma_{\M_\infty,\vV_1^-}$,
	and let $\Uu$ be the corresponding tree on $\vV_1$ (so via
	$\Psi_{\vV_1,\vV_1^-}$).
	Let
	\[ \pi_\alpha:M^\Tt_\alpha\to M^\Uu_\alpha\downarrow 0\sub M^\Uu_\alpha \]
	be the natural copy map
	\tu{(}where $\pi_0=\id$\tu{)}. Then:
	\begin{enumerate}[label=--]
		\item $[0,\alpha]_\Tt$ drops iff $[0,\alpha]_\Uu$ drops.
		\item If $[0,\alpha]_\Tt$ drops then $M^\Tt_\alpha=M^\Uu_\alpha=M^\Uu_\alpha\downarrow 0$.
		\item If $[0,\alpha]_\Tt$ does not drop then $M^\Tt_\alpha=M^\Uu_\alpha\downarrow 0$ and $M^\Uu_\alpha=\vV(M^\Tt_\alpha,\ell)$ where $\ell:M^\Tt_\alpha\to\M_\infty^{M^\Tt_\alpha}$
		is the $\Sigma_{M^\Tt_\alpha}$-iteration map, and in fact, $\Lambda^{M^\Uu_\alpha}\sub\Sigma_{N,N|\delta_0^N}$ where $N=M^\Tt_\alpha$.
		\item $\pi_\alpha=\id$; therefore, $i^\Tt_\alpha\sub i^\Uu_\alpha$.
	\end{enumerate}
\end{lem}
\begin{proof}
	We include the proof, mostly following that of \cite{Theta_Woodin_in_HOD}.
	Let $\Tt,\Uu$ be as above, of length $\alpha+1$.
	The interesting case is the non-dropping one, so consider this.

	Let $P=M^\Tt_\alpha$ and $P_\infty=\M_\infty^P$.
	Let $i_{\M_\infty P}$, $i_{P P_\infty}$ and $i_{\M_\infty P_\infty}$ be the iteration maps.
	We have
	\[ \M_\infty=\Hull^{\M_\infty}_1(\delta_0^{\M_\infty}\cup
	\mathscr{I}^{\M_\infty}); \]
	likewise for $P,P_\infty$ (maybe $\mathscr{I}^P\neq\mathscr{I}^{\M_\infty}$,
	but ${\mathscr{I}}^{P_\infty}=\mathscr{I}^P=i_{\M_\infty
		P}``\mathscr{I}^{\M_\infty}$).
	We have
	\[ P_\infty=\vV_1^P\downarrow 0\sub\vV_1^P.\]

	The analogue of the following claim
	was used in the proof of \cite[Claim 12]{Theta_Woodin_in_HOD},
	where it was implicitly asserted but the proof not explicitly given.
	We give the proof here. It is just a slight generalization of the proofs of
	\cite[Claims 8--10]{Theta_Woodin_in_HOD} (or see \cite[Lemma 2.15]{vm1}),
	the main conclusion of which is
	that if $j:M\to\M_\infty$ is the iteration map,
	then
	\[ \M_\infty\inter\Hull^{\M_\infty[*]}(\rg(j))=\Hull^{\M_\infty}(\rg(j))=\rg(j).\]

	\begin{clm*}  We have:
		\begin{enumerate}[label=\tu{(}\roman*\tu{)}]
			\item\label{item:Hull_unchanged_M_inf_to_P_inf}
			$P_\infty\inter\Hull^{P_\infty[*^P]}(\rg(i_{\M_\infty P_\infty}))=\rg(i_{\M_\infty
				P_\infty})$
			and
			\item\label{item:P-hull_conservative} $P_\infty\inter\Hull^{P_\infty[*^P]}(\rg(i_{P P_\infty}))=\rg(i_{P P_\infty})$.
		\end{enumerate}
	\end{clm*}
	\begin{proof}
		Consider \ref{item:Hull_unchanged_M_inf_to_P_inf}.
		Fix $\alpha\in\OR$ and $s\in
		[\mathscr{I}^{\M_\infty}]^{<\om}\cut\{\emptyset\}$ and
		$\beta<\delta_0^{\M_\infty}$
		and a term $t$ such that
		\[ \alpha=t^{P_\infty[*^P]}(i_{\M_\infty P_\infty}(s^-,\beta))\]
		and $\beta<\gamma^{\M_\infty}_s$.
		We need to see
		$\alpha\in\rg(i_{\M_\infty P_\infty})$.
		It suffices to see  $*^P(\alpha)\in\rg(i_{\M_\infty P_\infty})$,
		by arguments in \cite{Theta_Woodin_in_HOD}.
		But this holds just as in \cite{Theta_Woodin_in_HOD},
		except that we have a fixed term $u$ such that
		for each $N\in\mathscr{F}^P$,
		\[ \alpha=u^N(i_{\M_\infty N}(s,\beta)). \]
		This suffices. Part \ref{item:P-hull_conservative} is analogous.
	\end{proof}

	Let
	$\widetilde{P}=\cHull^{P_\infty[*^P]}(\rg(i_{P P_\infty}))$
	and $\widetilde{\M_\infty}=\cHull^{P_\infty[*^P]}(\rg(i_{\M_\infty
		P_\infty}))$, let
	$i_{P P_\infty}^+$  and $i_{\M_\infty P_\infty}^+$  the
	uncollapse maps, and
	$i_{\M_\infty P}^+=(i_{PP_\infty}^+)^{-1}\com i_{MP_\infty}^+$.
	By the claim, $i_{\M_\infty P}\sub i_{\M_\infty P}^+$.

	Also, $*^P\sub h$ where $h:P_\infty\to\M_\infty^{P_\infty}$
	is the iteration map. Letting
	\[ i^+_{\M_\infty P_\infty}(*')=i^+_{P
		P_\infty}(*'')=*^P\rest\delta_0^{P_\infty}, \]
	it easily follows that $*'$ and $*''$ agree with the iteration maps
	$\M_\infty\to(\M_\infty)^{\M_\infty}$ and $P\to P_\infty$
	respectively. Therefore
	$\widetilde{\M_\infty}=\M_\infty[*]$.

	Let $E$ be the
	$(\delta_0^{\M_\infty},\delta_0^{P})$-extender
	derived from $i_{\M_\infty P}$, or equivalently
	from $i_{\M_\infty P}^+$, also equivalently
	the $[0,\alpha]_\Tt$-branch extender of $\Tt$.
	So (recalling $\Uu$ is the corresponding tree on $\M_\infty[*]$)
	\[ M^\Uu_\alpha=\Ult(\M_\infty[*],E)\]
	and $i^\Uu_\alpha=i^{\M_\infty[*]}_E$.
	We also have $P=M^\Tt_\alpha=\Ult(\M_\infty,E)$
	and $i_{\M_\infty P}=i^\Tt_\alpha$.
	Let $\pi:M^\Uu_\alpha\to\widetilde{P}$
	be the natural factor map, i.e.
	\[ \pi(i^\Uu_\alpha(f)(a))=i^+_{\M_\infty P}(f)(a) \]
	whenever $f\in \M_\infty[*]$ and $a\in[\delta_0^P]^{<\om}$.
	Then $\pi$ is surjective, because if $\alpha\in\OR$
	then there is $f\in\M_\infty$ and $a\in[\delta_0^P]^{<\om}$
	such that $i_{\M_\infty P}(f)(a)=\alpha$,
	and since $i_{\M_\infty P}\sub i_{\M_\infty P}^+$,
	therefore $i_{\M_\infty P}^+(f)(a)=\alpha=\pi(\alpha)$.
	So in fact $M^\Uu_\alpha=\widetilde{P}$ and $\pi=\id$,
	so $i^\Uu_\alpha=i^+_{\M_\infty P}$,
	so $i^\Tt_\alpha\sub i^\Uu_\alpha$,
	and letting $\pi_\alpha:M^\Tt_\alpha\to(M^\Uu_\alpha\downarrow 0)$
	be the natural copy map, then $\pi_\alpha\sub\pi=\id$.

	It just remains to see that $\Lambda^{M^\Uu_\alpha}\sub\Sigma_{P}$ (still with $P=M^\Tt_\alpha$).
	First consider the case that for some
	correct normal above-$\kappa_0$ tree $\Vv$ on $M$ and $E=E^\Vv_\alpha$,
	we have $\crit(E)=\kappa_0$ and $E$ is $M$-total, and $P=\M_\infty^U$ where $U=\Ult(M,E)$.
	Here by Lemma \ref{lem:M_infty_of_kappa_0-sound_iterate}, $\M^U_\infty$ is indeed a $\delta_0^{\M^U_\infty}$-sound iterate
	of $\M_\infty$, and $i_{E}\rest\M_\infty$ is just the iteration map.
	Moreover, by Lemma \ref{lem:vV_1-translatable}, $\vV_1^U=i_E(\vV_1)$ is the corresponding iterate of $\vV_1$.
	But now the calculations that work for $\Lambda^{\vV_1}$ (the proof of Lemma \ref{V_1_knows_how_to_iterate_Minfty_up_to_its_woodin},
	using Lemma \ref{lem:M_infty_of_kappa_0-sound_iterate})
	also work for
	$\Lambda^{\vV_1^U}$.

	Now consider the general case. We will reduce this to the special case
	above via Lemma \ref{lem:branch_con}. Let $E\in\es^M$ be $M$-total
	with $\crit(F)=\kappa_0$, and $\bar{\delta}$ the least Woodin
	of $M|\lh(F)$ such that $\kappa_0<\bar{\delta}$. Form a genericity iteration
	at $\bar{\delta}$, above $\kappa_0$, making $P|\delta_0^P$ etc generic.
	Let $E$ be the eventual image of $F$. Then $E$ is as in the previous case;
	let $U=\Ult(M,E)$ and
	let $\M_\infty^U$. Recall $\widetilde{P}=M^\Uu_\alpha$
	is an iterate of $\vV_1$,
	and note $\vV_1^U$ is the corresponding iterate of $\widetilde{P}$;
	let
	$k:\widetilde{P}\to\vV_1$ be the iteration map.
	Let $\beta$
	be such that $\widetilde{P}|\beta$ is active with a long extender $G$,
	$\beta'=k(\beta)$, and $G'=F^{\vV_1^U|\beta'}$. Note that $k$
	is continuous at $\delta_0^P$ and $\beta$ (as $\cof^P(\beta)=\delta_0^P$).
	Let $j:P\to\Ult(P,G)$ and $j':\M_\infty^U\to\Ult(\M_\infty^U,G')$ be the ultrapower maps.
	Let $\Tt$ be the length $\beta$ tree from $P$ to $\Ult(P,G)$
	and $\Tt'$ the length $\beta'$ tree from $\M_\infty^U$ to $\Ult(\M_\infty^U,G')$;
	note that by first order considerations, these exist, and $G$ determines a $\Tt$-cofinal branch
	$b$ such that $M^\Tt_b=\Ult(P,G)$ and $j=i^\Tt_b$,
	and likewise for $G',\Tt',b',\M_\infty^U,j'$. We know that $\Tt'\conc b'$ is via $\Sigma_{\M_\infty^U}$, by the previous case. Note also that $\Tt$ is via $\Sigma_P$,
	because the Q-structure for each $\Tt\rest\lambda$ (for limit $\lambda<\beta$) does not overlap
	$\delta(\Tt\rest\lambda)$, and is embedded into the Q-structure
	for $\Tt'\rest k(\lambda)$. But
	\[ k\com j\rest(P|\delta_0^P)=j'\com k\rest(P|\delta_0^P) \]
	and $j(\delta_0^P)=\beta$ and $j'(\delta_\infty^U)=\beta'$.
	So by Lemma \ref{lem:branch_con}, $b=\Sigma_{P}(\Tt)$, as desired.
\end{proof}

\begin{dfn}\label{dfn:Psi_V,V^-}
	Given a non-dropping $\Psi_{\vV_1,\vV_1^-}$-iterate
	$\vV$ of $\vV_1$, let $\Psi_{\vV,\vV^-}$ be induced by $\Sigma_{\vV\downarrow 0}$
	just as $\Psi_{\vV_1,\vV_1^-}$ is induced by $\Sigma_{\M_\infty}$
	(this makes sense by Lemma \ref{lem:Gamma_A_is_good}).
\end{dfn}

\subsubsection{Condensation properties for full normalization}

The strategy $\Sigma_{\vV_1}$ (together with $\vV_1$) will have the properties required
for extending to a strategy for stacks with full normalization. We now
lay out the properties of $\vV_1$ needed for this.
	Recall the notions \emph{$n$-standard}, \emph{$(n+1)$-relevantly condensing}
	and \emph{$(n+1)$-sub-condensing} from \cite[***Definition 2.1]{fullnorm_v3}. We adapt
	these in an obvious manner to Vsps.

\begin{dfn}\label{dfn:norm_condensation} Let $m<\om$ and let $\vV$ be an $(m+1)$-sound Vsp. We say that $\vV$ is \emph{$(m+1)$-relevantly condensing}
	iff either $\vV$ is a premouse which
	is $(m+1)$-relevantly condensing,
	or $\vV$ is a sound base Vsp,
	or $\gamma^{\vV}<\OR^\vV$ and $\vV$ satisfies
	the requirements of \emph{$(m+1)$-relevantly condensing}
	from \cite[Definition 2.1]{fullnorm_v3} for $\pi:P\to\vV$
	such that $P$ is an $(m+1)$-sound Vsp, $\gamma^P<\OR^P$ and $\crit(\pi)>\delta_0^{\vV}$
	(so $\crit(\pi)>\gamma^{\vV}$). Likewise for \emph{$(m+1)$-sub-condensing}.

	For $n<\om$, a Vsp $\vV$ is \emph{$n$-standard}
	iff $\vV$ is $n$-sound and either $\vV$ is an $n$-standard premouse,
	or $\vV$ is a base Vsp and $\vV^\passive$ is $\om$-standard, or $\gamma^{\vV}<\OR^{\vV}$
	and $\vV$ is $(m+1)$-relevantly condensing for each $m<n$,
	and every $M\pins\vV$ is $(m+1)$-relevantly-condensing
	and $(m+1)$-sub-condensing for each $m<\om$.

	A Vsp is \emph{$\om$-standard} iff $n$-standard
	for each $n<\om$.
\end{dfn}

\begin{lem}\label{lem:vV_1_norm_condensation} $\vV_1$ is $\om$-standard. (Thus, we take \emph{$\vV_1$-like}
	to include \emph{$\om$-standard}.)
\end{lem}
\begin{proof}
	Let $\alpha>\gamma^{\vV}$ and $P$ be  a Vsp and $\pi:P\to\vV|\alpha$ be an embedding as in the definition
	of $(m+1)$-relevantly- or $(m+1)$-sub-condensing. We want to know
	that $P\pins\vV$. But note that there is a premouse $N$
	such that $M|\kappa_0^{+M}\pins N$
	and $\vV^N=P$ and $\pi$ extends to $\pi^+:N\to M|\alpha$,
	which also satisfies the conditions
	of $(m+1)$-relevantly- or $(m+1)$-sub-condensing, respectively.
	So $N\pins M$, so $P\pins\vV$.
\end{proof}

\begin{rem}\label{rem:iteration_preserves_standardness_at_degree}
	Let $\vV$ be $\vV_1$-like.
	Then as for premice, if $\Tt$ is a $0$-maximal
	tree on $\vV$ then $M^\Tt_\alpha$ is $\deg^\Tt_\alpha$-standard (see \cite[***Remark 2.2]{fullnorm_v3}).
\end{rem}

\subsubsection{Short-normal trees on $\vV_1$}

Recall that   \emph{short-normal} trees on $\vV_1$-like Vsps were defined in Definition
\ref{dfn:short-normal}.

\begin{dfn}
	Let $\vV$ be a (possibly dropping, putative) iterate of $\vV_1$,
	via a short-normal
	tree $\Tt\conc\Ss$ with lower and upper components $\Tt,\Ss$.
	We say that $\vV$ is \emph{good} iff $\Tt$ is via $\Psi_{\vV_1,\vV_1^-}$
	and if $b^\Tt$ does not drop then $\vV$ is wellfounded and for every long $E\in\es_+^\vV$,
	$\M_\infty^{\vV|\lh(E)}=N|\delta_0^N$ for some $\Sigma_{\vV\downarrow 0}$-iterate
	$N$ of $\vV\downarrow 0$, and $E$ is the corresponding iteration map.
	Say that a (partial) iteration strategy $\Psi$
	is \emph{good} iff all putative iterates via $\Psi$ are good.
\end{dfn}

Note we have already shown that $\Psi_{\vV_1,\vV_1^-}$ is good.
We now want to extend $\Psi_{\vV_1,\vV_1^-}$ to a good short-normal $0$-maximal strategy
$\Psi_{\sn}$ for $\vV_1$.
So we start by setting $\Psi_{\vV_1,\vV_1^-}\sub\Psi_{\sn}$.
As an easy next step, we deal with trees based on $\vV_1||\gamma^{\vV_1}$.

\begin{dfn}\label{dfn:Psi_for_vV_1|gamma}
	Write $\Psi_{\vV_1,\gamma^{\vV_1}}$ for the putative strategy $\Psi$ for $\vV_1$,
	for short-normal $0$-maximal trees based on $\vV_1||\gamma^{\vV_1}$, as follows:
	\begin{enumerate}
		\item $\Psi_{\vV_1,\vV_1^-}\sub\Psi$, and
		 \item given $\Tt$ via $\Psi_{\vV_1,\vV_1^-}$, of successor length $\alpha+1$,
		 where $[0,\alpha]_\Tt$ does not drop, and given a putative $0$-maximal tree $\Uu$ on $M^\Tt_\alpha||\gamma^{M^\Tt_\alpha}$, which is above $\delta_0^{M^\Tt_\alpha}$,
		 then $\Tt\conc\Uu$ is via $\Psi$
		 iff there is a tree $\Uu'$ on ${M^\Tt_\alpha\downarrow 0}$,
		 via $\Sigma_{M^\Tt_\alpha\downarrow 0}$, with the same extenders and tree order as $\Uu$.\qedhere
    \end{enumerate}
	\end{dfn}

Note here that by Lemma \ref{local-definability-of-that-structure},
$\rho_1^{M^\Tt_\alpha||\gamma^{M^\Tt_\alpha}}=\delta_0^{M^\Tt_\alpha}$,
a strong cutpoint of $(M^\Tt_\alpha||\gamma^{M^\Tt_\alpha})^\passive$,
so $\Tt\conc\Uu$ is indeed a putative $0$-maximal tree on $\vV_1$.

\begin{lem}\label{lem:Psi_vV_1,gamma^vV_1_good}
	$\Psi_{\vV_1,\gamma^{\vV_1}}$ is a short-normal $0$-maximal strategy (hence yields wellfounded models). Moreover, let
	 $\Tt\conc\Uu$ and $\Uu'$ be as in Definition \ref{dfn:Psi_for_vV_1|gamma}, with $\Uu\neq\emptyset$.
	Then:
	\begin{enumerate}
		\item $M^\Uu_0=M^\Tt_\alpha||\gamma^{M^\Tt_\alpha}$ and $\deg^\Uu_0=0$,
		\item $M^{\Uu'}_0={M^\Tt_\alpha\downarrow 0}$ and $\deg^{\Uu'}_0=0$,
		so $(M^\Uu_0)^\passive=M^{\Uu'}_0|\kappa_0^{+M^{\Uu'}_0}$,
		\item for $0<\beta<\lh(\Uu)$, $\beta\in\dropset_{\deg}^\Uu\Leftrightarrow\beta\in\dropset_{\deg}^{\Uu'}$,
		and $\deg^\Uu_\beta=\deg^{\Uu'}_\beta$,
		\item if $0<\beta<\lh(\Uu)$ and $[0,\beta]_\Uu$ drops then $M^\Uu_\beta=M^\Uu_{\beta'}$,
		\item if $0<\beta<\lh(\Uu)$ and $[0,\beta]_\Uu$ does not drop then
		$(M^\Uu_\beta)^\passive=M^{\Uu'}_\beta|\kappa_0^{+M^{\Uu'}_\beta}$,
		\item if $0<\beta+1<\lh(\Uu)$ and $[0,\beta+1]_\Uu$ drops then
		$M^{*\Uu}_{\beta+1}=M^{*\Uu'}_{\beta+1}$ and $i^{*\Uu}_{\beta+1}=i^{*\Uu'}_{\beta+1}$,
		\item if $0<\beta+1<\lh(\Uu)$ and $[0,\beta+1]_{\Uu'}$ does not drop
		then $i^{*\Uu}_{\beta+1}\sub i^{*\Uu'}_{\beta+1}$,
		\item\label{item:Ult(M^T_alpha_down_0,F)} if $0\leq\beta<\lh(\Uu)$ and $[0,\beta]_\Uu$ does not drop
		then $M^{\Uu'}_\beta$ is a $\kappa_0^{M^{\Uu'}_\beta}$-sound $\Sigma_{M^\Tt_\alpha\downarrow 0}$-iterate of ${M^\Tt_\alpha\downarrow 0}$,  $\M_\infty^{M^{\Uu'}_\beta}$
		is a $\delta_0^{\M_\infty^{M^{\Uu'}_\beta}}$-sound $\Sigma_{M^\Tt_\alpha\downarrow 0}$-iterate
		of ${M^\Tt_\alpha\downarrow 0}$,
		\[ \Ult(M^\Tt_\alpha\downarrow 0,F(M^\Tt_\beta))=\M_\infty^{M^{\Uu'}_\beta} \]
		and $F(M^\Tt_\beta)$ is the extender of the $\Sigma_{M^\Tt_\alpha\downarrow 0}$-iteration map.
	\end{enumerate}
Therefore $\Psi_{\vV_1,\gamma^{\vV_1}}$ is good.
	\end{lem}
\begin{proof}
We omit most of the proof, as it follows from the usual calculations.
However, for part \ref{item:Ult(M^T_alpha_down_0,F)},
just note that the action of the $\Uu$- and $\Uu'$-iteration maps $j$ and $j'$
on $M^\Uu_0=M^\Tt_\alpha||\gamma^{M^\Tt_\alpha}$ are identical (i.e. $j\sub j'$), since $(M^\Uu_0)^\passive=M^{\Uu'}_0|\kappa_0^{+M^{\Uu'}_0}$,
and so
\[ j(\M_\infty^{M^\Uu_0})=j'(\M_\infty^{M^{\Uu'}_0}|\delta_\infty^{M^{\Uu'}_0})=\M_\infty^{M^{\Uu'}_\beta}|\delta_\infty^{M^{\Uu'}_\beta};\]
but then the fact that $F(M^\Tt_\beta)$ agrees with the $\Sigma_{M^\Tt_\alpha\downarrow 0}$-iteration map
\[ ({M^\Tt_\alpha\downarrow 0})|\delta_0^{M^\Tt_\alpha\downarrow 0}\to \M_\infty^{M^{\Uu'}_\beta}\]
is a consequence, since $F^{\vV||\gamma^{\vV}}$ is likewise correct,
where $\vV=M^\Tt_\alpha$ (by Lemma \ref{lem:Gamma_A_is_good})
and $j'$ preserves indiscernibles.
\end{proof}

Note that if $\Tt\conc\Uu$ is as above, with last model $M^\Uu_\beta$,
then applying $F(M^\Uu_\beta)$ as the next extender (giving a non-short-normal tree),
the next model is again an iterate of $\vV_1$ via $\Psi_{\vV_1,\vV_1^-}$.

\begin{dfn}\label{dfn:Gamma_W}
Say that $(\Ww,U,\Tt)$ is \emph{$M$-standard} iff
$\Ww$ is a $\Sigma$-tree  on $M$
which is above $\kappa_0$, $\lh(\Ww)=\xi+2$ for some $\xi$,
 and letting $E=E^\Ww_\xi$, then $\crit(E)=\kappa_0$
 and $E$ is $M$-total, $U=\Ult(M,E)=M^\Ww_{\xi+1}$,
and $\Tt$ is the tree leading from $\vV_1$ to $\vV_1^U$; so
 \[ M^\Tt_\infty=\vV_1^U=\Ult(\vV_1,E\rest\vV_1)\]
and $i^\Tt\sub i^\Ww$.

Suppose $(\Ww,U,\Tt)$ is $M$-standard.
We define a strategy $\Gamma_{\Ww}$
for above-$\delta_0^{\vV_1^U}$ short-normal trees $\Ss$ on $\vV_1^U$.
\footnote{We use subscript $\Ww$, not $\Tt$, as one can have
$M$-standard $(\Ww',U',\Tt')\neq(\Ww,U,\Tt)$
with $\Tt'=\Tt$. We will later see that, in this case however,
$\Gamma_\Ww=\Gamma_{\Ww'}$.}
Let $\vV=\vV_1^U$ and $N=\vV\downarrow 0=\M_\infty^U$.

If $\lh(E^\Ss_0)<\gamma^{\vV}$,
then $\Gamma_\Ww$ follows $\Psi_{\vV_1,\gamma^{\vV_1}}$ (recalling that $\Tt$ is via $\Psi_{\vV_1,\vV_1^-}$).

Suppose $\lh(E^\Ss_0)>\gamma^{\vV}$.
Let $\Gamma$ be the above-$\kappa_0^{+U}$-strategy
for $U$ given by $\Sigma_U$. Then since
 $\vV=\vV_1^U$ is defined by P-construction,
$\Gamma$ induces a short-normal above-$\gamma^{\vV}$-strategy for $\vV$,
which $\Gamma_\Ww$ follows in this case.

We extend $\Gamma$ to $\Gamma^{V[G]}$,
for set-generic extensions $V[G]$ of $V$,
using that $\Sigma$ extends canonically to $\Sigma^{V[G]}$.
\end{dfn}

\begin{lem}
$\Gamma_\Ww$ is good, and hence so is each
$\Gamma^{V[G]}_\Ww$.
\end{lem}
\begin{proof} Clearly $\Gamma_\Ww$ yields wellfounded models,
and we already saw that $\Psi_{\vV_1,\gamma^{\vV_1}}$ is good.
So  with notation as in Definition \ref{dfn:Gamma_W},
suppose $\lh(E^\Ss_0)>\gamma^{\vV}$ and $\Ss$ has successor length,
and let $\Ss'$ be the corresponding
above-$\kappa_0^{+U}$ tree on $U$. Let $\beta\leq\OR(M^\Ss_\infty)$
with $F=F(M^\Ss_\infty||\beta)$
long, and $\beta>\gamma^{\vV}$.
Let $F'=F(M^{\Ss'}_\infty||\beta)$.
Then $\crit(F')=\kappa_0^U$ and $F'$ is $U$-total,
and $F\sub F'$. So goodness with respect to $F$
is an easy consequence of Lemma \ref{lem:M_infty_of_kappa_0-sound_iterate}
(note that $U$ and $U'=\Ult(U,F')$ are $\kappa_0^U$- and $\kappa_0^{U'}$-sound $\Sigma$-iterates of $M$).
\end{proof}

\begin{lem}\label{lem:M-standard_with_T_generic}
	Let $A$ be a set of ordinals.
	Then there is an $M$-standard $(\Ww,U,\Xx)$
	such that $A$ is $(U,\PP)$-generic for some $\PP\in U|\kappa_0^U$.
	Moreover, if $\vV$ is a non-dropping $\Psi_{\vV_1,\vV_1^-}$-iterate
	of $\vV_1$, via maximal tree $\Tt=A$,
	then $\vV_1^U$ is a $\Psi_{\vV,\vV^-}$-iterate of $\vV$ (see Definition \ref{dfn:Psi_V,V^-}).

	Likewise if $V[G]$ is a set-generic
	extension of $V$ and $A\in\pow(\OR)\cap V[G]$.
	\end{lem}
\begin{proof}
	Let $F\in\es^M$ be $M$-total with $\crit(F)=\kappa_0$.
	Then $\lambda_F$ is a limit of Woodin cardinals of $M|\lh(F)$,
	and $\nu_F<\lambda_F$.
	Let $\delta\in(\nu_F,\lambda_F)$ be Woodin in $M|\lh(F)$.
	Let $\Ww'$ be an above-$\nu_F$ genericity iteration of $M$,
	for the extender algebra of $M|\lh(F)$ at $\delta$,
	making $A$ generic. Note that $\rho_1^{M||\lh(F)}\leq\nu_F$,
	so $\Ww'$ drops immediately to $M||\lh(F)$, at degree $0$.
	Set $\Ww=\Ww'\conc\left<F(M^{\Ww'}_\infty)\right>$.
	The ``moreover'' clause is a routine consequence, and the extension to $V[G]$ is similar.
	\end{proof}

We can now define the full short-normal strategy $\Psi_{\sn}$ for $\vV_1$.
In the end the method used to define the corresponding strategy
for $\vV_2$, in \S\ref{subsec:short-normal_trees_on_vV_2}
(especially Definition \ref{dfn:Psi_vV_2^sn}), will be somewhat different,
instead of being a direct generalization. One could probably
use the methods of  \S\ref{subsec:short-normal_trees_on_vV_2} to define
a strategy for $\vV_1$ below, which would have benefit of making the construction more uniform. But historically, the approach below was found earlier, and the verification that it works involves ideas that do not come up for the methods of \S\ref{subsec:short-normal_trees_on_vV_2}, which and seem of interest. So in order to record more information, we use the two different methods, as opposed to aiming for succinctness through uniformity.

\begin{dfn}\label{dfn:Psi_sn_1}
Let $\mathscr{T}$ be the class of all trees $\Tt$
 on $\vV_1$ via $\Psi_{\vV_1,\vV_1^-}$, of successor length, with
$b^\Tt$ non-dropping.
For $\Tt\in\mathscr{T}$, letting
$\vV=M^{\Tt}_\infty$,
we will define a good above-$\delta_0^\vV$ short-normal strategy $\Psi_{\Tt}$
for $\vV$. We will then define
\[ \Psi_{\sn}=\Psi_{\vV_1,\vV_1^-}\cup\bigcup_{\Tt\in\mathscr{T}}\Psi_{\Tt}.\]

So fix $\Tt$. Let $(\Ww,U,\Xx)$ be  $M$-standard and
such that  $\Tt$ is $(U,\PP)$-generic
for some $\PP\in U|\kappa_0^U$, which exists by Lemma \ref{lem:M-standard_with_T_generic}.
Let $j:\vV\to\vV_1^U$ be the correct iteration map.
By \ref{lem:vV_1_norm_condensation}, $\vV_1^U$ is $\om$-standard.
We define
\[ \Psi_\Tt=\text{ the minimal }j\text{-pullback of }\Gamma_\Ww\]
(see \cite[10.3, 10.4]{fullnorm_v3} and Remark \ref{rem:min_j-pullback} below; by $\om$-standardness,
the  minimal $j$-pullback
is well-defined, but we verify below that $\Psi_\Tt$
is independent of the choice of $\Ww$\footnote{Of course,
we could have simply chosen $\Ww$ in a canonical fashion,
and then $\Psi_\Tt$ would be trivially well-defined.
But the independence from $\Ww$ will be important later.}).

We also generalize this to set-generic extensions $V[G]$
of $V$. Let $\mathscr{T}^{V[G]}$
be the class of all trees $\Tt$ on $\vV_1$ via $\Psi^{V[G]}_{\vV_1,\vV_1^-}$ (determined
by
$\Sigma^{V[G]}$ just as $\Psi_{\vV_1,\vV_1^-}$ is from $\Sigma$), of successor length,
with $b^\Tt$ non-dropping. Fix $\Tt\in\mathscr{T}^{V[G]}$.
Let $(\Ww,U,\Xx)$ be as above with respect to $\Tt$ (but still with $\Ww\in V$, so $U\sub V$ and $\Xx\in V$ also).
Let $\vV=M^\Tt_\infty$.
and $j:\vV\to\vV_1^U$ be the correct iteration map. Define
\[\psi_\Tt^{V[G]}=\text{ the minimal }j\text{-pullback of }\Gamma^{V[G]}_\Ww.\qedhere\]
\end{dfn}

\begin{rem}\label{rem:min_j-pullback}
	Let us  summarize how the minimal
	$j$-pullback determining $\Psi_\Tt$ is defined. It is like a standard copying construction, except that
the method for copying extenders is different. Let $E_j$ be the $(\delta_0^{\vV},\delta_0^{\vV_1^U})$-extender
derived from $j$. For $\Ss$ via $\Psi_\Tt$,
we will have a tree $\Ss'$ via $\Gamma_\Ww$, with the same length, tree order,
drop and degree structure,
and for $\alpha<\lh(\Ss)$, a $d=\deg^\Ss_\alpha$-embedding
\[ \pi_\alpha:M^\Ss_\alpha\to M^{\Ss'}_\alpha, \]
and moreover,
\[ M^{\Ss'}_\alpha=\Ult_d(M^\Ss_\alpha,E_j) \]
and $\pi_\alpha$ is the associated ultrapower map,
and if $\alpha+1<\lh(\Tt)$ then
\[ M^{\Ss'}_\alpha||\lh(E^{\Ss'}_\alpha)=\Ult_0(M^\Ss_\alpha||\lh(E^\Ss_\alpha),E_j),\]
and letting $d^*=\deg^\Ss(\alpha+1)$, then
\[ M^{*\Ss'}_{\alpha+1}=\Ult_{d^*}(M^{*\Ss}_{\alpha+1},E_j). \]
 The remaining details  are essentially as in \cite{ACPFMP} and \cite{fullnorm_v3},
using  $\om$-standardness for Vsps (the latter ensures that, for example,
when $E^\Ss_\alpha\neq F(M^\Ss_\alpha)$ or $\deg^\Ss_\alpha>0$,
the ultrapower above determining $E^{\Ss'}_\alpha$ does indeed produce
a segment of $M^{\Ss'}_\alpha$).
\end{rem}
\begin{lem}\label{lem:Psi_sn_good}\footnote{Although this lemma
		logically precedes later parts of the paper, its necessity and proof
		were actually found later,
         in particular after the proof of
		Theorem \ref{tm:vV_2_is_generic_HOD_and_mantle}, which was found after the proof of  Lemma \ref{lem:modified_P-con_works}.}
	 We have:
\begin{enumerate}
	\item\label{item:Psi_T_well-def} $\Psi_\Tt$
is well-defined for each $\Tt\in\mathscr{T}^{V[G]}$.
\item\label{item:Psi_sn_good}  $\Psi^{V[G]}_{\sn}$ is
 good.
 \end{enumerate}
\end{lem}
\begin{proof}Since $G$ doesn't make a significant difference, we assume $G=\emptyset$.

Part \ref{item:Psi_sn_good}: Let $\Ss$ on $\vV$ be via $\Psi_\Tt$,
and $\Ss'$ the minimal $j$-copy, as in Remark \ref{rem:min_j-pullback}.
Using the copy maps $\pi_\alpha:M^\Ss_\alpha\to M^{\Ss'}_\alpha$,
we can argue just like in the last part of the proof of Lemma \ref{lem:Gamma_A_is_good} to see that the long extenders in $\es_+(M^\Ss_\alpha)$
are correct.

Part \ref{item:Psi_T_well-def}:
  Let $(\Ww,U,\Xx)$ and $(\Ww',U',\Xx')$ be as in the definition,
 and let $\Psi,\Psi'$ respectively be the induced strategies for $\vV=M^\Tt_\infty$.

Roughly, we would like to compare $U$ with $U'$, producing a common iterate $U''$
and corresponding $\Ww'',\Xx''$, and show that $\Psi,\Psi'$ both agree with $\Psi_{\Ww''}$,
and hence are equal. However, a standard comparison of $U,U'$ doesn't work for this,
as the resulting iteration map could have critical point $\kappa_0^M$ on one side,
which would cause problems. Instead, we form a modified kind of comparison, as follows.

Let $D\in\es^U$ be the $U$-total order $0$ measure on $\kappa_0^U$.
Let $\delta$ be the least Woodin of $\Ult(U,D)$ such that $\delta>\kappa_0^U$.
Let $D',\delta'$ be likewise for $U'$.

Recall from \cite{iter_for_stacks} that the \emph{meas-lim extender algebra} of a premouse $N$
is like the usual extender algebra, except that we only induce axioms with extenders $E\in\es^N$
such that $\nu_E$ is a limit of measurable cardinals of $N$.
We will form a \emph{simultaneous genericity iteration} $(\Yy,\Yy')$ of $\Ult(U,D)$ and $\Ult(U',D')$
for the meas-lim extender algebras at $\delta,\delta'$, above $\kappa_0^U+1$ and $\kappa_0^{U'}+1$ respectively, arranging that $M^\Yy_\infty$ and $M^{\Yy'}_\infty$ are generic over one another, and
$\delta(\Yy)=i^\Yy_{0\infty}(\delta)=i^{\Yy'}_{0\infty}(\delta')=\delta(\Yy')$. To help ensure the latter, we also (i)
arrange genericity of $(U|\delta,U'|\delta')$,
which will allow $\Yy,\Yy'$ to be recovered in the generic extensions,
and (ii) insert short linear iterations which ensure that every measurable of $M^\Yy_\infty$ below $i^\Yy_{0\infty}(\delta)$ is a cardinal of $M^{\Yy'}_\infty$, and vice versa;
this is like similar arguments in \cite{odle_v2} and \cite{local_mantles_of_Lx_v2}.
However, executing this process in the most obvious manner,
using the process for genericity iteration with Jensen indexing
on both sides (as described in \cite[Theorem 5.8]{iter_for_stacks})
seems to lead to the possibility of the trees $\Yy,\Yy'$ being non-normal.
Thus, instead of this, we produce a sequence $\left<\Yy_\alpha,\Yy'_\alpha\right>_{\alpha\leq\iota}$ of normal trees approximating  the eventual
desired trees $\Yy=\Yy_\iota$ and $\Yy'=\Yy_\iota$, with the sequence converging in a natural way.

Here are the details.
We initially iterate linearly with the least measurable
of $\Ult(U,D)$ which is $>\kappa_0^U$ (hence ${<\delta}$),
and likewise $>\kappa_0^{U'}$ for $\Ult(U',D')$, until they reach some common closure
point $>\max(\delta,\delta')$.
Suppose we have defined $\Xx=\Yy_\alpha,\Xx'=\Yy'_\alpha$, at some point after this initial phase. These trees will be padded $0$-maximal, of successor length, and if $E^{\Xx}_\beta\neq\emptyset\neq E^{\Xx'}_\beta$
then $\lh(E^\Xx_\beta)=\lh(E^{\Xx'}_\beta)$. We will determine some extenders $E_\alpha,E'_\alpha$, or stop the process. First let $G_\alpha\in\es_+(M^\Xx_\alpha)$
be the extender selected for the purposes of genericity iteration,  for making
$(\es^{M^{\Xx'}_\infty},U|\delta,U'|\delta')$ generic, and given current tree $\Xx$ (but not demanding that $\lh(G_\alpha)>\lh(E^\Xx_\beta)$ for all $\beta+1<\lh(\Xx)$; if $\lh(G_\alpha)<\lh(E^\Xx_\beta)$ for some $\beta$, that is okay); we follow the extender selection procedure for genericity iteration for Jensen indexing here (see \cite[Theorem 5.8]{iter_for_stacks}), and
if $b^\Xx$ drops then $M^\Xx_\infty$ will be active,
and in this case $F(M^\Xx_\infty)$ is automatically set as $G_\alpha$,
if no extender with lower index is.
(If there is no such extender, set $G_\alpha=\emptyset$.) Define
$G'_\alpha$ symmetrically, for making $(\es^{M^{\Xx}_\infty},U|\delta,U'|\delta')$ generic.
If $G_\alpha\neq\emptyset$ then let $\gamma=\lh(G_\alpha)$;
if otherwise and $b^\Xx$ does not drop then let $\gamma=i^\Xx_{0\infty}(\delta)$,
and otherwise let $\gamma=\OR(M^\Xx_\infty)$.  Define $\gamma'$ symmetrically.
If there is $F\in\es(M^\Xx_\infty|\gamma)$ which is $M^\Xx_\infty$-total
and $\kappa_0^U<\crit(F)$ and $\crit(F)$ is not a cardinal of $M^{\Xx'}_\infty$,
then let $F_\alpha=$ the least such $F$, and otherwise let $F_\alpha=G_\alpha$.
(If $\crit(F)$ is not a cardinal of $M^{\Xx'}_\infty$
for the trivial reason that $\OR(M^{\Xx'}_\infty)\leq\crit(F)$,
and hence $b^{\Xx'}$ drops,
then it will follow from Claim \ref{clm:gen_it_terminates} below
that $M^{\Xx'}_\infty$ is active,
$G'_\alpha\neq\emptyset$ and so $\lh(G'_\alpha)\leq\lh(F)$, and in this case,
$F$ will actually be irrelevant.)
Define $F'_\alpha$ symmetrically.
If $F_\alpha\neq\emptyset$ and either $F'_\alpha=\emptyset$ or $\lh(F_\alpha)\leq\lh(F'_\alpha)$ then set $E_\alpha=F_\alpha$,
and otherwise set $E'_\alpha=\emptyset$; define $E'_\alpha$ symmetrically.

If $E_\alpha=\emptyset=E'_\alpha$ then we stop the process (setting $\iota=\alpha$).
If $E_\alpha\neq\emptyset$ then let $\beta$ be least such that $E_\alpha\in\es_+(M^\Xx_\beta)$, and  set $\Yy_{\alpha+1}=\Yy_\alpha\rest(\beta+1)\conc E_\alpha$ (as a $0$-maximal tree). If $E'_\alpha\neq\emptyset$, define
 $\Yy'_{\alpha+1}$ symmetrically. If $E_\alpha\neq\emptyset=E'_\alpha$
 then set $\Yy'_{\alpha+1}=\Yy_\alpha\rest(\beta+1)$, where $\beta$ is least
 such that either $\Yy'_\alpha=\Yy'_\alpha\rest(\beta+1)$
 or $\lh(E^{\Yy'_\alpha}_\beta)>\lh(E_\alpha)$. And if $E_\alpha=\emptyset\neq E'_\alpha$,
 proceed symmetrically.

Finally, given $\Yy_\alpha,\Yy'_\alpha$ for all $\alpha<\eta$, where $\eta$ is a limit,
define $\Yy_\eta$ as the natural lim inf of the $\Yy_\alpha$ for $\alpha<\eta$,
extended with the relevant iteration strategy $\Psi$ as necessary.
That is, $\Yy_\eta$ is via $\Psi$, and $\Yy_\eta$ uses an extender $E$ iff $\Yy_\alpha$ uses $E$ for eventually
all $\alpha<\eta$, and if this yields a limit length tree, then we extend it using $\Psi$
to successor length.

This determines the mutual genericity iteration.
The first claim below is much as in \cite[Theorem 5.8]{iter_for_stacks} and related arguments in \cite{odle_v2}:

\begin{clmtwo}\label{clm:gen_it_terminates}\
	\begin{enumerate}
		\item\label{item:1_Tt^P_alpha_is_0-max} $\Yy_\alpha$ is $0$-maximal
		and if $b^{\Yy_\alpha}$ drops then $M^{\Yy_\alpha}_\infty$ is active; likewise for $\Yy'_\alpha$,
		\item\label{item:1_termination} the process terminates
		at some $\iota<\infty$, giving $\Yy=\Yy_\iota$ and $\Yy'=\Yy'_\iota$, and
		\item\label{item:1_drop_implies_active}  $b^\Yy,b^{\Yy'}$ do not drop (although
		there \emph{can} be $\alpha<\lh(\Yy,\Yy')$ such that $[0,\alpha]_\Yy$ or $[0,\alpha]_{\Yy'}$ drops, because of Jensen indexing).
		\end{enumerate}
	\end{clmtwo}
\begin{proof}
	Part \ref{item:1_Tt^P_alpha_is_0-max}: The $0$-maximality
	is directly by definition; the rest is as in \cite{iter_for_stacks}.

	Part \ref{item:1_termination}: Suppose not. Fix some large enough regular cardinal
	$\chi$. Assume for the sake of illustration
	that for all $\alpha\in\OR$ there are $\beta,\beta'>\alpha$ such that $E_\beta\neq\emptyset$ and $E'_{\beta'}\neq\emptyset$; the other case is similar. Let $\Yy=\Yy_\chi$ and $\Yy'=\Yy'_\chi$, noting that $\chi+1=\lh(\Yy)$
	and $\chi=\delta(\Yy)$ and $\chi$ is a limit cardinal of $M^\Yy_\infty$, and likewise for $\Yy'$. Let $J\preccurlyeq\her_{\chi^+}$
	with everything relevant in $J$ and $\kappa=J\inter\chi\in\chi$, let $H$ be the transitive collapse of $J$ and $\pi:H\to\her_{\chi^+}$ the uncollapse map.
	So $\pi(\kappa)=\chi$ and we have $\bar{\Yy},\bar{\Yy}'\in H$ with $\pi(\bar{\Yy},\bar{\Yy}')=(\Yy,\Yy')$. Note that $\bar{\Yy}=\Yy_\kappa$ and  $\lh(\bar{\Yy})=\kappa+1$, $\delta(\bar{\Yy})=\kappa$ is a limit cardinal of $M^{\bar{\Yy}}_\infty$, $\kappa\in b^{\bar{\Yy}}= b^\Yy\inter(\kappa+1)$,
and for all $\alpha\in[\kappa,\chi]$, we have $\bar{\Yy}=\Yy_\alpha\rest(\kappa+1)$.
Likewise for $\bar{\Yy}'$.

As usual, $i^\Yy_{\kappa\chi}\rest\pow(\kappa)\sub\pi$. Let $E$ be the first extender used forming $i^\Yy_{\kappa\chi}$, so $\crit(E)=\kappa$.
Let $\alpha\in[\kappa,\chi)$ with $E=E_\alpha$.
Since $\kappa$ is a limit cardinal of  $M^{\bar{\Yy}'}_\infty$, and
	hence of  $M^{\Yy'_\alpha}_\infty$, either (i) $E=G_\alpha$ was  chosen for genericity iteration purposes, and let $\beta=\alpha$, or (ii) $b^{\Yy_\alpha}$ drops and $E_\alpha=F(M^{\Yy_\alpha}_\infty)$.
	But if (ii) holds then as is usual for genericity iteration with $\lambda$-indexing
	(see \cite{iter_for_stacks}), there is $\beta\in[\kappa,\alpha)$ such that
	$E_\beta=G_\beta$ is used in $\Yy$ (but not along $b^\Yy$),
	and used in $\Yy_\alpha$, and $G_\beta\rest\nu(G_\beta)=E_\alpha\rest\nu(G_\beta)$.
	So in either case (i) or (ii), $G_\beta$ is used in $\Yy$ and $G_\beta\rest\nu(G_\beta)$
	is derived from $i^\Yy_{\kappa\chi}$. But note then that $M^{\Yy'_\beta}_\infty|\nu(G_\beta)=M^{\Yy'}_\infty|\nu(G_\beta)$,
	and therefore we obtain a contradiction like in the usual proof that genericity iteration terminates.

	Part \ref{item:1_drop_implies_active} is by construction and  part
	\ref{item:1_Tt^P_alpha_is_0-max}.
\end{proof}

Let $\delta^\Yy=i^\Yy_{0\infty}(\delta)$ and $\delta^{\Yy'}=i^{\Yy'}_{0\infty}(\delta')$.
Note that $\delta^\Yy$ is a strong cutpoint of $M^\Yy_\infty$,
and $M^\Yy_\infty$ is $\delta^\Yy$-sound; likewise for $\delta^{\Yy'}$ and $M^{\Yy'}_\infty$.

\begin{clmtwo}$\delta^\Yy=\delta^{\Yy'}$.\end{clmtwo}
\begin{proof}We may assume $\delta^\Yy<\delta^{\Yy'}$.
By minimality of $\delta'$, $\delta^\Yy$ is not Woodin in $M^{\Yy'}_\infty$.
Let $Q'\pins M^{\Yy'}_\infty$ be the Q-structure for $\delta^\Yy$.
Because of the inserted linear iterations at  measurables,
$\delta^\Yy$ is a cardinal of $M^{\Yy'}_\infty$.
Note then that by choice of $D',\delta'$ (and smallness), $\delta^{\Yy}$ is a strong cutpoint
of $Q'$. But $(M^{\Yy'}_\infty|\delta^\Yy,U|\delta,U'|\delta')$ is meas-lim extender algebra generic
over $M^{\Yy}_\infty$ at $\delta^\Yy$, and $(M^\Yy_\infty|\delta^\Yy,U|\delta,U'|\delta')$ is likewise over $Q'$, also at $\delta^\Yy$, because $\delta^\Yy$ is a cardinal of $M^{\Yy'}_\infty$. Therefore these two premice
can be lifted to  premice $(M^\Yy_\infty)^+$ and $(Q')^+$ over $(M^\Yy_\infty|\delta^\Yy,M^{\Yy'}_\infty|\delta^\Yy,U|\delta,U'|\delta')$.
Comparing $(M^\Yy_\infty)^+$ with $(Q')^+$ (so the comparison is above $\delta^\Yy$)
and considering smallness and $\delta^\Yy$-soundness, we have $(Q')^+\pins(M^\Yy_\infty)^+$.

Now working in $(M^\Yy_\infty)^+$, where we have $U|\delta$, $M^\Yy_\infty|\delta^\Yy$, $U'|\delta'$, $M^{\Yy'}_\infty|\delta^\Yy$
and $Q'$, we can recover $\bar{\Yy}\rest\delta^\Yy$ and $\bar{\Yy'}\rest(\delta^{\Yy'}+1)$, where $\bar{\Yy}$
is just like $\Yy$ but as a tree on $U|\delta$, and $\bar{\Yy'}$ likewise. For (i) we have $U|\delta$ and $U'|\delta'$, and we just proceed by comparing $U|\delta$ with $M^\Yy_\infty|\delta^Y$ and $U'|\delta'$ with $Q'$, (ii)  $Q'$ determines
the branch of $\bar{\Yy'}$ at stage $\delta^\Yy$, and (iii) the intermediate Q-structures
used to guide $\bar{\Yy},\bar{\Yy'}$ are segments of $M^\Yy_\infty|\delta^\Yy$
and $M^{\Yy'}_\infty|\delta^\Yy$  (for example if $\xi=\delta(\bar{\Yy'}\rest\xi)$
and the Q-structure for $\bar{\Yy'}\rest\xi$ is non-trivial,
then $\xi$ is a limit of measurable cardinals of $M^{\Yy'}_\infty$,
hence a limit cardinal of $M^\Yy_\infty$, so  the Q-structures
on both sides do not overlap $\xi$, which implies that the next extenders used have index
beyond the Q-structures (and likewise at all stages after $\xi$), so the Q-structures are retained).

So in $(M^\Yy_\infty)^+$, where $\delta^\Yy$ is a regular cardinal, we can execute
a slight variant of the termination-of-genericity-iteration proof used
above for Claim \ref{clm:gen_it_terminates} (with $\delta^\Yy$ in the role of $\chi$  there).
(We may not have the sequence of trees $\left<\Yy_\alpha,\Yy'_\alpha\right>_{\alpha<\delta^\Yy}$
in $(M^\Yy_\infty)^+$, but note that the trees we do have are enough.)
\end{proof}

 Let $\delta^*=\delta^\Yy=\delta^{\Yy'}$ and $Q=M^\Yy_\infty$ and $Q'=M^{\Yy'}_\infty$.
Note now that (by $\delta^*$-soundness and as $\delta^*$ is a strong cutpoint on both sides),
$Q,Q'$ are equivalent modulo a generic above $\delta^*$, i.e.~letting
$Q^+$ be $Q[Q'|\delta^*]$, considered as a premouse over $(Q|\delta^*,Q'|\delta^*)$, and $(Q')^+$ likewise
(we no longer need $U|\delta$ and $U'|\delta'$), then $Q^+=(Q')^+$. It follows that $\kappa_0^Q=\kappa_0^{Q'}$
and $\vV_1^Q=\vV_1^{Q'}$, and by uniqueness of iteration strategies,
the above-$\delta^*$ strategy for $Q$ translates to that of $Q'$.
Finally let $\Zz$ be the normal tree equivalent to the stack $(\Ww,D,\Yy)$ and $\Zz'$ that to $(\Ww',D',\Yy')$;
it follows that $\Gamma_{\Zz}=\Gamma_{\Zz'}$.

Let $j:\vV\to\vV_1^U$ and $j':\vV\to\vV_1^{U'}$ be the correct iteration maps.
Now $i^{D,\Yy}\rest\vV_1^U$ is the correct iteration map $k:\vV_1^U\to\vV_1^Q$
(because we wanted this, we couldn't just compare $U$ with $U'$ in the usual manner).
Likewise $k'=i^{D',\Yy'}\rest\vV_1^{U'}$.
So
\[ k\com j=k'\com j':\vV\to\vV_1^Q=\vV_1^{Q'} \]
is also the correct iteration map.
So the minimal $k\com j$-pullback of $\Gamma_{\Zz}$ equals
the minimal $k'\com j'$-pullback of $\Gamma_{\Zz'}$
(a strategy for $\vV$); denote this by $\Psi^*$.
Recall that $\Psi$
is the minimal $j$-pullback of $\Gamma_\Ww$,
and $\Psi$ the minimal $j'$-pullback of $\Gamma_{\Ww'}$.
It suffices to see that $\Psi=\Psi^*$, since then by symmetry, $\Psi'=\Psi^*$ also.

Let $\Sigma^Q$ be the above-$\kappa_0^Q$ strategy for $Q$ given by $\Sigma_Q$.
Let $\Sigma^U$ be likewise for $U$.
Then by $\kappa_0^U$-soundness and since $\kappa_0^U$ is a cutpoint of $U$,
$\Sigma^U$ is the minimal $i^{D,\Yy}$-pullback of $\Sigma^Q$.
But then $\Gamma_\Ww$ is the minimal $k$-pullback of $\Gamma_{\Zz}$;
for trees $\Ss$ on $\vV_1^U$ with $\lh(E^\Ss_0)>\gamma^{\vV_1^U}$,
this uses that  $k\sub i^{D,\Yy}$ and the fine structural translation of $\Ss$
to a tree on $U$; for trees with $\lh(E^\Ss_0)<\gamma^{\vV_1^U}$,
it uses minimal hull condensation for $\Sigma_M$,
and the fact that the action of $\Gamma_\Ww$ and $\Gamma_{\Zz}$
on such trees in induced by $\Sigma_M$. Therefore $\Psi=\Psi^*$, as desired.
\end{proof}

\subsubsection{Normal trees on $\vV_1$}

So we have a good short-normal strategy $\Psi_\sn$,
extending $\Psi_{\vV_1,\vV_1^-}$. This extends easily to a normal strategy
$\Sigma_{\vV_1}$.

\begin{dfn}\label{dfn:Sigma_vV_1}
	We define a $0$-maximal iteration strategy $\Sigma_{\vV_1}$ for $\vV_1$,
	determined by the following properties:
	\begin{enumerate}
		\item $\Psi_{\sn}\sub\Sigma_{\vV_1}$.
		\item
Let $\Tt$ be on $\vV_1$, $0$-maximal, of length $\eta+2$, with $\Tt\rest\eta+1$
short-normal and via $\Psi_\sn$, and $E^\Tt_\eta$ long. Then (by goodness) $M^{\Tt}_{\eta+1}$
is a non-dropping iterate of $\vV_1$ via a tree $\Uu$
 according to $\Psi_{\vV_1,\vV_1^-}$,
and $\lh(E^\Tt_\eta)=\delta_0^{M^\Tt_{\eta+1}}$.
Then $\Sigma_{\vV_1}$ acts on trees normally extending $\Tt\rest\eta+2$ by
following $\Psi_{\Uu}$, until another long extender is used.
\item Likewise, whenever $\Tt$ is on $\vV_1$, via $\Sigma_{\vV_1}$,
and $E^\Tt_\eta$ is long, then $M^\Tt_{\eta+1}$ is a non-dropping
$\Psi_{\vV_1,\vV_1^-}$-iterate, via a tree $\Uu$,
and $\Sigma_{\vV_1}$ extends $\Tt\rest(\eta+2)$ by following $\Psi_{\Uu}$,
until another long extender is used.
\item
If $\lambda$ is a limit and there are long extenders
used cofinally in $\lambda$, there is a
 unique $\Tt\rest\lambda$-cofinal branch, and $M^\Tt_\lambda$
 is again an iterate via $\Psi_{\vV_1,\vV_1^-}$ (by  normalization for transfinite stacks). In this case,
 $\delta(\Tt\rest\lambda)$ is the least measurable of $M^\Tt_\eta$,
 so we can have $\Tt\rest[\lambda,\alpha)$
 based on $M^\Tt_\lambda|\delta_0^{M^\Tt_\lambda}$, with $\lambda<\alpha$.
 This interval is formed using $\Psi_{\vV,\vV^-}$ (see Definition \ref{dfn:Psi_V,V^-}).
Letting $\alpha$ be least such that $(\lambda,\alpha]_\Tt$ does not drop
and $\delta_0^{M^\Tt_\alpha}<\lh(E^\Tt_\alpha)$, then (by  normalization) there is a short-normal tree $\Uu$ via $\Psi_{\vV_1,\vV_1^-}$ with last model $M^\Tt_\alpha$,
and $\Sigma_{\vV_1}$ extends $\Tt\rest(\alpha+1)$ by following
$\Psi_{\Uu}$, until the next long extender.\qedhere
\end{enumerate}
 \end{dfn}

The following lemma is now easy to see:
\begin{lem} $\Sigma_{\vV_1}$ is a good $0$-maximal strategy
	for $\vV_1$. Moreover, for every successor length tree $\Tt$ via $\Sigma_{\vV_1}$
 there is a unique short-normal tree via $\Psi_{\sn}$ with the same last model.
\end{lem}

 \begin{lem}\label{lem:Sigma_vV_1_vshc} $\Sigma_{\vV_1}$ has minimal inflation condensation (mic).
\end{lem}
\begin{proof} We just consider short-normal trees; it is easy to
	extend this to arbitrary normal trees, and we leave this extension to the reader.

	Let $\Tt,\Uu$ be short-normal trees on $\vV_1$, via $\Psi_{\sn}$,
	such that $\Uu$ has length $\lambda+1$ for some limit $\lambda$,
	$\Tt$ has successor length,
	and $\Uu\rest\lambda$ is
	a minimal inflation of $\Tt$;
	we must show that $\Uu$ is also
	a minimal inflation of $\Tt$.
	Let $\Tt=\Tt_0\conc\Tt_1$ and $\Uu=\Uu_0\conc\Uu_1$ with lower components $\Tt_0,\Uu_0$ and upper $\Tt_1,\Uu_1$.

	Now $\Psi_{\vV_1,\vV_1^-}$ has mic,
because it follows $\Sigma_{\M_\infty}$, which has mic,
since $\Sigma_M$ does, and by \cite[***Theorem 10.2]{fullnorm_v3}.
So we may assume $\Tt_1\neq\emptyset$.
Therefore, $\Tt_0$ has successor length $\alpha+1$,
$[0,\alpha]_{\Tt_0}$ does not drop, and $\Tt_1$ is based on $M^{\Tt_0}_{\alpha}$
and is above $\delta_0^{M^{\Tt_0}_{\alpha}}$ and uses only short extenders.
And $\Uu_0,\Uu_1,\beta$ are likewise, and note that $\beta\in I^{\Tt\inflatearrow_{\min}\Uu}_{\alpha}$.
Let
\[ \Pi_0=\Pi^{\Tt\inflatearrow_{\min}\Uu}_{\beta}:
\Tt_0\hookrightarrow_{\min}\Uu_0 \]
(the minimal tree embedding
at stage $\beta$ of the inflation),
and $j:M^{\Tt_0}_{\alpha}\to M^{\Uu_0}_{\beta}$ be the copy map determined by  $\Pi_0$.
Then in fact, $M^{\Uu_0}_{\beta}$ is a $\Sigma_{\vV,\vV^-}$-iterate of $\vV$
where $\vV=M^{\Tt_0}_{\alpha}$, and $j$ is the iteration map (see \cite[***Lemma 4.5]{fullnorm_v3}).

Now it suffices to see that $\Psi_{\sn}$
has  minimal hull condensation (mhc) with respect to extensions of $\Pi_0$ ``above $\delta_0$'';
that is,  whenever $\Tt_2,\Uu_2$  are trees on $M^{\Tt_0}_{\alpha}$
and $M^{\Uu_0}_{\beta}$,
above $\delta_0^{M^{\Tt_0}_{\alpha}}$
and $\delta_0^{M^{\Uu_0}_{\beta}}$ respectively, with $\Uu_0\conc\Uu_2$
via $\Psi_{\sn}$, and
\[ \Pi:\Tt_0\conc\Tt_2\hookrightarrow_{\min}\Uu_0\conc\Uu_2 \]
is a minimal tree embedding  with $\Pi_0\sub\Pi$,
then $\Tt_0\conc\Tt_2$ is also via $\Psi_{\sn}$.

Let $\Tt_2'=j``\Tt_2$, a (putative) tree on $M^{\Uu_0}_{\beta}$.
Then $\Tt_2'$ has wellfounded models, and in fact, there is a minimal tree embedding
\[ \Pi':\Uu_0\conc\Tt_2'\hookrightarrow_{\min}\Uu_0\conc\Uu_2 \]
determined in the obvious manner: for $\alpha<\lh(\Uu_0)$, we have $I_\alpha^{\Pi'}=[\alpha,\alpha]$, and  for $\alpha<\lh(\Tt_2)$,
we have $I^{\Pi'}_{\lh(\Uu_0)+\alpha}=\gamma^\Pi_{\lh(\Tt_0)+\alpha}$;
this determines $\Pi'$.

\begin{clmthree}
	$\Uu_0\conc\Tt_2'$ is via $\Psi_{\sn}$.
	\end{clmthree}
\begin{proof}
	If $\lh(E^{\Tt_2'}_0)<\gamma^{M^{\Uu_0}_{\beta}}$ then
	this is just because $\Psi_{\Uu_0}$ follows the strategy induced by $\Sigma_{M^{\Uu_0}_{\beta}\downarrow 0}$ in this case,
		which has mhc.

		So suppose $\gamma^{M^{\Uu_0}_{\beta}}<\lh(E^{\Tt_2'}_0)$.
	Let $(\Ww,U,\Yy)$ be $M$-standard for $\Uu_0$. Let $\ell:M^{\Uu_0}_{\beta}\to\vV_1^U$ be the iteration map.
	Since $\Uu_2$ follows $\Psi_{\Uu_0}$, the minimal $\ell$-copy $\widetilde{\Uu_2}$
	of $\Uu_2$ (a tree on $\vV_1^U$) follows $\Gamma_\Ww$. Let $\widetilde{\Tt_2'}$
	be the minimal $\ell$-copy of $\Tt_2'$ (see \cite[***10.3, 10.4]{fullnorm_v3}). Then $\widetilde{\Tt_2'}$ has wellfounded models,
	and in fact there is a minimal tree embedding
	\[ \widetilde{\Pi'}:\widetilde{\Tt_2'}\hookrightarrow_{\min}\widetilde{\Uu_2}, \]
	determined in the obvious manner. But since $\widetilde{\Uu_2}$
	is via $\Gamma_\Ww$, and this strategy has mhc, because $\Sigma_U$ does,
	therefore $\widetilde{\Tt_2'}$ is also via $\Gamma_\Ww$,
	and therefore $\Tt_2'$ is via $\Psi_{\Uu_0}$, as desired.
\end{proof}
\begin{clmthree} $\Tt_0\conc\Tt_2$ is via $\Psi_{\sn}$.
	\end{clmthree}
\begin{proof}
If $\lh(E^{\Tt_2}_0)<\gamma^{M^{\Tt_0}_{\alpha}}$, this is again easy,
using mhc for $\Sigma_{M^{\Tt_0}_{\alpha}\downarrow 0}$. So suppose otherwise.
	Let $(\Ww,U,\Yy)$ be simultaneously
	$M$-standard for $\Tt_0$ and for $\Uu_0$.
Let $\ell$ be as before, and
 $k:M^{\Tt_0}_{\alpha}\to\vV_1^U$ be the correct iteration map. So $\Psi_{\Tt_0}$ and $\Psi_{\Uu_0}$ are the minimal $k$-pullback and $\ell$-pullback of $\Gamma_\Ww$ respectively. But $\ell\com j=k$, since these  are correct iteration maps, and therefore $\Psi_{\Tt_0}$ is the minimal $j$-pullback of $\Psi_{\Uu_0}$, which, since $\Tt_2'$ is via $\Psi_{\Uu_0}$, proves the claim.
	\end{proof}

This completes the proof.
 \end{proof}

The preceding proof does not seem
to give that $\Sigma_{\vV_1}$
has mhc, because it relies heavily
on the fact that $\ell\com j=k$,
and if $j$ were instead just a copy
map arising from an arbitrary minimal tree embedding, then it need not be an iteration map
(and in fact $M^{\Uu_0}_{\beta}$
need not be an iterate of $M^{\Tt_0}_{\alpha}$).

\section{The second Varsovian model $\vV_2$}\label{sec:vV_2}
\subsection{The $\delta_1$-short tree strategy for $\vV_1$}\label{sec:dl-rel_sts_vV_1}

\begin{dfn}\label{dfn:delta_1-maximal}
	Let $\vV$ be a non-dropping $\Sigma_{\vV_1}$-iterate of $\vV_1$.

	Let $\Tt$ be a short-normal tree on $\vV$ via $\Sigma_{\vV}$,
	based on $\vV|\delta_1^\vV$, of limit length. Let $b=\Sigma_{\vV}(\Tt)$.
	Say that $\Tt$ is \emph{$\delta_1$-short} iff either $b$ drops or $\delta(\Tt)<i^\Tt_b(\delta_1^\vV)$,
	and \emph{$\delta_1$-maximal} otherwise.

	We define \emph{$\delta_1$-short} and \emph{$\delta_1$-maximal}
	analogously for trees on $M$.
\end{dfn}

It will be shown in \cite{*-trans_add} that $\vV_1$ knows its own strategy for $\delta_1$-short trees,
and, moreover, has a \emph{modified P-construction} which
also computes the correct branch model $M^\Tt_b$ for $\delta_1$-maximal $\Tt$,
given that $\Tt$ is appropriate for forming a P-construction.

In this paper we  explain the main new idea needed to prove this,
illustrated with a restricted class of trees $\Tt$ (\emph{P-illustrative} trees) which suffice, for example,
for genericity iterations at $\delta_1^{\vV_1}$. This restriction will ensure that for such $\Tt$
and limits $\lambda$ with $\Tt\rest\lambda$ being $\delta_1$-short
and $Q$ be the correct Q-structure for $\Tt\rest\lambda$,
the only overlaps of $\delta(\Tt\rest\lambda)$ in $Q$ are long extenders,
and this will mean that we have no need for $*$-translation. We will
compute $Q$ via a modified P-construction; a key issue is that the P-construction has a new feature,
due to  the long extenders on the $\Tt$-side,
and the extenders with critical point $\kappa_0^M$ on the $M$-side.
Similarly, for $\delta_1$-maximal P-illustrative trees $\Tt$,
$M^\Tt_b$ will also be computed by a modified P-construction.
We will need a new argument (\ref{lem:modified_P-con_works}) to see that the P-construction does indeed compute the correct model.
We will actually first consider analogous P-illustrative trees on $M$,
and then transfer these results to trees on $\vV_1$.  We will
then adapt the results to $\delta_1^{\vV}$-sound non-dropping $\Sigma_{\vV_1}$-iterates $\vV$ of $\vV_1$.

P-illustrative trees suffice to construct $\vV_2$, and prove a significant amount about it.
However, in order to prove that it fully knows how to iterate itself, and related facts,
we need to consider arbitrary trees, including the full $\delta_1$-short tree strategy. Such facts are moreover used our proof that $\vV_2$ is the mantle (because it uses a comparison argument, which seems needs iterability
with respect to arbitrary trees). In order to deal with arbitrary trees,
we need to deal with trees having overlapped Q-structures, and therefore need
$*$-translation, adapted to incorporate the modified P-construction.
This material is deferred to \cite{*-trans_add}; at certain
points we summarize results from there we need.

\subsubsection{P-illustrative trees on $M$}

\begin{definition}
	Given a (strategy) premouse $N$ and $\kappa\leq\eta\in\OR^N$, we say that $\eta$ is a \emph{$\kappa$-cutpoint of $N$}
	iff for all $E\in\es_+^N$,
	if $\crit(E)<\eta<\lh(E)$ then $\crit(E)=\kappa$,
	and a \emph{strong $\kappa$-cutpoint} iff for all $E\in\es_+^N$,
	if $\crit(E)\leq\eta<\lh(E)$ then $\crit(E)=\kappa$.
\end{definition}

\begin{dfn}\label{dfn:P-illustrative}
	Let $\Tt$ be an iteration tree on an $\Mswsw$-like premouse $N$ and $\Sigma$ a partial strategy for $N$. We say that $\Tt$ is \emph{P-illustrative}
	iff there are $E_1,U,\lambda,\eta,\alpha_0,\mu$
	such that:
	\begin{enumerate}[label=\arabic*.,ref=\arabic*]
		\item either
		\begin{enumerate}[label=(\roman*)]
			\item  $E_1=\emptyset$,  $U=N$ and $\lambda=\kappa_1^N$, or
			\item  $E_1\in\es^N$, $E_1$ is $N$-total, $\crit(E_1)=\kappa_1^N$, $U=\Ult(N,E_1)$
			and $\lambda=\lambda(E_1)=\kappa_1^U$,
		\end{enumerate}
		\item $\Tt\in N$, $\Tt$ is normal, of limit length, is above $\kappa_0^N$
		and based on $N|\delta_1^N$, and letting $\Tt'$ be the corresponding tree on $N|\delta_1$,
		we have $\Tt'\in U|\lambda$,
		\item $\kappa_0^N<\eta<\lambda$ and  $\eta$ is a strong $\kappa_0^N$-cutpoint of $U$, and $\eta$ is a $U$-cardinal, and if $E_1\neq\emptyset$ then $\kappa_1^N<\eta$\footnote{\label{ftn:eta_not_a_lambda}Note there is no $F\in\es^U$
			with $\crit(F)=\kappa_0^N$ and $\lambda(F)=\eta$, since otherwise
			$N$ is past superstrong.}
		\item $0\leq\alpha_0<\eta$ and $\alpha_0<\lh(\Tt)$ and $[0,\alpha_0]_\Tt$ does not drop
		and $\mu=\kappa_0^{M^\Tt_{\alpha_0}}<\eta$,
		\item $\Tt\rest[\alpha_0,\lh(\Tt))$ is above $(\mu^+)^{M^\Tt_{\alpha_0}}$,
		\item\label{item:Woodins_bounded} for each $\beta+1\in(\alpha_0,\lh(\Tt))$,
		every Woodin of $M^\Tt_\beta|\lh(E^\Tt_\beta)$
		is ${<\mu}$,
		\item $\eta<\delta=\delta(\Tt)$,
		$\eta$ is the largest cardinal of $U|\delta$,
		$\Tt$ is definable from parameters over $U|\delta$,
		and $U|\delta$ is extender algebra generic over $M(\Tt)$.\qedhere
	\end{enumerate}
\end{dfn}

Condition \ref{item:Woodins_bounded} prevents us from requiring $*$-translation.
\begin{dfn}\label{dfn:modified_P-con}
	Let $\Tt$ be P-illustrative and $N,U$ as  in \ref{dfn:P-illustrative}.
	We define the \emph{P-construction} $\mathscr{P}^U(M(\Tt))$
	of $U$ over $M(\Tt)$. This is the largest premouse $P$ such that:
	\begin{enumerate}
		\item $M(\Tt)\ins P\sats$``$\delta$ is Woodin'', where $\delta=\delta(\Tt)$,
		\item for $\alpha\in[\delta,\OR^P]$, $P||\alpha$ is active iff $U||\alpha$ is active,\footnote{\label{ftn:M|delta_passive}
			$U||\delta$ is passive, because
			$\eta$ is the largest cardinal of $U||\delta$ and
			by Footnote \ref{ftn:eta_not_a_lambda}.}
		\item Let $\alpha_0$ be as in \ref{dfn:P-illustrative}, let $\alpha>\delta$ be such that $U||\alpha$ is active and
		let $E=F^{P||\alpha}$ and $F=F^{U||\alpha}$. Then either:
		\begin{enumerate}
			\item\label{item:no_overlap_rest_OR}
			$\crit(F)>\kappa_0^M$ (so $\crit(F)>\delta$) and
			$E\rest\OR=F\rest\OR$,
			or
			\item\label{item:overlap_rest_OR} $\crit(F)=\kappa_0^M$ and
			$E\com i^\Tt_{0\alpha_0}\rest\kappa_0^{+M}=F\rest\kappa_0^{+M}$.\qedhere
		\end{enumerate}
	\end{enumerate}
\end{dfn}

\begin{rem}
	A key point in the above definition is that in condition \ref{item:overlap_rest_OR},
	with $j=i^\Tt_{0\alpha_0}$, we only require
	$E\com j$ and $F$  to agree over \emph{ordinals},
	not the full model $N|\kappa_0^{+N}$ (although $N|\kappa_0^{+N}$ is an initial segment of both sides).
	In fact (as we require that $P$ is a premouse),
	\[ E\com j\rest(N|\kappa_0^{+N})\neq F\rest(N|\kappa_0^{+N}), \]
	because $P|\kappa_0=N|\kappa_0=U|\kappa_0$, but $P|\lambda(E)\neq N|\lambda(F)$,
	and because $E$ must cohere $P|\alpha$, therefore
	$E(j(N|\kappa_0))\neq F(N|\kappa_0)$.
	However, recall the following fact, which we leave as an exercise for the reader.
\end{rem}

\begin{fact}\label{fact:F_from_F_rest_X} Let $N$ be a passive premouse, $\kappa$ be a cardinal of $N$,
	$X\sub\kappa^{+N}$ be unbounded in $\kappa^{+N}$ and $f:X\to\OR^N$.
	Then there is at most one active premouse $N'$ whose reduct is $N$ and
	$F^N\rest\OR=f$, and in fact, $N'$ is $\Sigma_1$-definable over $(N,f)$.
\end{fact}

(Note though that it is important that $N$ is given;
there can actually be another active $N_1'$ with $N\neq N_1$,
but $F^{N_1'}\rest\OR=f$.)

\begin{rem}\label{rem:lev-by-lev_P-con}
	It is not immediate that the P-construction $P$ is well-defined, as we have defined it directly as the largest premouse with the above properties,
	and one also needs one small observation to see that, if well-defined, then $P$ is unique (and appropriately locally definable).

	Consider instead defining $P|\alpha$ and $P||\alpha$ by recursion on $\alpha$.
	Note that $\jbar=j\rest\kappa_0^{+N}\in U|\eta\sub U|\delta$,
	so $\jbar$ is available as a parameter when making definitions over $U||\beta$ for some $\beta\geq\delta$,
	and $\jbar$ is also in the generic extension $(P||\beta)[N|\delta]$.

	Now given  $P||\beta$ for all $\beta<\alpha$, $\alpha$ a limit,
	(all $P||\beta$ sound),
	we get a premouse $P|\alpha$ satisfying ``$\delta$ is Woodin''.
	Suppose $U||\alpha$ is active with $F$. We need to see that we get
	a premouse $P||\alpha$, fine structurally equivalent with $U||\alpha$ (modulo the generic).
	We need to in particular see that there is a unique premouse $P||\alpha$ with the right properties.
	If $\crit(F^{U||\alpha})=\kappa_0$, existence is not immediately clear, and will be verified in Lemma \ref{lem:modified_P-con_works}.
	Uniqueness and the manner in which $F^{P||\alpha}$ is determined,
	requires a short argument.
	We have $P|\alpha$ and can compute
	$F^{P||\alpha}\rest\rg(\jbar)$, and $\rg(\jbar)$ is cofinal in $\mu^{+(P|\alpha)}$ (where $\mu=\kappa_0^{M^\Tt_{\alpha_0}}=\crit(F^{P||\alpha})$).
	By Fact \ref{fact:F_from_F_rest_X} this (very locally) determines $F^{P||\alpha}$.
	For the case that $\crit(F^{U||\alpha})>\delta$, one makes the usual
	P-construction observations,
	although the generic equivalence here involves the parameter $\jbar$.

	Overall we maintain level by level that
	\[ P|\beta\text{ is }\Delta_1^{U|\beta}(\{P|\delta,\jbar\}),\]
	\[ P||\beta\text{ is }\Delta_1^{U||\beta}(\{P|\delta,\jbar\}),\]
	uniformly in $\beta>\delta$ (and recalling $P|\delta =M(\Tt)$),
	and also that
	\[ U|\beta=^*_\delta(P|\beta)[U|\delta]\text{ and }U|\beta\text{ is }\Delta_1^{(P|\beta)[U|\delta]}(\{U|\delta,\jbar\}), \]
	\[ U||\beta=^*_\delta(P||\beta)[U|\delta]\text{ and }U||\beta\text{ is }\Delta_1^{(P||\beta)[U|\delta]}(\{U|\delta,\jbar\}), \]
	uniformly in $\beta$, where $(P|\beta)[U|\delta]$ is a generic extension,
	which for definability purposes has has $P|\beta$ available as a predicate (and similarly for $P||\beta$),
	and where $=^*$ means that $U|\beta$ is the premouse which extends $U|\delta$,
	followed by the small forcing extension of extenders $E$ in $\es^{P|\beta}$ or $\es_+^{P||\beta}$
	when $\crit(E)>\delta$,
	and the extender determined by $E\com\jbar$ otherwise,
	and also that the usual fine structural correspondence holds between the two sides
	(employing an extender algebra name for $\jbar$).
\end{rem}

\begin{lem}\label{lem:modified_P-con_works}
	Let $\Tt$ be P-illustrative on $M$, via $\Sigma_M$, and $b=\Sigma_M(\Tt)$,
	and let $U$ be as in Definition \ref{dfn:P-illustrative}. If $\Tt$ is short let $Q=Q(\Tt,b)$,
	and otherwise let $Q=M^\Tt_b$. Then
	$\mathscr{P}^U(M(\Tt))=Q$.
\end{lem}
\begin{proof}
	Let $\delta=\delta(\Tt)$. Working in $V$, we ``compare the phalanx $\Phi(\Tt,Q)$ with the phalanx $((M,\delta),U)$, modulo the generic at $\delta$''.
	That is, in the ``comparison'', we \emph{only use extenders indexed above $\delta$},
	with least disagreements determined by the restrictions of extenders to the ordinals,
	and after composing extenders $E$ overlapping $\delta$ on the $\Phi(\Tt,Q)$ side with $\jbar$ (notation as above).
	That is, we define normal padded trees $\Uu$ on $\Phi(\Tt,Q)$ and $\Vv$ on $((M,\delta),U)$, with both corresponding to trees on $M$ via $\Sigma_M$,\footnote{For $((M,\delta),U)$,
		the notation means that the exchange ordinal associated to $M$ is $\delta$,
		so in fact, since $\delta<\lh(E^\Vv_0)$, and $\delta$ is a strong $\kappa_0$-cutpoint of $U$,
		if $\crit(E^\Vv_\alpha)<\delta$ then $\crit(E^\Vv_\alpha)=\kappa_0$. If $U=M$ then we $\Vv$ is directly equivalent to a tree on $M$ via $\Sigma_M$. If $U=\Ult(M,E)$ where $E\in\es^M$ and $\crit(E)=\kappa_1$,
		then a simple instance of normalization produces the tree $\Vv'$ on $M$, via $\Sigma_M$,
		corresponding to $\Vv$: If $\Vv\rest\beta$ is above $\delta$ but based on $U|\lambda^{+U}=U|\kappa_1^{+U}$,
		then $\Vv'\rest\beta$ is the  tree on $M$,  via $\Sigma_M$, which uses the same extenders and has the same tree structure as does $\Vv\rest\beta$. Note that because $\eta,\delta$ are strong $\kappa_0$-cutpoints of $U$, and by the Mitchell-Steel ISC, $E$'s natural length $\nu(E)\leq\eta$, so $\rho_1^{M||\lh(E)}\leq\eta$, so for each $\alpha+1<\beta$ such that $\pred^\Vv(\alpha+1)=0$
		and $\alpha+1\notin\dropset^\Uu$, we have that $\Vv'$ drops in model and degree at $\alpha+1$ to
		$(M^{*\Vv'}_{\alpha+1},\deg^{\Vv'}_{\alpha+1})=(M||\lh(E),0)$; $\Vv\rest\beta$ and $\Vv'\rest\beta$ otherwise agree in drop and degree structure, and so their models and embeddings agree in a simple manner. If $\Vv\rest\beta$ is as above but $[0,\beta]_\Vv$ does not drop and $\kappa_1^{+M^\Vv_\beta}<\lh(E^\Uu_\beta)$, then $\Vv'$ uses first $F(M^{\Vv'}_\beta)$ as an extra extender, and note then that $M^\Vv_\beta=M^{\Vv'}_{\beta+1}$. Since $\kappa_1^{M^\Vv_\beta}$
		is a  $\kappa_0$-cutpoint of $M^{\Vv}_\beta$, if $\Vv\rest[\beta,\gamma)$ is above $\kappa_1^{M^\Vv_\beta}$ then $\Vv\rest[\beta,\gamma)$ is directly equivalent to $\Vv'\rest[\beta+1,\gamma')$ where $\gamma'=\gamma+1$ or $\gamma=\gamma$ in the obvious manner.
		Finally if $\xi$ is least such that $\crit(E^\Vv_\xi)=\kappa_0$ (either with $\xi<\beta$ or $\beta<\xi$ as above) then $\Vv,\Vv'$ are again directly equivalent thereafter.}
	and such that, given $(\Uu,\Vv)\rest(\alpha+1)$, letting $\gamma>\delta$ be least such that
	$E=F^{M^\Uu_\alpha||\gamma}\neq\emptyset$ or $F=F^{M^\Vv_\alpha||\gamma}\neq\emptyset$ and either:
	\begin{enumerate}[label=--]
		\item $\crit(E)=\mu$ and $E\com j\rest\kappa_0^{+M}\neq F\rest\kappa_0^{+M}$, or
		\item $\crit(E)>\mu$ (so $\crit(E)>\delta$) and $E\rest\OR\neq F\rest\OR$, or
		\item $E=\emptyset\neq F$,
	\end{enumerate}
	then $E^\Uu_\alpha=E$ and $E^\Vv_\alpha=F$ (and the comparison terminates if there is no such $\gamma>\delta$).
	We need to see that this comparison is trivial, i.e. no extenders are used.
	So suppose otherwise.

	\begin{clmfour}
		The comparison terminates in set length.
	\end{clmfour}
	\begin{proof}
		Note that if $E^\Uu_\beta$ overlaps $\delta$ then $M^\Uu_{\beta+1}$
		is proper class and
		\[ \kappa_0^{M^\Uu_{\beta+1}}=\lambda(E^\Uu_\beta) \]
		is a cutpoint of $M^\Uu_{\beta+1}$,
		so $\Uu\rest[\beta+1,\infty)$ is above $\lambda(E^\Uu_\beta)$.
		Therefore there is at most one such $\beta$. Likewise for $\Vv$.
		But the comparison after these overlaps becomes standard comparison modulo the small generic at $\delta$,
		so the usual argument then shows that the comparison terminates.
	\end{proof}

	So say we get $\lh(\Uu,\Vv)=\alpha+1$.

	\begin{clmfour}\label{clm:some_overlap_used} At least one of the trees $\Uu,\Vv$ uses an extender overlapping $\delta$.
	\end{clmfour}
	\begin{proof}
		Suppose otherwise.
		Suppose $\OR(M^\Uu_\alpha)=\OR(M^\Vv_\alpha)$; the other case is similar.
		Then the level-by-level translation process described in \ref{rem:lev-by-lev_P-con} works
		between $M^\Uu_\alpha$ and $M^\Vv_\alpha$, and we get fine structural correspondence above $\delta$.
		Suppose that $\OR(M^\Uu_\alpha)<\OR$.
		Let $A$ be either the core of $M^\Uu_\alpha$ if $\rho_\om(M^\Uu_\alpha)>\delta$, and otherwise the $\delta$-core of $M^\Uu_\alpha$.
		Let $B$ be likewise from $M^\Vv_\alpha$.
		Let $\pi,\sigma$ be the core maps respectively. Then by the fine structural correspondence
		and forcing calculations,
		\[ \rg(\pi)\inter\OR=\rg(\sigma)\inter\OR,
		\]
		so $\OR^A=\OR^B$. But also, the core maps preserve the fact that the level-by-level translation works,
		so $B=^*_\delta A[M(\Tt)]$ etc (with fine structural agreement up to the relevant level). But we had $A\ins M^\Uu_\gamma$ and $B\ins M^\Vv_\gamma$ for some $\gamma$,
		and either $E^\Uu_\gamma$ or $E^\Vv_\gamma$ came from $\es_+^A$ or $\es_+^B$ respectively,
		a contradiction. So $\OR(M^\Uu_\alpha)=\OR$, so there is no dropping on main branches,
		and $\Tt$ is maximal. But now we can just replace ``$\delta$-core'' with
		the hull of $\delta\cup\mathscr{I}$, where $\mathscr{I}$ is the class of indiscernibles of $M^\Uu_\alpha$, or equivalently, $M^\Vv_\alpha$,
		and run the analogous argument, using that $Q,U$ are $\delta$-sound (if $U\neq M$, this is
		because $\eta$ is a strong $\kappa_0$-cutpoint of $U$, and hence if $U=\Ult(M,E)$
		where $E\in\es^M$ with $\crit(E)=\kappa_1$, then $\nu(E)\leq\eta$, where $\nu(E)$ is the natural length of $E$).
	\end{proof}

	Now let $\beta$ be least such that $E^\Uu_\beta$ or $E^\Vv_\beta$ overlaps $\delta$.
	The following claim is the most central issue:

	\begin{clmfour} Not both of $E^\Uu_\beta,E^\Vv_\beta$ overlap $\delta$.
	\end{clmfour}
	\begin{proof}
		Suppose $F=E^\Vv_\beta$ overlaps $\delta$.
		Then $F\rest\kappa_0^{+M}$ is a restriction of the iteration map
		\[ \M_\infty\to \M_\infty^{\Ult(M,F)}. \]
		Similarly, supposing $E=E^\Uu_\beta$ overlaps $\delta$, $E\rest\OR$ is a restriction of the iteration map
		\[ \M_\infty^{M^\Tt_{\alpha_0}}\to \M_\infty^{\Ult(M^\Tt_{\alpha_0},E)}, \]
		so $E\com j\rest\kappa_0^{+M}$ is the restriction of the iteration map
		\[ \M_\infty\to \M_\infty^{\Ult(M^\Tt_{\alpha_0},E)}. \]
		So it suffices to see that
		\begin{equation}\label{eqn:M_inftys_match} \M_\infty^{\Ult(M^\Tt_{\alpha_0},E)}=\M_\infty^{\Ult(M,F)},\end{equation}
		as then $E\com j\rest\kappa_0^{+M}=F\rest\kappa_0^{+M}$,
		contradicting the disagreement of extenders.
		But (\ref{eqn:M_inftys_match}) holds because
		\[ (M^\Uu_\beta|\lh(E))[M|\delta]=^*_\delta M^\Vv_\beta|\lh(F), \]
		since $\delta<\lambda(E)=\lambda(F)$ and
		$E,F$ constitute the least disagreement.
	\end{proof}

	Now assume for notational simplicity that $E^\Uu_\beta$ overlaps $\delta$;
	the other case is very similar. Note that:
	\begin{enumerate}[label=--]
		\item $[0,\beta+1]_\Uu$ does not drop in model, and
		\item  $\Uu\rest[\beta+1,\lh(\Uu))$ is above $\lambda(E^\Uu_\beta)$.
	\end{enumerate}

	As in the proof of Claim \ref{clm:some_overlap_used} we get:
	\begin{clmfour}
		Neither $(\Tt,b)\conc\Uu$ nor $\Vv$ drops on its main branch.
	\end{clmfour}

	It easily follows that there is $\gamma>\beta$ such that $E^\Vv_\gamma$ overlaps $\delta$.
	We have $\beta+1\in b^\Uu$ and $\gamma+1\in b^\Vv$.
	Let $\mathscr{I}=\mathscr{I}^{M^\Uu_\alpha}=\mathscr{I}^{M^\Vv_\alpha}$.
	Let $\lambda=\lambda(E^\Uu_\beta)$, so
	\[ \kappa_0^{+M}<\delta<\lambda<\lambda(E^\Vv_\gamma), \]
	so $F_\lambda=E^\Vv_\gamma\rest\lambda$ is a non-whole segment of $E^\Vv_\gamma$, by the ISC and  smallness of $M$.
	Let
	\[ H^\Uu=\cHull^{M^\Uu_\alpha}(\lambda\cup\mathscr{I})\text{ and }H^\Vv=\cHull^{M^\Vv_\alpha}(\lambda\cup\mathscr{I}), \]
	and $\pi^\Uu,\pi^\Vv$ the uncollapses.
	Then $\kappa_0^{H^\Uu}=\lambda$ because
	$H^\Uu=M^\Uu_{\beta+1}$
	(and note $i^\Uu_{\beta+1,\alpha}=\pi^\Uu$). Similarly,
	$H^\Vv=\Ult(M,F_\lambda)$,
	but $\lambda<i^M_{F_\lambda}(\kappa_0)=\kappa_0^{H^\Vv}$
	since $F_\lambda$ is not whole.
	But since $\delta<\lambda$,
	we also have
	$\pi^\Uu\rest\OR=\pi^\Vv\rest\OR$, so
	$\pi^\Vv(\lambda)=\pi^\Uu(\lambda)=\kappa_0^{M^\Uu_\alpha}=\kappa_0^{M^\Vv_\alpha}$,
	so $\lambda=\kappa_0^{H^\Vv}$, contradiction.
\end{proof}

We now want to consider similar P-constructions internal to $\vV_1$,
and also  iterates  of $\vV_1$ and their generic extensions.

\begin{dfn}	\label{dfn:dsr}
	Let $\vV$ be a non-dropping $\Sigma_{\vV_1}$-iterate of $\vV_1$. Let $\Tt$ be an iteration tree on $\vV$. We say that $\Tt$ is \emph{dl-somewhat-relevant (dsr)} iff
	there are $\Tt_0,\vV',\Tt_1$ such that:
	\begin{enumerate}
		\item $\Tt$ is short-normal,
		\item $\Tt$ has lower and upper components $\Tt_0,\Tt_1$ respectively,
		$\Tt_0$ has successor length,  $b^{\Tt_0}$ does not drop,
		$\vV'=M^{\Tt_0}_{\infty}$,
		and $\Tt_1$ (on $\vV'$) is above $\gamma^{\vV'}$,
		\item\label{item:no_new_Woodins_below_index} for each $\beta+1<\lh(\Tt_1)$, $\delta_0^{\vV'}$ is the unique Woodin of $M^{\Tt_1}_\beta|\lh(E^{\Tt_1}_\beta)$.
	\end{enumerate}
	Note that every dl-somewhat-relevant tree on $\vV$ is based on $\vV|\delta_1^\vV$.
\end{dfn}

Easily:
\begin{lem}
	Let $\vV$ be a non-dropping $\Sigma_{\vV_1}$-iterate of $\vV_1$.
	Let $\Tt$ on $\vV$ be via $\Sigma_{\vV}$ and $\delta_1$-maximal.
	Then $\Tt$ is dsr.
\end{lem}

\begin{dfn}
	For a non-dropping $\Sigma_{\vV_1}$-iterate $\vV$ of $\vV_1$,
	write $\Sigma_{\vV,\sss}$ for the restriction of $\Sigma_{\vV}$
	to $\delta_1$-short trees, and $\Sigma^{\dsr}_{\vV,\sss}$ for
	the restriction of $\Sigma_{\vV,\sss}$ to dsr trees.
\end{dfn}
\begin{dfn}\label{dfn:vV_1_P-suitable}
	Let $\vV$ be a non-dropping $\Sigma_{\vV_1}$-iterate of $\vV_1$. Let $\Tt$ be an iteration tree on $\vV$. Let $\PP\in\vV$ and $g$ be $(\vV,\PP)$-generic.
	Say $\Tt$ is \emph{P-suitable for $\vV[g]$} iff there are $\Tt_0,\vV',\Tt_1,E_1,U,E_0,\eta,\delta,\lambda$ such that:
	\begin{enumerate}[label=(\alph*)]
		\item $g$ is $(\vV,\PP)$-generic and $\Tt\in\vV[g]$,
		\item $\Tt$ is short-normal with
		lower and upper components $\Tt_0,\Tt_1$ respectively, $\Vv'=M^{\Tt_0}_\infty$
		exists and $b^{\Tt_0}$ is non-dropping, and $\Tt_1$ is based on $\vV'|\delta_1^{\vV'}$,
		\item $\Tt$ is via $\Sigma_{\vV}$,
		\item either
		\begin{enumerate}[label=(\roman*)]
			\item $E_1=\emptyset$ and $U=\vV$, or
			\item $E_1\in\es^{\vV}$ is short and $\vV$-total, $\crit(E_1)=\kappa_1^\vV$
			and $U=\Ult(\vV,E_1)$,
		\end{enumerate}
		\item $\PP\in U|\kappa_1^U$ and	$\Tt'\in (U|\kappa_1^U)[g]$ where $\Tt'$ on $\vV|\delta_1^\vV$ is equivalent to $\Tt$,
		\item either
		\begin{enumerate}[label=(\roman*)]
			\item $E_0=\emptyset$ and $\Tt_0$ is trivial (so $\vV'=\vV$), or
			\item $E_0\in\es^{\vV}$ is long, $\gamma^{\vV}<\lh(E_0)$ and
			the lower component $\Tt_0$ of $\Tt$ is just the (successor length) short-normal tree corresponding to $E_0$
			(so $\vV'=\Ult(\vV_,E_0)$),
		\end{enumerate}
		\item $\delta_0^{\vV'}<\eta<\kappa_1^U$ and $\eta$ is a strong $\delta_0^{\vV}$-cutpoint of $U$, $\eta$ is a $U$-cardinal, $\PP\in U|\eta$, and if $E_1\neq\emptyset$ then $\kappa_1^{\vV}<\eta$,
		\item $\Tt$ has limit length, $\eta<\delta=\delta(\Tt)<\kappa_1^U$, $\eta$ is the largest cardinal of $U|\delta$,
		$\Tt$ is definable from parameters over $(U|\delta)[g]$, and $(U|\delta,g)$
		is $\BB_{\delta\delta_0^{\vV'}}^{M(\Tt)}$-generic over $M(\Tt)$.
	\end{enumerate}
	Say $\Tt$ is \emph{dl-relevant (for $\vV$)} iff $\Tt\in\vV$ is $\delta_1$-maximal (hence dsr) P-suitable,
	as witnessed by $E_1=\emptyset$.
\end{dfn}

\begin{dfn}\label{dfn:vV_1_P-con}
	Let $\vV$ be a non-dropping $\Sigma_{\vV_1}$-iterate of $\vV_1$.
	Let $g$ be set-generic over $\vV$.
	Let $\Tt\in\vV[g]$ be P-suitable for $\vV[g]$, as witnessed by $U,\Tt_0$. Then $\mathscr{P}^{U,g}(M(\Tt))$
	denotes the P-construction $P$ of $U[g]$ over $M(\Tt)$, using $\es^U$, computed analogously to
	that in Definition \ref{dfn:modified_P-con}; so when $F=F^{U||\alpha}\neq\emptyset$
	and $F$ is long, then $E=F^{P||\alpha}$ is determined by demanding
	$E\com\jbar\sub F$, where $\jbar=i^{\Tt_0}_{0\infty}\rest\delta_0^{\vV}$.
\end{dfn}

It is now straightforward to deduce a version of Lemma \ref{lem:modified_P-con_works}
for dsr P-suitable trees $\Tt\in\vV_1$ on $\vV_1$, by translating them to trees on $M$
and applying \ref{lem:modified_P-con_works}. Combined with the minimal inflation method
used for $M$, this allows us to compute the $\delta_1$-short tree strategy for $\vV_1$ inside $\vV_1$, and also the models for forming the second direct limit system.
However, before we proceed to this, we want to also consider the analogous
issues for iterates $\vV$ of $\vV_1$ (and generic extensions $\vV[g]$ thereof).
The argument given above does not immediately adapt to such iterates $\vV$  in general,
because (i) $\vV$ need not be as sound as $\vV_1$,
and (ii) $\vV$  need not correspond appropriately to an iterate of $M$.
To deal with these possibilities, we will adjust somewhat the conclusion and argument for Lemma \ref{lem:modified_P-con_works}, in Lemma \ref{lem:general_P-correctness} below.
Also
note that
it is not immediate that P-suitability and dl-relevance are first-order over $\vV$ (or $\vV_1$),
because of the demand that $\Tt$ be via $\Sigma_{\vV}$.
We will address this issue also, in a manner similar to that for $M$.

\subsubsection{DSR trees on  iterates of $\vV_1$}\label{subsubsec:DSR_trees}

To deal with issue (ii) mentioned above, it turns out we can replace the use of $M$ (or some iterate thereof) with an $\Mswsw$-like \emph{generic extension} $N$ of $\vV$
(together with such an extension of a related iterate $\vV_1$; see below).
It inherits  iterability (above $\kappa_0^N$) from the corresponding iterate of $\vV_1$:

\begin{lem}\label{lem:generic_premouse_iterable}
	Let $\vV$ be a non-dropping $\Sigma_{\vV_1}$-iterate of $\vV_1$.
	Let $g$ be $\mathbb{L}^\vV$-generic and $N$ an $\Mswsw$-like premouse
	such that $N=^*\vV[g]$ and $\vV=\vV_1^N$.  Let $\mathscr{I}=\mathscr{I}^{\vV}$. Then:
	\begin{enumerate}
		\item \label{item:N_iterable} $N$ is $(0,\OR)$-iterable with respect
		to trees $\Tt$ with $\lh(E^\Tt_0)>\kappa_0^N$.
		\item\label{item:N_kappa_0-sound} $\mathscr{I}$ is the class $\mathscr{I}^N$ of Silver indiscernibles for $N$
		(with respect to the generator set $\kappa_1^{\vV}=\kappa_1^N$).
		If $\vV$ is $\eta$-sound where
		$\delta_0^{\vV}\leq\eta$ then $N$ is $\eta$-sound.
		If $\vV$ is $\delta_0^{\vV}$-sound then
		$N$ is $\kappa_0^N$-sound.
	\end{enumerate}
\end{lem}
\begin{proof}
	Part \ref{item:N_iterable}: Let $P\pins N$ with $\rho_\om^P=\kappa_0^N$. Then $P$ is $(0,\OR)$-iterable
	with respect to trees $\Tt$ with $\lh(E^\Tt_0)>\kappa_0^N$, because $\kappa_0^N$ is a cutpoint of $P$, and
	letting $E\in\es^{\vV}$ be long, $E$ extends to $E^+\in\es^N$,
	and $E^+(P)$ is above-$\lambda(E)$-$(0,\OR)$-iterable, 	since iterating $E^+(P)$ above $\lambda(E)$
	is equivalent to iterating $\vV||\OR(E^+(P))$ above $\lambda(E)$.

	So we may assume that $\kappa_0^{+N}=\delta_0^{\vV}<\lh(E^\Tt_0)$,
	and so $\gamma^{\vV}<\lh(E^\Tt_0)$. Consider  translatable trees $\Tt$ on $N$.
	We get an iteration strategy for such trees induced by $\Sigma_{\vV}$, and the resulting
	iterates of $N,\vV$ are related according to Lemma \ref{lem:vV_1-translatable}.
	Now suppose $\Tt\rest(\alpha+1)$ is translatable, and let $\Uu$ on $\vV$ be its translation
	to $\vV$,  but $\Tt\rest(\alpha+2)$ is not translatable. Then $\alpha$ is a limit ordinal
	and a limit of stages when $\Uu$ uses a long extender, and $\lh(E^\Tt_\alpha)<\delta_0^{M^\Uu_\alpha}$. But then $\Tt\rest[\alpha,\infty)$
	is just a tree on some $P\pins M^\Tt_\alpha$ where $\rho_\om^P=\kappa_0^{M^\Tt_\alpha}$,
	and $\kappa_0^{M^\Tt_\alpha}$ is a cutpoint of $P$, and this $P$ is also iterable above $\kappa_0^{M^\Tt_\alpha}$, like in the previous paragraph.

	Part \ref{item:N_kappa_0-sound}: By genericity, $\mathscr{I}$ form indiscernibles
	for $N$. Note that $\vV$ is $\kappa_1^{\vV}$-sound.
	Let $\eta\in[\delta_0^{\vV},\kappa_1^\vV]$ be least such that $\vV$ is $\eta$-sound.
	Note that $N=\Hull^N_1(\mathscr{I}\cup\eta)$. So $\mathscr{I}^N=\mathscr{I}$.   Finally suppose $\eta=\delta_0^{\vV}$, so
	$N=\Hull_1^N(\kappa_0^{+N}\cup\mathscr{I})$. But also
	$\Hull_1^{\vV}(\mathscr{I})$ is cofinal in $\delta_0^{\vV}$ and $\vV$ is definable over $N$,
	so $\Hull_1^N(\kappa_0^N\cup\mathscr{I})$ is cofinal in $\kappa_0^{+N}$ and transitive below $\kappa_0^{+N}$,
	hence contains all of $\kappa_0^{+N}$, and hence all of $N$.
\end{proof}

\begin{dfn}
	Under the hypotheses of Lemma \ref{lem:generic_premouse_iterable},
	let $\Phi_{N\supseteq \vV}$ be the above-$\kappa_0^N$-strategy for $N$ induced by $\Sigma_{\vV}$ as in the proof of the lemma. (Note that the components of this strategy which do not translate
	to a tree on $\vV$, i.e. on one of the projecting structures $P$ in the proof,
	are uniquely determined by $P$.)
\end{dfn}
\begin{lem}\label{lem:general_P-correctness}
	Let $\vV$ be a  $\Sigma_{\vV_1}$-iterate of $\vV_1$.
	Let $\Ss$ be the short-normal tree on $\vV_1$,
	via $\Sigma_{\vV_1}$, with last model $\vV$.  Let $\bar{\vV}=\cHull_1^{\vV}(\mathscr{I}^{\vV}\cup\delta_1^{\vV})$
	and $\bar{\bar{\vV}}=\cHull_1^{\vV}(\mathscr{I}^{\vV}\cup\delta_0^{\vV})$.
	Then
	\begin{enumerate}
		\item\label{item:vVbar_is_intermediate_iterate} $\bar{\vV}=M^{\Ss}_\alpha$ for some $\alpha\in b^\Ss$,
		and $\Ss\rest[\alpha,\infty)$ is equivalent to a tree  on $\bar{\vV}$
		via $\Sigma_{\bar{\vV}}$, $\bar{\vV}|\delta_1^{\bar{\vV}}=\vV|\delta_1^{\vV}$,
		and $\bar{\vV}$ is $\delta_1^{\bar{\vV}}$-sound. Likewise for $\bar{\bar{\vV}}$ and $\delta_0^{\vV}$.
	\end{enumerate}
	Let $\PP\in\vV$ and $g$ be $(\vV,\PP)$-generic.
	Let $\Tt\in\vV[g]$, on $\vV$, be dsr-P-suitable and $b=\Sigma_{\vV}(\Tt)$, and $U,\eta,\delta$ as in Definition \ref{dfn:vV_1_P-suitable}.
	Let $\bar{\Tt}$ be the tree on $\bar{\vV}$ equivalent to $\Tt$.
	If $\Tt$ is $\delta_1$-short let $\bar{Q}=Q=Q(\Tt,b)$,
	and otherwise let $Q=M^\Tt_b$ and
	\[ \bar{Q}=M^{\bar{\Tt}}_b=\cHull_1^Q(\mathscr{I}^Q\cup\delta).\] Let $P=\mathscr{P}^{U,g}(M(\Tt))$. If $P$ is set-sized let $\bar{P}=P$ and otherwise let $\bar{P}=\Hull_1^P(\mathscr{I}^U\cup\delta)$ and $\pi:\bar{P}\to P$ be the uncollapse. Then:
	\begin{enumerate}[resume*]
		\item\label{item:Pbar=Qbar} 	$\bar{P}=\bar{Q}$
		\item\label{item:delta_is_U-regular} if $\Tt$ is $\delta_1$-maximal
		then $\delta=(\eta^+)^U$  (and $P,\bar{P}$ are proper class),
		\item\label{item:P_is_Sigma-iterate_of_Pbar} if $\bar{P}\neq P$ then $P$ is an
		above-$\delta$, $\Sigma_{\bar{Q}}$-iterate of $\bar{Q}$ (and hence
		a $\Sigma_{\vV_1}$-iterate of $\vV_1$) and $\pi$ is the iteration map,
		so $\pi``\mathscr{I}^{\bar{Q}}=\mathscr{I}^P=\mathscr{I}^{\vV}$.
	\end{enumerate}
\end{lem}
\begin{proof}[Proof (sketch)]
	Part \ref{item:vVbar_is_intermediate_iterate}: This is straightforward
	and left to the reader.

	Part \ref{item:Pbar=Qbar}: We will prove this by considering a comparison
	of two phalanxes. It will take a little work to define the phalanxes
	and describe their relevant properties.

	Let $\Ss=\Ss_0\conc\Ss_1\conc\Ss_2$ where $\Ss_0$ is the lower component
	of $\Ss$, $\Ss_1\conc\Ss_2$ the upper component, with $\Ss_1$ based on $M^{\Ss_0}_\infty|\delta_1^{M^{\Ss_0}_\infty}$ and $\Ss_2$
	above $\delta_1^{M^{\Ss_1}_\infty}$. Then $\bar{\bar{\Vv}}=M^{\Ss_0}_\infty$
	is $\delta_0^{\bar{\bar{\vV}}}$-sound and $\Ss_1$ is above $\gamma^{\bar{\bar{\vV}}}$, and is via $\Sigma_{\bar{\bar{\vV}}}$, and $\bar{\vV}=M^{\Ss_1}_\infty$ is
	$\delta_1^{\bar{\vV}}$-sound and  $\Ss_2$ is via $\Sigma_{\bar{\vV}}$.

	Let $\vV^+$ be an $\Mswsw$-like generic extension of $\vV$ with $\vV=\vV_1^{\vV^+}$,
	and such that $\vV^+,g$ are mutually $\vV$-generic. Let $g_{\vV^+}\sub\mathbb{L}^{\vV}$ be the generic filter. Since $\mathbb{L}^{\vV}=\mathbb{L}^{\bar{\bar{\vV}}}$ is below $\crit(i^{\Ss_1\conc\Ss_2}_{0\infty})$, we have $\mathbb{L}^{\bar{\bar{\vV}}}=\mathbb{L}^{\vV}$ and $g_{\vV^+}$ is also $(\bar{\bar{\vV}},\mathbb{L}^{\bar{\bar{\vV}}})$-generic,
	and extends uniquely to an $\Mswsw$-like generic extension
	$\bar{\bar{\vV}}^+$  of $\bar{\bar{\vV}}$  with
	$\bar{\bar{\vV}}=\vV_1^{\bar{\bar{\vV}}^+}$; note $\bar{\bar{\vV}}^+|\kappa_0^{+\bar{\bar{\vV}}^+}=\vV^+|\kappa_0^{+\vV^+}$. Lemma \ref{lem:generic_premouse_iterable}
	applies to $(\bar{\bar{\vV}}^+,\bar{\bar{\vV}})$, so  $\Ss_1\conc\Ss_2$ translates to a tree $\Ss_1^+\conc\Ss_2^+$ on $\bar{\bar{\vV}}^+$,
	via $\Phi_{\bar{\bar{\vV}}^+\supseteq\bar{\bar{\vV}}}$,
	and note that $\vV^+=M^{\Ss_1^+\conc\Ss_2^+}_\infty$.

	Let $E_0,E_1\in\es^{\vV}$ be as in Definition \ref{dfn:vV_1_P-suitable}.
	Let $E_i^+\in\es^{\vV^+}$ with $\lh(E_i^+)=\lh(E_i)$ (or $E_i^+=\emptyset=E_i$).
	Let $\bar{\bar{U}}_0^+=\Ult(\bar{\bar{\vV}}^+,E_0^+)$.
	Since $\bar{\bar{\vV}}^+$ is $\kappa_0^{\bar{\bar{\vV}}^+}$-sound,   $\bar{\bar{U}}_0^+$ is  $\kappa_0^{\bar{\bar{N}}'}$-sound.
	(Note $\bar{\bar{U}}_0^+$ need not be $\sub \vV^+$.)

	Let $U^+=\Ult(\vV^+,E_1^+)$, if $E_1\neq\emptyset$,
	and $U^+=\vV^+$ otherwise. Since $\eta$ is a $\delta_0^U$-strong cutpoint of $U$
	and $U|\delta$ has largest cardinal $\eta$, so is $\delta$,
	and  $\delta$ is also a $\kappa_0^{U^+}$-strong
	cutpoint of $U^+$. Let $\bar{U}^+=\Hull_1^{U^+}(\mathscr{I}^{U^+}\cup\delta)$.
	Letting $\pi:\bar{U}^+\to U^+$ be the uncollapse, note that $\delta<\crit(\pi)$
	(as $\delta\leq(\eta^+)^{U^+}$ and $\delta$ is a strong $\kappa_0^{U^+}$-cutpoint), and $\delta$ is a $\kappa_0^{\bar{U}^+}=\kappa_0^{U^+}$-strong cutpoint of $\bar{U}^+$,
	and $\bar{U}^+$ is $\delta$-sound.
	Note that if $P$ is proper class then $\delta=\delta_1^P=\delta_1^{\bar{P}}$, $P$ is a ground of $U^+[g]$ via the extender algebra at $\delta$ (to reach $U[g]$)
	followed by some smaller forcing (to reach $U^+[g]$), $\delta=\eta^{+U}$,  $\mathscr{I}^P=\mathscr{I}^U=\mathscr{I}^{U^+}$, \[ \OR\inter\rg(\pi)=\OR\inter\Hull_1^R(\mathscr{I}^R\cup\delta) \text{	is independent of }R\in\{P,U,U^+\},\]
	$\bar{P}$ is $\delta$-sound, and letting $\bar{U}=\cHull_1^{U}(\mathscr{I}^U\cup\delta)$,
	then $\bar{P}\sub\bar{U}[g]\sub\bar{U}^+[g]$ are related as are $P\sub U[g]\sub U^+[g]$,
	so $\bar{P}=\mathscr{P}^{\bar{U}}(M(\Tt))$, etc.

	Let $Q^+=M^{\Tt^+}_b||\OR^Q$ where $\Tt^+\conc b$ is the translation of $\Tt\conc b$ to a tree on $\vV^+$. Let $\bar{Q}^+=Q^+$ if $Q$ has set size,
	and $\bar{Q}^+=\Hull_1^{Q^+}(\mathscr{I}^Q\cup\delta)$ otherwise.
	Write $\Tt=\Tt_0\conc\Tt_1$ for the lower and upper components of $\Tt$,
	and $\Tt^+=\Tt_0^+\conc\Tt_1^+$ correspondingly.
	Note that we can rearrange $\Tt^+$ as a tree $\bar{\bar{\Tt}}^+=\bar{\bar{\Tt}}^+_0\conc\bar{\bar{\Tt}}^+_1$ on $\bar{\bar{\vV}}^+$
	with $\bar{\bar{\Tt}}_0^+$ equivalent to $\Tt_0^+$ (so $M^{\bar{\bar{\Tt}}^+_0}_\infty=\bar{\bar{U}}_0^+$), and $\bar{\bar{\Tt}}_1^+$ given
	by normalizing a stack equivalent to $(j``\Ss_1,\Tt_1^+)$, where $j:\bar{\bar{\vV}}^+\to \bar{\bar{U}}_0^+$ is the iteration map and $j``\Ss_1$ is the minimal $j$-copy of $\Ss_1$. If $Q$ is set size
	then $\bar{Q}^+=Q(\bar{\bar{\Tt}}^+,b)$ and otherwise $\bar{Q}^+=M^{\bar{\bar{\Tt}}^+}_b$.
	Note that $\bar{Q}^+$ is $\delta$-sound.

	Now compare the phalanx $((\bar{\bar{U}}_0^+,\delta),\bar{Q}^+)$ versus the phalanx $((\bar{\bar{\vV}}^+,\delta),\bar{U}^+)$,
	``above $\delta$, modulo the generic at $\delta$ and translation for overlapping extenders'', just like in the proof of Lemma \ref{lem:modified_P-con_works}, using $\Phi_{\bar{\bar{\vV}}^+\supseteq\bar{\bar{\vV}}}$ to iterate the phalanxes (a little bit of normalization shows this works). Because   $\bar{\bar{U}}_0^+$ is $\kappa_0^{\bar{\bar{U}}_0^+}$-sound and
	$\bar{\bar{\vV}}^+$ is $\kappa_0^{\bar{\bar{\vV}}^+}$-sound and $\bar{Q}^+,\bar{U}^+$ are $\delta$-sound,
	essentially the same proof as before shows that the comparison is trivial, which  gives $\bar{P}=\bar{Q}$.

	We leave part \ref{item:delta_is_U-regular} to the reader.

	Part \ref{item:P_is_Sigma-iterate_of_Pbar}, sketch:\footnote{We don't
		really need this part of the lemma, but it is convenient to have it.} Suppose $\bar{P}\neq P$;
	then $P$ is proper class and $\bar{U}\neq U$. Letting $\Ss'=\Ss\conc E_1$ if $E_1\neq\emptyset$,
	and $\Ss'=\Ss$ otherwise,
	we have $U=M^{\Ss'}_\infty$. Now it need not be that $\bar{U}$ is $M^{\Ss'}_\alpha$ for some $\alpha$, but this is almost the case.  In fact,
	because $\delta$ is a $\delta_0^U$-strong cutpoint of $U$, we get the following: Let $\alpha$ be least such that $\lh(E^{\Ss'}_\alpha)>\delta$. Then there is a unique tree $\Ss^*$  extending
	$\Ss'\rest(\alpha+1)$ and such that
	$E^{\Ss^*}_{\alpha+i}\neq\emptyset$ iff $[0,\alpha+i]_{\Ss^*}$ drops,
	and if non-empty, $E^{\Ss^*}_{\alpha+i}=F(M^{\Ss^*}_{\alpha+i})$,
	and $M^{\Ss^*}_\infty=\bar{U}$. Moreover,
	$U$ is a $\Sigma_{\bar{U}}$-iterate of $\bar{U}$, via a tree $\widetilde{\Ss}$ which is a straightforward
	translation of $\Ss\rest[\alpha,\infty)$ via a little normalization (in  \cite{rule_conversion_v2} there are similar kinds of calculations, though here it is easier). But $\widetilde{\Ss}$ is above $\delta+1$.
	Therefore it translates to a tree $\Ss^{\dagger}$ on $\bar{P}$ whose last model
	is $P$. We have $\bar{Q}=\bar{P}$ is an iterate of $\vV_1$
	and $\delta=\delta_1^{\bar{Q}}$. But by the smallness of $M$,
	and since $\Ss^{\dagger}$ is above $\delta_1^{\bar{Q}}$ and does not drop on its main branch, it must be via
	$\Sigma_{\bar{Q}}$ (that is, $P$ has nothing remotely resembling a Woodin cardinal
	$>\delta_1^{\bar{Q}}$, so the Q-structures at limit stages are of $\Ss^{\dagger}$ are trivial). This completes the sketch.
\end{proof}

\subsubsection{Definability of $\Sigma_{\vV,\sss}$ and (variants of) $M^\Tt_b$}\label{subsubsec:def_Sigma_vV,sss^dsr}
We consider first the question of whether
$\vV_1$ can define its own extender sequence over its universe. We don't know
whether this is the case or not, but in this direction:

\begin{lem}\label{lem:vV_1_def_from_M_infty|kappa_0}
	We have:
	\begin{enumerate}
		\item \label{item:vV_1|delta_0_def} $\vV_1|\delta_0^{\vV_1}$ is definable over the universe of $\vV_1$.
		\item \label{item:vV_1_def_from_M_infty|kappa_0}
		$\vV_1$ is definable over its universe from the parameter $\M_\infty|\kappa_0^{\M_\infty}$.
	\end{enumerate}
\end{lem}
\begin{proof}
	Part \ref{item:vV_1|delta_0_def}: Since $\vV_1|\delta_0^{\vV_1}=\M_\infty|\delta_0^{\M_\infty}$ and these have the same $V_{\delta_0^{\vV_1}}$, this is  an easy corollary of Remark \ref{rem:E^M_def} (and its proof).

	Part \ref{item:vV_1_def_from_M_infty|kappa_0}: Let $U$ be the universe of $\vV_1$. Recall that $U$ is closed under $\Sigma_{\M_\infty}$ for maximal trees $\Tt$ via $\Sigma_{\M_\infty,\sss}$. Since $\delta_\infty$ is Woodin in $\vV_1$, $\Sigma_{\M_\infty}(\Tt)$ is in fact the unique $\Tt$-cofinal branch in $\vV_1$,
	for such $\Tt$.
	Moreover, by
	the (local) definability of the short tree strategy and of maximality,
	$\M_\infty|\kappa_0^{\M_\infty}$ can define the collection of trees in
	$\M_\infty|\kappa_0^{\M_\infty}$ which are maximal via $\Sigma_{\M_\infty,\sss}$.
	Therefore working in $U$,
	from parameter $\M_\infty|\kappa_0^{\M_\infty}$,
	$\N=\M_\infty^{\M_\infty}|\delta_\infty^{\M_\infty}$ can be computed.
	But then the branch through the tree from $\M_\infty|\delta_\infty$ to $\N$
	can be computed, and hence also $e^{\vV_1}$ also.
	Therefore we can compute $\Ult(U,e^{\vV_1})$, which is the universe of $\Ult(\vV_1,e^{\vV_1})=\vV_1^{\M_\infty}$ (by Lemma \ref{lem:vV_1_M_infty[*]_inter-def}).
	But $\vV_1^{\M_\infty}[\M_\infty|\kappa_0^{\M_\infty}]\ueq\M_\infty$,
	so we can identify the universe of $\M_\infty$, so by Remark \ref{rem:E^M_def},
	we can identify $\M_\infty$ itself. But from $e$, we therefore compute $*$,
	hence $\M_\infty[*]$, and hence $\vV_1$,
	by Lemma \ref{lem:vV_1_M_infty[*]_inter-def}.
\end{proof}

\begin{lem}\label{lem:vV_def_in_vV[g]_from_inseg}
	Let $\vV$ be a non-dropping $\Sigma_{\vV_1}$-iterate of $\vV_1$.
	Let $\lambda\in\OR$  with $\lambda\geq\delta_0^{\vV}$ and
	$\PP\in\vV|\lambda^{+\vV}$ and $g$ be $(\Vv,\PP)$-generic.\footnote{When we deal with
		such generic extensions of such $\vV$ (here and later), we allow $g$ to appear in some set-generic
		extension of $V$, as opposed to demanding $g\in V$.}
	Then $\vV$ is definable over the universe of $\vV[g]$ from the parameter
	$x=\vV|\lambda^{+\vV}$.
\end{lem}
\begin{proof}
	This is an immediate corollary of Lemma \ref{lem:vV_1_def_from_M_infty|kappa_0}
	and ground definability (from the parameter $\pow(\PP)$).
\end{proof}

\begin{lem}\label{lem:vV_1_short_tree_strategy}\ Let $\vV$ be a
	$\Sigma_{\vV_1}$-iterate of $\vV_1$. Let $\lambda\geq\delta_0^{\vV}$, let $\PP\in\vV|\lambda$ and $g$ be $(\vV,\PP)$-generic. Let $x=\vV|\lambda^{+\vV}$.
	Then:
	\begin{enumerate}
		\item\label{item:vV_1_closed_under_dsr_sts} $\vV[g]$ is closed under $\Sigma^{\dsr}_{\vV,\sss}$
		and $\Sigma^{\dsr}_{\vV,\sss}\rest\vV[g]$ is  definable over the universe of $\vV[g]$
		from the parameter $x$
		(hence lightface $\vV$-definable if $g=\emptyset$).
		\item \label{item:vV_1_dsr_etc_definable} The notions
		\begin{enumerate}
			\item \emph{dsr},
			\item \emph{$\delta_1$-short/$\delta_1$-maximal
				dsr via $\Sigma_{\vV}$},
			\item  \emph{dsr-P-suitable}, and
			\item  \emph{dl-relevant},
		\end{enumerate}
		are each definable over $\vV[g]$ from $x$
		(hence lightface $\vV$-definable if $g=\emptyset$).
		\item\label{item:P-suitable_delta_1-max} For each $\delta_1$-maximal (hence dsr) P-suitable tree $\Tt\in\vV[g]$, as witnessed by $U$, letting $\bar{\Tt}$ be the equivalent tree
		on $\bar{\vV}=\cHull_1^{\vV}(\mathscr{I}^{\vV}\cup\delta_1^{\vV})$ (which is via $\Sigma_{\bar{\vV}}$)
		and $b=\Sigma_{\bar{\vV}}(\bar{\Tt})$, we have \[M^{\bar{\Tt}}_b=\cHull_1^{\mathscr{P}^{U,g}(M(\Tt))}(\mathscr{I}^{U}\cup\delta(\Tt)). \]
		Therefore letting $\eta$ be least such that $\vV$ is $\eta$-sound (so $\eta\leq\kappa_1^{\vV}$), if $\eta\leq\delta=\delta(\Tt)$ then $M^{\bar{\Tt}}_b=\mathscr{P}^{U,g}(M(\Tt))$
		and the function $\Tt\mapsto M^{\bar{\Tt}}_b$ \tu{(}with  domain all such $\Tt$ with $\eta\leq\delta(\Tt)$\tu{)} is definable over the universe
		of $\vV[g]$ from the parameter $(x,\eta)$ (hence from $\eta$ over $\vV$ if $g=\emptyset$).
		\item\label{item:dl-relevant_density} Suppose $g=\emptyset$.
		Then for each $\delta_1$-maximal  tree $\Tt\in\vV$ via $\Sigma_{\vV}$,
		with $\lh(\Tt)<\kappa_1^{\vV}$,  letting
		$\bar{\Tt}$ be as in part \ref{item:P-suitable_delta_1-max},
		there is a dl-relevant tree $\Xx\in\vV$, on $\vV$, and such that, letting $\bar{\Xx}$ be likewise,
		then  $M^{\bar{\Xx}}_c$ is a $\Sigma_{M^{\bar{\Tt}}_b}$-iterate of $M^{
			\bar{\Tt}}_b$, where $b=\Sigma_{\bar{\vV}}(\bar{\Tt})$ and $c=\Sigma_{\bar{\vV}}(\bar{\Xx})$.

	\end{enumerate}
	Moreover, the definability is uniform  in $\vV,x$, and hence preserved by iteration maps.
\end{lem}

\begin{proof}[Proof Sketch]
	Part \ref{item:P-suitable_delta_1-max} is an immediate consequence
	of Lemma \ref{lem:general_P-correctness}.

	Parts  \ref{item:vV_1_closed_under_dsr_sts}, \ref{item:vV_1_dsr_etc_definable}:
	For simplicity we assume $g=\emptyset$, but the general case is very similar.
	Let $\Tt\in\vV$ be dsr of limit length and via $\Sigma_{\vV}$;
	we will determine whether $\Tt$ is $\delta_1$-short or $\delta_1$-maximal,
	and if $\delta_1$-short, compute $\Sigma_{\vV}(\Tt)$.  Let $\Tt=\Tt_0\conc\Tt_1$ with lower and upper components $\Tt_0,\Tt_1$ respectively. Let $E_0\in\es^{\vV}$ be long
	with $\gamma^{\vV}<\lh(E_0)$ and $\Tt_0\in \vV|\lambda'$ where $\lambda'=\lgcd(\vV||\lh(E_0))$. Let $\widetilde{\Tt}=\widetilde{\Tt_0}\conc\widetilde{\Tt_1}$
	where $\widetilde{\Tt_0}$ is the successor length tree corresponding to $E_0$,
	and letting $k:M^{\Tt_0}_\infty\to M^{\widetilde{\Tt_0}}_\infty$ be the iteration map
	(recall this is known to $\vV$), $\widetilde{\Tt_1}$ is the minimal $k$-copy of $\Tt_1$
	(so $\widetilde{\Tt}$ is also via $\Sigma_{\vV}$, by \cite[***10.3, 10.4]{fullnorm_v3}).
	It suffices to compute $\Sigma_{\vV}(\widetilde{\Tt})$. So instead assume that
	$\Tt$ is itself in the form of $\widetilde{\Tt}$.

	Let $E_1\in\es^{\vV}$
	be $\vV$-total with $\crit(E_1)=\kappa_1^{\vV}$ and $\Tt'\in \vV_1|\lambda$ where
	$\lambda=\lambda(E_1)$ and
	and $\Tt'$ is the tree on $\vV|\delta_1$ equivalent to $\Tt$.
	Let  $U=\Ult(\vV,E_1)$.
	Let $\eta<\lambda$
	be a strong $\delta_0^{\vV}$-cutpoint and cardinal of $U$ with $\Tt'\in U|\eta$. Now working in $U$,
	form a minimal inflation $\Xx$ of $\Tt_1'$, first iterating the least measurable $>\delta_0^{M^{\Tt_0}_\infty}$ out to $\eta$, and then folding in $\es^U$-genericity iteration.
	Now  $\Xx$ is dsr (the issue being that we do not introduce new Woodin cardinals
	below the index of some $E^\Xx_\beta$,  condition \ref{item:no_new_Woodins_below_index}
	in the definition of \emph{dsr} (see \ref{dfn:dsr})), because $\Tt$ is dsr and the inflationary extenders are only
	being used for genericity iteration (and the linear iteration at the start).
	The remaining details of the minimal inflation and overall process
	are as sketched in \S\ref{section-short-tree-strategy-for-M}
	(but the minimal variant, which is essentially the same), using that $\Sigma_{\vV}$ has minimal inflation condensation, by
	Lemma \ref{lem:Sigma_vV_1_vshc} and \cite[***Theorem 10.2]{fullnorm_v3}.

	Part \ref{item:dl-relevant_density} follows from the proof of part \ref{item:vV_1_dsr_etc_definable}
	in the case that $\Tt$ is $\delta_1$-maximal,
	since by \cite{fullnorm_v3}, both the conversion from $\Tt$ to $\widetilde{\Tt}$
	and minimal inflation yields a correct iterate.
\end{proof}

\begin{dfn}\label{dfn:Sigma_P,sh^dsr^vV}
	Let $\vV$ be a non-dropping $\Sigma_{\vV_1}$-iterate of $\vV_1$,
	and $g$ be $\vV$-generic.
	Let $\vV^-=\vV|\delta_1^{\vV}$.
	Then $\Sigma_{\vV^-}$ denotes the strategy
	for $\vV^-$ induced by $\Sigma_{\vV}$,
	and $\Sigma_{\vV^-,\sss}$ denotes its restriction
	to $\delta_1$-short trees, and $\Sigma^{\dsr}_{\vV^-,\sss}$ its
	restriction dsr-$\delta_1$-short trees.\footnote{Note that if $\vV,\vV'$
		are both such and $\vV^-=(\vV')^-$,
		then we get the same strategy for $\vV^-$
		induced by $\Sigma_\vV$ and $\Sigma_{\vV'}$.} Also
	if $R$ is a non-dropping $\Sigma_{\vV_1}$-iterate of $\vV_1$,
	and $R^-\in\vV[g]$ is a $\Sigma_{\vV^-}$-iterate of $\vV^-$,
	then $\Sigma^{\vV[g]}_{R^-,\sss}$ and $(\Sigma^{\dsr}_{R^-,\sss})^{\vV[g]}$ denote
	the restrictions of $\Sigma_{R^-,\sss}$ and $\Sigma^{\dsr}_{R^-,\sss}$
	to trees in $\vV[g]$.

	Let $\Tt\in\vV[g]$ be P-suitable for $\vV[g]$, as witnessed by $U$.
	Let $R=\mathscr{P}^{U,g}(M(\Tt))$ (so $R$ is a $\vV[g]$-class and is a $\Sigma_{\vV_1}$-iterate of $\vV_1$). Then $\Sigma_{R,\sss}^{\vV[g]}$ and $(\Sigma^{\dsr}_{R,\sss})^{\vV[g]}$ denote the restriction of $\Sigma_R$ to $\delta_1$-short and dsr-$\delta_1$-short trees in $\vV[g]$, respectively.\footnote{So in the case that $\Tt$ is P-suitable, $\Sigma^{\vV[g]}_{M(\Tt),\sss}$ and $\Sigma^{\vV[g]}_{R,\sss}$ are equivalent, as are $(\Sigma^{\dsr}_{M(\Tt),\sss})^{\vV[g]}$
		and $(\Sigma^{\dsr}_{R,\sss})^{\vV[g]}$.}
\end{dfn}

\begin{lem}\label{lem:vV_1_iterates_short_tree_strategy} Let $\vV$ be a
	$\Sigma_{\vV_1}$-iterate of $\vV_1$, $\lambda\geq\delta_0^{\vV}$,
	$\PP\in\vV|\lambda$ and $g$ be $(\vV,\PP)$-generic.	Let $x=\vV|\lambda^{+\vV}$. Let $\Tt\in\vV[g]$ be a $\delta_1$-maximal
	tree on $\vV^-$, via $\Sigma_{\vV^-}$, and if $\Tt$ is P-suitable for $\vV[g]$ then let
	$R=\mathscr{P}^{U,g}(M(\Tt))$, where $U$ is as above.
	Then:
	\begin{enumerate}
		\item
		$\vV[g]$ is closed under  $(\Sigma^{\dsr}_{M(\Tt),\sss})^{\vV[g]}$
		and $(\Sigma^{\dsr}_{M(\Tt),\sss})^{\vV[g]}$ is definable over $\vV[g]$ from $(\Tt,x)$, uniformly in $\Tt$; hence likewise for $(\Sigma^{\dsr}_{R,\sss})^{\vV[g]}$
		for $\delta_1$-maximal P-suitable trees $\Tt\in\vV[g]$.
		\item The notions
		\begin{enumerate}[label=--]
			\item \emph{dsr}, and
			\item  \emph{$\delta_1$-short/$\delta_1$-maximal dsr via $\Sigma_{M(\Tt)}$},
		\end{enumerate}
		applied to trees in $\vV[g]$ on $M(\Tt)$,
		are definable over $\vV[g]$ from $(\Tt,x)$, uniformly in $\Tt$.
		\item Suppose $\Tt$ is P-suitable. Let $\Tt'$ be the tree on $\vV^-$
		iterating out to $M(\Tt)=R^-$. Let $\Uu'\in\vV[g]$ be on $R^-$ and via $(\Sigma^{\dsr}_{R^-,\sss})^{\vV[g]}$. Then the stack $(\Tt',\Uu')$ normalizes to a tree on $\vV^-$ via $(\Sigma^{\dsr}_{\vV^-,\sss})^{\vV[g]}$.
	\end{enumerate}

	Moreover, the definability is uniform in $\vV,x$, and so preserved by the iteration maps.
\end{lem}

Lemma \ref{lem:general_P-correctness} suffers from a significant drawback,
which is that it is restricted to dsr trees. In \cite{*-trans_add}
there is a generalization of this to arbitrary trees, but this
involves a further modification of the P-construction, given
by merging the preceding methods with $*$-translation.
We now summarize the key consequences of this,
also  proven in \cite{*-trans_add},
which we will need later.

\begin{lem}\label{lem:vV_1_full_short_tree_strategy}\ Lemmas \ref{lem:vV_1_short_tree_strategy}
	and 	\ref{lem:vV_1_iterates_short_tree_strategy} both remain
	true after striking out every instance of the term \emph{dsr}.
\end{lem}

Note that that parts
parts \ref{item:P-suitable_delta_1-max} and \ref{item:dl-relevant_density}
of  Lemma \ref{lem:vV_1_short_tree_strategy} are not actually modified
by striking out \emph{dsr}, because all $\delta_1$-maximal trees are dsr anyway.

\subsection{The second direct limit system}\label{subsec:second_dls}

We now define a system of uniform grounds for $\vV_1$, and the associated
Varsovian model $\vV_2$. This is
analogous to the construction of $\vV_1$ in \S\ref{section_first_varsovian},
albeit slightly more involved. For the most part it is similar,
and so we omit details and remarks which are like before.
We use the results of \S\ref{sec:dl-rel_sts_vV_1}, and in particular
the modified P-construction,
dsr $\delta_1$-short tree strategies,
etc.

\subsubsection{The external direct limit system ${\mathscr D}_1^{\ext}$}

\begin{definition}\label{second_defn_points_from_the_system}
	Let ${\mathbb U}_1$ be the $\vV_1$-class
	of all dl-relevant iteration trees (Definition \ref{dfn:dsr}).
	Define
	\[ d_1=\{\vV_1|\delta_1^{\vV_1}\}\cup\{\M(\Uu)\bigm|\Uu\in\mathbb{U}_1\text{ is non-trivial}\}.\]
	For $p\in d_1$, set $\P_p=\mathscr{P}^{\vV_1}(p)$ (P-construction as in Definition \ref{dfn:vV_1_P-con}).
	Write
	${\mathscr F}_1=\{\mathcal{P}_p\mid p\in d_1\}$.
	Define $\preceq$ on $d_1$ and on $\mathscr{F}_1$ and maps $\pi_{pq}=i_{\mathcal{P}_p\mathcal{P}_q}$ for $p\preceq q$ as in \S\ref{section_first_varsovian}.
\end{definition}

By Lemma \ref{lem:vV_1_short_tree_strategy},
$(d_1,\left<\P_p\right>_{p\in d_1})$
is lightface $\vV_1$-definable,
as are $\mathbb{U}_1$ and ${\mathscr F}_1$.

\begin{lemma}\label{lem:vV_2_preceq_props} $\preceq$ is a directed partial order, is lightface $\vV_1$-definable,
	and the associated embeddings commute:
	if $P\preceq Q\preceq R$ then $i_{QR}\com i_{PQ}=i_{PR}$.
\end{lemma}
\begin{proof}[Proof Sketch]
	For the definability, that $\preceq$ is partial order, and the commutativity,
	see the proof of Lemma \ref{lem:vV_1_preceq_props}.
	For directedness, let $\Tt,\Uu\in\mathbb{U}_1$,
	with lower and upper components $\Tt_0,\Tt_1$ and $\Uu_0,\Uu_1$ respectively.
	Let $E_0,F_0\in\es^{\vV_1}$ with be $\vV_1$-total
	with $\crit(E_0)=\crit(F_0)=\kappa_0$
	and such that $\Tt_0,\Uu_0$ correspond to $E_0,F_0$ respectively.
	We may assume $\lh(E_0)\leq\lh(F_0)$, so if $E_0\neq F_0$
	then in fact $\lh(E_0)<\lambda(F_0)$.
	Therefore $M^{\Uu_0}_\infty$ is a (possibly trivial) iterate of $M^{\Tt_0}_\infty$.
	Let $j:M^{\Tt_0}_\infty\to M^{\Uu_0}_\infty$ be the iteration map.
	Let $\Tt_1'$ be the minimal $j$-copy of $\Tt_1$.
	Now proceed
	with a pseudo-comparison of $\Uu_0\conc\Tt_0'$ and $\Uu$
	intertwined with pseudo-genericity-iteration,
	as in Lemma \ref{lem:vV_1_preceq_props}.
\end{proof}

Define the external direct limit system $\mathscr{D}^{\ext}_1=(P,Q,i_{PQ} \colon P\preceq Q
\in {\mathscr F}_1 )$.
We have  \ref{item:M_pc_ZFC},
\ref{item:d,preceq_directed_po}, \ref{item:P_i_system},
\ref{item:pi_ij_elem},
\ref{item:pi_ij_commute}, \ref{item:M_infty_wfd}, \ref{item:i_0_exists},
and write
\begin{eqnarray}
	(\M_{\infty 1}^{\ext}, (i_{P\infty} \colon P \in {\mathscr F}_1))= \dirlim\ \mathscr{D}^{\ext}_1.
\end{eqnarray}

Let $P\in\mathscr{F}_1$.
Then $\vV_1|\delta_1^P$ is $(P,{\mathbb B}^P_{\delta_1^P\delta_0^P})$-generic  and
hence
$P[\vV_1|\delta]=^*_\delta \vV_1$, so $P$ is a ground for $\vV_1$ via the extender algebra $\BB_{\delta_1^P\delta_0^P}$ (at $\delta_1^P$, using extenders with critical point $\geq\delta_0^P$ (hence $>\delta_0^P$)). Thus:

\begin{dfn}\label{dfn:vV_1_tau^P}
	For $P\in\mathscr{F}_1$, let $\tau_1^P$ be the canonical class $\BB_{\delta_1^P\delta_0^P}$-name for $\vV_1$, like in Definition \ref{dfn:tau^P},
	but incorporating the appropriate conversion for the overlapping extenders
	(note the generic filter determines $\vV_1|\delta_0^P$, which in turn determines
	the ``key'' to this conversion).
\end{dfn}

\begin{lemma}\label{lem:vV_1_c_i_dense}
	\ref{item:c_i_dense} holds:
	for each $P\in\mathscr{F}_1$, $c^P_1=d_1\cap d^P_1$ is dense
	in $(d^P_1,\preceq^P)$ and dense in $(d_1,\preceq)$,
	and ${\preceq^P}\rest c^P_1={\preceq}\rest c^P_1$.\end{lemma}
\begin{proof}
	Let $P\in\mathscr{F}_1$.
	That ${\preceq^P}\rest c^P_1={\preceq}\rest c^P_1$ is
	by Lemma \ref{lem:vV_1_iterates_short_tree_strategy}.
	So let $Q\in\mathscr{F}_1$
	and $R\in\mathscr{F}_1^P$. We must find $S\in\mathscr{F}_1\inter\mathscr{F}_1^P$
	with $Q,R\preceq S$. 	Let $\Tt_P=\Tt_{P0}\conc\Tt_{P1}$ be the maximal tree leading from $\vV_1$
	to $P$, with lower and upper components $\Tt_{P0},\Tt_{P1}$ respectively,
	and likewise for $R$, and let $\Tt_{PQ},\Tt_{PQ0},\Tt_{PQ1}$ be likewise
	for $Q$ in $P$. Let $E_0\in\es^{\vV_1}$ be long with $\lh(E_0)<\kappa_1$
	be $\lh(E_0)$ sufficiently large that $P,Q,R$ are all translations of one another
	above some $\gamma<\lambda(E_0)$ and the various trees are in $\vV_1|\lambda(E_0)$.
	Letting $E^P_0\in\es^P$
	and $E^R_0\in\es^R$ with $\lh(E_0^P)=\lh(E_0^R)=\lh(E_0)$, then $E^R_0,E^P_0$ are translations of $E_0$,
	so $\Ult(\vV_1,E_0)$, $\Ult(P,E_0^P)$ and $\Ult(R,E_0^R)$ agree through their common least Woodin $\lh(E_0)$ (but not above there if $P\neq\vV_1$, as $\Ult(\vV_1,E_0)$ is $\delta_0^{\Ult(\vV_1,E_0)}$-sound, whereas then the others are not).
	Let $\sigma\in P$ be a $\BB_{\delta_1^P\delta_0^P}$-name
	for $Q$, and let $p_1\in\BB_{\delta_1^P\delta_0^P}$ be the Boolean value
	of the statement ``$\tau^P$ is $\vV_1$-like and $\sigma\in\mathscr{F}_1^{\tau^P}$''.
	Working in $P$, we will form a Boolean-valued comparison/genericity iteration
	of $\Ult(P,E_0^P)$, $\Ult(R,E_0^R)$ and all interpretations of $\Ult(\sigma,F^{\sigma||\lh(E_0)})$ below $p_1$,
	much as in the proof of Lemma \ref{lem:c_i_dense} (in particular incorporating Boolean-valued $\tau_1^P$-genericity). However, because we have not yet established that
	$\vV_1$ knows its own $\delta_1$-short tree strategy, we cannot quite argue as for
	Lemma \ref{lem:c_i_dense}. Thus, we tweak the comparison as in the following sketch
	(the process will be use an idea from \cite[\S7]{iter_for_stacks};
	see especially \cite[Corollary 7.5 and Theorem 7.3 (Claim 8)]{iter_for_stacks}).

	We define a $\BB_{\delta_1^P\delta_0^P}$-name
	for a padded tree $\dot{\Uu}$ on $\sigma$, and define padded trees $\Tt$ on $\Ult(P,E_0^P)$
	and $\Vv$ on $\Ult(R,E_0^R)$, recursively on length in the usual manner for comparison.
	Given (names for) the trees up to length $\alpha+1$, we will also have some condition $q_\alpha$,
	with $q_\alpha\leq q_0=p_1$.  Let $q_{\alpha+1}$ be the Boolean value, below $q_\alpha$,
	of the statement ``the least disagreement between $M^{\dot{\Uu}}_\alpha$ and $M^\Tt_\alpha$ and $M^\Vv_\alpha$, if it exists, involves a dsr extender'' (that is, satisfying condition \ref{item:no_new_Woodins_below_index} of Definition \ref{dfn:dsr}).
	We then take the least forced disagreement working below $q_{\alpha+1}$, and use this index and genericity iteration considerations to determine the next extender, etc.
	Given everything through some limit stage $\gamma$, which is short, the strategies $\Sigma_{\cdot,\sss}^{\dsr}$ determine branches (as required), and set $q_\gamma$ to be the infimum of $\left<q_\alpha\right>_{\alpha<\gamma}$. The rest is as usual. The conditions $q_\alpha$ are always non-zero, and in fact $q_\alpha\in g$ where $g$ is the generic
	adding $\vV_1|\delta_0^P$, because $\dot{\Uu}_g,\Tt,\Vv$ are then correct trees on $R,P,Q$,
	which were themselves iterates via dsr trees, and by the analysis of comparison in \cite[***\S8]{fullnorm_v3}, the least disagreement must be an ultrapower-image of one of the extenders used in those dsr trees, and hence be itself appropriate for dsr.
	Because we rule out the use of non-dsr extenders, the Q-structure(s) $Q_\xi$ used in the trees
	at limit stages $\xi$ do not overlap $\delta(\Tt\rest\xi)$
	(except possibly with long extenders). They also agree with one another (in $\Tt,\Vv$
	and all interpretations of $\dot{\Uu}$),
	and no extenders in $\es_+^{Q_\xi}$ are used later in the comparison (in particular for genericity iteration). This is because in $P$ and in $P[g']$, the trees are P-suitable,
	and the Q-structures are produced by P-construction, and because of the agreement between $P,\vV_1$, they are therefore identical.
\end{proof}

\subsubsection{The internal direct limit system $\mathscr{D}_1$}

We adapt Definition \ref{dfn:M_internal_dls} in the obvious manner, to which we refer the reader for details:

\begin{dfn}\label{dfn:vV_1_internal_dls}
	Work in $\vV_1$. Define \emph{(weak) $s$-iterability}
	for  $P \in {\mathscr F}_1$ and $s\in[\OR]^{<\om}\cut\{\emptyset\}$
	as in \ref{dfn:M_internal_dls}. If $P$ is $s$-iterable
	and $s\sub t$ and $Q$ is $t$-iterable with $P\preceq Q$,
	then likewise for
	$\gamma^P_s$,
	$H^P_s$
	and $\pi_{Ps,Qt}:H^P_s\to H^Q_t$. Define \emph{strong $s$-iterability}
	as before.

	Let $\mathscr{F}^+_1=\{(P,s)\bigm|P\in\mathscr{F}\text{ and }P\text{ is strongly
	}s\text{-iterable}\}$, and similarly let
	$d^+_1=\{(P|\delta_1^P,s)\bigm|(P,s)\in\mathscr{F}_1^+\}$.
	The order $\preceq$ on $d^+_1$
	is determined by \ref{item:d^+_order}.
	Define $\preceq$ on $\mathscr{F}^+$ likewise.
	Clearly if	$(P,s)\preceq(Q,t)\preceq (R,u)$ then
	\[ \pi_{Ps,Ru}=\pi_{Qt,Ru}\com\pi_{Ps,Qt}.\]

	Define the system $\mathscr{D}_1=
	(H^P_s,H^Q_t,\pi_{Ps,Qt} \colon (P,s) \preceq (Q,t) \in \mathscr{F}_1^+)$.

	Given $P\in\mathscr{F}_1$ and $s\in[\OR]^{<\om}$, recall that $s$ is
	\emph{$P$-stable}
	iff $\pi_{PQ}(s)=s$ for every $Q\in\mathscr{F}_1$ with $P\preceq Q$.
\end{dfn}

\begin{rem}
	As in Remark \ref{rem:M_s-it_implies_strong_s-it},
	$s$-iterability actually implies strong $s$-iterability.
\end{rem}

The following lemma yields properties  \ref{item:d^+},
\ref{item:d^+_order}, \ref{item:reduce_t}, \ref{item:increase_i},
\ref{item:D^+_M-def}, \ref{item:H_i^s}, \ref{item:pi_is_js},
\ref{item:pi_is_it}, \ref{item:pisjt_commute},
\ref{item:every_s_gets_stable}:

\begin{lemma}\label{lem:vV_1_first_properties_int_sys}
	We have:
	\begin{enumerate}[label=\tu{(}\alph*\tu{)}]
		\item\label{item:vV_1_stable_b} if $P \in {\mathscr
			F}_1$ and $s\in[\OR]^{<\om}\cut\{\emptyset\}$ and $s$ is $P$-stable, then
		$(P,s)\in\mathscr{F}^+_1$ and $(P,s)$ is true
		(see Definition \ref{dfn:ug_stable}).
		\item\label{item:vV_1_stable_c}  $(\mathscr{F}^+_1,\preceq)$ is
		directed -- for $(P,s)$, $(Q,t) \in \mathscr{F}^+_1$
		there is $(R,u) \in \mathscr{F}^+_1$ with $(P,s) \preceq (R,u)$ and
		$(Q,t) \preceq (R,u)$ \tu{(}note $u=s\cup t$ suffices\tu{)}.
		\item
		$\mathscr{D}_1$ is lightface $\vV_1$-definable.
	\end{enumerate}
\end{lemma}

\begin{dfn}
	Noting that $\mathscr{I}^{\vV_1}=\mathscr{I}^M$ is a club class of generating indiscernibles for $\vV_1$,
	define $\mathscr{I}^{P}=i_{\vV_1 P}``\mathscr{I}^{\vV_1}$ whenever $P$ is a non-dropping iterate of $\vV_1$.
\end{dfn}

For the following, see the proof of Lemma \ref{lem:M_indisc_stable}:
\begin{lemma}
	For each $P\in\mathscr{F}_1$, $P$ is $\{\alpha\}$-stable
	for every $\alpha\in\mathscr{I}^M=\mathscr{I}^P$.
	Therefore property \ref{item:every_x_gets_captured} holds,
	as witnessed by some $s\in[\mathscr{I}^M]^{<\om}$.
\end{lemma}

We can now (working in $\vV_1$) define the direct limit
\begin{equation}\label{vV_1_internal_dir_limit}
	({\cal M}_{\infty1}, \pi_{Ps,\infty} \colon (P,s) \in \mathscr{F}_1^+)
	= \dirlim\ \mathscr{D}_1,\end{equation}
and the associated $*$-map $*_1$.
This notation  is somewhat cumbersome, so let us also write $\N_\infty=\M_{\infty 1}$,
and we will often write $*$ instead of $*_1$, where there should be no cause for confusion.
By Lemmas \ref{c0} and \ref{c1}, $\chi:{\cal N}_{\infty} \to {\cal
	N}_{\infty}^{\ext}$ is the identity and $\N_{\infty}=\N_{\infty}^{\ext}$. Property \ref{item:struc_emb_correct}
holds as if $P\in\mathscr{F}_1$
and $\bar{Q}\in d_1^P\cap d_1$ then
$\P_{\bar{Q}}^P=\mathscr{P}^P(\bar{Q})=\mathscr{P}^{\vV_1}(\bar{Q})$,
because $\es^P$ is a translation of $\es^{\vV_1}$ above  $\delta_1^P$).
And \ref{item:all_s_eventually_internally_in} again holds if $s$ is $P$-stable.

So we have established \ref{item:M_pc_ZFC}--\ref{item:all_s_eventually_internally_in}.
For the remaining properties set $\delta=\delta_1^{\vV_1}$ and $\BB=\BB^{\vV_1}_{\delta\delta_0^{\vV_1}}$
(for the witnesses to those properties in \S\ref{sec:ground_generation}).
This gives \ref{item:delta_reg,B_delta-cc_cba}.
Recall we defined $\tau^P$ in Definition \ref{dfn:vV_1_tau^P}.
Write
$\delta_{1\infty}=i_{\vV_1\infty}(\delta_1)=\delta_1^{\N_{\infty}}$ (replacing
the notation $\delta_\infty$ of \S\ref{sec:ground_generation}). As for  Lemma \ref{key_facts_about_V1}:

\begin{lemma}\label{lem:vV_1_key_facts_about_V1}
	We have:	\begin{enumerate}
		\item \label{item:vV_1_tau_g_gives_M|alpha}	For each $\vV_1$-stable $\alpha\in\OR$
		and each $P\in\mathscr{F}_1$, letting
		and $g$
		be the $P$-generic filter for $\BB_{\delta_1^P\delta_0^P}^P$
		given by $\vV_1|\delta_1^P$, then
		$i_{\vV_1P}(\tau^{\vV_1}\rest\alpha)_g=\vV_1|\alpha$. Moreover,
		$\vV_1\ueq P[g]\ueq P[\vV_1|\delta_1^P]$.

		\item\label{item:vV_1_unif_grds_prop_holds}	\ref{item:unif_grds} holds.

		\item\label{item:vV_1_kappa_0^M=least_meas}  $\kappa_1^M=\kappa_1^{\vV_1}$ is the least measurable
		cardinal of ${{\cal N}_{\infty}}$.
		\item\label{item:vV_1_kappa_0^+M=delta_infty} $\kappa_1^{+M}=\kappa_1^{+\vV_1} = \delta_{1\infty}$.
	\end{enumerate}
\end{lemma}

\subsubsection{The second Varsovian model as $\N_{\infty}[*]$}

\begin{dfn}\label{dfn:vV_2_M_infty[*]}
	Recall that $*_1$ is the $*$-map associated to the preceding construction. We define the structure
	\begin{eqnarray}\label{vV_1_defn_first_varsovian_model}
		\N_{\infty}[*_1]=(L[\N_{\infty},*_1],\N_{\infty},*_1);
	\end{eqnarray}
	that is, with universe $L[\N_{\infty},*_1]$
	and predicates $\N_{\infty}$ and $*_1$. However, as mentioned above,
	we will often abbreviate $*_1$ with $*$, hence writing $\N_\infty[*]$.

	Note this structure has the universe of the abstract Varsovian model of \S\ref{sec:ground_generation}.
	Essentially by \S\ref{sec:ground_generation}, we have the elementary maps
	\[ \pi_{\infty1}:\N_{\infty}\to\N_{\infty}^{\N_{\infty}}, \]
	\[ \pi_{\infty1}^+:\N_{\infty}[*_1]\to\N_{\infty}^{\N_{\infty}}[*_1^{\N_{\infty}}], \]
	where $\N_{\infty}^{\N_{\infty}}$ is computed in $\N_{\infty}$
	just as $\N_{\infty}$ is computed in $\vV_1$, and   $*_1^{\N_{\infty}}$
	is the $*$-map as computed in $\N_{\infty}$. Recall $*_1\sub\pi_{\infty1}\sub\pi_{\infty1}^+$,
	and these maps are lightface definable over $\vV_1$.\end{dfn}

We next point out that $\N_{\infty}^{\N_{\infty}}$ is a $\Sigma_{\N_\infty}$-iterate of $\N_{\infty}$ and $\pi_{\infty1}$ is
the correct iteration map. We also want to generalize this to other iterates
of $\vV_1$.

\begin{dfn}
	Given a $\vV_1$-like Vsp $N$, let $\mathscr{D}^N_1$ and $\N_\infty^N$
	be defined over $N$ just as $\mathscr{D}_1$, $\N_\infty$ are defined
	over $\vV_1$, and likewise $*_1^N,\pi_{\infty1}^N,(\pi_{\infty1}^+)^N$.
	If $N$ is a correct iterate of $\vV_1$, also define
	$(\N_\infty^{\ext})_N$ (the external direct limit) relative to $N$, as for $\vV_1$: given a maximal tree $\Tt\in\mathbb{U}^N_1$ (considered as a tree on $N$),
	let $b=\Sigma_N(\Tt)$ and $M_{\Tt}=M^\Tt_b$,
	and let $(\N_\infty^{\ext})_N$ be the direct limit of these models $M_\Tt$
	under the iteration maps.
	If in fact $M_\Tt=\mathcal{P}^N_{M(\Tt)}$ (the model indexed by $M(\Tt)$ in the covering system $\mathscr{D}^{N}_1$) for each such $\Tt$, then define $\chi_N:\N_\infty^N\to(\N_\infty^{\ext})_N$ as in \S\ref{sec:ground_generation}.
\end{dfn}

\begin{lem}\label{lem:vV_1_delta_1-sound_iterate_M_infty}
	Let $N$ be a $\delta_1^N$-sound, non-dropping $\Sigma_{\vV_1}$-iterate of $\vV_1$.

	Then
	$M_\Tt=\mathscr{P}^N(M(\Tt))=\mathcal{P}^N_{M(\Tt)}$\footnote{Recall that the notation
		is $\mathscr{P}(M(\Tt))$ for P-construction over $M(\Tt)$,
		and $\mathcal{P}_{M(\Tt)}$ for the model of $\mathscr{D}_1$ indexed at $M(\Tt)$.} for each $\Tt\in\mathbb{U}^N_1$,
	$\N_\infty^N=(\N_\infty^{\ext})_N$ and $\chi_N=\id$,
	and $\N_\infty^N$ is a $\delta_1^{\N_\infty^N}$-sound,
	non-dropping $\Sigma_N$-iterate of $N$, and hence is
	a $\Sigma_{\vV_1}$-iterate of $\vV_1$.
	Moreover,
	\[ \pi_{\infty1}^N:\N_\infty^N\to\N_\infty^{\N_\infty^N} \]
	is the iteration map according to $\Sigma_{\N_\infty^N}$.
	This holds in particular for $N=\vV_1$ and for $N=\N_\infty$,
	so $\N_\infty^{\N_\infty}$ is a correct iterate of $\N_\infty$,
	and $\pi_{\infty 1}$ is the iteration map.
\end{lem}

\begin{proof}This is just Lemma \ref{lem:general_P-correctness} and a consequence thereof,
	and by standard arguments.
\end{proof}

Like with $M$, working in $\N_\infty[*_1\rest\delta_{\infty1}]$
we can compute $\pi_{\infty1}^+$, so $\N_\infty[*_1]$ has universe
\[ L[\N_\infty,*_1\rest\delta_{\infty1}]=L[\N_\infty,*_1]=L[\N_\infty,\pi_{\infty1}]=L[\N_\infty,\pi_{\infty1}^+].\]

\subsubsection{Uniform grounds of $\vV_1$}

\begin{lemma}\label{lem:vV_1_still-a-woodin-in-v} Write $\varepsilon=\delta_{1\infty}$. We have:
	\begin{enumerate}
		\item\label{item:vV_1_sys_1_V_delta_infty_preserved}
		$V_{\varepsilon}^{\N_\infty[*_1]} = V_{\varepsilon}^{{{\cal N}_\infty}}$.
		\item\label{item:vV_1_sys_1_delta_infty_Woodin_in_M_infty[*]} $\varepsilon$ is
		\tu{(}the second\tu{)} Woodin in $\N_\infty[*_1]$ (and $\delta_0^{\N_\infty}$ the first).
		\item\label{item:vV_1_V_delta_and_Woodin_from_gg}
		Property \ref{item:delta_infty-cc} of uniform grounds holds for $\N_\infty[*_1]$ at $\varepsilon$;
		that is, $\N_\infty[*_1]\sats$``$\varepsilon$ is
		regular and $\BB_\infty$ is
		$\varepsilon$-cc''. Moreover, $\N_\infty[*_1]\sats$``$\BB_\infty$ is a complete Boolean algebra''.
	\end{enumerate}
\end{lemma}

\begin{proof}Part \ref{item:vV_1_sys_1_V_delta_infty_preserved}: As usual we have $*_1 \upharpoonright
	\eta \in {{\cal N}_\infty}$ for every $\eta<\varepsilon$.
	Now  $\N_\infty^{\N_\infty}[*_1^{\N_\infty}]$ is a class of $\N_\infty$
	and
	\[ \pi_{\infty1}^+:\N_\infty[*_1]\to\N_\infty^{\N_\infty}[*_1^{\N_\infty}] \]
	is elementary.
	Let $A\in\pow(\OR)\inter\N_\infty[*_1]$. Then $\pi_{\infty1}^+(A)\in\N_\infty$.
	So if $A\sub\eta<\varepsilon$ then $\N_\infty$ can compute $A$ from the set $\pi_{\infty1}^+(A)$ and the map $\pi_{\infty1}^+\rest\eta=*_1\rest\eta$.
	The remaining parts are now as in Lemma \ref{still-a-woodin-in-v}.
\end{proof}

So by Theorem \ref{tm:Varsovian_is_ground}, $\N_\infty[*_1]$ is a ground
of $\vV_1$.

\subsection{The second Varsovian model as the strategy mouse $\vV_2$}\label{subsec:vV_2}

Let $j:\M_\infty|\delta_\infty\to\N_\infty|\delta_0^{\N_\infty}$ be the restriction of the $\Sigma_{\M_\infty}$-iteration map. Note that for each $\nu>\kappa_1$,
if $F=F^{\vV_1||\nu}\neq\emptyset$ is long, then $\kappa_1<\lambda(F)$
and $\N_\infty|\delta_0^{\N_\infty}$ is definable in the codes over $\vV_1|\kappa_1$,
and hence in $\N_\infty|\delta_0^{\N_\infty}\in\vV_1|\lambda$.
Moreover, letting $P=\Ult(\vV_1,F)\downarrow 0$, we have $\delta_0^P=\lh(F)$ and $P|\delta_0^P$ is an iterate of $\N_\infty|\delta_0^{\N_\infty}$. In this circumstance let
\[ k_\nu:\N_\infty|\delta_0^{\N_\infty}\to P|\delta_0^P \]
be the iteration map.
Now define
$\mathbb{F}^{\vV_1}_{>\kappa_1}$ as the class
of all tuples $(\nu,\alpha,\beta)\in\OR^3$ such that $\nu > \kappa_1$, $F=F^{\vV_1||\nu}\neq\emptyset$ and either
\begin{enumerate}[label=--]
	\item $F$ is short (so $\kappa_1\leq\crit(F)$) and $F(\alpha)=\beta$, or
	\item $F$ is long and $k_\nu(\alpha)=\beta$.
\end{enumerate}

\begin{lemma}\label{lem:vV_1_restr_of_extenders_are_there}
	${\mathbb F}^{\vV_1}_{>\kappa_1}$ is lightface definable over $\N_\infty[*_1]$.
\end{lemma}
\begin{proof}
	Write $*$ for $*_1$. Let $(\nu,\alpha,\beta)\in\OR^3$ with $\nu>\kappa_1$.
	Let $F'=F^{\N_\infty||\nu^*}$.
	We claim that $(\nu,\alpha,\beta)\in{\mathbb F}^{\vV_1}_{>\kappa_1}$
	iff either
	\begin{enumerate}[label=--]
		\item $F'\neq\emptyset$ is short and $F'(\alpha^*)=\beta^*$, or
		\item $F'\neq\emptyset$ is long and $F'(\alpha)=\beta^*$ (the argument to $F'$ is $\alpha$, not $\alpha^*$!),
	\end{enumerate}
	and moreover, if $F\neq\emptyset$ then $F$ is short iff $F'$ is short. This is proved
	like in Lemma \ref{restr_of_extenders_are_there}, but the case that $F$ is long is uses the modified P-construction.
\end{proof}

\begin{lemma}\label{vV_1_V_is_a_ground}
	Let
	${\mathbb L}=\mathbb{L}^{\N_\infty[*_1]}(\kappa_1)$
	(Definition \ref{defn_bukowsky-poset}, for adding a  subset of $\kappa_1$). Then	$\vV_1|\kappa_1$ is ${\mathbb L}$-generic over
	$\N_\infty[*_1]$ and
	$\N_\infty[*_1][\vV_1|\kappa_1] \ueq\vV_1$.
\end{lemma}

\begin{proof}
	This follows from Lemma  \ref{restr_of_extenders_are_there} almost like
	in the proof of Lemma \ref{V_is_a_ground},
	using the fact that $\N_\infty|\delta_0^{\N_\infty}$ and the iteration map $j$ used above are definable (in the codes) over $\vV_1|\kappa_1$, and hence available
	to $\N_\infty[*_1][\vV_1|\kappa_1]$.
\end{proof}

We now adapt Definition \ref{dfn:vV_1}, presenting
the second Varsovian model as a  strategy mouse
$\vV_2$ analogous to $\vV_1$.
The sequence $\es^{\vV_2}$ will have two kinds of long extenders, corresponding
to $\delta_0^{\N_\infty}$ and $\delta_1^{\N_\infty}$:

\begin{dfn}\label{dfn:vV_2}
	Write $\gamma_0^{\vV_2}=\gamma^{\N_\infty}$ and $\gamma_1^{\vV_2}=(\kappa_1^{\N_\infty})^{+\N_\infty}$.
	Note that \[ \kappa_1<\delta_0^{\N_\infty}<\gamma_0^{\vV_2}<\kappa_1^{+\vV_1}=\delta_1^{\N_\infty}<\gamma_1^{\vV_2}.\]

	Define the structure
	\[ \vV_2=(L[{\mathbb E}^{\vV_2}];\in,\mathbb{E}^{\vV_2}),\]
	with segments $\vV_2||\nu=(\J_{\nu}[\es^{\vV_2}\rest\nu];{\in},\es^{\vV_2}\rest\nu,\es^{\vV_2}_\nu)$ and their passivizations $\vV_2|\nu$, recursively in $\nu$ as follows:
	\begin{eqnarray}
		\es_\nu^{\vV_2} =
		\begin{cases}
			\es_\nu^{{\cal N}_\infty} & \mbox{ if } \nu< \gamma_1^{\vV_2} \\
			\pi_{\infty1} \rest ({\cal N}_\infty|\delta_1^{{\cal N}_\infty}) & \mbox{ if }
			\nu = \gamma_1^{\vV_2} \\
			\es^{\vV_1}_\nu \rest (\vV_2|\nu) & \mbox{ if } \nu >
			\gamma_1^{\vV_2}\mbox{ and }\es^{\vV_1}_\nu\text{ is short},
			\\
			k_\nu\rest(\N_\infty|\delta_0^{\N_\infty}) & \mbox{ if } \nu>\gamma_1^{\vV_2}\mbox{ and }\es^{\vV_1}_\nu\text{ is long},
		\end{cases}
	\end{eqnarray}
	and with $\es^{\vV_2}= \{ (\nu , x , y ) \colon \es_\nu^{\vV_1} \not= \emptyset
	\text{ and } y = \es_\nu^{\vV_1}(x) \}$ and
	$\es^{\vV_2} \rest \nu$ as usual. 	(We verify well-definedness in Lemma \ref{vV_2_local_correspondence}.)

	Definability etc over $\vV_2$ has the predicate $\es^{\vV_2}$ available by default.

	Write $e_i^{\vV_2}=\es^{\vV_1}_{\gamma_i^{\vV_2}}$ for $i=0,1$.
\end{dfn}

The fine structural concepts for segments of $\vV_2$ are defined directly
as for segments of $\vV_1$ (Definition \ref{dfn:vV_1_fine_structure}).
The next two lemmas are direct adaptations of Lemmas \ref{local-definability-of-that-structure}, \ref{local_correspondence} respectively:

\begin{lemma}\label{vV_2_local-definability-of-that-structure}
	Let $\bar{\vV}=\vV_2||\gamma_1^{\vV_2}$. Then:
	\begin{enumerate}[label=\tu{(}\alph*\tu{)}]
		\item\label{item:vV_2_forcing_def}
		${\mathbb L}=\mathbb{L}^{\N_\infty[*_1]}(\kappa_1)$ is
		$\Sigma_1$-definable over $\bar{\vV}$.
		\item\label{item:vV_2_gamma_structure_def}
		$\bar{\vV}$ is isomorphic to a structure which is
		definable without parameters over
		$\vV_1|\kappa_1^{+\vV_1}$.
		\item\label{item:vV_2||gamma_is_sound} $\bar{\vV}$ is sound, with
		$\rho_\om^{\bar{\vV}}=\rho_1^{\bar{\vV}}=\delta_{1\infty}$ and $p_1^{\bar{\vV}}=\emptyset$.
		\item\label{item:vV_2_better_size_of_gamma}
		$\OR^{\bar{\vV}}<\xi_0$,
		where $\xi_0$ is the least $\xi>\kappa_1^{+\vV_1}$ such that $\vV_1|\xi$ is
		admissible.
		Therefore  $\vV_2||\nu$ is passive for every $\nu\in(\gamma_1^{\vV_2},\xi_0]$.
	\end{enumerate}
\end{lemma}

\begin{lemma}\label{vV_2_local_correspondence}
	Let $g=g_{\vV_1|\kappa_1}$
	be the $(\N_\infty[*_1],{\mathbb L})$-generic determined by $\vV_1|\kappa_1$. For every
	$\nu\in\OR$:
	\begin{enumerate}
		\item\label{item:vV_2||nu_in_Minfty*} $\vV_2|\nu$ and $\vV_2||\nu$ are in
		$\N_\infty[*_1]$,
		\item\label{item:vV_2||nu_sound} $\vV_2|\nu$ and $\vV_2||\nu$ are sound,
		\item\label{item:vV_2_fs_correspondence} Suppose $\nu\geq\xi_0$ and let $E=
		F^{\vV_2||\nu}$ and $E'=F^{\vV_1||\nu}$.\footnote{
			Also, $M|\theta_0^M$ and $\vV_1||\gamma^{\vV_1}$ are ``generically equivalent in
			the codes'',
			and letting
			\[ f:(\theta_0^M,\xi_0)\to(\gamma^{\vV_1},\xi_0) \]
			be the unique surjective order-preserving map,
			then $M|\alpha=M||\alpha$ are likewise equivalent
			with $\vV_1|f(\alpha)=\vV_1||f(\alpha)$ for all $\alpha\in\dom(f)$,
			but we will not need this.}
		Then
		\begin{enumerate}[label=\tu{(}\alph*\tu{)}]
			\item ${\mathbb L} \in \vV_2|\nu$ and $g$ is $(\vV_2|\nu,\mathbb{L})$-generic,
			\item $(\vV_2|\nu)[g] =^* \vV_1|\nu$,\footnote{The notation is explained in
				\ref{rem:vV_2_=^*}.}
			\item $(\vV_2||\nu)[g] =^* \vV_1||\nu$,
			\item\label{item:vV_2_premouse_axioms} if $E'\neq\emptyset$ and
			$\crit(E')>\kappa_1$ then $\vV_2||\nu$ satisfies the usual premouse
			axioms with respect to  $E$ (with Jensen indexing;
			so $E$ is an extender over $\vV_2|\nu$ which coheres $\es^{\vV_2|\nu}$, etc),
			\item\label{item:vV_2_non_coherence} if $E'\neq\emptyset$ and
			$\crit(E')=\kappa_1$ then $E$ is a
			long $(\delta_{1\infty},\nu)$-extender
			over $\N_\infty$ and
			\[
			\Ult(\N_\infty|\delta_{1\infty},E)=i^{\vV_1}_{E'}
			(\N_\infty|\delta_{1\infty})=\N_\infty^{\Ult(\vV_1,E')}|i^{\vV_1}_{E'}(\delta_{1\infty})\]
			is a lightface proper class of $\vV_2|\nu$, uniformly in such $\nu$, and
			\item\label{item:vV_2_non_coherence_ii} if $E'\neq\emptyset$
			and $E'$ is long then $E$ is a long $(\delta_0^{\N_\infty},\nu)$-extender over $\N_\infty$ and
			\[ \Ult(\N_\infty|\delta_0^{\N_\infty},E)=i^{\vV_1}_{E'}(\vV_1|\delta_0^{\vV_1})=\M_\infty^{\Ult(\vV_1,E')}|i^{\vV_1}_{E'}(\delta_0^{\vV_1}) \]
			is a lightface proper class of $\vV_2|\nu$, uniformly in such $\nu$.
		\end{enumerate}
	\end{enumerate}
\end{lemma}
\begin{rem}\label{rem:vV_2_=^*}
	Here the notation $=^*$ is like in Remark \ref{rem:=^*},
	except that  when $E'$ is long, we have $E'=E\com j$, instead of $E\sub E'$; recall $j$ is encoded into $g$. An analogous consideration applies to the proof of part \ref{item:N_infty[*]_class_of_vV_2} in the next lemma; cf.~Lemmas
	\ref{lem:vV_1_M_infty[*]_same_univ} and \ref{lem:vV_1_M_infty[*]_inter-def} and their proofs:
\end{rem}

\begin{lem}\label{lem:vV_2_N_infty[*]_inter-def}\
	\begin{enumerate}
		\item\label{item:N_inf_*_V_2_same_univ} 	$\N_\infty[*_1]$ and $\vV_2$
		have the same universe.
		\item \label{item:vV_2^N_infty_=_Ult(vV_2,e)}
		$\Ult(\vV_2,e_1^{\vV_2})=\vV_2^{\N_\infty}$.
		\item\label{item:vV_2_lightface_class_of_N_infty[*]} $\vV_2$ is a lightface
		class of
		$\N_\infty[*_1]$.

		\item \label{item:N_infty[*]_class_of_vV_2} $\N_\infty[*_1]$ is a lightface
		class of $\vV_2$.
	\end{enumerate}
\end{lem}

\subsection{Iterability of $\N_\infty|\delta_{1\infty}$ in $\vV_1$ and $\vV_2$}

Adapting Definition \ref{dfn:N_alpha_P-stable}:
\begin{dfn}\label{dfn:vV_2_N_alpha_P-stable}Let $\vV$ be a non-dropping $\Sigma_{\vV_1}$-iterate of $\vV_1$.
	For $P\in\mathscr{F}^{\vV}_1$ let
	\[ H^P=\Hull_1^P(\delta_1^P\cup\mathscr{I}^{\vV}),\]
	$\bar{P}$ its transitive collapse
	and	$\pi_{\bar{P}P}:\bar{P}\to P$ the uncollapse map.
	Recall here that by Lemma \ref{lem:general_P-correctness}, $\bar{P}$
	is a $\delta_1^{\bar{P}}$-sound  $\Sigma_{\vV_1}$-iterate of $\vV_1$.
	Define $(\N_\infty^{\overline{\ext}})_{\vV}$ as the direct limit of
	the iterates $\bar{P}$ such for $P$.

	Recall that $\vV$ is automatically $\kappa_1^{\vV}$-sound. Let $\alpha\in\OR$ and $P\in\mathscr{F}_1^{\vV}$.
	We say that $\alpha$ is \emph{$(P,\mathscr{F}_1^{\vV})$-stable}
	iff whenever $P\preceq Q\in\mathscr{F}_1^{\vV}$, we have $\alpha\in H^Q$ and
	\[
	\pi_{\bar{Q}Q}\com i_{\bar{P}\bar{Q}}\com\pi_{\bar{P}P}^{-1}(\alpha)=\alpha.\qedhere\]
\end{dfn}

Adapting Lemmas \ref{lem:delta_0-sound_iterate_M_infty}, \ref{lem:pi_infty_is_iteration_map}, \ref{lem:M_infty_of_iterate_N},
\ref{lem:every_alpha_ev_stable} and \ref{lem:M_infty_of_kappa_0-sound_iterate}
and their proofs
(and using that non-dropping $\Sigma_{\vV_1}$-iterates of $\vV_1$ are always $\kappa^{\vV}_1$-sound), we have:

\begin{lem}\label{lem:vV_2_M_infty_of_iterate_N_etc}
	Let $\vV$ be a non-dropping $\Sigma_{\vV_1}$-iterate of $\vV_1$
	and $\bar{\vV}$ be the $\delta_1^{\vV}$-core of $\vV$.
	Let $\N=\N_\infty^{\vV}$.
	Then:
	\begin{enumerate}

		\item\label{item:vV_2_H^P_subset_H^Q} For each $P\preceq Q\in\mathscr{F}_1^{\vV}$,
		we have $H^P\inter\OR\sub H^Q\inter\OR$.
		\item\label{item:vV_2_each_alpha_ev_stable} For each $\alpha\in\OR$
		there is $P\in\mathscr{F}_1^{\vV}$
		such that
		$\alpha$ is $(P,\mathscr{F}_1^{\vV})$-stable.

		\item\label{item:vV_2_I^M_infty=I^M} $\mathscr{I}^{\N_\infty}=\mathscr{I}^{\vV_1}$
		and $i_{\vV_1\N_\infty}\rest\mathscr{I}^{\vV_1}=\id=*_1\rest\mathscr{I}^{\vV_1}$,
		\item\label{item:vV_2_I^M_infty^N=I^N} $\mathscr{I}^{\N_\infty^{\vV}}=\mathscr{I}^{\vV}$
		and $*_1^{\vV}\rest\mathscr{I}^{\vV}=\id$,

		\item\label{item:vV_2_M_infty^N_is_ext-bar_M_infty}
		$\N= \N_\infty^{\vV}=i_{\vV_1\vV}(\N_\infty)=(\N_\infty^{\overline{\ext}})_{\vV}$
		is a $\delta_1^{\N}$-sound $\Sigma_{\bar{\vV}}$-iterate of $\bar{\vV}$.
		Moreover, $\N=(\N_\infty^{\ext})_{\vV}$ iff $\bar{\vV}=\vV$ iff $\N$ is a $\Sigma_{\vV}$-iterate of $\vV$ iff $\vV$ is $\delta_1^{\vV}$-sound.
		\item\label{item:vV_2_iterate_of_M_infty} Let
		$\vV'$ be a non-dropping $\Sigma_{\vV}$-iterate $\vV$
		with  $\vV|\kappa_1^{\vV}\pins \vV'$. Then
		\begin{enumerate}
			\item\label{item:vV_2_M_infty^N_is_iterate}  $\N_\infty^{\vV'}$ is a $\Sigma_{\N_\infty^{\vV}}$-iterate
			of $\N_\infty^{\vV}$, and
			\item\label{item:vV_2_i_MN_restricts_to_iteration_map} $i_{\vV\vV'}\rest\N_\infty^{\vV}$ is just the $\Sigma_{\N_\infty^{\vV}}$-iteration map
			$\N_\infty^{\vV}\to\N_\infty^{\vV'}$.
		\end{enumerate}

		\item\label{item:vV_2_N_infty^N_infty^vV_is_sound_iterate} $\N_\infty^{\N}$ is a $\delta_1^{\N_\infty^{\N}}$-sound $\Sigma_{\N}$-iterate of $\N$ and
		$*_1^{\vV}\sub\pi_{1\infty}^{\vV}:\N\to\N_\infty^{\N}$ is the $\Sigma_{\N}$-iteration map.

	\end{enumerate}
\end{lem}

Recall that $\Sigma_{\vV,\vV|\alpha}$ denotes the restriction of
$\Sigma_{\vV}$ to trees based on $\vV|\alpha$. Write $(\N_\infty^{\vV})^-=\N_\infty^{\vV}|\delta_{1\infty}^{\vV}$.

\begin{lemma}\label{vV_2_M_knows_how_to_iterate_Minfty_up_to_its_woodin}
	Let $\vV$ be a non-dropping $\Sigma_{\vV_1}$-iterate of $\vV_1$. Then:
	\begin{enumerate}[label=\tu{(}\alph*\tu{)}]
		\item\label{item:vV_2_M_infty_strat_thru_delta_0} $\vV$ is closed under $\Sigma_{{{\cal N}_\infty^{\vV}},(\N_\infty^{\vV})^-}$
		and $\Sigma_{{{\cal N}_\infty^{\vV}},(\N_\infty^{\vV})^-}\rest\vV$ is lightface
		definable over $\vV$.
		\item\label{item:vV_2_M_infty_strat_thru_delta_0_in_M[g]} Let $\lambda\geq\delta_0^{\vV}$, let $\PP\in\vV|\lambda$, and
		$g$ be $(\vV,\PP)$-generic (with $g$ appearing in some generic extension of $V$).
		Then $\vV[g]$ is closed under $\Sigma_{\N_\infty^{\vV},(\N_\infty^{\vV})^-}$
		and $\Sigma_{\N_\infty^{\vV},(\N_\infty^{\vV})^-}\rest\vV$ is
		definable over the universe of $\vV[g]$
		from the parameter
		$x=\vV|\lambda^{+\vV}$, uniformly in $\lambda$.\footnote{Regarding trees $\notin V$,
			cf.~Footnote \ref{ftn:trees_not_in_V}.}
	\end{enumerate}
	Moreover, the definability is uniform in $\vV,x$.
\end{lemma}

\begin{proof}
	By Lemma \ref{lem:vV_def_in_vV[g]_from_inseg},
	we can define $\vV$ from $x$ in $\vV[g]$, and uniformly so.
	To compute the $\delta_1$-short tree strategy
	(for $\N_\infty^{\vV}$) and determine $\delta_1$-maximality,
	use  Lemma \ref{lem:vV_1_full_short_tree_strategy} (recall this involves $*$-translation).
	The computation of branches at $\delta_1$-maximal stages is like in the proof of Lemma \ref{M_knows_how_to_iterate_Minfty_up_to_its_woodin},
	using Lemmas \ref{lem:vV_2_M_infty_of_iterate_N_etc} and \ref{lem:branch_con_2} (or arguing as in Footnote \ref{ftn:find_branch_via_normalization}
	in place of Lemma \ref{lem:branch_con_2}).
\end{proof}

By Lemma \ref{lem:Sigma_vV_1_vshc} and \cite[***Theorem 10.2]{fullnorm_v3}, $\Sigma_{\N^{\vV}_\infty}$ has minimal
inflation condensation.
So like in Remark \ref{rem:M_computes_stacks_strategy_for_M_infty}, it follows that $\vV[g]$ can also compute
the tail strategy $\Gamma_{\N_\infty^{\vV},(\N_\infty^{\vV})^-}$
for stacks on $\N_\infty^{\vV}$, based on $(\N_\infty^{\vV})^-$ (restricted to
stacks in $\vV[g]$), as in fact
\[ \Gamma_{\N_\infty^{\vV},(\N_\infty^{\vV})^-}=(\Sigma_{\N_\infty^{\vV},(\Sigma_{\infty}^{\vV})^-})^{\stk}.\]

Similarly:

\begin{lem}\label{lem:vV_2_iterates_N_infty_up_to_its_Woodin}
	$\vV_2$ is closed under $\Sigma_{\vV_2,\vV_2^-}$ and $\Sigma_{\vV_2,\vV_2^-}\rest\vV_2$ is lightface
	definable over $\vV_2$.
\end{lem}
\begin{proof}
	To compute the $\delta_0$-short and $\delta_1$-short tree strategies in $\vV_2$,
	proceed much
	as in the proof of Lemmas \ref{vV_2_M_knows_how_to_iterate_Minfty_up_to_its_woodin} and \ref{rem:vV_1_as_strategy_premouse}, naturally adapted to $\vV_2$. Since $\vV_2$ is a ground of $\vV_1$
	via $\mathbb{L}^{\vV_2}$ and because of the correspondence between
	$\es^{\vV_2}$, $\es^{\vV_1}$ and $\es^M$, we can perform the relevant
	P-constructions above $\gamma_1^{\vV_2}$ using $\es^{\vV_2}$ in the natural way.
	For $\delta_0$-maximal and $\delta_1$-maximal trees, we use the
	$0$-long and $1$-long extenders in $\es^{\vV_2}$ as usual.
\end{proof}

\subsection{2-Varsovian strategy premice}

\begin{dfn}
	For  a $\vV_1$-like $\vV$, we define the lightface $\vV$-classes $\N_\infty^{\vV}$, $*_1^{\vV}$, $\N_\infty[*_1]^{\vV}$ and $\vV_2^{\vV}$
	over $\vV$ just as the corresponding
	classes are defined over $\vV_1$.

	Also given a $\vV_1$-like $\vV$ and $\bar{\vV}\ins \vV$
	with $\kappa_1^{+\vV}\leq\OR^{\bar{\vV}}$,
	we define $\vV_2^{\bar{\vV}}$ by recursion on $\OR^{\bar{\vV}}$
	by setting $\vV_2^{\vV||(\kappa_1^{+\vV}+\alpha)}=\vV_2^\vV||(\gamma+\alpha)$,
	where $\gamma=\gamma_1^{\vV_2^{\vV}}$.
	Noting that this definition is level-by-level,
	we similarly define $\vV_2^{\bar{\vV}}(\kappa)$
	whenever $\bar{\vV}$ is a $\vV_1$-small Vsp such that $\gamma_0^{\bar{\vV}}$ exists
	and $\kappa$ is an inaccessible limit of $\delta_0^{\bar{\vV}}$-cutpoints
	of $\bar{\vV}$ and $\kappa<\OR^{\bar{\vV}}$,
	level-by-level (starting by defining
	$\vV_2^{\bar{\vV}|\kappa^{+\bar{\vV}}}$
	as $\vV_2||\gamma_1^{\vV_1}$ is defined (in the codes)
	over $\vV_1|\kappa_1^{+\vV_1}$). We will often suppress the $\kappa$
	from the notation, writing just $\vV_2^{\bar{N}}$.
\end{dfn}

We now want to axiomatize structures in the hierarchy of $\vV_2$ to some extent,
just like for $\vV_1$, adapting Definitions
\ref{dfn:local_vV_1}, \ref{dfn:Vsp}
and  \ref{dfn:vV_1-like}. These are very straightforward adaptations,
and the reader could fill it in him/herself, but because they are
reasonably detailed, we write them out for convenience:

\begin{dfn}\label{dfn:local_vV_2}
	A \emph{base 2-Vsp} is an amenable transitive structure $\vV=(P_\infty,F)$ such that
	in some forcing extension
	there is $P$ such that:
	\begin{enumerate}
		\item $P,P_\infty$ are 1-Vsps which model $\ZFC^-$, $\gamma_ 0^P<\OR^P$ and $\gamma_0^{P_\infty}<\OR^{P_\infty}$, and $P,P_\infty$ are $\vV_1$-small
		(that is, $P$ has no active segments
		satisfying ``There are $\delta_1',\kappa_1'$ such that $\gamma_0^P<\delta_1'<\kappa_1'$
		and $\delta_1'$ is Woodin and $\kappa_1'$ is strong'', and likewise for $P_\infty$).
		\item $P$ has a unique Woodin cardinal $\delta_1^P>\gamma_0^P$
		and a largest cardinal $\kappa_1^P>\delta_1^P$, and $\kappa_1^P$ is inaccessible in $P$
		and a limit of $\delta_0^P$-cutpoints of $P$; likewise for $P_\infty$,
		\item  $\OR^P=\delta_1^{P_\infty}$,
		$\kappa_1^P$ is the least measurable of $P_\infty$ and
		$\vV_2^{P}=\cHull_1^{\vV}(\delta_1^{P_\infty})$,
		\item $\N_\infty^{P_\infty}$ (defined over $P_\infty$
		like $\N_\infty|\gamma_1^{\vV_2}$ is defined over $\vV_1|\kappa_1^{+\vV_1}$) is  well-defined,
		and has least measurable $\kappa_1^{P_\infty}$ and second Woodin $\delta_1^{\N_\infty^{P_\infty}}=\OR^{P_\infty}$,
		\item $\N_\infty^{P_\infty}|\delta_1^{\N_\infty^{P_\infty}}$ is obtained
		by iterating $P_\infty|\delta_1^{P_\infty}$, via a short-normal tree $\Tt$ of length $\delta_1^{\N_\infty^{P_\infty}}$,
		\item $F$ is a cofinal $\Sigma_1$-elementary
		(hence fully elementary) embedding \[ F:P_\infty|\delta_1^{P_\infty}\to\N_\infty^{P_\infty}|\delta_1^{\N_\infty{P_\infty}},\]
		and there is a $\Tt$-cofinal branch $b$ such that
		$F\sub i^\Tt_b$, and $i^\Tt_b(\delta_1^{P_\infty})=\delta_1^{\N_\infty^{P_\infty}}$ (so
		$b$ is intercomputable with $F$, and note that by amenability of $\vV$, $F$ is amenable to $P_\infty$, and hence so is $b$),
		\item $\rho_1^{\vV}=\delta_1^{P_\infty}=\OR^P$ and $p_1^{\vV}=\emptyset$
		(so $\core_1(\vV)=\vV_2^P$)
		and  $\delta_1^{P_\infty}$ is Woodin in $\J(\core_1(\vV))$, as witnessed by $\es^{P_\infty}$.

		\item $P$ is $(\J(\core_1(\vV)),\mathbb{L}^{\vV})$-generic,
		where $\mathbb{L}^{\vV}$ is defined over $\vV$ as
		$\mathbb{L}$ above was defined over $\vV_2||\gamma_1^{\vV_2}$.
		\qedhere
	\end{enumerate}
\end{dfn}

Remark \ref{rem:base_Vsp} carries over directly.

\begin{dfn}
	A \emph{2-Varsovian strategy premouse (2-Vsp)} is a structure
	\[ \vV=(\J_\alpha^\es,\es,F) \]
	for some sequence $\es$ of extenders, where either $\vV$ is a premouse or a 1-Vsp, or:
	\begin{enumerate}
		\item $\alpha\leq\OR$ and $\vV$ is an amenable acceptable J-structure,
		\item $\vV$ has at least two Woodin cardinals, the least two of which
		are $\delta_0^{\vV}<\delta_1^{\vV}$,
		and has an initial segment $\vV||\gamma$ which is a  base 2-Vsp,
		\item $\delta_1^\vV<\gamma$, so $\delta_1^\vV$ is the second Woodin of $\vV||\gamma$,
		\item if $F\neq\emptyset$ and $\gamma<\OR^\vV$ then either:
		\begin{enumerate}
			\item  $\vV$ satisfies the premouse
			axioms (for Jensen indexing) with respect to $F$, and $\gamma<\crit(F)$, or
			\item\label{item:2-Vsp_long_0} $\vV$ satisfies
			the 1-Vsp axioms for a long extender, i.e. clause \ref{item:long}
			of Definition \ref{dfn:Vsp}, for giving an iteration map on (the premouse)  $\vV|\delta_0^{\vV}$, or
			\item\label{item:2-Vsp_long_1}
			\begin{enumerate}
				\item $\vV^\passive=(\J_\alpha^\es,\es,\emptyset)\sats\ZFC^-$,
				\item $\vV$ has largest cardinal $\mu$, which is inaccessible in $\vV$
				and a limit of $\delta_0^{\vV}$-cutpoints of $\vV$ (where \emph{$\delta_0^{\vV}$-cutpoint} applies to both short extenders
				and long extenders over $\vV|\delta_1^{\vV}$)
				\item   $\N=\N_\infty^{\vV^\passive}$ is a well-defined,
				and satisfies the axioms of a 1-Vsp with $\delta_1^{\N}$ existing (but $\N$ is possibly illfounded),
				and $\N$ is $(\OR^{\vV}+1)$-wellfounded, with $\delta_1^{\N}=\OR^{\vV}$,
				\item
				$\N|\delta_1^{\N}$  is a proper class
				of $\vV^\passive$ has least measurable
				$\mu$,  \item  $F$ is a cofinal $\Sigma_1$-elementary
				embedding $F:\vV|\delta_1^{\vV}\to \N|\delta_1^{\N}$,
				\item $\N|\delta_1^{\N}$ is pseudo-iterate of $\vV|\delta_1^{\vV}$,
				via short-normal tree $\Tt$, and there is a $\Tt$-cofinal branch $b$ such that
				$F\sub i^\Tt_b$ (hence $b$ is amenable to $\vV$
				and inter-definable with $F$ over $\vV^\passive$),
			\end{enumerate}
		\end{enumerate}
		\item each proper segment of $\vV$ is a sound 2-Vsp
		(defining \emph{2-Vsp} recursively),
		where the fine structural language for active segments
		is just that with symbols for $\in,\es,F$,
		\item some $p\in\mathbb{L}^\vV=\mathbb{L}^{\vV||\gamma}$ forces
		that the generic object is a 1-Vsp $P$ of height $\delta_1^{\vV}$
		with $\vV_2^P=\vV||\gamma$, and there is an extension $P$ to a 1-Vsp $P^+$
		such that $\vV_2^{P^+}=\vV$ (so $P^+$ is level-by-level definable over $\vV$,
		via inverse P-construction).
	\end{enumerate}

	We write $\gamma_1^{\vV}=\gamma$ above (if $\vV$ is not a 1-Vsp).
\end{dfn}

\begin{dfn}\label{dfn:vV_2-like}
	A  2-Vsp $\vV$ is \emph{$\vV_2$-like} iff it is proper class
	and in some set-generic extension,  $\vV=\vV_2^N$ for some $\Mswsw$-like premouse $N$. (Note this is first-order over $\vV$.)

	We write $\vV_2\downarrow 1=\N_\infty$ and $\vV_2\downarrow 0=\N_\infty\downarrow 0$.
	Let $\vV$ be  $\vV_2$-like.
	We define $\vV\downarrow 1$ and $\vV\downarrow 0$ analogously (first-order over $\vV$ as in the proof of
	Lemma \ref{lem:vV_2_N_infty[*]_inter-def} part \ref{item:N_infty[*]_class_of_vV_2}).
	In fact, let us define $\vV\downarrow i$ more generally, in the same first-order manner, but allowing  $\vV$ to be illfounded,
	but  $\vV_2$-like with respect to first-order properties.
	Also if $N$ is a 1-Vsp, let $N\downarrow 1=N$.
	We write  $\vV^-$ for (the 1-Vsp) $\vV|\delta_1^{\vV}$.
	We write $\Lambda^{\vV}_i$ for the putative strategy for $\vV\downarrow i$
	for trees based on $\vV|\delta_i^{\vV}$, defined over $\vV$
	just as the corresponding restrictions of $\Sigma_{\vV_2\downarrow i}$ are defined over $\vV_2$, via the proof of Lemma \ref{lem:vV_2_iterates_N_infty_up_to_its_Woodin}.

	We write
	$\vV_2=\vV(\N_\infty,*_1)=\vV(\N_\infty,*_1\rest\delta_{1\infty})=\vV(\N_\infty,e_1^{
		\vV_2})$.
	Given a pair
	$(N,*')$ or $(N,*'\rest\delta)$ or $(N,e)$ where
	$N$ is $\vV_1$-like and the  pair has
	similar first-order properties
	as does $(\N_\infty,*_1)$ or $(\N_\infty,*_1\rest\delta_{1\infty})$ or
	$(\N_\infty,e_1^{\vV_2})$ respectively,
	we define $\vV(N,*')$ or $\vV(N,*'\rest\delta)$ or $\vV(N,e)$ analogously
	(via the proof of Lemma \ref{lem:vV_2_N_infty[*]_inter-def} part
	\ref{item:vV_2_lightface_class_of_N_infty[*]}).
\end{dfn}

\subsection{Iterability of $\vV_2$}\label{subsec:iterability_vV_2}

In this subsection we will define a normal iteration strategy $\Sigma_{\vV_2}$
for $\vV_2$ in $V$.

\begin{dfn}
	For $i\leq 1$, an \emph{$i$-long} extender is a $(\delta_i^{\vV},\delta)$-extender
	over $\vV$, for some premouse, 1-Vsp or 2-Vsp $\vV$, and some $\delta$.
\end{dfn}

\begin{definition}
	Let $\vV$ be a $\vV_2$-like 2-Vsp. A \emph{$0$-maximal iteration tree $\Tt$ on $\vV$ of length $\lambda\geq 1$} is a system
	with the usual properties for $0$-maximality, except that when $E^\Tt_\alpha$ is a $i$-long extender, then
	then $\pred^\Tt(\alpha+1)$ is the least $\beta\leq\alpha$
	such that $[0,\beta]_\Tt$ does not drop and $\delta_i^{M^\Tt_\beta}<\lh(E^\Tt_\alpha)$.

	\emph{Iteration strategies} and \emph{iterability} for (such trees on) $\vV_2$ are defined in the obvious manner (one detail here is that if $[0,\alpha+1]_\Tt$ does not drop
	then $M^\Tt_{\alpha+1}$ is a (putative) 2-Vsp, including when
	$E^\Tt_\alpha$ is $0$-long).
\end{definition}

\begin{dfn}\label{dfn:vV_2_short-normal}
	A \emph{short-normal tree} on a $\vV_2$-like 2-Vsp $\vV$
	is a $0$-maximal tree that uses no long extenders.
	Note that a short-normal tree is of the form
	$\Tt_0\conc\Tt_1\conc\Ss$,
	where $\Tt_0$ is based on $\vV|\delta_0^{\vV}$,
	either
	\begin{enumerate}[label=\tu{(}\roman*\tu{)}]
		\item\label{item:vV_2_short-extender_i_0} [$\Tt_0$ has limit length or $b^{\Tt_0}$ drops]
		and $\Tt_1=\Ss=\emptyset$, or
		\item\label{item:vV_2_short-extender_ii_0}  $\Tt_0$ has successor length,
		$b^{\Tt_0}$ does not drop and $\Tt_1$ is above $\gamma_0^{M^{\Tt_0}_\infty}$
		and based on $M^{\Tt_0}_\infty|\delta_1^{M^{\Tt_0}_\infty}$,
	\end{enumerate}
	and if $\Tt_1\neq\emptyset$ then either
	\begin{enumerate}[label=\tu{(}\roman*\tu{)}]
		\item\label{item:vV_2_short-extender_i_1} [$\Tt_1$ has limit length or $b^{\Tt_1}$ drops]
		and $\Ss=\emptyset$, or
		\item\label{item:vV_2_short-extender_ii_1}  $\Tt_1$ has successor length,
		$b^{\Tt_1}$ does not drop and $\Ss$ is above $\gamma_1^{M^{\Tt_1}_\infty}$.
	\end{enumerate}
	Say that $\Tt_0\conc\Tt_1$ and $\Ss$ are the \emph{lower, upper components} respectively,
	and $\Tt_i$ the \emph{$i$-lower component}.
\end{dfn}

\subsubsection{Condensation properties for full normalization}

\begin{dfn}\label{dfn:vV_2_norm_condensation} We define the notions
	\emph{$(m+1)$-relevantly condensing}, \emph{$(m+1)$-sub-condensing}
	and \emph{$n$-standard} for 2-Vsps just as for 1-Vsps (see Definition \ref{dfn:norm_condensation}), replacing the role of premice there with 1-Vsps, and replacing $\delta_0^{\vV},\gamma_0^{\vV}$ with $\delta_1^{\vV},\gamma_1^{\vV}$.\end{dfn}

By Lemma \ref{lem:vV_1_norm_condensation} and its proof we have the following,
and Remark \ref{rem:iteration_preserves_standardness_at_degree} carries over directly:
\begin{lem}\label{lem:vV_2_norm_condensation} $\vV_2$ is $\om$-standard. (Thus, we take \emph{$\vV_2$-like}
	to include \emph{$\om$-standard}.)
\end{lem}

\subsubsection{Tree translation from $\vV_1$ to $\vV_2$}

\begin{definition} Let $\vV$ be $\vV_1$-like.
	We define \emph{1-translatable} trees $\Tt$ on $\vV$
	like in Definition \ref{dfn:translatable},
	with $0$ replaced by $1$ as appropriate,
	but add the demand that $\Tt$ uses no $0$-long extenders.

	Let $\Tt$ on $\vV$ be 1-translatable. The \emph{1-translation} of $\Tt$ is the tree on $\vV_2^{\vV}$  defined just as in Definition \ref{dfn:vV_1-translation}.
\end{definition}

Remarks \ref{rem:translatable}  and \ref{rem:vV_1_of_dropped_iterate} carry over directly, replacing $0$ with $1$ as appropriate.
Likewise  Lemma \ref{lem:vV_1-translatable}
and its proof:

\begin{lemma}\label{lem:vV_2-translatable}
	Let $\Tt$ on $\vV$ be $1$-translatable, where $\vV$ is $\vV_1$-like. Then:
	\begin{enumerate}
		\item The 1-translation $\Uu$ of $\Tt$
		exists and is unique.
		\item $M^\Uu_\alpha=\vV_2^{M^\Tt_\alpha}$ and $\gamma_1^{M^\Uu_\alpha}<\OR(M^\Uu_\alpha)$ for all $\alpha<\lh(\Tt)$.
		\item $i^\Uu_{\alpha\beta}=i^\Tt_{\alpha\beta}\rest M^\Uu_\alpha$ for all $\alpha<_\Tt\beta$ such that $(\alpha,\beta]_\Tt$ does not drop.
		\item $M^{*\Uu}_{\alpha+1}=\vV_2^{M^{*\Tt}_{\alpha+1}}$ for all $\alpha+1<\lh(\Tt)$.
		\item $i^{*\Uu}_{\alpha+1}=i^{*\Tt}_{\alpha+1}\rest M^{*\Uu}_{\alpha+1}$ for all $\alpha+1<\lh(\Tt)$.
	\end{enumerate}
\end{lemma}

\subsubsection{Trees based on $\N_\infty|\delta_1^{\N_\infty}$}

Toward defining
$\Sigma_{\vV_2}$, we first consider trees on $\vV_2$ based on
$\vV_2^-  =\N_\infty|\delta_1^{\N_\infty}$,
adapting Definition \ref{dfn:Psi_vV_1,vV_1^-}:

\begin{dfn}
	Write $\Sigma^{\sn}_{\N_\infty,\vV_2^-}$
	for the  strategy for $\N_\infty$ for short-normal trees based on $\vV_2^-$, induced
	by $\Sigma_{\N_\infty}$.
	Let $\Psi_{\vV_2,\vV_2^-}$ denote the putative  strategy for
	short-normal trees on $\vV_2$ based on $\vV_2^-$,
	induced by $\Sigma^{\sn}_{\N_\infty,\vV_2^-}$.
	This makes sense
	by Lemma \ref{lem:vV_1_still-a-woodin-in-v}.
\end{dfn}

Remark \ref{rem:basic_trees_via_Psi_vV_1,vV_1^-} adapts routinely.
We now partially adapt Lemma \ref{lem:Gamma_A_is_good}, but omit the clause
``and in fact, $\Lambda^{M^\Uu_\alpha}\sub\Sigma_{M^\Tt_\alpha}$'',
as we will prove this in more generality later, in Lemma \ref{lem:vV_2_Psi^sn_good}.
The proof of the rest is a direct adaptation:

\begin{lem}\label{lem:vV_2_Gamma_A_is_good}
	$\Psi_{\vV_2,\vV_2^-}$ yields wellfounded models.
	Moreover, let $\Tt$ be on $\N_\infty$,
	via $\Sigma^{\sn}_{\N_\infty,\vV_2^-}$,
	and let $\Uu$ be the corresponding tree on $\vV_2$ (so via
	$\Psi_{\vV_2,\vV_2^-}$).
	Let
	\[ \pi_\alpha:M^\Tt_\alpha\to M^\Uu_\alpha\downarrow 1\sub M^\Uu_\alpha \]
	be the natural copy map
	\tu{(}where $\pi_0=\id$\tu{)}. Then:
	\begin{enumerate}[label=(\roman*)]
		\item $[0,\alpha]_\Tt$ drops iff $[0,\alpha]_\Uu$ drops.
		\item If $[0,\alpha]_\Tt$ drops then $M^\Tt_\alpha=M^\Uu_\alpha=M^\Uu_\alpha\downarrow 1$ (cf.~Remark \ref{rem:basic_trees_via_Psi_vV_1,vV_1^-} adapted ).
		\item\label{item:M^U_alpha=V(M^T_alpha,l)} If $[0,\alpha]_\Tt$ does not drop then $M^\Tt_\alpha=M^\Uu_\alpha\downarrow 1$ and $M^\Uu_\alpha=\vV(M^\Tt_\alpha,\ell)$ where $\ell:M^\Tt_\alpha\to\N_\infty^{M^\Tt_\alpha}$
		is the correct iteration map,
		\item $\pi_\alpha=\id$; therefore, $i^\Tt_\alpha\sub i^\Uu_\alpha$.
	\end{enumerate}
\end{lem}

\begin{dfn}\label{dfn:vV_2_Psi_V,V^-}
	Given a non-dropping $\Psi_{\vV_2,\vV_2^-}$-iterate
	$\vV$ of $\vV_2$, let $\Psi_{\vV,\vV^-}$ be induced by $\Sigma_{\vV\downarrow 1}$
	just as $\Psi_{\vV_2,\vV_2^-}$ is induced by $\Sigma_{\N_\infty}$
	(this makes sense by Lemma \ref{lem:vV_2_Gamma_A_is_good}).
\end{dfn}

\subsubsection{Short-normal trees on $\vV_2$}\label{subsec:short-normal_trees_on_vV_2}

\begin{dfn}
	Let $\vV$ be a (possibly dropping, putative) iterate of $\vV_2$,
	via a short-normal
	tree $\Tt\conc\Ss$ with lower and upper components $\Tt,\Ss$.
	We say that $\vV$ is \emph{good} iff $\Tt$ is via $\Psi_{\vV_2,\vV_2^-}$, $\vV$ is wellfounded and for every $i$-long $E\in\es_+^\vV$,
	$\M_{i\infty}^{\vV|\lh(E)}=P|\delta_i^P$ for some $\Sigma_{\vV\downarrow i}$-iterate
	$P$ of $\vV\downarrow i$, and $E$ is the corresponding iteration map.

	Say that a (partial) iteration strategy $\Psi$ for $\vV_2$
	is \emph{good} iff all putative iterates via $\Psi$ are good.
\end{dfn}

We now extend $\Psi_{\vV_2,\vV_2^-}$ to a good short-normal $0$-maximal strategy
$\Psi_{\sn}$ for $\vV_2$.
We first deal with trees based on $\vV||\gamma_1^{\vV}$:

\begin{dfn}\label{dfn:vV_2_Psi_for_vV_1|gamma}
	Write $\Psi_{\vV_2,\gamma_1^{\vV_2}}$ for the putative strategy $\Psi$ for $\vV_2$,
	for short-normal $0$-maximal trees based on $\vV_2||\gamma_1^{\vV_2}$, as follows:
	\begin{enumerate}
		\item $\Psi_{\vV_2,\vV_2^-}\sub\Psi$, and
		\item given $\Tt$ via $\Psi_{\vV_2,\vV_2^-}$, of successor length $\alpha+1$,
		where $[0,\alpha]_\Tt$ does not drop, and given a putative $0$-maximal tree $\Uu$ on $M^\Tt_\alpha||\gamma_1^{M^\Tt_\alpha}$, which is above $\delta_1^{M^\Tt_\alpha}$,
		then $\Tt\conc\Uu$ is (equivalent to a tree) via $\Psi$
		iff there is a tree $\Uu'$ on ${M^\Tt_\alpha\downarrow 1}$,
		via $\Sigma_{M^\Tt_\alpha\downarrow 1}$, with the same extenders and tree order as $\Uu$.\qedhere
	\end{enumerate}
\end{dfn}

We adapt Lemma \ref{lem:Psi_vV_1,gamma^vV_1_good}:
\begin{lem}
	$\Psi_{\vV_2,\gamma_1^{\vV_2}}$ is a short-normal $0$-maximal strategy (hence yields wellfounded models). Moreover, let
	$\Tt\conc\Uu$ and $\Uu'$ be as in Definition \ref{dfn:vV_2_Psi_for_vV_1|gamma}, with $\Uu\neq\emptyset$.
	Then:
	\begin{enumerate}
		\item $M^\Uu_0=M^\Tt_\alpha||\gamma_1^{M^\Tt_\alpha}$ and $\deg^\Uu_0=0$,
		\item $M^{\Uu'}_0={M^\Tt_\alpha\downarrow 1}$ and $\deg^{\Uu'}_0=0$,
		so $(M^\Uu_0)^\passive=M^{\Uu'}_0|\kappa_1^{+M^{\Uu'}_0}$,
		\item for $0<\beta<\lh(\Uu)$, $\beta\in\dropset_{\deg}^\Uu\Leftrightarrow\beta\in\dropset_{\deg}^{\Uu'}$,
		and $\deg^\Uu_\beta=\deg^{\Uu'}_\beta$,
		\item if $0<\beta<\lh(\Uu)$ and $[0,\beta]_\Uu$ drops then $M^\Uu_\beta=M^\Uu_{\beta'}$,
		\item if $0<\beta<\lh(\Uu)$ and $[0,\beta]_\Uu$ does not drop then
		$(M^\Uu_\beta)^\passive=M^{\Uu'}_\beta|\kappa_1^{+M^{\Uu'}_\beta}$,
		\item if $0<\beta+1<\lh(\Uu)$ and $[0,\beta+1]_\Uu$ drops then
		$M^{*\Uu}_{\beta+1}=M^{*\Uu'}_{\beta+1}$ and $i^{*\Uu}_{\beta+1}=i^{*\Uu'}_{\beta+1}$,
		\item if $0<\beta+1<\lh(\Uu)$ and $[0,\beta+1]_{\Uu'}$ does not drop
		then $i^{*\Uu}_{\beta+1}\sub i^{*\Uu'}_{\beta+1}$,
		\item\label{item:vV_2_Ult(M^T_alpha_down_0,F)} if $0\leq\beta<\lh(\Uu)$ and $[0,\beta]_\Uu$ does not drop
		then $M^{\Uu'}_\beta$ is a ($\kappa_1^{M^{\Uu'}_\beta}$-sound) $\Sigma_{M^\Tt_\alpha\downarrow 1}$-iterate of ${M^\Tt_\alpha\downarrow 1}$,  $\N_\infty^{M^{\Uu'}_\beta}$
		is a $\delta_1^{\N_\infty^{M^{\Uu'}_\beta}}$-sound $\Sigma_{M^\Tt_\alpha\downarrow 1}$-iterate
		of ${M^\Tt_\alpha\downarrow 1}$,
		\[ \Ult(M^\Tt_\alpha\downarrow 1,F(M^\Tt_\beta))=\N_\infty^{M^{\Uu'}_\beta} \]
		and $F(M^\Tt_\beta)$ is the extender of the $\Sigma_{M^\Tt_\alpha\downarrow 1}$-iteration map.
	\end{enumerate}
	Therefore if $\Tt,\Uu$ each have successor length, then $M^{\Tt\conc\Uu}_\infty$ is good with respect to extenders
	indexed $\leq\gamma_1^{M^{\Tt\conc\Uu}_\infty}$ (or all extenders
	in $\es_+(M^{\Tt\conc\Uu}_\infty)$, if  $b^{\Tt\conc\Uu}$ drops).
\end{lem}
\begin{proof}
	The (last) ``therefore'' clause is because $\Sigma_{\vV_1}$ is good.
	The rest of the proof is like for Lemma \ref{lem:Psi_vV_1,gamma^vV_1_good}
	(although we did not yet prove that $\Psi_{\vV_2,\vV_2^-}$ is good,
	Lemma \ref{lem:vV_2_Gamma_A_is_good}\ref{item:M^U_alpha=V(M^T_alpha,l)} does give the instance of this with respect to $F^{\vV||\gamma_1^{\vV}}$, where $\vV=M^\Tt_\alpha$,
	which is enough to prove part \ref{item:vV_2_Ult(M^T_alpha_down_0,F)}
	as in Lemma \ref{lem:Psi_vV_1,gamma^vV_1_good}).
\end{proof}

We now prove a couple of variants of the branch condensation lemma \ref{lem:branch_con_2}
for trees on $\vV_1$:
\begin{lem}\label{lem:vV_1_branch_con_2}
	Let $\Tt,\Uu$ be short-normal on $\vV_1$, via $\Sigma_{\vV_1}$,
	based on $\vV_1|\delta_1^{\vV_1}$, with $\Tt$ of limit length, $\Uu$ successor length
	with $b^\Uu$ non-dropping and $\delta(\Uu)=\delta_1^{M^\Uu_\infty}$. Let $G$ be $V$-generic.
	Let $b,k\in V[G]$ where $b$ is a non-dropping $\Tt$-cofinal branch
	with $i^\Tt_b(\delta_1^{\vV_1})=\delta(\Tt)$ and
	and
	\[ k:M^\Tt_b|\delta_1^{M^\Tt_b}\to M^\Uu_\infty|\delta_1^{M^\Uu_\infty} \]
	is elementary with $k\com i^\Tt_b=i^\Uu_{0\infty}\rest(\vV_1|\delta_1^{\vV_1})$.
	Then $b=\Sigma_{\vV_1}(\Tt)$.
\end{lem}
\begin{proof}
	We may assume $b,k\in V$. 	Let $\vV=M^\Tt_b$.
	Let $\alpha\in b$ be least with either $\alpha+1=\lh(\Tt)$ or $\delta_0^{M^\Tt_\alpha}<\crit(i^\Tt_{\alpha b})$. So $\bar{\vV}=M^\Tt_\alpha$
	is $\delta_0^{\bar{\vV}}$-sound and $\Tt\rest[\alpha,\infty)$
	is on $\bar{\vV}$, above $\delta_0^{\bar{\vV}}$.
	Let $\beta\in b^\Uu$ be analogous for $\Uu$.

	Define $\pi:\vV\to M^\Uu_\infty$ in the natural way, extending $k$,
	like in Lemma \ref{lem:branch_con_2}. Also define $\bar{k}:\bar{\vV}\to M^\Uu_\beta$
	analogously. So $\bar{k}\rest\delta_0^{\vV}\sub k$.

	The phalanx $\ph=((\bar{\vV},\delta_0^{\bar{\vV}}),\vV,\delta(\Tt))$ is iterable, via lifting trees with $(\bar{k},k)$.

	Let $c=\Sigma_{\vV_1}(\Tt)$ and $Q_c=Q(\Tt,c)$ be the Q-structure,
	or $Q_c=M^\Tt_c$ if $i^\Tt_c(\delta_1^{\vV_1})=\delta(\Tt)$.
	Let $\mathfrak{Q}$ be the phalanx $((\bar{\vV},\delta_0^{\bar{\vV}}),Q_c,\delta(\Tt))$,
	which is also iterable.

	Let $\bar{\vV}^+$ be a generic expansion of $M^\Tt_\alpha$ (so $\bar{\vV}^+$ is an $\Mswsw$-like premouse and $\bar{\vV}=\vV_1^{\bar{\vV}^+}$).
	Then $\Tt\rest[\alpha,\infty)$ can be translated to a tree $\Tt^+$ on $\bar{\vV}^+$,
	which is above $\kappa_0^{+\bar{\vV}^+}$. Let $\vV^+=M^{\Tt^+}_b$ and $Q_c^+=Q(\Tt^+,c)$ or $Q_c^+=M^{\Tt^+}_c$ accordingly. Then the phalanxes
	\[ \ph^+=((\bar{\vV}^+,\kappa_0^{\bar{\vV}^+}+1),\vV^+,\delta(\Tt))\text{ and }
	\mathfrak{Q}^+=((\bar{\vV}^+,\kappa_0^{\bar{\vV}^+}+1),Q_c^+,\delta(\Tt)) \]
	are iterable, since trees on them translate to trees on $\ph,\mathfrak{Q}$.
	Note that $\bar{\vV}^+$ is $\kappa_0^{\bar{\vV}^+}$-sound,
	and $\vV^+,Q^+_c$ are $\delta(\Tt)$-sound.
	But then comparing $\ph^+$ versus $\mathfrak{Q}^+$ gives $b=c$.
\end{proof}

\begin{lem}\label{lem:vV_1_branch_con_Q}
	Let $\Tt,\Uu$ be short-normal dsr trees on $\vV_1$, via $\Sigma_{\vV_1}$,
	based on $\vV_1|\delta_1^{\vV_1}$, with $\Tt,\Uu$ of limit length. Suppose there is $\alpha<\lh(\Tt)$
	such that $[0,\alpha]_\Tt$ does not drop and $\Tt\rest[\alpha,\infty)$
	is above $\delta_0^{M^\Tt_\alpha}$, and there is an analogous such $\beta<\lh(\Uu)$;
	fix the least such $\alpha,\beta$.
	Let $c=\Sigma_{\vV_1}(\Uu)$.
	Suppose that if $c$ is non-dropping then $\delta(\Uu)<\delta_1^{M^\Uu_c}$. Let $G$ be $V$-generic.
	Let $b\in V[G]$ where $b$ is a $\Tt$-cofinal branch such that if $b$ is non-dropping
	then $\delta(\Tt)<\delta_1^{M^\Tt_b}$. Let $k\in V[G]$
	be such that
	\[ k:Q(\Tt,b)\to Q(\Uu,c) \]
	is elementary and $k\com i^\Tt_b\rest(\vV_1|\delta_0^{\vV_1})=i^\Uu_{c}\rest(\vV_1|\delta_0^{\vV_1})$.
	Then $b=\Sigma_{\vV_1}(\Tt)$.
\end{lem}
\begin{proof}
	This is via a straightforward variant of the proof of Lemma \ref{lem:vV_1_branch_con_2},
	noting that because $\Tt$ is dsr, $Q(\Tt,b)$ can only overlap $\delta(\Tt)$ with long extenders.
\end{proof}
It turns out that the method we used to define $\Psi_{\vV_1}^{\sn}$ is not so well suited to $\vV_2$. Instead we proceed as follows:

\begin{dfn}\label{dfn:Psi_vV_2^sn}
	$\Psi_{\vV_2}^{\sn}$ denotes the (putative) short-normal strategy $\Psi$ for $\vV_2$, defined as follows.
	Firstly, $\Psi_{\vV_2,\gamma_1^{\vV_2}}\sub\Psi$.
	Secondly, let $\Tt$ be via $\Psi_{\vV_2,\vV_2^-}$, of successor length, such that $b^\Tt$ does not drop, and $\vV=M^\Tt_\infty$. We define the action of $\Psi$
	on above-$\gamma_1^{\vV}$ trees on $\vV$.

	Let $\P=\vV\downarrow 1=i_{\vV_2\vV}(\N_\infty)$,
	so (by Lemma \ref{lem:vV_2_Gamma_A_is_good}) $\P$ is a $\Sigma_{\N_\infty}$-iterate of $\N_\infty$ and $\P\sub\vV$.
	Let $e=F^{\vV||\gamma_1^{\vV}}$.
	Let (and by Lemma \ref{lem:vV_2_N_infty[*]_inter-def})
	\[ i_e^{\vV}:\vV\to\Ult(\vV,e)=\vV_2^{\P}\]
	be the ultrapower map.	Let $\Lambda$ be the above-$\gamma_1^{\vV_2^{\P}}$ strategy
	for $\vV_2^{\P}$ determined by translating above-$\kappa_1^{+\P}$ (hence 1-translatable) trees on $\P$ via $\Sigma_{\P}$.
	Then for above-$\gamma_1^{\vV}$ trees $\Uu$ on $\vV$,
	$\Tt\conc\Uu$ is via $\Psi$ iff $\Uu$ is via the minimal $i^{\vV}_e$-pullback
	of $\Gamma$.
\end{dfn}

\begin{lem}\label{lem:vV_2_Psi^sn_good}
	$\Psi_{\vV_2}^{\sn}$ is good.
\end{lem}
\begin{proof}
	Clearly $\Psi_{\vV_2}^{\sn}$ is well-defined and yields wellfounded models.
	So let $\Tt\conc\Uu$ be via $\Psi_{\vV_2}^{\sn}$, as in Definition \ref{dfn:Psi_vV_2^sn},
	and $\vV$, $\P$, $e$, $i_e=i^{\vV}_e$ be as there.
	Let $i_e``\Uu$ be the minimal $i_e=i^{\vV}_e$-copy
	of $\Uu$ to a tree on $\vV_2^{\P}$.
	Let $\wW=M^{\Tt\conc\Uu}_\infty$ and $\gamma_1^{\wW}\leq\alpha\leq\OR^{\wW}$ be such that
	$\wW||\alpha$ is active with an $i$-long extender. Let $\wW'=\Ult_0(\wW||\alpha,e)$.
	By Lemma \ref{lem:vV_2_norm_condensation}, $\wW'\ins M^{i_e``\Uu}_\infty$.

	Now $F^{\wW'}$ is a correct iteration extender (via $\Sigma_{\N_\infty^{\P}}$)
	based on $\N_\infty^{\P}|\delta_i^{\N_\infty^\P}$.
	For let ${\Uu'}^+$ be the translation of $i_e``\Uu$ to a tree on $\P$, and ${\wW'}^+=M^{{\Uu'}^+}_\infty||\OR^{\wW'}$.
	If $i=1$ (so $\crit(F^{{\wW'}^+})=\kappa_1^\P$) the correctness of $F^{\wW'}$ is by Lemma \ref{lem:vV_2_M_infty_of_iterate_N_etc} part \ref{item:vV_2_iterate_of_M_infty} applied to $\P$ and $\P'=\Ult(\P,F^{{\wW'}^+})$
	and $i_{\P\P'}$. If $i=0$  (so $F^{{\wW'}^+}$ is long)
	it is because $\Sigma_{\vV_1}$ is good (so $F^{{\wW'}^+}$ is correct)
	and how $\vV_2^{\P}$ is defined.

	Let $\Ss$ be the limit length tree leading from $\P|\delta_i^{\P}$ to $\M_{i\infty}^{\wW||\alpha}$, and $\Ss'$ likewise, so $\Ss'=j(\Ss)$
	where $j=i_e^{\wW||\alpha,0}$ is the ultrapower map.
	Let $\Rr$ be the successor length tree leading from $\P|\delta_i^{\P}$
	to $\N_\infty^{\P}|\delta_i^{\P}$ (given by $e\rest(\P|\delta_i^\P)$).
	We know that	 $\Ss'$ is via the short tree strategy for $\N_\infty^{\P}|\delta_i^{\N_\infty^\P}$,
	and $F^{\wW'}$ yields the branch $\Sigma_{\N_\infty^\P}(\Ss')$.
	We claim the same holds for $\Ss$ and $F^{\wW||\alpha}$; that is, $\Ss$ is via the $\delta_i$-short tree strategy
	for $\P\downarrow i$, and
	$F^{\wW||\alpha}$ yields $b=\Sigma_{\P}(\Ss)$.

	For if $i=0$, the Q-structure
	used in $\Ss$ for the limit stage $\Ss\rest\eta$ does not overlap $\delta(\Ss\rest\eta)$, and is embedded by $j$ into an iterable Q-structure
	used in $\Ss'$. And if $i=1$,
	it is likewise through the $0$-lower component of $\Ss$
	(until reaching $\delta_0^{M(\Ss)}$),
	and above there, Lemma \ref{lem:vV_1_branch_con_Q} applies to the normalizations
	of $(\Tt,\Ss\rest\eta)$
	and $(\Tt,\Rr,j(\Ss\rest\eta))$,
	using a restriction of $j$ as the map $k$. (Here $\Tt$ and $\Tt\conc\Rr$
	can be $\delta_1$-maximal, but one should literally apply Lemma \ref{lem:vV_1_branch_con_Q}
	to the short-normal trees $\Tt'$ on $\vV_1$, iterating to $M(\Ss\rest\eta)$,
	and $\Uu'$ on $\vV_1$, iterating to $M(j(\Ss\rest\eta))$).

	Finally let $b$ be the $\Ss$-cofinal branch
	determined by $F^{\wW||\alpha}$ and $b'$ that determined by $F^{\wW'}$.
	Then we can apply
	Lemma \ref{lem:vV_1_branch_con_2} to $(\Ss,b)$ and the stack $(\Rr,(\Ss',b'))$, using
	$k=j\rest(M(\Ss))$. Therefore $b$ is correct.
\end{proof}

\subsubsection{Normal trees on $\vV_2$}

Much like in Definition \ref{dfn:Sigma_vV_1}, it is now easy to see:

\begin{lem}\label{lem:vV_2_Sigma_vV_1}
	There is a unique $0$-maximal strategy $\Sigma$ for $\vV_2$ such that $\Psi_{\vV_2}^{\sn}\sub\Sigma$. We write $\Sigma_{\vV_2}=\Sigma$.
	Every iterate of $\vV_2$ via $\Sigma_{\vV_2}$ is a short-normal iterate of $\vV_2$ via $\Psi_{\vV_2}^{\sn}$, and hence $\Sigma_{\vV_2}$ is good.
\end{lem}

\begin{rem}
	Consider a  $0$-maximal tree $\Tt$ on $\vV_2$ and some limit $\lambda<\lh(\Tt)$
	such that $\Tt$ uses $1$-long extenders cofinally below $\lambda$. Then $\delta(\Tt)$ is the least measurable of $M^\Tt_\lambda$, and in particular $\delta(\Tt)<\delta_0^{M^\Tt_\lambda}$. Suppose $E^\Tt_\lambda$ is short
	with $\crit(E^\Tt_\lambda)<\delta_1^{M^\Tt_\lambda}$ and total over $M^\Tt_\lambda$,
	or $E^\Tt_\lambda$ is $0$-long. Then $\pred^\Tt(\lambda+1)=\lambda$ and $M^\Tt_{\lambda+1}=\Ult(M^\Tt_\lambda,E^\Tt_\lambda)$,
	and note that the short-normal tree $\Uu$ via $\Psi^{\sn}_{\vV_2}$ has $\delta(\Uu)>\lh(E^\Tt_\lambda)$. This could be unnatural; letting $\Uu_\lambda$
	be the short-normal tree with last model $M^\Tt_\lambda$, it might be better
	to define $0$-maximality by taking $\beta<\lh(\Uu_\lambda)$ least such that
	$E^\Tt_\lambda\in\es_+^{M^{\Uu_\lambda}_\beta}$, and defining
	$M^\Tt_{\lambda+1}$ to be the model produce by normally extending $\Uu_\lambda\rest(\beta+1)$ with $E^\Tt_\lambda$. However,
	for our purposes here, the more naive notion of $0$-maximality suffices.
\end{rem}

\begin{lem}\label{lem:Psi_2_pullback-con}
	Let $\bar{\vV}$ be a non-dropping
	$\Psi_{\vV_2,\vV_2^-}$-iterate of $\vV_2$,
	and $\vV$ a non-dropping $\Psi_{\bar{\vV},\bar{\vV}^-}$-iterate of $\bar{\vV}^-$. Let $\pi:\bar{\vV}\to \vV$ be the iteration map.
	Let $\bar{\Psi}$ be the above-$\delta_1^{\bar{\vV}}$ short-normal strategy
	for $\bar{\vV}$ given by $\Sigma_{\vV_2}$, and $\Psi$ likewise for $\vV$. Then $\bar{\Psi}$ is the minimal $\pi$-pullback of $\Psi$ (see \cite[***10.3, 10.4]{fullnorm_v3}).
	\end{lem}
\begin{proof}
	Let $\Tt_1$ on $\bar{\vV}$ be via the
	$\pi$-pullback of $\Psi$; we want to
	see that $\Tt_1$ is via $\bar{\Psi}$.
	Let $\Xx_1=\pi``\Tt_1$, which is via $\Psi$.  Let $\pi'=\pi\rest(\bar{\vV}\downarrow 1)$. So $\pi'= i_{\bar{\vV}\downarrow 1,\vV\downarrow 1}$, also an iteration map.

	If $\lh(E^{\bar{\Tt_1}}_0)<\gamma_1^{\bar{\vV}}$ then the desired conclusion follows
	from the fact that $\Sigma_{\N_\infty}$
	has mic. So suppose otherwise.

	Let $\bar{j}:\bar{\vV}\to\Ult(\bar{\vV},e_1^{\bar{\vV}})$ and
	$j:\vV\to\Ult(\vV,e_1^{\vV})$ be the ultrapower maps, and recall
	$\Ult(\vV,e^{\vV_1})=\vV_2^{\vV\downarrow 1}$ and likewise for $\bar{\vV}$.
	So the minimal $j$-copy $j``\Xx_1$ of $\Xx_1$, on $\vV_2^{\vV\downarrow 1}$,
	translates to a tree $\Xx_1'$ on $\vV\downarrow 1$ which is via $\Sigma_{\vV\downarrow 1}$ (and above $\gamma_1^{\vV}=\kappa_1^{+\vV\downarrow 1}$). We need to see that the minimal $\bar{j}$-copy of $\Tt_1$ translates  to a tree $\Tt_1'$
	on $\bar{\vV}\downarrow 1$ via $\Sigma_{\bar{\vV}\downarrow 1}$.
	 Since $\Sigma_{\bar{\vV}\downarrow 1}$ has mic (Lemma \ref{lem:Sigma_vV_1_vshc} and \cite[***Theorem 10.2]{fullnorm_v3})
and by \cite[***10.3, 10.4]{fullnorm_v3}, it therefore suffices to see that
that $\Xx_1'=\pi'``\Tt_1'$.

Let $\sigma'=\pi'\rest(\bar{\vV}|\gamma_1^{\bar{\vV}})$
and
\[ \sigma=\sigma'\rest(\vV_2^{\bar{\vV}\downarrow 1}|\delta_1^{\vV_2^{\bar{\vV}\downarrow 1}})=\pi\rest\Ult(\bar{\vV}|\delta_1^{\bar{\vV}},e_1^{\bar{\vV}}).\]
We have
\[ \Ult(\bar{\vV}\downarrow 1,\sigma')=\Ult(\bar{\vV}\downarrow 1,\sigma'\rest\delta_1^{\bar{\vV}})=\vV\downarrow 1 \]
and the associated ultrapower map is just
$\pi\rest(\bar{\vV}\downarrow 1)$. Given
the fine structural correspondence between
$\vV_2^{\bar{\vV}\downarrow 1}$ and $\bar{\vV}\downarrow 1$, therefore \[ \Ult(\vV_2^{\bar{\vV}\downarrow 1},\sigma)=\vV_2^{\vV\downarrow 1} \]
and the $\sigma$-ultrapower map $\vV_2^{\bar{\vV}\downarrow 1}\to\vV_2^{\vV\downarrow 1}$
is just $\pi\rest\vV_2^{\bar{\vV}\downarrow 1}$.  Although $\sigma$ is not the restriction of an iteration map on $\vV_2^{\bar{\vV}\downarrow 1}$, it is straightforward to see we still have
$\widetilde{\Xx_1}=\sigma``\widetilde{\Tt_1}$ (that is, $\widetilde{\Xx_1}$ is the minimal $\sigma$-copy of $\widetilde{\Tt_1}$), meaning that:
\begin{enumerate}[label=--]
	\item  $\widetilde{\Xx_1}$ has the same tree, drop and degree structure as has $\widetilde{\Tt_1}$,
	\item for each $\alpha+1<\lh(\widetilde{\Tt_1})$,
	we have    $M^{\widetilde{\Xx_1}}_\alpha||\lh(E^{\widetilde{\Xx_1}}_\alpha)=\Ult_0(M^{\widetilde{\Tt_1}}_\alpha||\lh(E^{\widetilde{\Tt_1}}_\alpha),\sigma)$,
	\item for each $\alpha<\lh(\widetilde{\Tt_1})$, if $d=\deg^{\Tt_1}_\alpha$
	then $M^{\widetilde{\Xx_1}}_\alpha=\Ult_d(M^{\widetilde{\Tt_1}}_\alpha,\sigma)$,
	and if $\alpha$ is a successor
	then $M^{*\widetilde{\Xx_1}}_\alpha=\Ult_d(M^{*\widetilde{\Tt_1}}_\alpha,\sigma)$,
	and
	\item the resulting ultrapower maps $M^{\widetilde{\Tt_1}}_\alpha\to M^{\widetilde{\Xx_1}}_\alpha$ and $M^{*\widetilde{\Tt_1}}_\alpha\to M^{*\widetilde{\Xx_1}}_\alpha$ (via $\sigma$) commute with the iteration maps
	of $\widetilde{\Tt_1}$ and $\widetilde{\Xx_1}$.
\end{enumerate}
These are just standard properties of minimal copying,
so we already know the corresponding properties hold with respect to $(\pi,\Tt_1,\Xx_1)$,
$(\bar{j},\Tt_1,\widetilde{\Tt_1})$, and $(j,\Xx_1,\widetilde{\Xx_1})$.
One can now deduce them for $(\sigma,\widetilde{\Tt_1},\widetilde{\Xx_1})$
with some commutativity, and in particular
that \[ j\com\pi\rest(\bar{\vV}|\delta_1^{\bar{\vV}})=
\pi\rest(\vV_2^{\bar{\vV}\downarrow 1})\com\bar{j}.\]

But then because $\Tt_1',\Xx_1'$ are translations of $\widetilde{\Tt_1},\widetilde{\Xx_1}$,
and given the fine structural correspondence between $\bar{\vV}\downarrow 1$ and $\vV_2^{\bar{\vV}\downarrow 1}$, and likewise between $\vV\downarrow 1$ and $\vV_2^{\vV\downarrow 1}$,
it follows that $\Xx_1'=\pi'``\Tt_1'$, as desired.
	\end{proof}
\begin{lem}\label{lem:vV_2_Sigma_vV_1_vshc} $\Sigma_{\vV_2}$ has minimal inflation condensation (mic).
\end{lem}
\begin{proof}
We just discuss short-normal trees.
		Let $\Tt=\Tt_0\conc\Tt_1$ and $\Xx=\Xx_0\conc\Xx_1$ be as before,
		but with respect to $\vV_2$ and $\Sigma_{\vV_2}$; in particular we have
		\[ \Tt_0\conc\Tt_1\inflatearrow_{\min}\Xx_0\conc(\Xx_1\rest\lambda), \]
		where $\lambda$ is a limit ordinal
		and $\lambda+1=\lh(\Xx_1)$. We must show that $\Xx$ is a minimal inflation of $\Tt$.
	Now $\Psi_{\vV_2,\vV_2^-}$ has mic,
	since $\Sigma_{\vV_1}$ does, by
	Lemma \ref{lem:Sigma_vV_1_vshc} and \cite[***Theorem 10.2]{fullnorm_v3}.
	So we may assume $\Tt_1\neq\emptyset$,
	so we get $\alpha,\beta,\Pi_0,\pi$ like before,
	with analogous properties (with $\delta_1$ replacing $\delta_0$).
	Let $\eta<\lh(\Tt_1)$ be the limit ordinal and $c$ the $\Tt_1\rest\eta$-cofinal branch and
	\[ \Pi:\Tt_0\conc(\Tt_1\rest\eta)\conc c\hookrightarrow_{\min}\Xx \]
	the minimal tree embedding determined by extending the inflation $\Tt\inflatearrow_{\min}\Xx_0\conc(\Xx_1\rest\lambda)$
	to $\Xx$ in the unique possible way.
	We want $c=[0,\eta)_{\Tt_1}$.

Let $\bar{\vV}=M^{\Tt_0}_\alpha$, $\vV=M^{\Xx_0}_\beta$ and $\bar{j}:\bar{\vV}\to\Ult(\bar{\vV},e_1^{\bar{\vV}})$ and
	$j:\vV\to\Ult(\vV,e_1^{\vV})$ be the ultrapower maps, and recall
	$\Ult(\vV,e^{\vV_1})=\vV_2^{\vV\downarrow 1}$ and likewise for $\bar{\vV}$.
	So the minimal $j$-copy $j``\Xx_1$ of $\Xx_1$, on $\vV_2^{\vV\downarrow 1}$,
	translates to a tree $\Xx_1'$ on $\vV\downarrow 1$ which is via $\Sigma_{\vV\downarrow 1}$
	and is above $\kappa_1^{+(\vV\downarrow 1)}$. Likewise, $\Tt_1$ translates
	to a tree $\Tt_1'$
	on $\bar{\vV}\downarrow 1$ via $\Sigma_{\bar{\vV}\downarrow 1}$ which is
	above $\kappa_1^{+(\bar{\vV}\downarrow 1)}$.

 Let $\widehat{\Tt_1}=\pi``\Tt_1$.
By Lemma \ref{lem:Psi_2_pullback-con},
$\Xx_0\conc\widehat{\Tt_1}$ is via
 $\Sigma_{\vV_2}$.
Lifting with $\pi$,
	it is easy to see that
	\[\Xx_0\conc\widehat{\Tt_1}\inflatearrow_{\min}\Xx_0\conc(\Xx_1\rest\lambda) \]
	and that it suffices to see that
	\[ \Xx_0\conc\widehat{\Tt_1}\inflatearrow_{\min}\Xx_0\conc\Xx_1\]
 (see \cite[***Theorem 10.7]{fullnorm_v3} for details;
 there is  a straightforward
 correspondence between these inflations
 and those for $\Tt_0\conc\Tt_1$).

So relabelling, we may assume  $\Tt_0=\Xx_0$ and $\Pi\rest\Tt_0=\id$, so $\bar{\vV}=\vV$ and $\pi=\id$ and $j=\bar{j}$. Let $\widetilde{\Tt_0}=\widetilde{\Xx_0}$ be the short-normal tree
	leading from $\vV_2$ to $\Ult(\vV,e_1^{\vV})$. Then
	\[ \widetilde{\Tt_0}\conc \bar{j}``\Tt_1\inflatearrow_{\min}\widetilde{\Xx_0}\conc (j``\Xx_1\rest\lambda), \]
	as can be seen by lifting all relevant structures up by the extender $e^{\bar{\vV}_1}=e^{\vV_1}$ with the relevant degree ultrapowers.
	Letting $\Tt_0'=\Xx_0'$ be $\Tt_0=\Xx_0$ but as a tree on $\N_\infty$
	(and recall $\Tt_1',\Xx_1'$ were introduced above),
	it follows that
	\[ \Tt_0'\conc\Tt_1'\inflatearrow_{\min}\Xx_0'\conc(\Xx_1'\rest\lambda).\]
Since $\Sigma_{\N_\infty}$ has mic, therefore
\[ \Tt_0'\conc\Tt_1'\inflatearrow_{\min}\Xx_0'\conc\Xx_1'. \]
But the ultimate minimal tree embedding
$\Pi'$
determined by this inflation is induced naturally by $\Pi$ above, and in particular $c=[0,\eta)_{\Tt_1'}$, so $c=[0,\eta)_{\Tt_1}$, as desired.
\end{proof}

\subsection{Self-iterability of $\vV_2$}

\begin{lem}\label{lem:vV_2_def_from_M_infty|kappa_1}
	$\vV_2$ is definable over its universe from the parameter $\N_\infty|\kappa_1^{\N_\infty}$.
\end{lem}
\begin{proof}
	This is an essentially direct adaptation of the proof of Lemma \ref{lem:vV_1_def_from_M_infty|kappa_0},
	but using Lemma \ref{lem:vV_1_def_from_M_infty|kappa_0}
	at the point that Remark \ref{rem:E^M_def} was used there.
\end{proof}
Note that in Theorem \ref{tm:vV_2_is_generic_HOD_and_mantle}, we will improve
the lemma above, showing that, in fact, $\vV_2$ is definable without any parameters
over its universe.
But just using Lemma \ref{lem:vV_2_def_from_M_infty|kappa_1}
and adapting Lemma \ref{lem:vV_def_in_vV[g]_from_inseg}, we have:
\begin{lem}\label{lem:vV_2_def_in_vV_2[g]_from_inseg}
	Let $\vV$ be a non-dropping $\Sigma_{\vV_2}$-iterate of $\vV_2$.
	Let $\lambda\in\OR$  with $\lambda\geq\delta_1^{\vV_2}$ and
	$\PP\in\vV|\lambda^{+\vV}$ and $g$ be $(\Vv,\PP)$-generic.
	Then $\vV$ is definable over the universe of $\vV[g]$ from the parameter
	$x=\vV|\lambda^{+\vV}$.
\end{lem}

We will now state a key fact on the self-iterability of $\vV_2$ (and more).
As usual, we will give the proof in a special case which illustrates the main new features,
but the full proof will be handled by \cite{*-trans_add}, as it involves $*$-translation:
\begin{tm}\label{tm:vV_2_self-it} 	Let $G\sub\lambda$ be set generic over $\vV_2$,
	where $\lambda\geq\delta_1^{\vV_2}$. Let $x=\vV_2|\lambda^{+\vV_2}$. Then:
	\begin{enumerate}
		\item 	$\vV_2$ is closed under $\Sigma_{\vV_2}$ and $\Sigma_{\vV_2}\rest\vV_2$ is lightface
		definable over $\vV_2$.

		\item $\vV_2[G]$ is closed under $\Sigma_{\vV_2}$ and $\Sigma_{\vV_2}\rest(\vV_2[G])$ is definable over the universe of $\vV_2[G]$ from the parameter $x$, uniformly in $x$.

		\item 	$\vV_2$ is closed under $\Sigma_{\N_\infty}$ and $\Sigma_{\N_\infty}\rest\vV_2$
		is lightface definable over $\vV_2$. 	(Recall that $\N_\infty=\vV_2\downarrow 1$ is a
		$\Sigma_{\vV_1}$-iterate of $\vV_1$.)

		\item	 $\vV_2[G]$ is closed under $\Sigma_{\N_\infty}$ and $\Sigma_{\N_\infty}\rest(\vV_2[G])$ is definable over the universe of $\vV_2[G]$ from the parameter $x$, uniformly in $x$.
	\end{enumerate}
\end{tm}

In order to prove the theorem, we again use modified P-constructions (in general,
incorporating $*$-translation),
in the context of the following notions of P-suitability.
The full proof will rely on $*$-translation,
and so will be given in \cite{*-trans_add}. Here we will restrict our attention to dsr (defined in this context below) trees only,
for illustration purposes (but the notion of P-suitability below
does not have such a restriction). We restrict to trees in $\vV_2$ (as opposed to $\vV_2[G]$),
as this simplifies things, and we can reduce other trees to this case.

\begin{dfn}\label{dfn:vV_2_self-it_P-suitability}
	Let  $\Tt\in\vV_2$ be an iteration tree on $\vV_2$.
	Say that $\Tt$ is \emph{P-suitable for $\vV_2$} iff
	there are $\Tt_0,\Tt_1,\Tt_2,E,F,\vV,\eta,\delta,\iota$ such that:
	\begin{enumerate}
		\item $\Tt=\Tt_0\conc\Tt_1\conc\Tt_2$ is short-normal
		on $\vV_2$, according to $\Sigma_{\vV_2}$, with $0$-lower and $1$-lower
		components $\Tt_0,\Tt_1$ and upper component $\Tt_2\neq\emptyset$,
		\item $E,F\in\es^{\vV_2}$ are $\vV_2$-total and $1$-long,
		\item $\Tt_0\conc\Tt_1$ is the successor-length tree on $\vV_2$ induced by $E$,
		\item  $\vV=M^{\Tt_0\conc\Tt_1}_\infty=\Ult(\vV_2,E)$,
		\item  $\delta_1^{\vV}=\lh(E)<\lambda^{++\vV_2}<\iota=\lgcd(\vV_2|\lh(F))<\lh(F)$,
		\item $\Tt_2\in\vV_2|\iota$; note $\Tt_2$ is on $\vV$ and is above $\delta_1^{\vV}$,
		\item $\eta$ is a strong $\{\delta_0^{\vV_2},\delta_1^{\vV_2}\}$-cutpoint of $\vV_2$,
		\item $\Tt_2$   has limit length, $\lambda^{++\vV_2}\leq\eta<\delta=\delta(\Tt_2)<\iota$, $\eta$ is the largest cardinal of $\vV_2|\delta$,
		$\Tt_2$  is definable from parameters
		over $\vV_2|\delta$, and $\vV_2|\delta$ is generic over $M(\Tt_2)$ for
		the above-$\xi$ extender algebra
		of $M(\Tt_2)$ at $\delta$, for some $\xi<\delta$.
	\end{enumerate}

	Now let $\Tt\in\vV_2$ be a tree on $\N_\infty=\vV_2\downarrow 1$.
	Say that $\Tt$ is \emph{P-suitable for $\vV_2$} iff
	the conditions above hold, except that $\Tt$ is short-normal on $\N_\infty$, according to $\Sigma_{\N_\infty}$, $\Tt_0$ is the lower component of $\Tt$,
	$\Tt_1$ is based on $M^{\Tt_0}_\infty|\delta_1^{M^{\Tt_0}_\infty}$ and is above
	$\delta_0^{M^{\Tt_0}_\infty}$, $\Tt_0\conc\Tt_1$ has successor length and does not drop,
	and $\Tt_2$ is on $\N=M^{\Tt_1}_\infty$ and is above $\delta_1^{\N}$
	(so $\N=\Ult(\N_\infty,E)$, and note that $\Tt_0\conc\Tt_1$ can also be considered as a tree on $\vV_2$, with properties as above).
\end{dfn}

The corresponding P-constructions are as follows; no proper class models
show up, because we are now working up above the real Woodin cardinals.
We must now restrict our attention to dsr trees (as defined immediately below).

\begin{dfn}
	Let $\PP\in\vV_2$ and $\delta_1^{\vV_2}\leq\lambda\in\OR$ with $\PP\sub\lambda$, and $G$ be $(\vV_2,\PP)$-generic.

	Let  $\Tt\in\vV_2$, on either $\vV_2$ or $\N_\infty$, be P-suitable for $\vV_2$,
	and adopt notation as in Definition \ref{dfn:vV_2_self-it_P-suitability}.
	Say that $\Tt$ is \emph{dsr} iff $M(\Tt)$ has only 2 Woodin cardinals.

	Suppose that $\Tt$ is dsr, but $M(\Tt)$ is not a Q-structure for itself.
	Then the \emph{P-construction} $\mathscr{P}^{\vV_2|\iota}(M(\Tt))$ of
	$\vV_2|\iota$ over $M(\Tt)$ (recalling that $\iota$ is the largest cardinal
	of $\vV_2|\lh(F)$) is defined like the P-constructions
	used to compute the $\delta_1$-short tree strategy for $\vV_1$,
	noting that the iteration map $j:\vV_2|\delta_0^{\vV_2}\to M^{\Tt_0}_\infty|\delta_0^{M^{\Tt_0}_\infty}$, which is determined by $E$, is in $\vV_2|\eta$
	(and note that
	there are no 1-long extenders in $\es^{\vV_2}\rest[\delta,\iota]$).
\end{dfn}

\begin{lem}\label{lem:vV_2_P-construction_correct}
	Let $\Tt$ be dsr P-suitable for $\vV_2$,
	on either $\vV_2$ or  $\N_\infty$.
	Suppose $M(\Tt)$ is not a Q-structure for itself.
	Then the P-construction $\mathscr{P}^{\vV_2|\iota}(M(\Tt))$ reaches
	the Q-structure $Q(\Tt,b)$, where $b=\Sigma_{\vV_2}(\Tt)$ or $b=\Sigma_{\N_\infty}(\Tt)$.
\end{lem}

\begin{proof}
	We first consider P-suitable trees $\Tt$ on $\N_\infty$ in $\vV_2$.
	So adopt the notation of Definition \ref{dfn:vV_2_self-it_P-suitability} for this,
	with $\N=\Ult(\N_\infty,E)$.
	Because $\Tt_2$ is above $\delta_1^{\N}$, the Q-structure $Q=Q(\Tt_2,b)$ exists,
	where $b=\Sigma_{\N_\infty}(\Tt)$. Suppose $Q\neq M(\Tt_2)$.
	We want to see that the P-construction reaches $Q$. To verify this,
	we run a comparison analogous some earlier in the paper,
	modulo the generic at $\delta$, and after appropriate translation of long extenders.
	We need to specify the phalanxes we compare.

	On the P-construction side, we just have $M$.

	On the Q-structure side, we proceed as follows. Let $\bar{\N}$ be the
	$\delta_0^\N$-core of $\N$ (this is not $M^{\Tt_0}_\infty$,
	as $\N_\infty$ itself is not $\delta_0^{\N_\infty}$-sound). Note that $\delta_0^{\bar{\N}}=\delta_0^\N$ and $\bar{\N}$ is $\delta_0^{\bar{\N}}$-sound. Let $\bar{\N}^+$ be a generic expansion of $\bar{\N}$
	(to a premouse), via a filter which is $M$-generic (for the same forcing $\mathbb{L}^{\bar{\N}}$. So $\bar{\N}=\vV_1^{\bar{\N}^+}$. Then $\Tt_1\conc\Tt_2$ translates
	to a tree $\Tt_1^+\conc\Tt_2^+$ on $\bar{\N}^+$ which is above $\kappa_0^{+\bar{\N}^+}=\delta_0^{\bar{\N}}$. Let $Q^+=Q(\Tt_2^+,b)$. Define the phalanx
	\[ \mathfrak{Q}^+=((\bar{\N}^+,\kappa_0^{\bar{N}}+1),Q^+,\delta).\]
	Note that $\mathfrak{Q}^+$ is iterable, as it corresponds to iterating the phalanx
	\[ \mathfrak{Q}=((\bar{\N},\delta_0^{\bar{\N}}),Q,\delta).\]
	(Note here that the only extenders overlapping $\delta$ in $\es_+^Q$
	are long, since $\Tt$ is dsr.)

	We now compare $\mathfrak{Q}^+$ with $M$, above $\delta$, modulo the generic at $\delta$,
	translating extenders with critical point $\kappa_0^{\bar{\N}^+}$
	on the $\mathfrak{Q}^+$ side, and those with critical point $\kappa_0^M$ on the $M$ side,  much as before.   Much like in the proof
	of Lemma \ref{lem:general_P-correctness}, and using the $\kappa_0^{\bar{\N}^+}$-soundness of $\bar{\N}^+$
	(which is by Lemma \ref{lem:generic_premouse_iterable}), the comparison is trivial, so the P-construction reaches $Q$, as desired. Regarding the equivalence modulo the generic at $\delta$,
	although $\vV_2|\delta$ is extender algebra generic over $Q$,
	(for an extender algebra $\BB$ at $\delta$, above some $\xi$), it doesn't seem
	immediate that it is also generic for the corresponding extender algebra
	of $Q^+$ (although the extenders correspond, it seems there might
	still be further axioms in $Q^+$ which cause problems). However,
	this is not a problem. Note that
	\[ \pow({<\delta})\inter (Q[\vV_2|\delta])=\pow({<\delta})\inter (M(\Tt)[\vV_2|\delta])=
	\pow({<\delta})\inter(\vV_2|\delta). \]
	We can force over $Q[\vV_2|\delta]$ with $\mathbb{L}'=\mathbb{L}^{\vV_2}*\dot{\mathbb{L}^{\vV_1}}\in \vV_2|\delta$, adding $(\vV_1|\kappa_1^{\vV_1},M|\kappa_0^M)$, which results in $Q[M|\delta]$, and similarly
	\[ \pow({<\delta})\inter(Q[M|\delta])=\pow({<\delta})\inter (M(\Tt)[M|\delta])=\pow({<\delta})\inter(M|\delta).\]
	Since $\bar{\N}^+|\kappa_0^{\bar{\N}^+}$ was taken $M$-generic for $\mathbb{L}^{\bar{\N}}\in M(\Tt)$, we can therefore
	force further with $\mathbb{L}^{\bar{\N}}$ to reach $Q[M|\delta,\bar{\N}^+|\kappa_0^{\bar{\N}^+}]$,
	which computes $Q^+$.
	But the product $(\BB*\dot{\mathbb{L}'})\times\mathbb{L}^{\bar{\N}}$ can be reversed,
	and so $M|\delta$ is also $Q^+$-generic for $\BB*\dot{\mathbb{L}'}$.
	Moreover, $\BB$ is definable from parameters
	over $M(\Tt^+)=Q^+|\delta$. The same holds for all models that appear above $\mathfrak{Q}^+$ in the comparison. This gives the usual fine structural correspondence
	between models above $\mathfrak{Q}^+$ and their generic extensions given by adjoining $M|\delta$.
	On the $M$-side, the extension to $M[\bar{\N}^+|\kappa_0^{\bar{\N}^+}]$ is
	via $\mathbb{L}^{\bar{\N}}$, which is small relative to $\delta$ in $M$.  Likewise
	for all models on the $M$-side of the comparison. So we also get the appropriate fine structural correspondence on the $M$-side.

	Now consider trees $\Tt$ on $\vV_2$; we adopt the relevant notation from Definition \ref{dfn:vV_2_self-it_P-suitability}.
	If $\lh(E^{\Tt_2}_0)<\gamma_1^{\vV}$, then since $\rho_1^{\vV||\gamma_1^{\vV}}=\delta_1^{\vV}$, $\Tt_2$ immediately drops in model
	to $\vV||\gamma_1^{\vV}$, and this cannot be undone (since $\Tt_2$ is short-normal).
	Since $\vV|\gamma_1^{\vV}=\N||\kappa_1^{+\N}$, where $\N$ is as in the previous case,
	with corresponding iteration strategy for such trees,
	everything in this situation is as above. So suppose $\gamma_1^{\vV}<\lh(E^{\Tt_2}_0)$.
	Let $E^+\in\es^M$ with $\lh(E^+)=\lh(E)$. Let $R=\Ult(M,E^+)$. Then
	$\vV=\vV_2^R$ and (since $\gamma_1^{\vV}<\lh(E^{\Tt_2}_0)$), $\Tt_2$ translates to a tree $\Tt_2^+$ on $R$, which is above $\kappa_1^{+R}$. Let $Q^+=Q(\Tt_2^+,b)$.
	It is straightforward to see that $Q=Q(\Tt_2,b)$ has no $1$-long extenders overlapping $\delta$ (recall that $Q\neq M(\Tt_2)$, and use the smallness of $M$), so $Q^+$ can only have extenders overlapping $\delta$ with critical point $\kappa_0$.
	Define the phalanx
	\[ \mathfrak{Q}=((M,\kappa_0+1),Q^+,\delta).\]
	Clearly $\mathfrak{Q}$ is iterable. We compare $\mathfrak{Q}$ versus $M$,
	again above $\delta$ etc, like before. This time the equivalence modulo the generic
	is a little different, because $R|\kappa_1^{+R}$ is not $M$-generic, but instead
	is in $M$. However, over $Q$, we adjoin $\vV_2|\delta$ with the extender algebra,
	then adjoin $M|\kappa_1^{+M}$, reaching $Q[M|\delta]$. Like in the previous case,
	this two-step forcing iteration is definable from parameters over $Q|\delta$
	and the Woodinness of $\delta$ ensures genericity. But $R|\kappa_1^{+R}\in M|\delta$,
	so $Q^+$ is (simply) definable from parameters over $Q[M|\delta]$,
	and $Q^+$ (simply) defines $Q$ from parameters. This (together with exactly how these definitions are made and the parameters used) is enough for the fine structural analysis of the comparison.
\end{proof}

Without discussing $*$-translation, we are limited to sketching the proof of that
$\vV_2$ can iterate itself and $\N_\infty$:

\begin{proof}[Sketch of proof for Theorem \ref{tm:vV_2_self-it}]\

	Lemma \ref{lem:vV_2_iterates_N_infty_up_to_its_Woodin} handles
	trees based on $\N_\infty|\delta_1^{\N_\infty}=\vV_2|\delta_1^{\vV_2}$.

	So consider dsr trees $\Uu=\Uu_0\conc\Uu_1\conc\Uu_2$, with lower component
	$\Uu_0$, $\Uu_1$ on $M^{\Uu_0}_\infty$, based on $M^{\Uu_0}_\infty$,
	and above $\delta_0^{M^{\Uu_0}_\infty}$, $b^{\Uu_0\conc\Uu_1}$ non-dropping,
	and $\Uu_2$ on $M^{\Uu_0\conc\Uu_1}_\infty$, above $\delta_1^{M^{\Uu_0\conc\Uu_1}_\infty}$.
	Let $E\in\es^{\vV_2}$ be such that $E$ is $1$-long and $\lambda^{+\vV_2}<\iota=\lgcd(\vV_2|\lh(E))$ and $\Uu_0\conc\Uu_1\in\vV_2|\iota$. Then letting $\Tt_0\conc\Tt_1$ on $\N_\infty$ result
	from $E$, $\N=M^{\Tt_0\conc\Tt_1}_\infty=\vV_1^{\Ult(\vV_2,E)}$ is a correct iterate of $M^{\Uu_0\conc\Uu_1}_\infty$, and $\vV_2$ knows the iteration map $j$. So given
	that $\vV_2$ computes the restriction of $\Sigma_{\N}$ to above-$\delta_1^{\N}$ trees
	correctly, it can use $j$ to form minimal copies of trees $\Uu_2$ (of the form above) to correct trees $\Tt_2$, and then $\Uu_2$ is correct, because $\Sigma_{\vV_1}$ has minimal inflation condensation
	(Lemma \ref{lem:Sigma_vV_1_vshc}), and hence so does $\Sigma_{\N_\infty}$,
	by \cite[***Theorem 10.2]{fullnorm_v3}. (Note also that dsr-ness is preserved by the copying.)
	Finally, arbitrary (dsr) trees $\Tt_2$ can be reduced to P-suitable trees by the usual minimal genericity inflation technique.

	By \cite{*-trans_add}, the foregoing generalizes to arbitrary (not just dsr) trees,
	so that $\Sigma^{\vV_2}_{\N_\infty}$ is definable over $\vV_2$.
	Since $\Sigma_{\N_\infty}$ has minimal inflation condensation, $\vV_2\sats$``$\Sigma^{\vV_2}_{\N_\infty}$ has minimal inflation condensation''.

	Since the least $\vV_2$ indiscernible is countable in $V$, and $\Sigma_{\N_\infty}$
	has minimal inflation condensation in $V$,
	by \cite[***Remark 9.2]{fullnorm_v3},
$\Sigma^{\vV_2}_{\N_\infty}$
	extends to canonically to set-generic extensions $\vV_2[G]$ of $\vV_2$ (via
	the method in the proof of \cite[***Remark 9.2]{fullnorm_v3}),
	and letting $\Sigma^{\vV_2[G]}_{\N_\infty}$ be the extension,
	every tree via $\Sigma^{\vV_2[G]}_{\N_\infty}$ embeds via a minimal tree embedding arising from minimal inflation
	into some tree in $\vV_2$ via $\Sigma^{\vV_2}_{\N_\infty}$,
	and therefore $\Sigma^{\vV_2[G]}_{\N_\infty}$ also agrees
	with $\Sigma^{V[G]}_{\N_\infty}$
	if $G$ is $V$-generic.
	For the definability in $\vV_2[G]$ from the parameter $x$,
	use \ref{lem:vV_2_def_in_vV_2[g]_from_inseg}
	to recover $\vV_2$,
	from which we compute(d) the strategy.

	For trees on $\vV_2$, i.e. computing $\Sigma^{\vV_2}_{\vV_2}$ and $\Sigma^{\vV_2[G]}_{\vV_2}$,
	it is very similar,
	using Lemma \ref{lem:vV_2_Sigma_vV_1_vshc}.
\end{proof}
\subsection{The mantle and eventual generic HOD of $M$}\label{subsec:vV_2_is_mantle}

In this section we prove the main facts regarding
the eventual generic $\HOD$ and the mantle:
\begin{theorem}\label{tm:vV_2_is_generic_HOD_and_mantle} Let $U_2$ be the universe of $\vV_2$. Then:
\begin{enumerate}
	 \item\label{item:U_2=HOD^U_2[G]} $U_2=\HOD^{U_2[G]}$ for all $\vV_2$-generics $G\sub\Coll(\om,\lambda)$, for all $\lambda\geq\delta_1^{\vV_2}=\kappa_1^{+M}$,
	and likewise $U_2=\HOD^{M[H]}$ for all $M$-generics $H\sub\Coll(\om,\lambda)$.
 \item\label{item:U_2_no_proper_ground} $U_2$ has no proper ground, so $U_2$ is the mantle and smallest ground of $M$.
 \item\label{item:U_2_is_generic_mantle} $U_2$ is the mantle of all set generic extensions of $M$.
 \item\label{item:vV_2_def_over_U_2[g]}
 $\vV_2$ is definable without parameters over $U_2$,
 and in fact over any set-generic extension of $U_2$.
\end{enumerate}
\end{theorem}

\begin{proof}
Work in $\vV_2[G]$ where $G\sub\Coll(\om,\lambda)$ is $\vV_2$-generic and $\lambda\geq\delta_1^{\vV_2}$.
Say that $\vV$ is a \emph{$\lambda$-candidate} iff
 $\vV$ is a $\vV_2$-like 2-Vsp and there is $H\sub\Coll(\om,\lambda)$ which is $\vV$-generic
and $\vV[H]\ueq \vV_2[G]$.

Note that $\vV$ is determined in $\vV_2[G]$ by $\bar{\vV}=\vV|\lambda^{+\vV}$, by Lemma \ref{lem:vV_2_def_in_vV_2[g]_from_inseg}, and moreover, by the uniformity of its proof, $\bar{\vV}\mapsto\vV$ is definable over (the universe of) $\vV_2[G]$.
(We can recover the universe $U$ of $\vV$ from $\bar{\vV}$,
via (the proof of) Woodin-Laver,
and  we can recover $\vV$ from $\bar{\vV}$ and $U$ via  (the proof of) Lemma \ref{lem:vV_2_def_in_vV_2[g]_from_inseg}.)
So there are only set-many $\lambda$-candidates.
Note that $\vV_2$ is a $\lambda$-candidate.

Recall here that \emph{$\vV_2$-like} is assumed to include whatever
first-order facts satisfied by $\vV_2$ to make our arguments work.
In particular, it should include the statement/proof of Lemma \ref{lem:vV_2_def_in_vV_2[g]_from_inseg}, and also the statements
\begin{enumerate}[label=--]
	\item $\vV$ is fully iterable in every set generic extension $\vV[g]$ of $\vV$, via the strategy $\Sigma^{\vV,g}_{\vV}$ defined as in the proof of Theorem \ref{tm:vV_2_self-it}; and
	\item $\vV\downarrow 1$ is fully iterable in every set generic extension $\vV[g]$ of $\vV$, via the strategy $\Sigma^{\vV,g}_{\vV\downarrow 1}$ defined as in the proof of Theorem \ref{tm:vV_2_self-it}.
\end{enumerate}

Now using this iterability (which holds in $\vV_2[G]$ with respect to each $\lambda$-candidate $\vV$), we want to define a kind of simultaneous ``comparison'' of all $\lambda$-candidates.
For this, we will not directly attempt to compare the $\lambda$-candidates $\vV$ themselves
by least disagreement
(due to familiar problems with showing that the comparison terminates), but,
as we have done elsewhere in the paper,
instead compare generic expansions of the $\vV\downarrow 1$, and then use this to infer a
comparison of the $\lambda$-candidates (and it doesn't seem obvious that
this comparison of $\lambda$-candidates is by least disagreement).

However, we only have iterability for the generic expansions $N$ above their $\kappa_0^N$,
which isn't enough to expect a standard comparison of these premice by least disagreement either (they need not agree below their $\kappa_0^N$, as this part is just generic).
Instead, like in the proof of Lemma \ref{lem:Psi_sn_good}, we will first form a ``mutual genericity iteration'' at an appropriate
Woodin cardinal, and after this converges, move to comparison ``modulo a generic'' above
that point.

So, work in $\vV_2[G,G']$, where $G'\sub\Coll(\om,\lambda^{+\vV_2[G]})$ is $\vV_2[G]$-generic. Fix for each $\lambda$-candidate $\vV$ a generic expansion $P=P_\vV$ of $\vV\downarrow 1$. So $(\vV_1\downarrow 1)\sub P$
and  $(\vV_1\downarrow 1)=\vV_1^{P}$. (Moving to $\vV_2[G,G']$ ensures these $P$s exist.) Let $\Sigma_{P}$ be the iteration strategy
for $P$ for $0$-maximal trees $\Tt$ with $\lh(E^\Tt_0)>\kappa_0^{P}$
given by the proof of Lemma \ref{lem:Psi_sn_good} (translating to trees via $\Sigma^{\vV_2[G,G']}_{\vV}$, etc).
Let $D^P\in\es^P$ be the least $P$-total extender with $\crit(D^P)=\kappa_0^P$.
Let $\delta^P$ be the least Woodin cardinal of $P|\lambda(D^P)$ such that $\delta^P>\kappa_0^P$
(so $\delta^P>\kappa_0^{+P}$).

Recall the meas-lim extender algebra (see \cite{odle_v2}), used in the proof of Lemma \ref{lem:Psi_sn_good}. Write $\BB^P$ for the (meas-lim) extender
algebra of $P|\lh(D^P)$, at $\delta^P$,
formed with extenders  $E\in\es^P$ such that $\crit(E)>\kappa_0$
and $\nu(E)$ is a limit of measurables of $P|\lh(D^P)$, as witnessed by $\es^P$.
We will now form a mutual  genericity iteration of all $P$ as above,
for the image of $\BB^P$,
producing padded trees $\Tt_P$ on $P$, above $\kappa_0^{+P}$,
based on $P|\delta^P$ (and hence $\Tt_P$ immediately drops in model to $P|\lh(D^P)$,
noting that $\rho_1^{P|\lh(D^P)}=\kappa_0^{+P}$),
inserting some linear iterations at successor measurables to space things conveniently.
Let $\mathscr{P}$ be the set of all $P_{\vV}$, for $\lambda$-candidates $\vV$
(where ``$\lambda$-candidate'' is still as computed in $\vV[G]$, but $\mathscr{P}\sub\vV[G,G']$). Fix an enumeration $\left<P_\beta\right>_{\beta}$ of $\mathscr{P}$,
and let $C$ be a set of ordinals coding $\left<P_\beta|\delta^{P_\beta}\right>_\beta$.
We define a sequence $\left<\Tt^P_\alpha\right>_{\alpha\leq\iota}$ of approximations
to the final trees $\Tt_P=\Tt^P_\iota$. We start with $\Tt^P_0$ being the trivial tree on $P$.
Suppose we have defined $\Tt^P_\alpha$ for each $P\in\mathscr{P}$.
This will be a $0$-maximal successor-length padded tree on $P$, based on $P|\lh(D^P)$,
above $\kappa_0^{+P}$. Let $\delta_\alpha=\sup_{P\in\mathscr{P}}\delta(\Tt^P_\alpha)$,
where $\delta(\Tt)=\sup_{\beta+1<\lh(\Tt)}\lh(E^{\Tt}_\beta)$.
If $b^{\Tt^P_\alpha}$ drops below the image of $P|\lh(D^P)$
then let $\gamma^P_\alpha=\OR(M^{\Tt^P_\alpha})$,
and otherwise let $\gamma^P_\alpha=j^{\Tt^P_\alpha}(\delta^P)$,
where $j^{\Tt^P_\alpha}:P|\lh(D^P)\to M^{\Tt^P_\alpha}_\infty$ is the iteration map.
Let $K^P_\alpha=M^{\Tt^P_\alpha}||\gamma^P_\alpha$.
Let $D_\alpha=(C,C_\alpha)$ where $C_\alpha$ codes $\left<(K^{P_\beta}_\alpha)^\passive\right>_{\beta}$
as
\[ C_\alpha=\{(\beta,\gamma)\in\OR^2\bigm|\gamma\in\es(K^{P_\beta}_\alpha)\} \]
(where $\es(K^{P_\beta}_\alpha)$ is taken as a set of ordinals in a canonical fashion).
Let $G^P_\alpha$ be the least $E\in\es_+(K^P_\alpha)$
such that $E$ is $K^P_\alpha$-total and
\begin{enumerate}
	\item $E=F(K^P_\alpha)$, or
	\item\label{item:D_alpha_bad}   $\nu(E)$ is a limit of measurable cardinals of $K^P_\alpha$,
as witnessed by $\es^{K^P_\alpha}$, and $E\rest\nu(E)$ induces an extender algebra
axiom false of $D_\alpha$, or
\item\label{item:crit(E)_non-card} $\crit(E)<\sup(C)$, or $\crit(E)$ is not a cardinal in $\vV_2[G,G']$,

\end{enumerate}
if there is such an $E$, and $G^P_\alpha=\emptyset$ otherwise.
If there is $P\in\mathscr{P}$ such that  $G^P_\alpha=\emptyset$
and $b^{\Tt^P_\alpha}$ does not drop below the image of $P|\lh(D^P)$ and $\gamma^P_\alpha\leq\lh(G^{P'}_\alpha)$
for all $P'$ such that $G^{P'}_\alpha\neq\emptyset$, then we stop the construction,
and set $\iota=\alpha$, and $\Tt^P=\Tt^P_\alpha$ for all $P$.
Otherwise, let $\xi_\alpha=\min_{P\in\mathscr{P}}\lh(G^P_\alpha)$,
and set $E^P_\alpha=G^P_\alpha$ if $\lh(G^P_\alpha)=\xi_\alpha$,
and $E^P_\alpha=\emptyset$ otherwise.
Let $\gamma$ be least such that either $\gamma+1=\lh(\Tt^P_\alpha)$ or $\xi_\alpha<\lh(E^{\Tt^P_\alpha}_\gamma)$,
and set
$\Tt^P_{\alpha+1}=\Tt^P_\alpha\rest(\gamma+1)\conc\left<E^P_\alpha\right>$
(as a $0$-maximal tree, with last extender used being $E^P_\alpha$,
which might be empty).

Now suppose we have defined $\Tt^P_\alpha$ for all $\alpha<\eta$ and $P\in\mathscr{P}$,
where $\eta$ is a limit. Let $\xi=\liminf_{\alpha<\eta}\xi_\alpha$.
Then $\Tt^P_\eta$ is the natural lim inf of the sequence $\left<\Tt^P_\alpha\right>_{\alpha<\eta}$.
That is,
$E^{\Tt^P_\eta}_\gamma=E$ iff $\lh(E)<\xi$ and eventually all $\alpha<\eta$ have $E^{\Tt^P_\alpha}_\gamma=E$, and $\Tt^P_\eta$ is via $\Sigma_P$,
and has successor length. This determines $\Tt^P_\eta$.

This determines the entire construction.
The first claim is very much like in the proof of Lemma \ref{lem:Psi_sn_good}:
\begin{clmfive}\label{comp_terminates} We have:
\begin{enumerate}
	\item\label{item:Tt^P_alpha_is_0-max} Each $\Tt^P_\alpha$ is $0$-maximal on $P$,
	and if $b^{\Tt^P_\alpha}$ drops below the image
	of $P|\lh(D^P)$ then $K^P_\alpha$ is active, so $G^P_\alpha\neq\emptyset$.
	\item\label{item:termination} $\iota<\infty$.
	\item\label{item:drop_implies_active}
	Therefore there is $P'\in\mathscr{P}$ such that $b^{\Tt^{P'}_\iota}$ does not drop below the image of $P'|\lh(D^{P'})$.
\end{enumerate}
\end{clmfive}

\begin{clmfive}
	For every $P\in\mathscr{P}$, $b^{\Tt^P}$ does not drop below
	the image of $P|\lh(D^P)$,
	and $j^{\Tt^P}(\delta^P)=j^{\Tt^{P'}}(\delta^{P'})$.
	\end{clmfive}
\begin{proof}
	Because $b^{\Tt^{P'}_\iota}$ does not drop below the image of $P'|\lh(D^{P'})$,
	and
	$j^{\Tt^{P'}}(\delta^{P'})$  is a limit of measurables of $K^{P'}_\iota$,
	and $G^{P'}_\iota=\emptyset$, $j^{\Tt^{P'}}(\delta^{P'})$ must be a limit
	cardinal of $\vV_2[G,G']$. Therefore every $\Tt^{P}$ uses cofinally
	many non-empty extenders indexed below $j^{\Tt^{P'}}(\delta^{P'})=
\xi=\liminf_{\alpha<\iota}\xi_\alpha$,  $\xi$ is a limit cardinal
of $M^{\Tt^P}_\infty$, and note that  $(C,C_\iota)$ is generic
over $M(\Tt^P)$ for its extender algebra
at $\xi$, since $\xi\leq\lh(G^P_{\iota})$ if $G^P_\iota\neq\emptyset$.

Now suppose that $b^{\Tt^P}$ drops below the image of $P|\lh(D^P)$,
or $j^{\Tt^P}(\delta^P)>\xi$. Let $Q\ins M^{\Tt^P}_\infty$ be the Q-structure for $\xi$.
It is straightforward to see that $Q$ does not overlap $\xi$,
and note that we can compare $Q$ versus $U=\Ult(P',F(M^{\Tt^{P'}}_\infty))$ as premice
$Q^+$ and $U^+$
over $(M(\Tt^P),M(\Tt^{P'}))$, so by $\Mswsw$-likeness, it follows
that $Q^+\pins U^+$. Expanding $U^+$ to $U^+[C,C_\iota]$, where $\xi$ is still regular,
we can now argue like in the proof of Claim \ref{clm:gen_it_terminates} of the proof of Lemma \ref{lem:Psi_sn_good}
for a contradiction. (Although $\Tt^P\in U^+[P|\delta^P]$, we work in $U^+[C,C_\iota]$
because the reasons for the extenders used in $\Tt^P$ are encoded into $(C,C_\iota)$,
 and we need this to obtain the contradiction.)
	\end{proof}

So $M^{\Tt^P}_\infty$ is active with an image $E^P$ of $D^P$, and $E^P$ is $P$-total
with $\crit(E^P)=\kappa_0^P$. Let  $U_P=\Ult(P,E^P)$.
Then $\xi$ is a strong cutpoint of $U_P,U_{P'}$, and both extend to premice over $(U_P|\xi,U_{P'}|\xi)$.
So we can simultaneously compare all $U_P$ above $\xi$, modulo this generic equivalence.
(With a simple instance of normalization, the resulting trees can easily be rearranged as trees
on the phalanxes $\Phi(\Tt^P)$.)
This produces a final iterate $W_P$ of $P$, with $W_P=^*_\xi W_{P'}$,
 $\xi$ is Woodin in $W_P$ and $\xi<\kappa_0^{W_P}$.

\begin{clmfive} $\vV_1^{W_P}=\vV_1^{W_{P'}}$ for all $P,P'\in\mathscr{P}$.\end{clmfive}
\begin{proof}
By \S\ref{subsec:HOD_E}, $\vV_1=\vV_1^M$ depends only on the equivalence class
$\mathscr{E}$. So the corresponding fact holds for $W_P$.
But we can take $G,G'\sub\Coll(\om,{<\kappa_0})$ which are $W_P,W_{P'}$-generic respectively
with $W_P[G]\ueq W_{P'}[G']$, and then $\mathscr{E}^{W_P}=\mathscr{E}^{W_{P'}}$, which suffices.
	\end{proof}

Now let $\vV$ be a $\lambda$-candidate and $P=P_{\vV}$, so $\vV\downarrow 1=\vV_1^P$.
Note that $\Tt_P\conc\left<E_P\right>\conc\Uu_P$, after normalization,
is translatable, where $\Uu_P$ is the tree leading from $U_P$ to $W_P$. Let $\Tt_{\vV\downarrow 1}$ on ${\vV\downarrow 1}=\vV_1^P$ be its (short-normal) translation on ${\vV\downarrow 1}$.
Then $M^{\Tt_{\vV}}_\infty=\vV_1^{W_P}$, which by the previous claim is independent of $\vV$.
So the trees $\Tt_{\vV\downarrow 1}$ iterate the various $\vV\downarrow 1$
to a common model $\vV_1^*$.

Let $\Tt_{\vV\downarrow 1}=\Tt_0\conc\Tt_1$ where $\Tt_0$ is based  on $\vV|\delta_1^{\vV}=(\vV\downarrow 1)|\delta_1^{\vV\downarrow 1}$,
and $\Tt_1$ is on $M^{\Tt_0}_\infty$,
and is above $\delta_1^{M^{\Tt_0}_\infty}=\delta_1^{\vV_1^*}$. Then $\Tt_0$ translates
to a tree $\Uu_0$ on $\vV$, and $M^{\Tt_0}_\infty={M^{\Uu_0}_\infty\downarrow 1}$. Now
$\Tt_1$ essentially translates to a tree $\Uu_1$ on $M^{\Uu_0}_\infty$.
The extenders used in $\Uu_1$ are just those with indices those used in $\Tt_1$, together with 1 further $1$-long extender, which is
an image of $e_1(M^{\Uu_0}_\infty)$.  That is, let $\alpha_0=0$, and let $\alpha_1$ be least such that
$[0,\alpha_1]_{\Tt_0}$ is non-dropping
and $\kappa_1^{+M^{\Tt_0}_{\alpha_1}}<\lh(E^{\Tt_0}_{\alpha_1})$.
Then $\Uu_1\rest(\alpha_1+1)$ is a direct translation of $\Tt_1$,
though note that if it is non-trivial, it drops in model immediately
to $M^{\Uu_0}_\infty||\gamma_1^{M^{\Uu_0}_\infty}$, or some segment thereof.
Then, $\Uu_1$ uses an extra extender; if $\alpha_1=0$ then $E^{\Uu_1}_{\alpha_1}=e_1(M^{\Uu_0}_\infty)$, and otherwise
$E^{\Uu_1}_{\alpha_1}=F(M^{\Uu_0}_\infty)$ (which is an image of $e_1(M^{\Uu_0}_\infty)$).
This results in $M^{\Uu_1}_{\alpha_1+1}=\vV_2(M^{\Tt_1}_{\alpha_1})$.
After this, noting that $\Tt_1\rest[\alpha_1,\infty)$ is 1-translatable on $M^{\Tt_1}_{\alpha_1}$ (in particular, uses no $0$-long extenders), we set $\Uu_1\rest[\alpha+1,\infty)$ to be its 1-translation. Write $\Uu_{\vV}=\Uu_0\conc\Uu_1$.

So we end up with $M^{\Uu_{\vV}}_\infty=\vV_2(M^{\Tt_{\vV}}_\infty)$,
but $M^{\Tt_{\vV}}_\infty$ was independent of  $\vV$. So write $\vV_2^*$ for this common iterate of the $\lambda$-candidates $\vV$.

Using the strategies $\Sigma_P$, we can define $\Uu_{\vV}$ uniformly in $\vV$.
So let $\Gamma$ be the proper class of all ordinals fixed by all the iteration maps
$i^{\Uu_{\vV}}$. Let $\bar{\vV}_2=\cHull_1^{\vV_2^*}(\Gamma)$
and $\pi:\bar{\vV}_2\to\vV_2^*$ the uncollapse map.
Since $\rg(i_{\Uu_{\vV_2}})\sub\rg(\pi)$, this determines an elementary
$\pi:\bar{\vV}_2\to\vV_2$ by factoring. But $\vV_2$ is a set-ground of $\vV_2[G,G']$,
so by \cite{gen_kunen_incon}, $\bar{\vV}_2=\vV_2$ and $\pi=\id$.

So we have defined $\vV_2$ over the universe of $\vV_2[G,G']$ from the parameter $\lambda$,
and so by homogeneity, in fact over the universe of $\vV_2[G]$ from $\lambda$.
The uniformity then gives that we can define $\vV_2$ over the universe of any set-generic extension of $\vV_2$, from no parameter.

We can now easily complete the proof of the theorem.
Part \ref{item:vV_2_def_over_U_2[g]} was just established above.
Part \ref{item:U_2=HOD^U_2[G]}:
Let  $U_2$ be the universe of $\vV_2$ and $G\sub\Coll(\om,\lambda)$ be $\vV_2$-generic,
where $\lambda\geq\delta_1^{\vV_2}$. It now easily follows that $U_2\sub\HOD^{U_2[G]}$,
so actually $U_2=\HOD^{U_2[G]}$.  And if $H$ is $\Coll(\om,\lambda)$-generic over $M$,
then since $M$ is an $\mathbb{L}^{\vV_2}$-extension of $\vV_2$,
it follows that $U_2=\HOD^{M[H]}$.
Part \ref{item:U_2_no_proper_ground}:
If $W\sub U_2$ is a ground of $U_2$,
then we get $\HOD^{U_2[G]}\sub W_2$ if $\lambda$ is sufficiently large,
so $U_2\sub W$, so $U_2=W$.
Part \ref{item:U_2_is_generic_mantle}: The fact that $U_2$ is the mantle of all set-generic extensions of $U_2$ now follows from the set-directedness of set-grounds.
\end{proof}

\begin{cor}\label{cor:HOD^M[G],kappa_1-mantle} We have:
	\begin{enumerate}
		\item The $\kappa_1$-mantle $\mathscr{M}^M_{\kappa_1}$ of $M$ is the universe of $\vV_2$, so $\kappa_1$ is the least ordinal with this property.
\item If
	$G\sub\Coll(\om,{<\kappa_1})$ is $M$-generic
	then  $\HOD^{M[G]}$ is the universe of $\vV_2$.
	\end{enumerate}
\end{cor}
\begin{proof}
	Note that $\vV_2\sub\HOD^{M[G]}\sub\mathscr{M}^M_{\kappa_1}$, using Theorem \ref{tm:vV_2_is_generic_HOD_and_mantle}.  So we just need to see that $\mathscr{M}^M_{\kappa_1}\sub\vV_2$. Let $X$ be a set
	of ordinals in $\mathscr{M}^M_{\kappa_1}$.
	Let $E\in\es^M$ be $M$-total with $\crit(E)=\kappa_1$, and $U=\Ult(M,E)$.
	Then $j(X)\in\mathscr{M}^U_{j(\kappa_1)}$ and $j\rest\sup X\in\vV_2$,
	so it suffices to find a ${<j(\kappa_1)}$-ground  $W$ of $U$ with $W\sub\vV_2$.
	But this can be done like in the proof of Theorem \ref{prop:vV_1_is_<kappa_0-mantle},
	or as follows:\footnote{The method used here was actually
		the method used in the original proof of Theorem \ref{prop:vV_1_is_<kappa_0-mantle}.}  Let $W$ be the result of the P-construction of $U$
	over $\vV_2|\gamma_1^{\vV_2}$ (in the style of that used to construct $\vV_2$). Note that $(\vV_1|\kappa_1,M|\kappa_0)$ is generic over $W$
	for the two-step forcing iteration given by $\mathbb{L}^{\vV_2}$ followed by $\mathbb{L}^{\vV_1}$, and $W[\vV_1|\kappa_1,M|\kappa_0]\ueq U$. So $W$ is a ${<j(\kappa_1)}$-ground of $U$. But $W\sub\vV_2$, since $W||\lh(E)=\vV_2|\lh(E)$,
and  we can inductively compute the extender sequence of $W$ above $\lh(E)$
	using $E\rest\vV_2$ and $\Ult(\vV_2,E\rest\vV_2)=\vV_2^{U}$.
	\end{proof}

\section*{Acknowledgements}

Schindler gratefully acknowledges support by
the DFG grant SCHI 484/8-1, ``Die Geologie Innerer Modelle.''
Schindler and Schlutzenberg gratefully acknowledge partial support by the Deutsche Forschungsgemeinschaft (DFG, German Research Foundation) under Germany's Excellence Strategy EXC 2044\-390685587, Mathematics M\"unster: Dynamics–Geometry–Structure.
Schlutzenberg
teilweise gef\"ordert durch die Deutsche Forschungsgemeinschaft (DFG) -- Projektnummer 445387776.
Schlutzenberg partly supported by the Deutsche Forschungsgemeinschaft (DFG, German Research Foundation) -- project number 445387776.

\bibliographystyle{plain}
\bibliography{../bibliography/bibliography}

\end{document}